\documentclass[10pt,reqno]{amsart}

\makeatletter

\renewcommand{\tocsection}[3]{%
 \indentlabel{\@ifnotempty{#2}{\bfseries\ignorespaces#1 #2\quad}}\bfseries#3}
\renewcommand{\tocsubsection}[3]{%
  \indentlabel{\@ifnotempty{#2}{\ignorespaces#1 #2\quad}}#3}

\newcommand\@dotsep{4.5}
\def\@tocline#1#2#3#4#5#6#7{\relax
  \ifnum #1>\c@tocdepth 
  \else
    \par \addpenalty\@secpenalty\addvspace{#2}%
    \begingroup \hyphenpenalty\@M
    \@ifempty{#4}{%
      \@tempdima\csname r@tocindent\number#1\endcsname\relax
    }{%
      \@tempdima#4\relax
    }%
    \parindent\z@ \leftskip#3\relax \advance\leftskip\@tempdima\relax
    \rightskip\@pnumwidth plus1em \parfillskip-\@pnumwidth
    #5\leavevmode\hskip-\@tempdima{#6}\nobreak
    \leaders\hbox{$\m@th\mkern \@dotsep mu\hbox{.}\mkern \@dotsep mu$}\hfill
    \nobreak
    \hbox to\@pnumwidth{\@tocpagenum{\ifnum#1=1\bfseries\fi#7}}\par
    \nobreak
    \endgroup
  \fi}
\AtBeginDocument{%
\expandafter\renewcommand\csname r@tocindent0\endcsname{0pt}
}
\def\l@subsection{\@tocline{2}{0pt}{2.5pc}{5pc}{}}
\makeatother

\usepackage{lipsum}

\usepackage{tikz}
\usetikzlibrary{decorations.pathreplacing}
\usepackage{amsfonts,latexsym,rawfonts,amsmath,amssymb,amsthm, mathrsfs, lscape}

\usepackage[all]{xy}
\usepackage{graphicx}
\usepackage{dsfont}
\usepackage{bm}

\usepackage{array, tabularx}

\usepackage{setspace}
\usepackage[english]{babel}
\usepackage{color}
\usepackage{enumerate}

\usepackage{hyperref}

\usepackage{anysize}
\marginsize{2.5cm}{2.5cm}{2cm}{2.5cm}

\newtheorem{theorem}{Theorem}[section]

\newtheorem{corollary}{Corollary}[theorem]
\newtheorem{definition}[theorem]{Definition}
\newtheorem{example}[theorem]{Example}
\newtheorem{lemma}[theorem]{Lemma}
\newtheorem{problem}[theorem]{Problem}
\newtheorem{conjecture}[theorem]{Conjecture}
\newtheorem{notation}[theorem]{Notation}
\newtheorem{proposition}[theorem]{Proposition}
\newtheorem{remark}[theorem]{Remark}

\numberwithin{equation}{section}
\numberwithin{figure}{section}

\def \P{\mathbb{P}}

\newcommand{\I}{\bf{I}}
\newcommand{\II}{\bf{II}}
\newcommand{\III}{\bf{III}}
\newcommand{\IV}{\bf{IV}}

\newcommand{\bI}{\mathbb{I}}
\newcommand{\bo}{\bm{0}}

\newcommand{\bp}{\bar{\partial}}
\newcommand{\bv}{\bm{v}}
\newcommand{\bx}{\bm{x}}
\newcommand{\by}{\bm{y}}

\newcommand{\C}{\mathbb C}
\newcommand{\Ca}{\mathcal{C}}

\newcommand{\cN}{\mathcal N}
\newcommand{\cO}{\mathcal{O}}
\newcommand{\cV}{\mathcal V}
\newcommand{\cX}{X}
\newcommand{\cP}{\mathcal P}

\newcommand{\dC}{\mathbb{C}}
\newcommand{\DelH}{\Delta}

\newcommand{\dN}{\mathbb{N}}
\newcommand{\dR}{\mathbb{R}}
\newcommand{\dT}{\mathbb{T}}
\newcommand{\dZ}{\mathbb{Z}}
\newcommand{\Err}{\mathrm{Err}_{CY}}
\newcommand{\Errt}{\mathrm{Err}_{t}}

\newcommand{\fr}{\mathfrak{r}}
\newcommand{\fs}{\mathfrak{s}}
\newcommand{\fS}{\mathfrak{S}}
\newcommand{\ft}{\mathfrak{t}}

\newcommand{\hX}{\widehat{\mathcal X}}

\newcommand{\M}{\mathcal M}

\newcommand{\nmv}{\underline{\nu}}
\newcommand{\p}{\partial}
\newcommand{\R}{\mathbb R}
\newcommand{\rmi}{\tilde{r}_-}

\newcommand{\tO}{\widetilde{O}}
\newcommand{\sq}{\sqrt{-1}}

\newcommand{\U}{\mathcal U}
\newcommand{\vf}{\varphi}
\newcommand{\vr}{\varrho}

\newcommand{\wI}{C^{1, \alpha}_{\delta, \nu,\mu}}
\newcommand{\wII}{C^{2, \alpha}_{\delta, \nu,\mu}}
\newcommand{\wIII}{C^{3, \alpha}_{\delta, \nu,\mu}}

\newcommand{\y}{\tilde y}
\newcommand{\Z}{\mathbb Z}
\newcommand{\e}{\bm{e}}

\DeclareMathOperator{\Area}{Area}

\DeclareMathOperator{\diam}{Diam}
\DeclareMathOperator{\dvol}{dvol}
\DeclareMathOperator{\Exp}{Exp}

\DeclareMathOperator{\FT}{\mathfrak{T}}
\DeclareMathOperator{\IIs}{II}

\DeclareMathOperator{\Res}{Res}
\DeclareMathOperator{\Ric}{Ric}
\DeclareMathOperator{\Rm}{Rm}

\DeclareMathOperator{\Tr}{Tr}

\DeclareMathOperator{\Vol}{Vol}

\begin{document}

\title{Complex structure degenerations and collapsing of Calabi-Yau metrics} 
\author{Song Sun}
\address{Institute for Advanced Study in Mathematics, Zhejiang University, Hangzhou 310058, China}
\address{
Department of Mathematics, University of California, Berkeley, CA 94720, U.S.A.} 
\email{songsun@zju.edu.cn}

\author{Ruobing Zhang}
 \thanks{The first author was supported by NSF Grant DMS-1708420, an Alfred P. Sloan Fellowship, and the Simons Collaboration Grant on Special Holonomy in Geometry, Analysis and Physics. \\ The second author is supported by NSF Grants DMS-1906265 and DMS-2304818.}

\address{Department of Mathematics, University of Wisconsin-Madison, Madison, WI, 53706}
\email{rzhang573@wisc.edu}

\begin{abstract} In this paper, we make progress on understanding the collapsing behavior of Calabi-Yau metrics on a degenerating family of polarized Calabi-Yau manifolds. In the case of a family of smooth Calabi-Yau hypersurfaces in projective space degenerating into the transversal union of two smooth Fano hypersurfaces in a generic way,  we obtain a complete result in all dimensions establishing explicit and precise relationships between the metric collapsing and complex structure degenerations. This result is  new even in complex dimension two.  This is achieved via gluing and singular perturbation techniques, and a key geometric ingredient involving the construction of certain (not necessarily smooth) K\"ahler metrics with torus symmetry. We also discuss possible extensions of this result to more general settings. 
\end{abstract}

\maketitle{}

\tableofcontents

\section{Introduction}
\label{s:introduction}

\subsection{Background and main results}
\label{ss:1.1}
Let $n\geq 2$ be a positive integer and $\Delta$ be the unit disc in $\C$. Let  $p: (\mathcal{X}, \mathcal L)\rightarrow \Delta$ be a flat polarized degenerating family of $n$-dimensional Calabi-Yau varieties. More precisely, we assume that $\mathcal X$ is normal with $X_t\equiv p^{-1}(t)$ smooth for $t\neq 0$ and $X_0$ singular,  the relatively canonical line bundle $K_{\mathcal X/\Delta}$ is trivial, and $\mathcal L$ is relatively ample. Yau's proof  of the Calabi conjecture \cite{Yau}  yields  for each $t\neq 0$ a unique smooth Ricci-flat K\"ahler metric (\emph{the Calabi-Yau metric}) $\omega_{CY,t}$ on $X_t$ in the cohomology class $2\pi c_1(\mathcal L|_{X_t})$. The following is a folklore problem:    
\begin{problem}
\label{pb:1-1} What is the limiting geometric behavior of $(X_t,\omega_{CY,t})$ in the Gromov-Hausdorff sense as $|t|\to0$, and how to describe its connection with the algebraic geometry associated to this degeneration?
\end{problem}
 A particularly intriguing and challenging situation is when \emph{collapsing} occurs, or equivalently, when  the diameters of $(X_t,\omega_{CY, t})$ tend to infinity. In this case if we rescale the diameters to be fixed then the Gromov-Hausdorff limit is a lower dimensional space.   In the special case of large complex structure limits, the above problem is related to  the limiting version of the SYZ Conjecture, as proposed by Gross-Wilson \cite{GW} and Kontsevich-Soibelman \cite{KonSo}. Based on the main results in this paper (Theorem \ref{t:main-theorem}), we will propose a more general conjecture (Conjecture \ref{cj:generalized-SYZ}) for general degenerations.  

The goal in this paper is to give an answer to Problem \ref{pb:1-1} for a
 special class of complex structure degenerations. More precisely, we give a complete description of the collapsing geometry of the family of Calabi-Yau metrics {along the degeneration of complex structures}. We mainly focus on the special class of examples below, but the crucial techniques involved apply to more abstract situations, and the strategy can be extended to deal with more general classes of collapsing (see the discussions in  Section \ref{s:discussions}).

Let $f_1, f_2, f$ be homogeneous polynomials in $n+2$ variables of degree $d_1$, $d_2$ and 
$d_1+d_2=n+2$ respectively. Let $\mathcal X\subset \C\P^{n+1}\times \Delta$ be the family of Calabi-Yau hypersurfaces $X_t$ in $\C\P^{n+1}$ defined by the equation $F_t(x)=0$ (see Figure \ref{f: the degeneration}), where 
 \begin{equation} \label{eqn1.1}
 F_t(x)\equiv f_1(x)f_2(x)+tf(x),\quad x\in\dC\P^{n+1},
 \end{equation}
and $t$ is a complex parameter on the unit disc $\Delta\subset \C$. The projection map $p: \mathcal X\rightarrow \Delta$ is the natural one. The relative ample line bundle $\mathcal L$ comes from  restriction of the natural $\cO(1)$ bundle over $\C\P^{n+1}$. 
We further assume that $f_1, f_2, f$ are sufficiently general so that the following hold:
\begin{enumerate}[(i)]
\item   $Y_1=\{f_1(x)=0\}$ and $Y_2=\{f_2(x)=0\}$ are smooth hypersurfaces in $\C\P^{n+1}$; 
\item $X_t$ is smooth for $t\neq 0$ and $|t|\ll1$;
\item $D=\{f_1(x)=f_2(x)=0\}$ is a smooth complete intersection in $\C\P^{n+1}$;
\item $H=\{f_1(x)=f_2(x)=f(x)=0\}$ is a smooth complete intersection in $\C\P^{n+1}$. 
\end{enumerate}
Then $Y_1 \cap Y_2 = D$ and $H \subset D$. Naturally we may view $Y_1, Y_2$ and $D$ as subvarieties in $X_0\subset \mathcal X$, and $X_0=Y_1\cup Y_2$,  as illustrated in Figure \ref{fig1.1}. By adjunction formula, $D$ is also Calabi-Yau, and sits as an  anti-canonical divisor in both $Y_1$ and $Y_2$.  Moreover, for $i\in\{1,2\}$, the normal bundle $L_i$ of $D$ in $Y_i$ is given by $\cO(d_{3-i})|_D$.
Notice the total space $\mathcal X$ has singularities  along $H\times \{0\}$, and transverse to $H\times \{0\}$ the singularities are locally modeled on a three dimensional ordinary double point. The dual intersection complex of the singular fiber $X_0$  is a one dimensional interval.   

For $i\in\{1, 2\}$, it has been shown by Tian-Yau \cite{TY} that $Y_i\setminus D$ admits a complete Ricci-flat K\"ahler metric $\omega_{TY, i}$ with interesting asymptotics governed by the \emph{Calabi model space} (c.f. Section \ref{ss:Calabi model space}) , and the latter in turn depends on the Calabi-Yau metric on $D$ in the cohomology class $2\pi \cO(1)|_D$.   Notice that the construction of Tian-Yau can be viewed as a generalization of Yau's proof of the Calabi conjecture to the non-compact case, but it remains an interesting question to understand in what sense the metrics $\omega_{TY, i}$ are uniquely or canonically associated to the pair $(Y_i, D)$. 

\begin{figure}
\label{fig1.1}
\begin{tikzpicture}[scale=0.7]
\draw (-5, 3) to [out=-75, in=75] (-5, -3); 
\draw (-5, 3) to (-3.5, 4); 
\draw (-5, -3) to  (-3.5, -2); 
\draw (-3.5, 4) to [out=-75, in=75]  (-3.5, -2); 
\draw (-3, 3) to [out=-75, in=75]  (-3, -3); 
\draw (-3, 3) to (-1.5, 4); 
\draw (-3, -3) to (-1.5, -2); 
\draw (-1.5, 4) to [out=-75, in=75]  (-1.5, -2); 
\draw (3.0, 3) to [out=-75, in=75] (3.0, -3); 
\draw (3.0, 3) to (4.5, 4); 
\draw (3, -3) to  (4.5, -2); 
\draw (4.5, 4) to [out=-75, in=75]  (4.5, -2); 
\draw (1.0, 3) to [out=-75, in=75]  (1.0, -3); 
\draw (1.0, 3) to (2.5, 4); 
\draw (1.0, -3) to (2.5, -2); 
\draw (2.5, 4) to [out=-75, in=75]  (2.5, -2); 
\draw (-1, 3) to (0.5, 4); 
\draw (-1, 3) to [out=-65, in=80] (-0.7, -0.5);

\draw (-0.7, -0.5) to [out=-70, in=65] (-1, -3);

\draw (-1, -3) to (0.5, -2); 
\draw (0.5, 4) to  [out=-65, in=80]  (0.8, 0.5);
\draw  (0.8, 0.5) to [out=-70, in=65] (0.5, -2);
\draw[blue, very thick] (-0.7, -0.5) to (0.8, 0.5);
\node[red] at (0.05, 0) {$\bullet$};
\node[red] at (-1.75, 0.1) {$\bullet$};
\node[red] at (-3.75, 0.4) {$\bullet$};
\node[red] at (2.25, 0.2) {$\bullet$};
\node[red] at (4.25, 0.6) {$\bullet$};

\node at (-4, 2.5) {$X_t$};
\node at (0, 2.5) {$Y_1$}; 
\node at (0, -1.5) {$Y_2$}; 
\node[blue] at (0.3, 0.7) {$D$};

\node[red] at (-6, 1.4) {$H\times\{t\}$};
\draw[->] (-5.2, 1.3) to (-3.9, 0.5);
\node at (0, -4.5) {$X_0=Y_1\cup_D Y_2$};
\draw[->] (0, -4) to (0, -2.5);
\draw[thick, red, dashed] plot[smooth] coordinates {(-3.75, 0.4) (-1.75, 0.1) (0.05, 0) (2.25, 0.2) (4.25, 0.6)};
\end{tikzpicture}
\caption{The algebraic family $\mathcal X$}
 \label{f: the degeneration}
\end{figure}
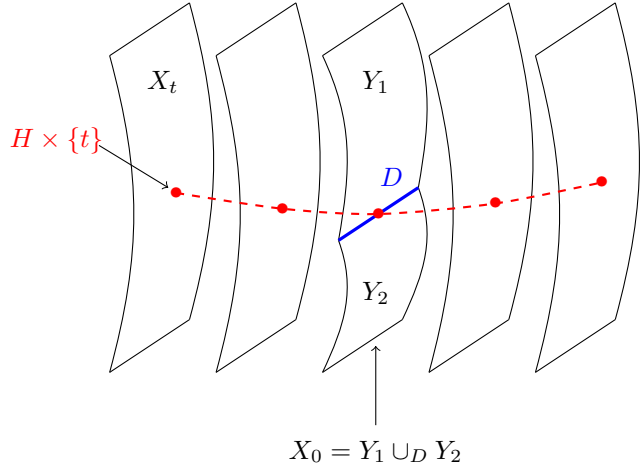

The main result of this paper provides an answer to Problem \ref{pb:1-1} for the above family.  Roughly speaking, we prove that for $|t|$ sufficiently small, the Calabi-Yau metrics $\omega_{CY, t}$ on $X_t$ can be constructed by gluing the Tian-Yau metrics on $Y_1\setminus D$ and $Y_2\setminus D$, together with an approximate Calabi-Yau metric on a neck region.  Let us denote  the {\it  renormalized metric} and  the {\it renormalized measure} as follows,
\begin{align}
\tilde\omega_{CY,t} \equiv  (\diam_{\omega_{CY, t}}(X_t))^{-2}\cdot \omega_{CY, t}
,\quad d\nmv_t\equiv (\Vol_{\omega_{CY, t}}(X_t))^{-1}\cdot\dvol_{\omega_{CY, t}}, 
\end{align}
so that $\tilde\omega_{CY, t}$ has unit diameter and $d\nmv_t$ is a probability measure. 
The more precise version of the main result is as follows (see Figure \ref{f:glued-manifold} for a geometric description).
\begin{theorem}
\label{t:main-theorem} 
The following statements hold:

\begin{enumerate}
\item (Diameter estimate) There exists a constant $C>0$ such that for $0<|t|\ll1$ we have 
\begin{equation}
C^{-1}\cdot  (\log|t|^{-1})^{\frac{1}{2}}\leq \diam_{\omega_{CY, t}}(X_t)\leq C \cdot (\log|t|^{-1})^{\frac{1}{2}}.
\end{equation}

 \item (Convergence) As $|t|\rightarrow 0$, the spaces $(X_{t}, \tilde\omega_{CY,t}, d\nmv_t)$ converge, in the measured Gromov-Hausdorff sense,  to $(\bI, dx^2, d\nmv_{\bI})$, where $\bI= [0,1]\subset\dR$ is a unit interval endowed with the standard metric $dx^2$ and the singular measure has an explicit expression    
   \begin{align}
d\nmv_{\bI} &=C_{n,d_1,d_2}\cdot \mathscr{V}_{\bI}(x)dx,\\ \mathscr{V}_{\bI}(x)&=
\begin{cases}
(\frac{x}{d_1})^{\frac{n-1}{n+1}}, & x\in[0,\frac{d_1}{d_1+d_2}],
\\
(\frac{1-x}{d_2})^{\frac{n-1}{n+1}}, &  x \in[\frac{d_1}{d_1+d_2}, 1],
\end{cases}\end{align}
for some constant $C_{n,d_1,d_2}>0$ depending only on $n$, $d_1$ and $d_2$.

 \item (Singular fibration) For $0<|t|\ll 1$, there is a continuous surjective map $\mathcal{F}_t: X_t\rightarrow\bI$ with the following properties: 
 \begin{enumerate}[(a)]
 \item (Almost distance preserving)  For all $p,q\in X_t$,
 \begin{equation}\Big||\mathcal{F}_t(p)-\mathcal{F}_t(q)|-d_{\tilde\omega_{CY,t}}(p,q)\Big|<\tau(t),\quad \lim\limits_{t\to0}\tau(t)=0.\end{equation} \item (Regular fiber) For each $x\in (0,1)\setminus\{\frac{d_1}{d_1 + d_2}\}$,  $\mathcal{F}_t^{-1}(x)$ is an $S^1$-bundle over $D$ with the first Chern class \begin{align}
 c_1(\mathcal{F}_t^{-1}(x))
=
\begin{cases}
c_1(\cO(d_2)|_D), & x\in (0,\frac{d_1}{d_1+d_2}),
\\
c_1(\cO(-d_1)|_D), & x\in  (\frac{d_1}{d_1+d_2},1).
\end{cases}
 \end{align}

 \item (Singular fiber and deepest bubble) The fiber $\mathcal{F}_t^{-1}(\frac{d_1}{d_1 + d_2})$ is a singular $S^1$-fibration over $D$ with vanishing circles along $H\subset D$. Suitable rescalings around the vanishing circles on $\mathcal{F}_t^{-1}(\frac{d_1}{d_1 + d_2})$ converge to the Riemannian product $\C_{TN}^2\times\C^{n-2}_{flat}$, where $\dC_{TN}^2$ 
is the Taub-NUT space and $\C^{n-2}_{flat}$ is the Euclidean space. 
 \item (End bubble) Suitable rescalings around the  ends $x=0$ and $x=1$ converge to the complete  Tian-Yau metrics $\omega_{TY, 1}$ and $\omega_{TY, 2}$ on $Y_1\setminus D$ and $Y_2\setminus D$ respectively. 

\end{enumerate}

\end{enumerate}
\end{theorem}
\begin{figure}
\begin{tikzpicture}[scale=0.8]
\draw (0,0) to [out = 90, in = 180] (1,.5) to (2,.5);
\draw (2,.5) to [out = 0, in = 235] (4,2);
\draw (4,2) to [out = 55, in = 180] (6,3);
\draw (6,3) to [out = 0, in = 125] (8,2);
\draw (8,2) to [out = -55, in = 180] (10,.5);
\draw (10,.5) to (11,.5);
\draw (11,.5) to [out = 0, in = 90] (12,0);
\draw (0,0) to [out = 270, in = 180] (1,-.5) to (2,-.5);
\draw (2,-.5) to [out = 0, in = 125] (4,-2);
\draw (4,-2) to [out = -55, in = 180] (6,-3);
\draw (6,-3) to [out = 0, in = -125] (8,-2);
\draw (8,-2) to [out = 55, in = 180] (10,-.5);
\draw (10,-.5) to (11,-.5);
\draw (11,-.5) to [out = 0, in = -90] (12,0);
\draw[blue](2,0) ellipse (.2 and .5);
\draw[red](2,0) ellipse (.2 and .05);

\draw[blue] (3,0) ellipse (.1 and .82);

\draw[red](3,0) ellipse (.1 and .05);
\draw[blue](4,0) ellipse (.05 and 2);

\draw[red ](4,0) ellipse (.05 and .02);
\draw[blue](10,0) ellipse (.2 and .5);
\draw[red](10,0) ellipse (.2 and .05);

\draw[blue](9,0) ellipse (.1 and .82);

\draw[red](9,0) ellipse (.1 and .05);
\draw[blue](8,0) ellipse (.05 and 2);

\draw[red ](8,0) ellipse (.05 and .02);

\draw[red](5,0) ellipse (.03 and .02);
\draw[blue](5,0) ellipse (.03 and 2.82);
\draw[red](7,0) ellipse (.03 and .02);
\draw[blue](7,0) ellipse (.03 and 2.82);

\draw (1.2,-.05) arc [ radius = .5, start angle = 45, end angle = 135];

\draw (11.5,-.05) arc [ radius = .5, start angle = -45, end angle = -135];
\draw[very thick] (6,-1.5) -- (6, 1.5);
\node at (6.25, 0) {$H$};

\draw (0,-5) -- (12,-5);
\draw (0,-5.1)   -- (0,-4.9);
\draw (12,-5.1)   -- (12,-4.9);
\draw (6,-5.1)   -- (6,-4.9);
\draw[->, very thin] (0,-.3) -- (0,-4.5);
\draw[->, very thin] (1.25,-.6) -- (.2,-4.5);
\draw[->, very thin] (3,-.9) -- (3,-4.5);
\draw[->, very thin] (6,-3.1) -- (6,-4.5);
\draw[->, very thin] (5,-2.9) -- (5.8,-4.5);
\draw[->, very thin] (12,-.3) -- (12,-4.5);
\draw[->, very thin] (10.75,-.6) -- (11.8,-4.5);
\draw[->, very thin] (9,-.9) -- (9,-4.5);
\draw[->, very thin] (6,-3.1) -- (6,-4.5);
\draw[->, very thin] (7,-2.9) -- (6.2,-4.5);
\draw[->, thick] (-2,1.5) -- (-2,-4.5);
\node at (-1.5,-1.5) {$\mathcal{F}_t$};

\node at (1,.8) {$Y_1\setminus D$};
\node at (11,.8) {$Y_2\setminus D$};
\node at (6,3.3) {$\M$};
\node at (0,-5.5) {$x=0$};
\node at (6,-5.5) {$x=\frac{d_1}{d_1+d_2}$};
\node at (-2,2.0) {$X_t$};
\node at (-2,-5.0) {$\bI$};
\node at (12,-5.5) {$x=1$};
\end{tikzpicture}
\caption{The collapsing Calabi-Yau metric $\omega_{CY,t}$ on $X_t$}
 \label{f:glued-manifold}
\end{figure}
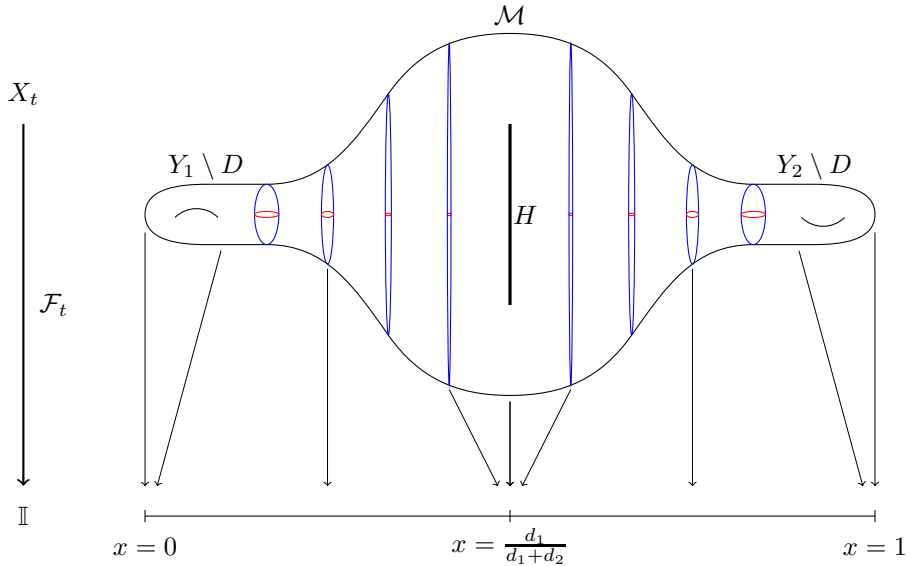
\begin{remark}
A consequence of the theorem is that the Tian-Yau metrics $\omega_{TY, i}$ on $Y_i\setminus D$, though not a priori canonical by construction, is indeed canonically associated to the above degeneration family of compact Calabi-Yau manifolds.   
\end{remark}

\begin{remark}
Using the gluing construction in this paper, we obtain a fairly precise description of the multi-scale collapsing of the Calabi-Yau metrics $(X_t, \tilde\omega_{CY, t})$ as $t\rightarrow 0$. For each $x\in (0, 1)\subset \bI$,  not only does $\mathcal{F}^{-1}_{t}(x)$ collapse as $t\rightarrow0$, each $S^1$-fiber of $\mathcal{F}^{-1}_{t}(x)$ collapses also, but at a faster rate.
This iterated collapsing in effect gives rise to various bubbles of different geometric properties. 
See Section \ref{ss:regularity-scales} for detailed studies on the rescaled limit geometries. \end{remark}

\begin{remark}
 Transverse to the divisor $H\subset D$, the singular fibration $\mathcal{F}_t$ is topologically modeled on the composition of the Hopf fibration
 \begin{equation}\C^2\rightarrow \C\oplus \R\simeq \R^3, \quad  (z_1, z_2)\mapsto \Big(z_1z_2, \frac{1}{2}(|z_1|^2-|z_2|^2)\Big), \end{equation}
and the projection map 
\begin{equation}\C\oplus \R\rightarrow \R,\quad  (y, z)\mapsto z.\end{equation} 
\end{remark}

\begin{remark} 
In Theorem \ref{t:main-theorem},
the singularities of $d\nmv_{\bI}$ coincide with the singular loci of the map $\mathcal F_t$.  The explicit formula of $d\underline\nu_{\bI}$ has applications in understanding the properties of general collapsing limits with lower Ricci curvature bounds. See Section \ref{ss:renormalized-measure} for more discussions. 
\end{remark}

It is worth emphasizing that, in the volume collapsing case, Theorem \ref{t:main-theorem} provides the first {\it explicit and precise} descriptions on the metric limiting behavior and singularity formation along an algebraically degenerating family, beyond the easy case of abelian varieties.
Prior to our complete answer to Problem \ref{pb:1-1} for the special family obtained in this paper, an initial progress was established in \cite{HSVZ} which focused on a gluing construction of  hyperk\"ahler metrics for $n=2$.  However, \cite{HSVZ} only described the construction at the level of {\it Riemannian metrics} without studying the relationships with the complex structure degenerations. So both the results in Theorem \ref{t:main-theorem} and the ideas in the proof are essentially new even in the case of complex dimension two.

In this paper, the formulation and the proof of Theorem \ref{t:main-theorem} naturally require us to directly work on a fixed algebraic degeneration family and perform the metric construction at the level of K\"ahler potentials. This framework puts us into a much more rigid situation and leads to a primary challenge, which demands a number of new ideas and techniques.

A more challenging part of this paper, compared to the 2 dimensional case studied in \cite{HSVZ},  lies in the  construction of the neck region which are significantly harder. This involves studying the reduction of the Calabi-Yau equation with an $S^1$-symmetry, and in our setting we need to allow the $S^1$-action to have fixed points, which  means we need to study the solutions of the reduced equation with singularities which generically are of real codimension 3. 
 When $n=2$, this reduction gives rise to a \emph{linear} PDE (Gibbons-Hawksing ansatz), and the singularities are \emph{isolated}. When $n>2$, the reduced equation is a PDE, however this time \emph{non-linear}, an equation we call the \emph{non-linear Gibbons-Hawking ansatz}. Moreover,  we need to study the solutions with non-isolated singularities, which involves new technical ingredients (see Sections \ref{s:Greens-currents}-\ref{s:neck}).    
A fundamental strategy is to exploit the adiabatic feature of the problem, i.e., the smallness of the $S^1$-orbits, to construct \emph{approximate} solutions to the nonlinear Gibbons-Hawking ansatz via \emph{linearization}.

\subsection{Outline of the proof of the main theorem}

The proof of Theorem \ref{t:main-theorem} consists of four parts.

The first part involves an algebraic modification of the family $\mathcal X$ (see Section \ref{ss:algebraic-geometry}). The overarching strategy in the proof of Theorem \ref{t:main-theorem} is to start with the Tian-Yau metrics on $X_0\setminus D=(Y_1\setminus D)\cup (Y_2\setminus D)$, and graft them to nearby fibers $X_t$ for $|t|$ small to get K\"ahler metrics which are \emph{approximately} Calabi-Yau. However, the existence of singularities of the total space $\mathcal X$ along $H$ imposes  difficulties in performing a reasonable construction. To overcome this issue, we first modify the family $\mathcal X\rightarrow \Delta$ to obtain another family $\hX\rightarrow\Delta$ by using base change and birational modifications (see Figure \ref{f: the modified family}). The new family $\hX$ essentially agrees with $\mathcal X$ away from $X_0$, and the new central fiber $\widehat{X}_0$ comprises of a chain of \emph{three} components, with the two end components isomorphic to $Y_1, Y_2$ respectively, and the middle component $\mathcal N$ is given by a conic bundle over $D$, and  is realized as a natural hypersurface in the projective bundle $\P(L_1\oplus L_2\oplus \mathbb C)$ cut out by the equation $s_1\otimes s_2+s_3f(x)=0$ (viewing $f$ as a section of $L_1\otimes L_2$).  Recall that  $L_i$ is the normal bundle of $D$ in $Y_i$ for $i \in \{1,2\}$. The conic fibers degenerate precisely along the divisor $H$ in $D$. The component $\mathcal N$ intersects transversally with $Y_1, Y_2$ along $D_1, D_2$, which are both naturally isomorphic to $D$.   Notice that $\hX$ is not necessarily smooth. Instead it has singularities along $D_1\cup D_2$ which is of codimension $2$. However, it turns out that working with $\hX$ is the correct thing to do.

The second part involves the \emph{construction of the neck region}. We aim at constructing Calabi-Yau metrics on the smooth locus of the central fiber of $\hX$. For the two end components, these are provided by the complete Tian-Yau metrics.  For the middle component, with a moment's thought, one realizes that  it is not possible to construct a complete Calabi-Yau metric on $\mathcal N^0=\mathcal N\setminus (D_1\cup D_2)$. The reason is that  if such a metric existed, it would have two ends, and hence Cheeger-Gromoll splitting theorem would imply that the complete Ricci-flat manifold $\mathcal N^0$ must isometrically split off a line, which is not compatible with the complex geometry of $\mathcal N^0$. Instead we will try to construct a family of incomplete Calabi-Yau metrics defined on larger and larger open subsets in $\mathcal N^0$. The fact that $\mathcal N$ has a natural  holomorphic $\C^*$-action suggests us to look for Calabi-Yau metrics with an $S^1$-symmetry. 

In complex dimension 2, the hyperk\"ahler structure of the neck region was achieved in \cite{HSVZ} using the classical Gibbons-Hawking ansatz and studying the reduced linear equations, where the underlying complex manifold of the neck was not identified. In higher dimensions, the technologies employed in the construction are much more involved due to the appearance of the much wilder singularity behavior, and the framework is substantially    
different from the hyperk\"ahler triple method in \cite{HSVZ}. This is based on the study of higher dimensional generalizations of the Gibbons-Hawking ansatz, see Section \ref{s:torus-symmetries}. An  essential difference from the case in complex dimension $2$ is that the corresponding reduced equations in higher dimensions are still nonlinear. Via a linearization we are led to study certain solutions to a linear elliptic PDE with singularities along a submanifold. The existence and local regularity of such solutions, namely \emph{Green's currents}, are studied in detail in Section \ref{s:Greens-currents},  which will be used to construct a family of incomplete K\"ahler metrics on open subsets of $\mathcal N^0$. The fact that the singularities of the Green's currents are non-isolated causes essential difficulties in understanding the regularity of the K\"ahler metrics. In actuality we only prove the metrics are $C^{2,\alpha}$ and this suffices for our purpose. Another difference in higher dimensions is that these metrics are only \emph{approximately} Calabi-Yau. Since we work on the fixed algebraic family $\hX$, this requires us to obtain a precise formula and establish some uniform estimates for the K\"ahler potentials of these K\"ahler metrics, which is crucial for our gluing construction. In Section \ref{ss:glued-metrics}, we graft the incomplete approximately Calabi-Yau metrics constructed in Section \ref{s:neck} and the complete Tian-Yau metrics on $Y_i\setminus D$ to $C^{1,\alpha}$-K\"ahler metrics on $X_t$ for $|t|\ll1$, which are approximately Calabi-Yau.

The third part involves using \emph{weighted analysis} and the implicit function theorem to perturb the above approximately Calabi-Yau metrics on $X_t$ to genuine Calabi-Yau metrics. A crucial point is to prove some uniform weighted estimates, see Proposition \ref{p:estimate-L-neck} and Proposition \ref{p:global-injectivity-estimate} for both the  incomplete and complete versions. To do this, including the analysis on the regularity scales in Section \ref{s:neck}, another crucial analytic input we need is a Liouville theorem on the Tian-Yau spaces in higher dimensions, which is proved in  \cite{SZ-Liouville}.
One can directly see from the implicit function theorem that the Gromov-Hausdorff collapsing behavior of the Calabi-Yau metrics.

Finally, to prove Theorem \ref{t:main-theorem}, we  start with the family of approximately Calabi-Yau metrics constructed above and the Tian-Yau metrics on the two end components of $\widehat X_0$. We then perturb them to approximately Calabi-Yau metrics on the nearby fibers $\widehat X_t$ for $|t|\ll1$. In this step we need to work at the level of K\"ahler potentials and use suitable cut-off functions. 
Finally we perturb them to  genuine Calabi-Yau metrics on $\widehat X_t$ in the K\"ahler class $2\pi c_1(\mathcal O(1)|_{\widehat X_t})$, so by well-known uniqueness of Calabi-Yau metrics we know the latter must agree with $\omega_{CY, t}$. This means that we have obtained an approximate geometric description of the metrics $\omega_{CY, t}$ for $|t|\ll1$. Then conclusion of Theorem \ref{t:main-theorem} follows.

\subsection{Organization of the paper}

We briefly overview the remainder of the paper.

In Section \ref{s:torus-symmetries}, we  study the Calabi-Yau equation with torus symmetry and particularly we will focus on the reduced equations which yield certain singular behavior. We also formulate the linearization of such singular equations in Section \ref{ss:linearized equation}. This serves as model objects in understanding the degenerations.

In Section \ref{s:Greens-currents}, we discuss both the existence and local regularity theory for the linearized equations with singularities along a submanifold as formulated in Section \ref{ss:linearized equation}. To do this,  this section will provide detailed analysis on the existence and local regularity theory of Green's currents. 

In Section \ref{s:neck}, we construct a family of degenerating family of  K\"ahler metrics which are approximately Calabi-Yau on the neck region. In Sections \ref{ss:kaehler-structures}-\ref{ss:complex-geometry}, we will identify the underlying complex manifold and carry out detailed analysis on the K\"ahler potentials of the metrics.
To precisely understand the metric degenerations,  we study in Section \ref{ss:regularity-scales} the regularity of the metrics in a quantitative fashion and classify the various rescaled limiting geometries for this family of metrics. In Section \ref{ss:neck-weighted-analysis}, we set up the weighted spaces and prove several fundamental weighted estimates. 
Section \ref{ss:perturbation of complex structures}
focuses on the estimates for the perturbation of the complex structures. This is necessary since our actual degenerating family of complex structures is only a perturbation of the neck region.

In Section \ref{s:neck-perturbation}, we apply the implicit function theorem and weighted estimates to show that the family of approximately Calabi-Yau metrics on the neck can be perturbed to genuine Calabi-Yau metrics, all of which will be done in Sections \ref{ss:perturbation-framework}-\ref{ss:incomplete weighted analysis}.
Here an important subtlety is that, to perturb those incomplete K\"ahler metrics, we will invoke Neumann boundary conditions instead of Dirichlet boundary conditions, which more naturally fits in with the formulation of Proposition \ref{p:neck-uniform-injectivity}. 
 In Section \ref{ss:renormalized-measure},  we also discuss the renormalized limit measures and their singular behavior, which may be of its own interest.

In Section \ref{s:gluing}, we will complete the proof of 
Theorem \ref{t:main-theorem}.
There are three main points in the this section. The algebraic modification as mentioned above is formulated in Section \ref{ss:algebraic-geometry}. 
In Section \ref{ss:glued-metrics}, we will construct the glued metrics on $X_t$  which are approximately Calabi-Yau.  Finally, we will perform the perturbation in Section \ref{ss:global analysis}. It is worth mentioning that in the proof of  Theorem \ref{t:main-theorem} we do not use  genuine Calabi-Yau metrics on the neck region as in Theorem \ref{t:neck-CY-metric}. But the details of the final perturbation step  is essentially the same as the proof of Theorem \ref{t:neck-CY-metric}. With the detailed perturbation argument on the neck region, we can safely omit the details in the final perturbation step.

In Section \ref{s:discussions}, we will give some possible extensions of our results and we will propose a conjecture  (Conjecture \ref{cj:generalized-SYZ}) based on Theorem \ref{t:main-theorem}. 

\subsection{Acknowledgements}
 We would like to thank Lorenzo Foscolo, Mark Haskins, and Shouhei Honda for helpful discussions. We thank Ronan Conlon for comments which improved the presentation. 
We are also grateful to Hans-Joachim Hein and Jeff Viaclovsky for stimulating discussions on the study of collapsing hyperk\"ahler metrics on K3 surfaces which led to an earlier joint paper \cite{HSVZ}. We thank Yang Li for communications regarding the draft of his preprint \cite{Li} and the first draft of the current paper in January 2019. 
Substantial parts of this paper were written when the second author was visiting Princeton University, Academia Sinica, and ShanghaiTech University in the academic year 2018-2019. He  would like to thank those institutions for their hospitality and support.

\section{Calabi-Yau metrics with torus symmetry}
\label{s:torus-symmetries}
In this section, we discuss the Calabi-Yau metrics which are preserved by a compact torus action, and the symmetry reduction of the Calabi-Yau equation.  We will explain why we expect these metrics to provide local models for the collapsing of Calabi-Yau metrics when the complex structure degenerates. The ideas of this section will be used in Section \ref{s:neck} to construct the approximately Calabi-Yau neck region. 

In Section \ref{ss:dimension reduction} we explain the motivation and study the dimension reduction of the Calabi-Yau equation for the $S^1$-action. In Sections \ref{ss:Calabi model space}-\ref{ss:2d standard model}, we will give some exact solutions to the reduced equation, which will serve as important local models in later sections. In Section \ref{ss:linearized equation} we consider the linearized equation and explain a natural class of singular solutions given by Green's currents.   Section \ref{ss:higher rank torus} is to discuss the Calabi-Yau metrics admitting a torus action of higher rank. 

\subsection{Motivation and dimension reduction of the Calabi-Yau equation}
\label{ss:dimension reduction}

We begin by recalling the familiar theory in complex dimension two. In this case, Calabi-Yau metrics are locally hyperk\"ahler and such metrics with free  $S^1$-action are locally given by the classical \emph{Gibbons-Hawking ansatz}, in terms of a positive harmonic function on a domain in $\R^3$. Notice that in the classical Gibbons-Hawking ansatz, the hyperk\"ahler metrics admit an $S^2$-family of parallel compatible complex structures. A priori there is no preferred choice. If we do make a choice of complex structure, then the base $\R^3$ also has a natural splitting into $\C\oplus \R$, and we refer to Section \ref{ss:2d standard model} for further discussions. When the $S^1$-action is not free,  the fixed points correspond to simple poles, i.e., Dirac type singularities of the harmonic function,  locally given by $\frac{1}{2r}$ plus a smooth function, where $r$ is the Euclidean distance to the origin  $0^3\in\dR^3$. The local topological model for the $S^1$-fibration near a singularity is the standard Hopf fibration $\pi: \R^4\rightarrow \R^3$. Applying the Gibbons-Hawking ansatz to the entire $\R^3$ with a positive Green's function $G_{\lambda}\equiv\frac{1}{2r}+\lambda$ for $\lambda>0$, one obtains a homothetic scaling family of the \emph{Taub-NUT} metrics $g_{\lambda}$ on $\R^4$, so that $g_{\lambda}$ converges to the Euclidean metric on $\R^4$ when $\lambda\rightarrow 0$, and to the Euclidean metric on $\R^3$ when $\lambda\rightarrow\infty$. 

We can also apply the Gibbons-Hawking ansatz to three dimensional flat manifolds with slower than cubic volume growth, but then there will not exist any non-trivial global positive harmonic function with only simple poles.
Nevertheless, the Gibbons-Hawking construction still yields various interesting families of \emph{incomplete} hyperk\"ahler metrics. Important examples are given by the Green's functions on $S^1\times \R^2$ (the Ooguri-Vafa metric, c.f. \cite{GW})  and $\dT^2\times \R$ (c.f.  \cite{HSVZ}). These metrics are important in understanding the collapsing behaviors of hyperk\"ahler metrics on K3 surfaces \cite{GW, HSVZ}.

In higher dimensions, the algebro-geometric consideration concerning complex structure degenerations also suggests the significance of Calabi-Yau metrics with torus symmetry. We will explain the motivation in a local model situation.  Let $p:\mathcal{N}\to\Delta\subset \dC$ be a degenerating family of smooth complex algebraic varieties  $\mathcal{N}_t\equiv p^{-1}(t)$ such that as $t\to0$, $\mathcal{N}_t$ degenerates into $\mathcal N_0$ which is a union of irreducible components. Our primary observation is that, in the generic situation, near a point on $\mathcal N_0$ where the $(k+1)$-components intersect transversally, the degenerating family is locally modeled by the equation 
\begin{equation}z_0\cdot \ldots \cdot z_{k}=t\cdot(f(z_{k+1}, \cdots, z_{n})+g),\end{equation}
where $g$ is contained in the analytic ideal generated by $z_0 , \ldots ,  z_k$. Near a point with $z_0=\ldots=z_{k}=0$, this can be further approximated by omitting the term $g$, which results in a $(\C^*)^k$-fibration 
\begin{equation}z_0\cdot \ldots \cdot z_k=t\cdot f(z_{k+1}, \ldots, z_n)\end{equation}
over an $(n-k)$-dimensional base. The fibers are orbits of the $(\C^*)^k$-action, where $(\C^*)^k$ is naturally a subgroup in 
$(\C^*)^{k+1}=\{(\lambda_0, \cdots, \lambda_k)|\lambda_i\in \C^*\}$
defined by the relation $\lambda_0\cdots\lambda_k=1$.

More generally, one can consider a  complex manifold $D$ and $(k+1)$ holomorphic line bundles $L_0, \ldots, L_k$ over $D$. Let us denote  $E\equiv\bigoplus\limits_{j=0}^k L_j$ and fix a holomorphic section $f$ of the tensor product $L_0\otimes \cdots\otimes L_k\cong \det (E)$. Then we can consider the hypersurface $\mathcal N$ in $E\times \C$ cut-out by the equation 
\begin{equation}s_0\otimes\ldots\otimes s_k=t\cdot f(x),\end{equation}
where $(x, [s_0, \cdots, s_k])$ is a point in $E$ and $t\in \C$. We can view $\mathcal N$ as a family of hypersurfaces in $E$ parametrized by $t\in \C$. There is a natural $\C^*$ action on $\mathcal N$ given by 
\begin{equation}
\lambda(\zeta)\cdot \Big(x, [s_0, \cdots, s_k], t\Big) \equiv \Big(x, [\zeta s_0, \cdots, \zeta s_k], \zeta^{k+1} t\Big).
\end{equation}
It induces isomorphisms between $\mathcal N_t$ and $\mathcal N_1$ for all $t\neq 0$, and it preserves $\mathcal N_0$.

 For simplicity we only consider the generic case when the zeroes of $f$ form smooth hypersurface, then for $t\neq 0$, $\mathcal N_t$ is smooth but the projection map $\pi_t: \mathcal N_t\rightarrow D$ is still singular precisely along the union of $\Pi_{ij}\equiv \{(x, [s_0, \cdots, s_k]\in \mathcal N|s_i=s_j=0, f(x)=0\}$ for all pairs $(i, j)$ with $i\neq j$. Notice this union is also the singular set of the total space $\mathcal N$. When $t=0$, $\mathcal N_0$ is simply the union of the zero sections of $L_j$.

Suppose that the base $D$ has a holomorphic volume form $\Omega_D$ and a Calabi-Yau metric $\omega_D$, satisfying the equation $\omega_D^n=C\Omega_D\wedge\bar\Omega_D$. Then one can easily write down a $(\C^*)^k$-invariant holomorphic volume form $\Omega_t$ on $\mathcal N_t$ for $t\neq 0$, which is given by 
\begin{equation}
\Omega_t=\sum_{j=0}^k (-1)^j \frac{ds_0}{s_0}\wedge\cdots \widehat{\frac{ds_j}{s_j}}\wedge\cdots\wedge \frac{ds_k}{s_k}\wedge \pi_t^*\Omega_D.
\end{equation}
Here the notation $\frac{ds_j}{s_j}(j=0, \cdots k)$ should be understood after choosing a local holomorphic section of $L_j$, but it is easy to see that this definition of $\Omega_t$ does not depend on the particular choice. Also a priori $\Omega_t$ is defined away from the singular fibers of the projection map $\pi_t$, and it is not difficult to see that $\Omega_t$ extends to a nowhere vanishing holomorphic volume form on $\mathcal N_t$. 

Let $T^k=(S^1)^k\subset (\C^*)^k$ be the obvious maximal compact subgroup. Naturally one would ask for $T^k$-invariant Calabi-Yau metrics $\omega_t$ on (part of) $\mathcal N_t$, satisfying the equation $\omega_t^n=C\Omega_t\wedge \bar\Omega_t$, and we are then led to study dimension reduction of the Calabi-Yau equation under the $T^k$ action. This has been written down by Matessi \cite{Matessi} and we now explain the details in the case $k=1$, and we briefly discuss the case of general $k$ in Section \ref{ss:higher rank torus}.

\

Let $(X, \omega, J)$ be an $n$-dimensional K\"ahler manifold admitting an $S^1$-action which is holomorphic and Hamiltonian  with a moment map function $z$, i.e., 
\begin{equation}\label{momentmap}
dz=\xi\lrcorner\omega,
\end{equation}
where $\xi$ is the vector field generating the $S^1$-action.
 We now assume in addition that the $S^1$-action is free. Locally in a neighborhood of an $S^1$-orbit we can complexify the $S^1$-action and obtain a complex quotient $D$ which is an $(n-1)$-dimensional complex manifold. Then the local $S^1$-quotient $Q$ can be identified as a differentiable manifold with  $D\times I$, where $I$ is an interval with coordinate function $z$. 
 
 Denote by $\{w_1, \cdots, w_{n-1}\}$ the local holomorphic coordinates on $D$. Then they can be viewed as local holomorphic functions on $X$. Let $t$ be an arbitrary local function with $\xi(t)=1$,  Then $\{z, t, w_1, \cdots, w_{n-1}\}$ gives a local coordinate system on $X$, and we have $\xi=\p_t$. Let us write $w_i=x_i+\sqrt{-1} y_i$. Then we can express the complex structure $J$ on $X$ in terms of the local coordinates as 
\begin{equation}
Jdx_i=dy_i, \quad Jdy_i=-dx_i, \quad Jdz=h^{-1}\Theta,
\end{equation}
where $h>0$ is a local function and $\Theta$ is a local 1-form which can be written as
\begin{equation}
\Theta=-dt+\theta,
\end{equation}
such that $\theta$ does not have $dt$ component. The above negative sign appears due to the fact that 
\begin{equation}
(Jdz)(\p_t)=-dz(J\p_t)=-\omega(\xi, J\xi)<0.
\end{equation}
This also gives an intrinsic geometric interpretation for $h^{-1}$, as the squared norm of the Killing field $\xi$. In particular, $h$ is $S^1$-invariant and hence it descends to a function on $Q$.  

By the $S^1$-invariance
\begin{equation}\mathcal L_{\xi} (Jdz)=0,\quad  \mathcal L_{\xi}(dt)=0,
\end{equation}
we obtain
$\mathcal L_{\xi}\theta=0$.
Therefore, $\theta$ can be also viewed as a 1-form on $Q$.

We can write the K\"ahler form $\omega$ on $X$ as 
\begin{equation}
\omega=dz\wedge (-dt+\theta)+\tilde\omega,
\end{equation}
where $\tilde\omega$ is a $(1,1)$-form without $dz$ or $dt$ component. This is due to \eqref{momentmap} and the fact that $\omega$ is of type $(1,1)$.
Since $\mathcal L_{\xi}\omega=0$ we also have $\mathcal L_{\xi}\tilde\omega=0$, so the coefficients of $\tilde\omega$ also descend to $Q$. In particular, we may view $\tilde\omega=\tilde\omega(z)$ as a family of $(1,1)$-forms  on $D$. The condition $d\omega=0$ is equivalent to 
\begin{equation}
\label{omegaequation}
\begin{cases}
d_D\tilde\omega(z)=0\\
\p_z\tilde\omega(z)=d_D\theta, 
\end{cases}
\end{equation}
where $d_D$ denotes the differential along $D$.  

Now we consider the integrability condition of the complex structure $J$. One can check that 
\begin{equation}J\p_t=h^{-1}\theta_z\p_t+h^{-1}\p_z, \end{equation}
so the holomorphic vector field generating the $\C^*$-action is given by 
\begin{equation}\xi^{1,0}=\frac{1}{2}(\p_t-\sq J\p_t)=\frac{1}{2}(1-\sq h^{-1}\theta_z)\p_t-\frac{1}{2}\sq h^{-1}\p_z.
\end{equation}
The dual holomorphic $(1,0)$-form is
\begin{equation}\kappa=\sq(hdz+\sq\Theta+\kappa'),\end{equation}
where $\kappa'$ only involves $dx_i, dy_i$. The integrability condition for $J$ can be expressed as
\begin{equation}
d\kappa\wedge \kappa\wedge dw_1\wedge\cdots\wedge dw_{n-1}=0,\end{equation}
which is equivalent to
\begin{equation}\label{thetaequation}
\begin{cases}
d_D\theta\wedge dw_1\wedge \cdots\wedge dw_{n-1}=0\\
\p_z\theta=-d_D^ch , 
\end{cases}
\end{equation}
where $d_D^c\equiv J_D \circ d_D$. 
The first equation follows from the second equation in \eqref{omegaequation} which implies $d_D\Theta$ is of type $(1,1)$ on $D$. Notice that \eqref{omegaequation} and \eqref{thetaequation} together can be re-organized as  a system
\begin{equation}\label{omegahequation}
\begin{cases}\p_z^2\tilde\omega+d_Dd_D^ch=0\\
d\Theta=\p_z\tilde\omega-dz\wedge d_D^ch.
\end{cases}
\end{equation}
It is not difficult to globalize the above discussion and the upshot is that a K\"ahler form on $X$ with a free $S^1$-action gives rise to a family of K\"ahler forms $\tilde\omega(z)$ on a complex manifold $D$, together with a positive function $h$ on $D\times I$, satisfying \eqref{omegahequation}. This is the familiar procedure in \emph{K\"ahler reduction}. The 1-form $-\sq\Theta$ can be viewed as a family of connection 1-forms on the natural $S^1$-bundle over $D$, so as a consequence $\p_z\tilde\omega=d_D\Theta$ determines an integral cohomology class in $2\pi H^2(D; \Z)$. 

Conversely, suppose that we are given $(\tilde\omega(z),h)$ that satisfies \eqref{omegahequation}, and suppose that the deRham class $[\p_z\tilde\omega_z]$ lies in $2\pi H^2(D; \Z)$. Then by general theory we can choose a connection 1-form $-\sqrt{-1}\Theta$ on an $S^1$-bundle $X$ over $D\times I$ satisfying \eqref{omegahequation}, and this gives rise to a K\"ahler structure on $X$ with $S^1$-action.
The isomorphism class of the resulting K\"ahler structures on $X$ depends only on the gauge equivalent classes of the connection 1-form, so is unique if the first Betti number of $D$ vanishes.

Now we specialize to study Calabi-Yau metrics with $S^1$-symmetry, so we additionally assume that $X$ has a  holomorphic volume form $\Omega$. 
Denote the holomorphic $(n-1)$-form on $X$
\begin{equation}
\tilde\Omega=\xi^{1,0}\lrcorner \Omega.
\end{equation}
The fact that $\Omega$ is $S^1$-invariant and holomorphic implies that $\tilde\Omega$ descends to a holomorphic $(n-1, 0)$-form $\Omega_D$ on $D$, and we also have 
$\Omega=\kappa\wedge\Omega_D$.
By definition, 
\begin{align}\omega^{n}&=-n\cdot dz\wedge dt\wedge \tilde\omega^{n-1},
\\
\Omega\wedge\bar\Omega&=2\sq (-1)^{n-1} h\cdot  dz\wedge dt\wedge \Omega_D\wedge\bar\Omega_D.
\end{align}
So the Calabi-Yau equation on $X$ 
\begin{equation}\frac{\omega^{n}}{n!}=\frac{(\sq)^{n^2}}{2^n}\cdot\Omega\wedge\bar\Omega
\end{equation}
becomes  
\begin{equation}\label{volumeformequation}
\frac{\tilde\omega^{n-1}}{(n-1)!}=\frac{(\sq)^{(n-1)^2}}{2^{n-1}}\cdot h\cdot \Omega_D\wedge\bar\Omega_D.
\end{equation}
Combining \eqref{omegahequation} and \eqref{volumeformequation}, we obtain that
\begin{equation} \label{nonlinearGH}
\p_z^2\tilde\omega+d_Dd_D^c\frac{2^{n-1}\tilde\omega^{n-1}}{(\sq)^{(n-1)^2}\Omega_D\wedge\bar\Omega_D}=0.
\end{equation}
Again it is easy to see this discussion can be globalized so we get a complex Calabi-Yau manifold $(D, \Omega_D)$ together with a family of K\"ahler forms $\tilde\omega(z)$ satisfying \eqref{nonlinearGH}. 
Conversely, the study of $n$-dimensional Calabi-Yau manifolds $(X, \omega, \Omega)$ with a free $S^1$-action is reduced to the study of the equation \eqref{nonlinearGH}. 

Now we make a few observations. First, when $n=2$, equation \eqref{nonlinearGH} is reduced to a linear equation. Notice that $\frac{\sq}{2}\Omega_D\wedge\bar\Omega_D$ is a flat K\"ahler form when $n=2$. Then we can write 
\begin{equation}
\tilde\omega=\frac{\sq}{2}\cdot  V \cdot \Omega_D\wedge\bar\Omega_D,
\end{equation}
for a real-valued function $V$ on $Q=D\times I$. Then equation \eqref{volumeformequation} is equivalent to 
\begin{equation} \label{3dGH}
\p_z^2V-\Delta_D V=0,
\end{equation}
where $\Delta_D=d_D^*d_D$ is the Hodge Laplace operator with respect to the above flat metric on $D$. Equation \eqref{3dGH} is now exactly the Laplace equation on $Q$, and the above discussion is reduced to the classical Gibbons-Hawking ansatz which produces hyperk\"ahler 4-manifolds. The minor  difference is that here we have made a distinguished choice of the complex structure so that $Q$  naturally splits as $D\times I$. 

When $n>2$, \eqref{nonlinearGH} is a non-linear equation, which is much harder to deal with. We  call \eqref{nonlinearGH} the \emph{non-linear Gibbons-Hawking ansatz} for the Calabi-Yau metrics with $S^1$-symmetry This equation was first written down by Matessi \cite{Matessi}, in a slightly different form.

\subsection{Calabi model spaces}
\label{ss:Calabi model space}
In general it is not easy to directly solve equation \eqref{nonlinearGH}, but we can easily see some special solutions, which will be important for us. 

Suppose that $(D, \Omega_D)$ is  an $(n-1)$-dimensional compact Calabi-Yau manifold, and $\omega_D$ is a Calabi-Yau metric on $D$ with $[\omega_D]\in 2\pi H^2(D;\Z)$, satisfying 
\begin{equation}
\frac{\omega_D^{n-1}}{(n-1)!}=\frac{(\sq)^{(n-1)^2}}{2^{n-1}} \Omega_D\wedge\bar\Omega_D.
\end{equation}
If we set
\begin{equation}
\label{e:Calabi model solution}
\begin{cases}
\tilde\omega(z)=z\cdot \omega_D \\
h=z^{n-1},
\end{cases}
\end{equation}
then as long as $z>0$, $(\tilde\omega, h)$ clearly satisfies \eqref{nonlinearGH} and the integrality condition $[\p_z\tilde\omega(z)]\in 2\pi H^2(D;\Z)$ is also achieved. So the above gives incomplete Calabi-Yau metrics in dimension $n$. 

This metric has already appeared in K\"ahler geometry, which is usually expressed in terms of a K\"ahler potential. To explain this, 
we fix a holomorphic line bundle $L_D$ with first Chern class $\frac{1}{2\pi} [\omega_D]$, and also fix a hermitian metric on $L_D$ whose curvature form is $-\sq\omega_D$. Then we consider the subset $\mathcal{C}^n$ of the total space of $L_D$ consisting of all elements $\xi$ with $0<|\xi|< 1$. It is endowed with a holomorphic volume form $\Omega_{\mathcal{C}^n}$ and a Ricci-flat K\"ahler metric $\omega_{\mathcal{C}^n}$ which is incomplete as $|\xi| \to 1$ and complete as $|\xi| \to 0$. The holomorphic volume form $\Omega_{\mathcal{C}^n}$ is given by (as in Section \ref{ss:complex-geometry})
\begin{equation}
\Omega_{\mathcal{C}^n}=\sq\cdot\frac{d\xi}{\xi}\wedge \Omega_D.
\end{equation}
The metric  $\omega_{\mathcal{C}^n}$ is given by the \emph{Calabi ansatz}
 \begin{equation}\label{calabiansatz}\omega_{\mathcal{C}^n}=\frac{n}{n+1} \sq\p\bp (-{\log |\xi|^2})^{\frac{n+1}{n}}. \end{equation}
It is straightforward to check that
 \begin{equation}\omega_{\mathcal{C}^n}^n=\frac{1}{n\cdot 2^{n-1}}(\sq)^{n^2} \Omega_{\mathcal{C}^n}\wedge\overline\Omega_{\mathcal{C}^n}.\end{equation}
Clearly the Calabi-Yau structure $(\omega_{\mathcal{C}^n}, \Omega_{\mathcal{C}^n})$ is invariant under the natural $S^1$-action on $L_D$. Applying the $S^1$-reduction as in Section \ref{ss:dimension reduction}, we get that the moment map is given by
\begin{equation}
z=(-{\log |\xi|^2})^{1/n},
\end{equation}
and the reduced family of K\"ahler metrics on $D$ is given by 
\begin{equation}
\tilde\omega=z\cdot\omega_D.
\end{equation}
The function $h$ is 
\begin{equation}
h=\frac{2}{n}\cdot z^n.
\end{equation}
So we see this  gives rise to the above solution to \eqref{e:Calabi model solution} (up to a multiplicative constant on $h$), We call the space $(\mathcal{C}^n, \omega_{\mathcal{C}^n}, \Omega_{\mathcal{C}^n})$ a \emph{Calabi model space}. In  Remark \ref{r:Calabi model potential} we will see the formula \eqref{calabiansatz} can also be recovered from \eqref{e:Calabi model solution}, and this works in a more general situation. 

Now in the above Calabi ansatz the connection 1-form $-\sqrt{-1}\Theta$ is given by the Chern connection 1-form on $L_D$. We claim that by varying the holomorphic structures on $L_D$ we obtain all possible gauge equivalence classes of $\Theta$. This follows from the fact that there is a natural isomorphism between the group  $\mathcal S_h$ of  the isomorphism classes of holomorphic line bundles with $c_1=0\in H^2(D;\R)$ and the group $\mathcal S_f$ of  gauge equivalence classes of flat $U(1)$ connections on $D$. Abstractly, we know the first group fits into an exact sequence
\begin{equation}
0\rightarrow \frac{H^1(D;\cO)}{H^1(D;\Z)}\rightarrow \mathcal S_h\rightarrow H^2_{\text{tor}}\rightarrow 0
\end{equation}
obtained from the exponential sequence on $D$, 
where $H^2_{tor}$ denotes the torsion subgroup in $H^2(D;\Z)$. The second group is given by $\text{Hom}(\pi_1(D), S^1)=\text{Hom}(H_1(D; \Z), S^1)$, and it
 fits into a short exact sequence
\begin{equation}
0\rightarrow \frac{H^1(D; \R)}{H^1(D;\Z)}\rightarrow \mathcal S_f\rightarrow \text{Hom}(H_{1, \text{tor}}, S^1)\rightarrow 0 ,
\end{equation}
where $H_{1, tor}$ is the torsion subgroup in $H_1(D; \Z)$. The isomorphism between $\mathcal S_h$ and $\mathcal S_f$ induces an isomorphism on the torsion quotients, which coincides with the isomorphism\begin{equation}
H^2_{\text{tor}}\simeq \text{Ext}(H_1(D;\Z), \Z)\simeq \text{Hom}(H_{1, tor}, S^1)
\end{equation}
given by the universal coefficient theorem. 

We mentioned above that gauge equivalent choices of the connection 1-form $-\sqrt{-1}\Theta$ yield isomorphic Calabi-Yau structures on $\mathcal{C}^n$. Now we observe that for different choices of gauge equivalence classes which differ only by an element in the identity component of $\mathcal S_f$, the resulting Calabi-Yau structures are also isomorphic, via a diffeomorphism that covers a holomorphic isometry on $D$.

To see this, let us fix a choice of $\Theta$. Then given any vector field $Z$ on $D$, let $\hat{Z}$ be the horizontal lift of $Z$ to the $S^1$-bundle with respect to the connection form $-\sqrt{-1}\Theta$. The infinitesimal variation of $\Theta$ along the flow of $\hat{Z}$ is $\mathcal  L_{\hat{Z}} \Theta=\hat{Z}\lrcorner d\Theta$, which is the pull-back of the form  $Z\lrcorner \omega_D$ on $D$. 
For a different choice $\Theta'$, modulo gauge equivalence we may assume $\Theta'-\Theta$ is a harmonic 1-form on $D$, so it must be parallel by Bochner's theorem. Let $Z$ be the vector field on $D$ that satisfies $Z\lrcorner \omega_D=\Theta'-\Theta$. Then $Z$ is also parallel, so $Z$ is holomorphic and Killing. In particular, $\mathcal{L}_{Z}\omega_D=0$.  Then it follows that $\Theta'$ and $\Theta$ are related by the flow of the lifted vector field $\hat{Z}$.

\subsection{Two dimensional standard model spaces}
\label{ss:2d standard model}
In the Gibbons-Hawking ansatz, to get non-trivial topology one often needs to allow the $S^1$-action to have fixed points. This corresponds to the harmonic function $V$ having  Dirac type singularities. For the convenience of later discussions, let us briefly recall the relevant formulae in this model situation, using our  description with a preferred complex structure.

Let $\C^2$ be equipped with the standard holomorphic coordinates $(u_1, u_2)$  and the flat K\"ahler metric 
\begin{equation}
\begin{cases}
\omega_{\C^2}=\frac{\sq}{2}(du_1\wedge d\bar u_1+du_2\wedge d\bar u_2)\\
\Omega_{\C^2}=du_1\wedge du_2.
\end{cases}
\end{equation}
Consider the $S^1$-action on $\C^2$,
\begin{equation}e^{\sq t}\cdot (u_1, u_2)\equiv(e^{-\sq t}u_1, e^{\sq t}u_2)\label{e:2D-model-action}
\end{equation}
with an infinitesimal generator
\begin{equation}\p_t=-\sq(u_1\p_{u_1}-u_2\p_{u_2})+\sq(\bar u_1\p_{\bar u_1}-\bar u_2\p_{\bar u_2}).
\end{equation}
Then we have a moment map $z$ for the $S^1$-action with respect to $\omega_{\C^2}$ and a complex moment map 
$y$ for the complexified $\C^*$-action with respect to $\Omega_{\C^2}$. So we obtain the standard Hopf map \begin{align}\pi: \C^2\rightarrow Q_0\equiv \C\times \R,\quad (u_1,u_2)\mapsto (y, z),\end{align} where 
\begin{equation}
\begin{cases}
z=\frac{1}{2}(|u_1|^2-|u_2|^2)\\
y=u_1u_2.
\end{cases}
\end{equation}
Then the holomorphic quotient is $D_0=\C$ with holomorphic coordinate $y=y_1+\sq y_2$. Let us define \begin{equation} \label{modelquantities}
\begin{cases}
\tilde\omega_0\equiv\frac{\sq}{4r}dy\wedge d\bar y\\
\Omega_0\equiv dy\\
h_0\equiv\frac{1}{2r}\\
V_0\equiv\frac{1}{2r},
\end{cases}
\end{equation}
where $r=\sqrt{y_1^2+y_2^2+z^2}$ is the standard radial function on $Q_0$, and we have the relation 
\begin{equation}
r=\frac{1}{2}(|u_1|^2+|u_2|^2).
\end{equation}
The connection 1-form $-\sqrt{-1}\Theta_0$ on $\C^2$ can also be written down explicitly as
\begin{equation}
\Theta_0=h_0 \cdot J_0(dz)=\sq\cdot \frac{u_1d\bar u_1-\bar u_1 du_1+\bar u_2du_2-u_2d\bar u_2}{2(|u_1|^2+|u_2|^2)}.\label{e:model-connection}
\end{equation}
Let us define the curvature 2-form on $Q_0$,
\begin{equation}
\Upsilon_0\equiv \p_z\tilde\omega_0-dz\wedge d_{\C}^ch_0.
\end{equation}
So we have that
\begin{equation}\label{modelcurvature}
\Upsilon_0=-\frac{\sq}{4r^3} (zdy\wedge d\bar y+yd\bar y\wedge dz-\bar y dy\wedge  dz).
\end{equation}
From our above discussion we have the following holds
\begin{equation}
\label{e:d Theta0}
\begin{cases}d\Theta_0=\Upsilon_0\\
\tilde \omega_0+dz\wedge \Theta_0=\omega_{\C^2},
\end{cases}
\end{equation}
where we have implicitly viewed a form on $Q_0$ as a form on $\C^2$ using  the pull-back $\pi^*$. In other words, the flat metric on $\C^2$ together with the above $S^1$-action can be recovered via the Gibbons-Hawking ansatz applied to the function $V_0=\frac{1}{2r}$ on $Q_0=\C\times \R$. 

The above flat metric admits a one-parameter non-flat perturbation, which corresponds to  replacing $V_0$ in the Gibbons-Hawking ansatz by $V_0+\lambda$ for a positive constant $\lambda$ . We define
\begin{equation}
\begin{cases}
\tilde\omega_{0, \lambda}\equiv (\frac{1}{2r}+\lambda)\frac{\sq}{2}dy\wedge d\bar y\\
h_{0, \lambda}\equiv\frac{1}{2r}+\lambda,\\
\end{cases}
\end{equation}
which yields a family of \emph{Taub-NUT metrics} $(\omega_{TN, \lambda}, \Omega_{TN, \lambda})$ on $\R^4$ with 
\begin{equation}
\label{Taub-NUT}
\begin{cases}\omega_{TN, \lambda}\equiv (\frac{1}{2r}+\lambda)\frac{\sq}{2}dy\wedge d\bar y+dz\wedge \Theta_0\\
\Omega_{TN, \lambda}\equiv \sq ((\frac{1}{2r}+\lambda)dz+\Theta_0)\wedge dy.
\end{cases}
\end{equation}
We can still view these metrics as defined on $\R^4$ with coordinates $u_1, u_2, \bar u_1, \bar u_2$ via the above Hopf map, but the coordinate functions $u_1, u_2$ are no longer holomorphic. Indeed one can write down explicitly the holomorphic volume form 
\begin{equation}
\Omega_{TN, \lambda}=du_1\wedge du_2+\frac{\lambda}{2}dz\wedge dy.
\end{equation}
We also have 
\begin{equation}\label{e: modelholomorphic1form}
h_{0, \lambda}\cdot dz+\sq \Theta_{0}=\lambda\cdot dz+\frac{1}{2}(\frac{du_1}{u_1}-\frac{du_2}{u_2}).
\end{equation}
LeBrun \cite{LeBrun} showed that if we make a (non-holomorphic) coordinate change on $\C^2$,
\begin{equation}
\begin{cases}
\eta_+=u_1e^{\frac{\lambda}{2}(|u_1|^2-|u_2|^2)}\\
\eta_-=u_2e^{\frac{\lambda}{2}(|u_2|^2-|u_1|^2)},
\end{cases}
\end{equation}
then we have 
$\Omega_{TN, \lambda}=d\eta_+\wedge d\eta_-$. 
So the underlying complex manifold is still bi-holomorphic to $\C^2$ with holomorphic coordinates $\eta_{+}$ and $\eta_-$, and one can write down a global K\"ahler potential 
\begin{equation}
\label{e:TaubNUT potential}
\begin{cases}
\omega_{TN, \lambda}=\sq \p\bp \varphi_{\lambda}
\\ 
\varphi_{\lambda}=\frac{1}{2}(|u_1|^2+|u_2|^2)+\frac{T}{4}(|u_1|^4+|u_2|^4).
\end{cases}
\end{equation}
 Again in  Remark \ref{r:TaubNUT potential} we will see this follows from a more  general fact.

\subsection{Linearized equation and singularities}
\label{ss:linearized equation}
Now we return to the higher dimensional situation. One interesting local model is given by the product of a flat space $\C^{n-2}$ with trivial $S^1$-action and the above 2-dimensional model $(\C^2, \omega_{\C^2}, \Omega_{\C^2})$. Let $Q_0\equiv \dC\times\dR$ and $Q\equiv\C^{n-2}\times Q_0$. Then the singular set of $\tilde\omega$ and $h$ is a real codimension-3 subspace $P=\C^{n-2}\oplus\{0\}\subset Q$, and they both have transversal Dirac type singularities along $P$.   In our applications, we need to consider the non-linear situation. So $D$ is an $n-1$ dimensional complex manifold and $H\subset D$ is a smooth complex hypersurface, and we want our solution $(\tilde\omega, h)$ to the equation \eqref{nonlinearGH} to satisfy a distributional equation on $Q$ of the form 
\begin{equation} \label{distributionequation}
\Big(\p_z^2\tilde\omega+d_Dd_D^c\frac{2^{n-1}\tilde\omega^{n-1}}{(\sq)^{(n-1)^2}\Omega_D\wedge\bar\Omega_D}\Big)\wedge dz=2\pi \cdot \delta_{P},
\end{equation}
where $P\equiv H\times\{0\}$ and $\delta_P$ is a degree $3$ current given by integration along $P$.  This equation has appeared in the literature \cite{Zharkov} in a slightly different form. A solution  to this equation, with suitable regularity, will give rise to a Calabi-Yau metric with an $S^1$-action whose fixed point locus is a complex codimension-2 submanifold and transverse to which the action is given the model \eqref{e:2D-model-action}.  This is exactly what we are motivated to search for from the algebro-geometric discussion at the beginning of this section. 

Unfortunately, solving the equation \eqref{distributionequation} in general seems very difficult.
Motivated by recent results in the study of adiabatic limits of $G_2$-manifolds \cite{Don, FHN}, we attempt to study the equation when the $S^1$-orbit is very small. Again suppose $(D, \omega_D, \Omega_D)$ is $(n-1)$-dimensional Calabi-Yau, then for $T$ large we know there are trivial constant solutions to \eqref{nonlinearGH}with $\tilde\omega=T\omega_D$ and $h=T^{n-1}$. Now we look for a perturbation $\tilde\omega=T\omega_D+\psi$ for $T$ large. To the first order we know $\psi$ must satisfy the linearized equation at $T\omega_D$, hence 
\begin{equation}
\p_z^2\psi+T^{n-2}d_Dd_D^c \Big(\Tr_{\omega_D}\psi\Big)=0, 
\end{equation}
which by K\"ahler identities is equivalent to 
\begin{equation}
\p_z^2\psi-T^{n-2}d_Dd_D^*\psi=0.
\end{equation}
Up to a scaling of the $z$ variable this is equivalent to the equation 
\begin{equation}
\p_z^2\psi-d_Dd_D^*\psi=0. 
\end{equation}
If $\psi$ also satisfies $d_D\psi=0$,  then  $\psi\wedge dz$ is a harmonic $3$-form on the product $Q=D\times \R_z$. Again, the interesting case is when $\psi$ has singularities, and we want to study the case when the singular set is of the form $H\times \{0\}\subset Q$ for $H$ a smooth hypersurface in $D$, and correspondingly $\psi\wedge dz$ satisfies 
\begin{equation}
\Delta_Q (\psi \wedge dz) = 2\pi\cdot \delta_P.
\end{equation}
This is a generalization of Green's functions to 3-forms and we call $\psi \wedge dz$  a \emph{Green's current}, which is our main object of study in Section \ref{s:Greens-currents}. 

When $n=2$, the above Green's current is simply a Green's function which has been used in \cite{HSVZ} to obtain  exact solutions to a family of incomplete Calabi-Yau metrics by Gibbons-Hawking construction. In higher dimensions, using Green's currents, we can apply \eqref{omegahequation} to define a family of approximately Calabi-Yau metrics.  This is our main object  in Section \ref{s:neck}.

\subsection{Higher rank torus symmetry}\label{ss:higher rank torus}

Now we assume that an $n$-dimensional K\"ahler manifold $(X, \omega, J)$ admits an $T^k$-action which is holomorphic and Hamiltonian. We first assume the action is free. Let $(z_1, \cdots, z_k)$ be the moment map. Then similar discussion to that in Section \ref{ss:dimension reduction} yields locally a family of K\"ahler forms $\tilde\omega$ on the complex quotient, parametrized by $(z_1, \ldots z_k)\in \R^k$, a family of connection $1$-forms $-\sq\Theta_j(j=1, \cdots, k)$ and a positive definite  $k\times k$ real symmetric matrix $W=(W_{ij})$ with the inverse matrix
$W^{ij}=\langle \p_{t_i}, \p_{t_j}\rangle$  such that the following system of equations hold 
\begin{equation}
\begin{cases}\p_{z_j}\tilde\omega = d_D\Theta_j
\\
\p_{z_j}\Theta_i=-d^c_D W_{ij}
\\
\p_{z_l}W_{ij}=\p_{z_j}W_{il}.
\end{cases}\end{equation}
 As before the first two equations combine to give an equation on $(\tilde\omega, W_{ij})$
\begin{equation}
\label{e:higher Tk}
\p_{z_i}\p_{z_j}\tilde\omega+d_Dd_D^cW_{ij}=0. 
\end{equation}
Now suppose the complex quotient $D$ is Calabi-Yau with a holomorphic volume form $\Omega_D$, then the Calabi-Yau equation on $X$ becomes 

\begin{equation}
\frac{\tilde{\omega}^{n-k}}{(n-k)!}=\frac{(\sq)^{(n-k)^2}}{2^{n-k}} \cdot  \det(W_{ij}) \cdot \Omega_D\wedge\bar\Omega_D.
\end{equation}
This equation has been derived by Matessi and Zharkov (see \cite{Matessi}, \cite{Zharkov}). Again when the $T^k$-action is not free one should replace \eqref{e:higher Tk} by a distributional equation. We will discuss a simple example in Section \ref{ss:general-situations}.  In the most extreme case when $k=n$ is the complex dimension of $X$, this becomes the real Monge-Amp\`ere equation 
$\det(W_{ij})=C$.

\section{Green's currents}
\label{s:Greens-currents}
In this section, we study in detail some existence and regularity theory of {\it Green's currents}. Our main motivation for studying these arises from Section \ref{s:torus-symmetries}, where we see that the Green's currents appear as Dirac type singular solutions to the linearization of dimension reduced Calabi-Yau equation by   the $S^1$-symmetry. It is possible that this study will also have applications to other geometric problems, especially to those concerning adiabatic limits. 

This section is organized as follows. In Section \ref{ss:geodesic-coordinates} we recall the generalized geodesic normal coordinates for an embedded submanifold. In Section \ref{ss:real-case} we discuss the definition, local existence and regularity properties of Green's currents in the general Riemannian setting. In Section \ref{ss:complex-greens-currents}, we refine these results in the special case related to K\"ahler geometry. In Section \ref{ss:global-existence}, we  prove a global existence result which will be immediately used in Section \ref{s:neck}.

\subsection{Normal coordinates for an embedded submanifold}\label{ss:geodesic-coordinates}

We start our discussion by introducing the basic notions of the {\it normal exponential map} and {\it normal coordinates} with respect to an embedded submanifold. 
This part seems to be standard in Riemannian geometry, in the language of Fermi coordinates (see \cite{Gray} for example). Since we cannot find a reference for the precise formulae we will need, we include detailed 
discussions and proofs in this section. 

Let $(Q,g)$ be an oriented Riemannian manifold of dimension $m$ and let $P\subset Q$ be a closed embedded oriented submanifold of codimension $k_0$ in $Q$.
In our later applications, we only need the case $k_0=3$.    
Denote by ${N}$ the normal bundle of $P$ in $Q$, equipped with the induced fiberwise Riemannian inner products. For any $p\in P$, denote by $N(p)$ the fiber of $N$ over $p$. 
The {\it normal exponential map} of $P$ in $Q$ is defined as 
\begin{equation}\Exp_P: {N}\rightarrow Q,\ (p, v)\mapsto \Exp_p(v),\ v\in N(p),\end{equation}
where $\Exp_p:T_pQ\to Q$ is the usual exponential map at $p\in Q$.
By implicit function theorem, we know that $\Exp_P: {N}\to Q$ is a diffeomorphism from some neighborhood of the zero section in ${N}$ to some tubular neighborhood of  $P$ in $Q$.

Now we define the {\it normal coordinates} for $P$.  
Fix $p\in P$. First we choose local  coordinates $\{x_1', \ldots, x_{m-k_0}'\}$ in a small neighborhood $U\subset P$ of $p$ such that at $p$
\begin{equation}\langle \p_{x_i'}, \p_{x_j'}\rangle=\delta_{ij},\  1\leq i,j\leq m-k_0.\end{equation}  We also assume that $\p_{x_1'}\wedge\cdots\wedge \p_{x_{m-k_0}'}$ is compatible with the orientation on $U$. 
Next, we pick local orthonormal sections $\{e_1,  \ldots, e_{k_0}\}$ of the normal bundle $N$
such that 
\begin{equation}
\langle e_{\alpha}, e_{\beta}\rangle=\delta_{\alpha\beta}, \ 1\leq \alpha,\beta\leq k_0,
\end{equation}
on $U$. Again we assume that $e_1\wedge\cdots\wedge e_{k_0}$ is compatible with the orientation on $N$, i.e., $e_1\wedge \cdots \wedge e_{k_0}\wedge \p_{x_1'}\wedge\cdots \p_{x_{m-k_0}'}$ is compatible with the orientation on $Q$. Then we can find $\epsilon>0$ such that $\Exp_p$ restricts to a diffeomorphism from 
\begin{equation}\fS_0=\{(q, v)\in N|q\in U,\ |v|<\epsilon\}\end{equation}
to a neighborhood $\mathcal U$ of $p$. In particular, $U=\mathcal U\cap P$. 

\begin{definition}[Normal coordinates] For any $(q,v)\in \mathfrak{S}_0$ with $v=\sum\limits_{\alpha=1}^{k_0}v_{\alpha}e_{\alpha}$, the local normal coordinates are defined as follows 
\begin{align}
\begin{cases}
x_j(\Exp_P(q,v))\equiv x_j'(q), & 1\leq j\leq m-k_0, 
\\
y_{\beta}(\Exp_P(q,v))\equiv v_{\beta}, & 1\leq\beta\leq k_0. 
\end{cases}
\end{align}
\end{definition}
By definition for each fixed point $(y_1,\ldots, y_{k_0},x_1,\ldots, x_{m-k_0})$ in $\mathcal{U}$, the curve \begin{equation}\vartheta(t)\equiv  (ty_1,\ldots, ty_{k_0},x_1,\ldots, x_{m-k_0}),\label{e:normal-geodesic}\end{equation} is a {\it normal geodesic} which is orthogonal to $U\subset P$.
Let 
\begin{equation}
\p_{y_1},\ldots, \p_{y_{k_0}}, \p_{x_1},\ldots, \p_{x_{m-k_0}}
\end{equation}
be the induced coordinate vector fields. Then  $\p_{y_\alpha}|_{y=0}=e_\alpha$ when both are viewed as sections of $N$ over $U$. 
For $1\leq i,j\leq m-k_0$
and $1\leq\alpha\leq k_0$, we denote 
\begin{equation}g_{ij}\equiv\langle\p_{x_i}, \p_{x_j}\rangle,\ g_{i\alpha}\equiv\langle\p_{x_i}, \p_{y_\alpha}\rangle, \ g_{\alpha\beta}\equiv\langle\p_{y_\alpha}, \p_{y_\beta}\rangle.\end{equation}
Obviously we have 
\begin{equation}g_{ij}(0, 0)=\delta_{ij},\   g_{\alpha\beta}(x, 0)=\delta_{\alpha\beta},\ g_{i\alpha}(x,0)=0,\end{equation}
 for all $x\in U\subset P$. 
The second fundamental form  of the embedding $U\hookrightarrow\mathcal U$ can be written as $\IIs=\IIs_{ij}^{\alpha} \p_{y_\alpha}\otimes (dx_i\otimes dx_j)$, where 
\begin{equation}\IIs^\alpha_{ij}\equiv\langle\nabla_{\p_{x_i}}e_\alpha, \p_{x_j}\rangle\end{equation}
Denote by $\overrightarrow{H}=H^{\alpha}e_{\alpha}$ the mean curvature vector. Then we have  
$H^\alpha\equiv g^{ij}\IIs^\alpha_{ij}$.

Let us define the function $r$ 
in the normal coordinates as follows, \begin{equation}r(x,y)\equiv|y|=\Big(\sum\limits_{\alpha=1}^{k_0} y_{\alpha}^2\Big)^{\frac{1}{2}}.\label{e:normal-distance-function}\end{equation}
  The following lemma gives an extension of the usual Gauss Lemma.
  
   \begin{lemma}[Generalized Gauss Lemma] \label{l:generalized-Gauss} In the above notations, for any $p\in P$, let $\{x_i, y_{\alpha}\}$ be a normal coordinate system defined in  a neighborhood  $\mathcal{U}$ of $p$ in $Q$ with $U=\mathcal{U}\cap P$.  
Let $r$ be the function defined by \eqref{e:normal-distance-function}. Then $r$ satisfies the following properties in $\mathcal{U}$:  
 \begin{enumerate}
 \item $\nabla r = \partial_r$ holds in $\mathcal{U}\setminus U$. In particular,  $r$ is the normal distance function
  in $\mathcal{U}$, i.e., $r(q)=d(q, P)$ for all $q\in \mathcal{U}$.

 \item $\partial_r$ is orthogonal to $\p_{x_i}$'s   in  $\mathcal{U}\subset Q$, and hence \begin{equation}\label{Gausslemma}
\langle\partial_{x_i},r\partial_r\rangle=\langle \partial_{x_i}, \sum_{\alpha=1}^{k_0} y_\alpha \p_{y_\alpha}\rangle=\sum_{\alpha=1}^{k_0} y_\alpha g_{i\alpha}=0, \ 1\leq  i\leq m-k_0.
\end{equation}

 \end{enumerate}
\end{lemma}
The proof is by straightforward computations, so we omit the details. 

For  the convenience of later discussions, we introduce several notations concerning the {\it normal  regularity order} near the submanifold $P$. It will be frequently used throughout the paper. 
\begin{definition}[Normal regularity order] \label{d:normal-regularity}
Let  $T(x,y)$ be a tensor field locally defined on $\mathcal{U}$ which is $C^{\infty}$ on  $\mathcal U\setminus U$. Then for every $k\in\dN$, we define the following notion as $r\rightarrow 0$: 
\begin{enumerate}
\item
$T(x,y)=O'(r^k)$  if for each $\epsilon>0$
 \begin{align}
 \Big|\p_t^I\p_{n}^{J}T(x,y)\Big|=\begin{cases}
 O(1),  & |J|\leq k-1,\\
 O(r^{k-|J|-\epsilon}), & |J|\geq k,
 \end{cases}\end{align}
 for all multi-indices $I$ and $J$. Here we work on the normal coordinate system and take the usual derivatives on the coefficients of $T$. The tangential derivative $\p_t$ denotes one of the $\p_{x_j}$'s, and the normal derivative $\p_n$ denotes one of the $\p_{y_\alpha}$'s. 

\item $T(x,y)=r^kO'(1)$
 if $r^{-k}T(x,y)=O'(1)$. 
 \item 
$T(x,y) = \tO(r^k)$ if $T(x,y)\in C^{\infty}(\mathcal{U})$ and 
$T(x,y)=r^kO'(1)$. In other words,  $T(x,y)$ is smooth on $\mathcal{U}$
and has vanishing normal derivatives along $U$ up to order $k-1$.

\end{enumerate}
\end{definition}

 Notice that that defining condition does not depend on the choice of the local coordinates $\{y_\alpha, x_i\}$, as long as they satisfy that $y_\alpha=0$ along $P$ for $1\leq\alpha\leq k_0$. We also give the following instructive example. 
 \begin{example} \label{ex:regularity} In the above notations,  we have the following: 
 \begin{enumerate} 
 \item If $T=O'(1)$, then $T \in L^q(\mathcal{U})$ for any $q>1$.
 \item  If $T\in C^{\infty}(\mathcal{U})$, then  $T=O'(r^k)$ for all $k\in\dZ_+$.
 \item  If $T=r^k O'(1)$ for some $k\in\dZ_+$, then $T=O'(r^k)$ and $T$ vanishes along $U$. 

 \item  For any $k\in\dZ_+$ and $\alpha\in(0,1)$, if $T=O'(r^k)$, then $T$ is $C^{k-1, \alpha}(\mathcal{U})$.

 \end{enumerate}

 \end{example}

\begin{remark}
The purpose of introducing this notation is because we will frequently meet the terms such as $r^{-j}\cdot y_{\alpha_1}\cdots y_{\alpha_l}$ which naturally lie in $O'(r^{l-j})$. The reason to allow the defect of $\epsilon$ is because such tensors are well-behaved under the $W^{2,q}$-elliptic estimate in Lemma \ref{l:higher-regularity}. 
\end{remark}

The following metric tensor expansion will be used in our regularity analysis.
\begin{lemma} \label{l:metric-expansion}
In the above normal coordinates, we have the following expansions of the metric tensor $g$ of $Q$ along the normal directions,
\begin{align}
g_{\alpha\beta}|_{(x,y)}&=\delta_{\alpha\beta}-\frac{1}{3}\Rm_{\alpha\gamma\xi\beta}\Big|_{(x,0)}y_{\gamma} y_{\xi} +\tO(r^3),
\\
g_{ij}|_{(x, y)}&=g^P_{ij}(x)+2\IIs^\alpha_{ij}\Big|_{(x,0)} y_{\alpha} -(\Rm_{i\gamma \xi j}   -  \langle \nabla_{\p_{x_i}} \p_{y_\gamma}, \nabla_{\p_{x_j}} \p_{y_\xi}\rangle)\Big|_{(x,0)}y_\gamma y_\xi+\tO(r^3),
\\
g_{i\alpha}|_{(x,y)}&=\langle \nabla_{\p_{x_i}}\p_{y_\gamma}, \p_{y_\alpha} \rangle\Big|_{(x,0)} y_\gamma -\frac{2}{3}\Rm_{i\gamma\xi\alpha}\Big|_{(x,0)}y_{\gamma}y_{\xi}  + \tO(r^3), 
\end{align}
where $g^P=(g^P_{ij})$ is the restriction of the metric $g$ to $U\subset P$,  $\Rm$ is the Riemann curvature tensor of $g$.

\end{lemma}

\begin{proof}
This can be proved using the Jacobi fields. For any fixed point $q=(0^{k_0},x_1,\ldots, x_{m-k_0})\in U\subset P$, we choose a unit vector $v=\sum\limits_{\alpha=1}^{m-k_0}v_{\alpha}\p_{y_{\alpha}}\in N(q)\cong \R^{k_0}$ with $|v|=1$. Let $\vartheta$ be the  radial geodesic in $\mathcal{U}$ 
\begin{equation}\vartheta(t)\equiv \Exp_P(q,tv) = \Exp_q(tv).\end{equation}
In particular $\vartheta(0)=q$, and $\vartheta'(0)=v$. In the normal coordinates, the geodesic $\vartheta$ can  represented as 
$\vartheta(t)=(tv_1,\ldots, tv_{k_0}, x_1,\ldots,x_{m-k_0})$. 

For each $1\leq \alpha\leq k_0$ and $1\leq i\leq m - k_0$, we define the geodesic variations
\begin{align}
\sigma_{\alpha}(t,s) &\equiv ((tv_1,\ldots, t(v_{\alpha}+s), \ldots, tv_{k_0}, x_1,\ldots,x_{m-k_0}),  
\\
\sigma_i(t,s) &\equiv (tv_1,\ldots, tv_{k_0}, x_1, \ldots, x_i + s, \ldots x_{m-k_0}).
\end{align}
Then the variation fields of $\sigma_{\alpha}(t,s)$ and $\sigma_i(t,s)$ at $s=0$ give 
the following Jacobi fields along the radial geodesic $\vartheta(t)$ respectively:
\begin{align}
\begin{cases}
J_{\alpha}(t)=t\cdot  \p_{y_{\alpha}}, & 1\leq  \alpha\leq k_0
\\
J_{i}(t)= \p_{x_i}, & 1\leq i \leq m- k_0.
\end{cases}
\end{align}
By definition,
\begin{equation}
J_{\alpha}(0)=0,\ J_i(0)=\p_{x_i}.
\end{equation}
Taking covariant derivatives at $t=0$ we get
\begin{equation}
J_{\alpha}'(0) = \p_{y_{\alpha}},\ J_i'(0)=v_{\alpha}\nabla_{\p_{x_i}}\p_{y_{\alpha}}. 
\end{equation}
Then applying the Jacobi equation along the geodesic $\vartheta$
\begin{align}
 \begin{cases}
 J_{\alpha}''+\Rm(J_{\alpha},\vartheta')\vartheta'=0,\\
  J_{i}''+\Rm(J_{i},\vartheta')\vartheta'=0,\\
 \end{cases}
 \end{align}
 where $\Rm(X,Y)Z\equiv \nabla_X\nabla_YZ - \nabla_Y\nabla_X Z - \nabla_{[X,Y]}Z$ is the Riemann curvature tensor of $g$, we obtain 
\begin{align}
J''_\alpha(0)=0, \ J'''_\alpha(0)=-v_\gamma v_\xi \Rm(\p_{y_\alpha},\p_{y_\gamma})\p_{y_\xi},\\
J''_i(0)=-v_\gamma v_\xi \Rm(\p_{x_i},\p_{y_\gamma})\p_{y_\xi}.
\end{align}
Therefore, 
\begin{align}g_{\alpha\beta}&=t^{-2}g(J_\alpha, J_\beta)=\delta_{\alpha\beta}-\frac{1}{3}\Rm_{\alpha\gamma\xi\beta}v_\gamma v_\xi t^2+\tO(t^3),
\\
g_{ij}&=g(J_i, J_j)=g^P_{ij}+2 v_\alpha\IIs_{ij}^\alpha t-\Big(\Rm_{i\gamma\xi j}  - \langle \nabla_{\p_{x_i}} \p_{y_\gamma}, \nabla_{\p_{x_j}} \p_{y_\xi}\rangle\Big)v_\gamma v_\xi t^2+\tO(t^3),
\\
g_{i\alpha}&=t^{-1}g(J_i, J_\alpha)= \langle \nabla_{\p_{x_i}}\p_{y_\gamma}, \p_{y_\alpha}\rangle v_\gamma t-\frac{2}{3}\Rm_{i\gamma\xi\alpha}v_{\gamma}v_{\xi}t^2+\tO(t^3).\end{align}
Notice $y_{\alpha} = tv_{\alpha}$, and the conclusion follows.
\end{proof}

As a digression we briefly discuss the intrinsic meaning of the above expansion. The point is that locally the Riemannian metric $g$ is approximated by a Riemannian metric $g_N$ on the normal bundle $N$ up to the first order. Notice we have the natural projection $\pi: N\rightarrow P$ and $N$ is a Riemannian vector bundle together with an induced ``normal" connection, given by the normal component of the Levi-Civita connection of $g$. The latter hence gives rise to a distribution of horizontal subspaces at each point of $N$, which in our coordinates is spanned by $\p_{x_i}-A_{i\alpha\beta}y_\beta \p_{y_\alpha}$, where 
\begin{equation} \label{eqn2-10}
A_{i\alpha\beta}(x)\equiv\langle \nabla_{\p_{x_i}}\p_{y_\beta}, \p_{y_\alpha}\rangle|_{y=0}
\end{equation}
is a smooth function on $U$. 
 We define $g_N$ so that at each point of $N$, the vertical and horizontal subspaces are orthogonal and on the vertical part is given by the bundle metric on $N$, and on the horizontal part is given by the perturbation of the base metric $g_P$ using the second fundamental form. In this way we get a coordinate free description of the above expansion up to the first order.
 
 We also define 
\begin{equation}A_{ij\alpha\beta}\equiv \frac{1}{2}(\p_{x_i} A_{j\alpha\beta}-\p_{x_j}A_{i\alpha\beta}).\end{equation}
Then the curvature of the normal connection is given by 
\begin{equation}\Omega_{ij\alpha\beta}\equiv A_{ij\alpha\beta}+\frac{1}{2}(A_{i\alpha\gamma}A_{j\gamma\beta}-A_{i\beta\gamma}A_{j\gamma\alpha}).\end{equation}

\subsection{Green's currents for Riemannian submanifolds} \label{ss:real-case}

First we  recall and introduce the basic terminology. Let $(Q, g)$ be an oriented Riemannian manifold of dimension $m$.  Denote by $\Omega^l_0(Q)$  the space of differential $l$-forms with compact support in $Q$. 
\begin{definition}
[$k$-current] A $k$-current on $Q$ is a linear functional \begin{equation}
 T: \Omega^{m-k}_0(Q)\to \dR, \quad  \chi\mapsto (T, \chi),
 \end{equation}
which is continuous in the sense of distributions, i.e., suppose that $\chi_j\in \Omega_0^{m-k}(Q)$ is a sequence of differential forms with all derivatives uniformly converging to $0$ as $j\to\infty$, then  
$\lim\limits_{j\to\infty}(T, \chi_j)= 0$.  
\end{definition}
The notion of currents unifies the notion of differential forms and submanifolds. In particular, a locally integrable $k$-form $\beta$ can be naturally viewed as a $k$-current via the pairing
\begin{equation}
(\beta, \chi)\equiv \int_Q\beta\wedge\chi, \ \ \ \ \chi\in \Omega^{m-k}_0(Q), 
\end{equation}
and an oriented submanifold $P$ of \emph{codimension}-$k$ also defines a $k$-current $\delta_P$ via 
\begin{equation}
(\delta_P, \chi)\equiv  \int_P \chi, \ \ \ \ \chi\in \Omega^{m-k}_0(Q).
\end{equation}

The usual exterior differential operator $d$ and the Hodge star operator $*$ acting on differential forms naturally extend to currents. Given a $k$-current $T$ and $\chi\in\Omega_0^{m-k}(Q)$, then we define 
\begin{align}
(dT, \chi) & \equiv  (-1)^{k+1}(T, d\chi), 
\\(*T, \chi) & \equiv (-1)^{k(m-k)}(T, *\chi).\end{align}
Let $d^*$  be the codifferential operator and denote by $\DelH \equiv dd^*+d^*d$ the Hodge Laplacian, then it follows that for every $k$-current $T$ and $\chi\in \Omega_0^{m-k}(Q)$,
\begin{align}
(d^*T, \chi)& =(-1)^k(T, d^*\chi),
\\ 
(\DelH T, \chi)&=(T, \DelH \chi).\end{align}
A $k$-current $T$ is called {\it harmonic} if $\DelH T=0$. It follows from the standard elliptic regularity theory that a harmonic $k$-current can be represented by a smooth harmonic $k$-form.

Now let $P\subset Q$ be a (not necessarily closed)  embedded oriented  submanifold. Although the following discussion applies to more general setting, for our purpose in this paper we will only consider the case when $P$ is of codimension-$3$ in $Q$. The importance of the codimension-$3$ case  in our setting is related to the fact that there is a Hopf fibration $\R^4\rightarrow \R^3$ which is a singular $S^1$-fibration with a smooth total space and a codimension-3 discriminant locus on the base. The codimension-3 condition also appears in other geometric settings, for example, Hitchin's theory of Gerbes \cite{Hitchin}.

\begin{definition}[Green's current] Suppose that $P$ is of codimension-3 in $Q$. 
A Green's current $G_P$ for $P$ in $Q$ is a locally integrable $3$-form which  solves the following current equation on $Q$,
\begin{equation}\label{Greencurrent}
\Delta G_P=2\pi \cdot \delta_P.
\end{equation}
\end{definition}
\begin{example}The above normalization constant is chosen such that in the case $Q\equiv \dR^3$ and $P\equiv 0^3\in\dR^3$, then 
\begin{equation}G_P=\frac{1}{2|y|}dy_1\wedge dy_2\wedge dy_3\end{equation}
solves the current equation $\Delta_0 G_P=2\pi\cdot \delta_P$ for the standard Hodge Laplacian $\Delta_0$ on $\dR^3$. \end{example}

In particular, $G_P$ is harmonic  and hence smooth outside $P$. Notice that a Green's current $G_P$ for $P$ is unique up to the addition of a harmonic $3$-form. So the singular behavior near $P$ does not depend on the particular choice of $G_P$. Also it is clear that if $Q'\subset Q$ is an open submanifold, then the restriction of $G_P$ to $Q'$ is a Green's current for $P'=P\cap Q'$ in $Q'$, so that we can study the regularity problem in a local fashion.  Our goal of this subsection is to understand the local existence and regularity  of $G_P$ via the approximation by the standard model, i.e., the product Euclidean space $\R^3\times \R^{n-3}$.    

To begin with, we have the following simple regularity result for $d(G_P)$. 

\begin{proposition}\label{regularityd}
Given a Green's current $G_P$ for $P$ in $Q$,  its differential $d(G_P)$ extends to a smooth $4$-form across $P$. 
\end{proposition}
\begin{proof}
This is a local result so we can work with the geodesic ball $B_r(p)$ for any $p\in P$ such that $\overline{B_r(p)}\subset \subset Q$ and $\overline{B_r(p)}\cap P \subset \subset P$. Given  any test form $\chi\in \Omega^{m-4}_0(B_r(p))$ we have 
\begin{equation}(\Delta (d(G_P)), \chi)=(d\Delta G_P, \chi)=(\Delta G_P, d\chi)=(2\pi \delta_P, d\chi)=2\pi\int_P  d\chi=2\pi\int_{\p B_r(p)\cap P} \chi=0. \end{equation}
Therefore, $d(G_P)$ is a harmonic $4$-current in $B_r(p)$ and hence it is smooth in $B_r(p)$. 
\end{proof}

From now on, let us fix a point $p\in P$ and let  $\{x_j, y_\alpha\}$ be a normal coordinate system for $P$ in a neighborhood $\mathcal U$ of $p$ in $Q$ with $U=\mathcal U\cap P$. We also need the following notation.
\begin{notation}
\label{n:differential-form-notation}  Given $k\in\dZ_+$ and $0\leq q\leq 3$, the notations $\Pi_q^{(k)}$   always represents a general $q$-form for, which is defined in a neighborhood of $p$ in $Q$, and can be expressed in normal coordinates as
\begin{equation}\Pi_q^{(k)}\equiv A_{i,\alpha_1,\ldots, \alpha_{q-1}}^{(k)}(x,y)(dy_{\alpha_1}\wedge\ldots  \wedge dy_{\alpha_{q-1}})\wedge dx_i + B_{\alpha_1,\ldots, \alpha_q}^{(k)}(x,y)dy_{\alpha_1}\wedge\ldots\wedge  dy_{\alpha_q},\label{e:3-form-polynomial-coe}\end{equation}
where $A_{i,\alpha_1,\ldots, \alpha_{q-1}}^{(k)}(x,y)$ and $B_{\alpha_1,\ldots, \alpha_q}^{(k)}(x,y)$ are  homogeneous polynomial functions in $y$ of degree $k$ with coefficient functions smooth on $U\subset P$.\end{notation}

\begin{remark}
 Notice that this expression depends on the choice of local coordinates, but under a change of coordinates, a $p$-form $\Pi_p^{(k)}$ still yields such an expression, modulo a term of order $\tO(r^{k+1})$. 
\end{remark}

Now we are ready to state the first main theorem of this section, which gives a local existence for Green's current, and describes its leading singular behavior. 

\begin{theorem} \label{t:Green-expansion}  Given the above oriented codimension-3 submanifold $P$ in $Q$. For any $p\in P$, shrinking $\mathcal U$ if necessary. Then there exists a  Green's current $G_{U}$ for $U=\mathcal U\cap P$ which satisfies 
\begin{align}
\Delta G_U = 2\pi\cdot \delta_U \quad \text{in} \ \mathcal{U},
\end{align}
and in normal coordinates we have the local expansion 
\begin{align} \label{e:Green expansion}
G_U 
=&\frac{1}{2r}(1-\frac{H^\alpha y_\alpha}{2}) dy_1\wedge dy_2\wedge dy_3+\frac{1}{2r}y_{\beta}A_{i\alpha\beta}dx_i \wedge dy_{\widehat{\alpha}}
-\frac{1}{4} A_{ij\alpha\beta}r\cdot dy_{\widehat{\alpha\beta}}\wedge dx_i\wedge dx_j\nonumber
\\+& \frac{3}{16}(A_{i\alpha, \alpha+1}A_{j\alpha, \alpha+2}-A_{i\alpha, \alpha+2}A_{j\alpha, \alpha+1})d(ry_\alpha)\wedge dx_i\wedge dx_j
+r^{-3}\Pi_3^{(4)}+O'(r^2),\end{align}
 where $\overrightarrow{H}=H^{\alpha}e_{\alpha}$ is the mean curvature of $P\subset Q$ and the $3$-form $\Pi_3^{(4)}$ is defined in \eqref{e:3-form-polynomial-coe}.

\end{theorem}

Here and in the following we use the notation that for $\alpha\in \{1, 2, 3\}$, $\widehat{\alpha}=(\alpha+1, \alpha+2)$ (with the convention $3+1=1$), $dy_{\alpha\beta}\equiv dy_\alpha\wedge dy_\beta$, and that
\begin{align}
dy_{\widehat{\alpha\beta}}\equiv\begin{cases}
dy_{\alpha+2},   & \beta=\alpha+1,\\
-dy_{\alpha+1}, & \beta=\alpha+2,\\
0, & \beta=\alpha.
\end{cases}
\end{align}

\begin{remark}
Any Green's current $G_P$ for $P$ yields the same singular expansion \eqref{e:Green expansion} near $p\in P$.
\end{remark}

\begin{remark}
Theorem \ref{t:Green-expansion} is certainly not optimal. Our choice of order of expansion is dictated by our applications, and the above exactly suits our purposes; see Proposition \ref{p:complex-Green-expansion},  Remark \ref{r:C2alpharegularity} and Remark \ref{r:C1alphavsC2alpha}. In general one would expect a ``poly-homogeneous" expansion.
Similarly
in the proof we will not keep track of the explicit form of $\Pi_3^{(4)}$ because it is not needed in our applications. However, it is possible to obtain the precise expression with more work. Given the above expansion, there are also some constraint for $\Pi_3^{(4)}$ following from the fact that $d(G_{U})$ is smooth by Proposition \ref{regularityd}.
\end{remark}

Before starting the proof of Theorem \ref{t:Green-expansion}, we need some preliminary computations. For the convenience of our calculations,   we introduce three $1$-forms 
\begin{equation}\eta_\alpha \equiv dy_\alpha+p_{i\alpha}dx_i \label{e:definition-of-eta-alpha}\end{equation} such that 
\begin{equation}\label{eqn3-29}
\langle\eta_\alpha, dx_j\rangle=0\end{equation}
for all $j$ and $\alpha$ at all points of $\mathcal U$.  Then the linear span of the $\eta_\alpha$'s is orthogonal to the linear span of the $dx_i$'s. 

The following elementary lemma is  crucial in the proof of Theorem \ref{t:Green-expansion}.

\begin{lemma} \label{l:coe-crossing} The coefficient $p_{i\alpha}$ in \eqref{e:definition-of-eta-alpha} has the expansion
 \begin{equation}p_{k\alpha}=g_{k\alpha}+\tO(r^3), \quad 1\leq k\leq m-3,\ 1\leq \alpha\leq 3.
 \end{equation}
\end{lemma}

\begin{proof}
We write the full matrix expression of the metric $g$ as
\begin{equation}
   g=  \left[ {\begin{array}{cc}
   g_{\alpha\beta}& 0 \\
  0 & g_{ij} \\
  \end{array} } \right]+
  \left[ {\begin{array}{cc}
   0 & S \\
   S^t & 0 \\
  \end{array} } \right],
\end{equation}
where $S=(g_{\alpha i})=\tO(r)$. We denote by $(h_{\alpha\beta})$ and $(h_{ij})$ the inverse matrix of $(g_{\alpha\beta})$ and $(g_{ij})$ respectively. 
Then by elementary consideration
\begin{eqnarray}
   g^{-1} &=& \left[ {\begin{array}{cc}
  h_{\alpha\beta}& 0 \\
  0 & h_{ij} \\
  \end{array} } \right]-\left[ {\begin{array}{cc}
   h_{\alpha\beta}& 0 \\
  0 & h_{ ij} \\
  \end{array} } \right]
  \left[ {\begin{array}{cc}
   0 & S \\
   S^t & 0 \\
  \end{array} } \right]\left[ {\begin{array}{cc}
   h_{\alpha\beta}& 0 \\
  0 & h_{ij} \\
  \end{array} } \right]
  \nonumber\\
  &+&  \left[ {\begin{array}{cc}
 h_{\alpha\beta}& 0 \\
  0 & h_{ij} \\
  \end{array} } \right] \left[ {\begin{array}{cc}
   0 & S \\
   S^t & 0 \\
  \end{array} } \right]\left[ {\begin{array}{cc}
   h_{\alpha\beta}& 0 \\
  0 & h_{ ij} \\
  \end{array} } \right]
  \left[ {\begin{array}{cc}
   0 & S \\
   S^t & 0 \\
  \end{array} } \right]\left[ {\begin{array}{cc}
   h_{\alpha\beta}& 0 \\
  0 & h_{ ij} \\
  \end{array} } \right]
+\tO(r^3).
\end{eqnarray}
Notice that the third term does not have off-diagonal contributions, so the matrix $g^{-1}=(g^{IJ})$ satisfies
\begin{align}g^{i\alpha} &=-h_{ij}g_{j\beta}h_{\beta\alpha}+\tO(r^3)
\label{e:inverse-crossing}
\\
 g^{ij} &= h_{ij}+\tO(r^2),\label{eqn3-35}
\\
g^{\alpha\beta} &= h_{\alpha\beta}+\tO(r^2)=\delta_{\alpha\beta}+\tO(r^2),\label{e:inverse normal}
\end{align}
where we used  Lemma \ref{l:metric-expansion}.
By \eqref{eqn3-29} we have
$p_{i\alpha}g^{ij}+g^{j\alpha}=0$. 
Let $(\hat{g}_{ij})$ be the inverse of the matrix $(g^{ij})$
with $1\leq i,j\leq m-3$, i.e., $g^{ij}\hat{g}_{jk}=\delta_{ik}$. Then
\begin{equation}
p_{k\alpha} = -g^{j\alpha}\hat{g}_{jk}.\label{e:coe-coor}
\end{equation}
We claim that for any $1\leq i,j \leq m-3$,
\begin{equation}
\hat{g}_{ij} - g_{ij} =\tO(r^2).
\end{equation}
In fact, since $
g^{ij}g_{jk} + g^{i\alpha}g_{\alpha k} =\delta_{ik}$, 
 we get 
\begin{equation}
\hat{g}_{li}(g^{ij}g_{jk} + g^{i\alpha}g_{\alpha k})=\hat{g}_{lk}.\end{equation}
So this implies that
\begin{equation}
\hat{g}_{lk} - g_{lk} = \hat{g}_{li} g^{i\alpha} g_{\alpha k} = \tO(r^2), \label{e:db-inverse-order}
\end{equation}
which proves the claim.
Therefore, combining  \eqref{e:inverse-crossing},\eqref{e:inverse normal}, \eqref{e:coe-coor} and  \eqref{e:db-inverse-order}, we obtain
\begin{eqnarray}
p_{k\alpha}&=& 
-(g_{kj}+\tO(r^2))g^{j\alpha}
\nonumber\\
&=& -g_{kj}g^{j\alpha} + \tO(r^3)
\nonumber\\
&=&h_{\alpha\beta}g_{k\beta}+\tO(r^3)
\nonumber\\
&=&g_{k\alpha}+\tO(r^3).\end{eqnarray}
\end{proof}

The following symmetry property of $p_{i\alpha}$ will also frequently used in our calculations below.
By Lemma \ref{l:metric-expansion} and Lemma \ref{l:coe-crossing}, we may write 
\begin{equation} \label{e:symmetry-of-coefficient-p}
p_{i\alpha}=A_{i\alpha\beta}y_\beta+\frac{1}{2}B_{i\alpha\beta\gamma} y_\beta y_\gamma+\tO(r^3).
\end{equation}
Here the connection term  $A_{i\alpha\beta}\equiv \langle\nabla_{\p_{x_i}}\p_{y_{\beta}}, y_{\alpha}\rangle|_{(x,0)}$
is skew-symmetric in $\alpha$, $\beta$, and the curvature term  $B_{i\alpha\beta\gamma}\equiv -\frac{2}{3}(R_{i\beta\gamma\alpha}+R_{i\gamma\beta\alpha})$ is symmetric in $\beta$, $\gamma$.

\begin{lemma}
For every $1\leq \alpha,\beta\leq 3$,
\begin{equation} \label{etaorthogonal}
\langle \eta_\alpha, \eta_\beta\rangle=\delta_{\alpha\beta}+\tO(r^2).
\end{equation}
\end{lemma}
\begin{proof}
By definition 
$\langle \eta_\alpha, \eta_\beta\rangle=\langle dy_\alpha+p_{i\alpha} dx_i, dy_\beta+p_{j\beta}dx_j\rangle$. By \eqref{e:inverse normal} we get
\begin{equation}\langle dy_\alpha, dy_\beta\rangle=\delta_{\alpha\beta}+\tO(r^2).\end{equation}
Also we have $p_{i\alpha}=\tO(r)$ and $\langle dx_i, dy_\beta\rangle=\tO(r)$ for all $i$ and $\alpha$. The conclusion then follows.
\end{proof}

Using the above differential $1$-forms $\eta_{\alpha}$'s, we can decompose the volume form $\dvol_g$ in the horizontal and vertical directions, which will substantially simplify the computations regarding the Hodge Laplacian. 
The volume form of $g$ is given by 
\begin{equation}\dvol_g=\sqrt{\det(g)}\cdot dy_1\wedge dy_2\wedge dy_3\wedge dx_1\wedge \cdots\wedge dx_{m-3},\end{equation}
where we have used the orientation fixed above. 
We define  the normal and tangential volume forms by 
\begin{eqnarray}\begin{cases}
{\dvol}_N \equiv \eta_1\wedge \eta_2\wedge \eta_3\\
{\dvol}_T \equiv \sqrt{\det(g_{ij}^P)}\cdot dx_1\wedge \cdots\wedge dx_{m-3}.
\end{cases}
\end{eqnarray}
By the expansion formula \eqref{e:symmetry-of-coefficient-p}, the normal volume form $\dvol_N$ has the following expansion,
\begin{eqnarray} \label{e:dvol-N-expansion}
{\dvol}_N=dy_1\wedge dy_2\wedge dy_3+\Big(A_{i\alpha\beta} y_{\beta}+\frac{1}{2}B_{i\alpha\beta\gamma}y_\beta y_\gamma \Big) dx_i\wedge  dy_{\widehat\alpha}\nonumber \\
+A_{i,\alpha+1,\beta}A_{j, \alpha+2, \gamma} y_\beta y_\gamma dy_{\alpha}\wedge dx_i\wedge dx_j+
\tO(r^3).
\end{eqnarray}
In addition, by the definition of $\eta_{\alpha}$'s, it holds that for each $1\leq \alpha \leq 3$,
\begin{equation}
\eta_{\alpha} \wedge \dvol_T = dy_{\alpha} \wedge \dvol_T,
\end{equation}
and hence
\begin{equation}
\dvol_N\wedge \dvol_T = dy_1\wedge dy_2\wedge dy_3 \wedge \dvol_T =  \frac{\sqrt{\det(g_{ij}^P)}}{\sqrt{\det(g)}}\cdot \dvol_g.
\end{equation}

\begin{lemma}\label{l:v-T-N}Denote by $*:\Omega^k(Q)\to\Omega^{m-k}(Q)$ the Hodge $*$-operator. Then we have the following:
\begin{enumerate}
\item The tangential volume form $\dvol_T$ satisfies
\begin{equation}
*\dvol_T=(-1)^{m+1}\dvol_N\cdot (1-H^{\alpha}y_{\alpha}+\tO(r^2)).
\end{equation}
\item For any $\alpha\in\{1,2,3\}$,
\begin{equation}
*(\eta_{\alpha}\wedge \dvol_T) = \lambda_{1}\eta_{\widehat{\alpha}} + \lambda_{2}\eta_{\widehat{\alpha+1}} + \lambda_{3}\eta_{\widehat{\alpha+2}},
\end{equation}
where $\lambda_1=1-H^{\alpha}y_{\alpha}+\tO(r^2)$, $\lambda_2=\tO(r^2)$, $\lambda_3=\tO(r^2)$.
\end{enumerate}

\end{lemma}
\begin{proof}
First, we prove item (1). By \eqref{eqn3-29} we have 
\begin{equation}(-1)^{m+1}*\dvol_T=\lambda\cdot\dvol_N\end{equation}
for a function $\lambda>0$. The function $\lambda$ is given by 
 \begin{equation}\lambda=\frac{|\dvol_T|^2\sqrt{\det(g)}}{\sqrt{\det(g_{ij}^P)}}=\sqrt{\det(g)}\det(g^{ij})\sqrt{\det(g_{ij}^P)}.\label{e:coe-def}\end{equation}
Now we compute the expansion of $\lambda$. Applying the expansions of $g_{ij}$, $g_{\alpha\beta}$ and $g_{i\alpha}$ in Lemma \ref{l:metric-expansion}, one can directly obtain the following,
 \begin{align}\det(g)&=\det(g_{\alpha\beta})\cdot\det(g_{ij})+\tO(r^2),\label{e:v}
 \\\det(g_{\alpha\beta})&=1+\tO(r^2),\label{e:vertical-v} 
 \\\det(g_{ij})&=\det(g_{ij}^P)\cdot(1+2H^\alpha y_\alpha)+\tO(r^2).\label{e:horizontal-v}\end{align}
Plugging \eqref{e:vertical-v} and \eqref{e:horizontal-v} into \eqref{e:v},
\begin{equation}
\det(g) = \det(g_{ij}^P)\cdot (1 + 2H^{\alpha} y_{\alpha}) + \tO(r^2).\label{e:g-expansion}
\end{equation} 
Let $(h_{ij})$ be the inverse of the matrix $(g_{ij})$.  Since $g^{ij}=h_{ij} + \tO(r^2)$ by (\ref{eqn3-35}), so it follows that
 \begin{equation}\det(g^{ij})=\det(h_{ij})+\tO(r^2)=(\det(g_{ij}))^{-1}+\tO(r^2).
 \end{equation}
 Plugging \eqref{e:horizontal-v} into the above, 
 \begin{equation}\det(g^{ij})=\det(g_{ij}^P)^{-1}\cdot (1-2H^{\alpha}y_{\alpha})+\tO(r^2).\label{e:inv-expansion}\end{equation}
Therefore, substituting  \eqref{e:g-expansion} and \eqref{e:inv-expansion} into \eqref{e:coe-def}, 
\begin{equation}
\lambda = 1 - H^{\alpha} y_{\alpha} + \tO(r^2),
\end{equation}
which completes the proof of item (1).

Now we prove item (2). 
For each $\alpha\in\{1,2,3\}$, we can write 
\begin{equation}
*(\eta_\alpha\wedge \dvol_T)=\lambda_1 \cdot \eta_{\widehat{\alpha}}+\lambda_2\cdot \eta_{\widehat{\alpha+1}}+\lambda_3 \cdot \eta_{\widehat{\alpha+2}}.\end{equation}
Taking point-wise wedge product with $\eta_\alpha\wedge \dvol_T$, and noticing $\eta_{\widehat{\alpha+1}}\wedge \eta_{\alpha}$, $\eta_{\widehat{\alpha+2}}\wedge \eta_{\alpha}$  are both  zero,  we obtain
\begin{equation}
\lambda_1 \cdot \eta_{\widehat{\alpha}} \wedge \eta_\alpha\wedge \dvol_T = (\eta_\alpha\wedge \dvol_T) \wedge  *(\eta_\alpha\wedge \dvol_T).
\end{equation}
Therefore, by \eqref{eqn3-29} and \eqref{etaorthogonal} we get 
\begin{equation}
\lambda_1 =\frac{|\eta_{\alpha}\wedge\dvol_T|^2\sqrt{\det(g)}}{\sqrt{\det(g_{ij}^P)}}=1-H^{\alpha}y_{\alpha}+\tO(r^2),
\end{equation}
Similarly taking wedge product with $\eta_{\alpha+1}\wedge \dvol_T$ and $\eta_{\alpha+2}\wedge \dvol_T$ respectively, and  again by \eqref{etaorthogonal} we obtain
that 
\begin{equation}
\lambda_2=\tO(r^2), \lambda_3=\tO(r^2). 
\end{equation}
These imply that
\begin{equation}
*(\eta_\alpha\wedge \dvol_T) =\lambda_1 \cdot \eta_{\widehat{\alpha}} +\lambda_2 \cdot \eta_{\widehat{\alpha+1}} + \lambda_3 \cdot \eta_{\widehat{\alpha+2}},
\end{equation}
where $\lambda_1 = 1+H^{\alpha}y_{\alpha}+\tO(r^2)$, $\lambda_2 = \tO(r^2)$ and $\lambda_3 = \tO(r^2)$.
\end{proof}

Now we proceed to prove Theorem \ref{t:Green-expansion}. This will be done in several steps. 

\

\noindent {\bf Step 1:}
We start 
by defining a 3-form 
\begin{equation}
\phi_1\equiv(-1)^{m+1}\frac{1}{2r} *\dvol_T. \label{e:def-phi-1}\end{equation}
By item (1) of Lemma \ref{l:v-T-N}, immediately we have 
\begin{align}
\phi_1=\frac{\dvol_N}{2r}\cdot (1-H^{\alpha}y_{\alpha}+\tO(r^2)).\label{e:phi-1-expansion}
\end{align}
Then applying the expansion of $\dvol_N$ in \eqref{e:dvol-N-expansion}, $\phi_1$ has a further expansion, 
\begin{align}\phi_1&=\frac{1-H^\alpha y_\alpha}{2r}dy_1\wedge dy_2 \wedge dy_3+\frac{1}{2r}A_{i\alpha\beta} y_\beta dx_i\wedge dy_{\widehat{\alpha}}\nonumber\\
&+\frac{1}{2r}A_{i,\alpha+1,\beta}A_{j, \alpha+2, \gamma} y_\beta y_\gamma dy_{\alpha} \wedge dx_i\wedge dx_j+r^{-1}\Pi_3^{(2)}+O'(r^2)\nonumber\\
&=\frac{1-H^\alpha y_\alpha}{2r}dy_1\wedge dy_2 \wedge dy_3+\frac{1}{2r}A_{i\alpha\beta} y_\beta dx_i\wedge dy_{\hat\alpha}\nonumber\\
&+\frac{1}{4}(A_{i\alpha, \alpha+1}A_{j\alpha, \alpha+2}-A_{i\alpha, \alpha+2}A_{j\alpha, \alpha+1})y_\alpha \cdot dr\wedge dx_i\wedge dx_j+r^{-1}\Pi_3^{(2)}+O'(r^2),
\label{e:phi1expansion-2}
\end{align}
where $\Pi_3^{(2)}$ is the $3$-form introduced in Notation \ref{n:differential-form-notation} and 
the last step follows from the following lemma:
\begin{lemma}[Re-arrangement Lemma]
\label{l:rearrange}
\begin{equation}
 A_{i,\alpha+1,\beta}A_{j, \alpha+2, \gamma} y_\beta y_\gamma dy_{\alpha} \wedge dx_i\wedge dx_j=\frac{1}{2}(A_{i\alpha, \alpha+1}A_{j\alpha, \alpha+2}-A_{i\alpha, \alpha+2}A_{j\alpha, \alpha+1})y_\alpha \cdot rdr\wedge dx_i\wedge dx_j.
\end{equation}
\end{lemma}
\begin{proof}
First by writing out the terms and re-arranging the subscripts and using the skew symmetry of $A_{i\alpha\beta}$ we get
\begin{eqnarray}
&&A_{i,\alpha+1,\beta}A_{j, \alpha+2, \gamma} y_\beta y_\gamma dy_{\alpha}\nonumber \\
&=&\Big(A_{i, \alpha+1, \alpha}A_{j, \alpha+2, \alpha} y_\alpha^2+A_{i, \alpha+1, \alpha}A_{j, \alpha+2, \alpha+1} y_\alpha y_{\alpha+1}\nonumber\\
&&+A_{i, \alpha+1, \alpha+2}A_{j, \alpha+2, \alpha} y_\alpha y_{\alpha+2} +A_{i, \alpha+1, \alpha+2}A_{j, \alpha+2, \alpha+1} y_{\alpha+1}y_{\alpha+2}\Big)dy_\alpha\nonumber\\
&=&A_{i, \alpha, \alpha+1}A_{j, \alpha, \alpha+2} y_\alpha^2dy_\alpha-A_{i\alpha, \alpha+2} A_{j\alpha, \alpha+1} y_\alpha y_{\alpha+2}dy_{\alpha+2}\nonumber\\
&&-A_{i, \alpha, \alpha+2}A_{j, \alpha, \alpha+1}y_\alpha y_{\alpha+1}dy_{\alpha+1}-A_{i, \alpha, \alpha+1}A_{j, \alpha, \alpha+1}y_\alpha y_{\alpha+1}dy_{\alpha+2}.
\label{eqn3-98}\end{eqnarray}
So we have
\begin{equation} A_{i,\alpha+1,\beta}A_{j, \alpha+2, \gamma} y_\beta y_\gamma dy_{\alpha} \wedge dx_i\wedge dx_j=\frac{1}{2} (A_{i,\alpha+1,\beta}A_{j, \alpha+2, \gamma}-A_{j,\alpha+1,\beta}A_{i, \alpha+2, \gamma}) y_\beta y_\gamma dy_{\alpha} \wedge dx_i\wedge dx_j.\end{equation}
Correspondingly by skew-symmetrizing each term of (\ref{eqn3-98}) with respect to $i$ and $j$, we get
\begin{eqnarray}
&&\frac{1}{2} (A_{i,\alpha+1,\beta}A_{j, \alpha+2, \gamma}-A_{j,\alpha+1,\beta}A_{i, \alpha+2, \gamma})y_\beta y_\gamma dy_\alpha\nonumber\\
&=&\frac{1}{2}(A_{i, \alpha, \alpha+1}A_{j, \alpha, \alpha+2}- A_{i\alpha, \alpha+2}A_{j\alpha, \alpha+1}) y_\alpha (y_\alpha dy_\alpha+y_{\alpha+1}dy_{\alpha+1}+y_{\alpha+2}dy_{\alpha+2})\nonumber\\
&=& \frac{1}{2} (A_{i\alpha, \alpha+1}A_{j\alpha, \alpha+2}-A_{i\alpha, \alpha+2}A_{j\alpha, \alpha+1})y_\alpha \cdot rdr.
\end{eqnarray}
\end{proof}
\

\noindent {\bf Step 2:} In this step will explicitly compute the singular terms in the expansion of $\DelH\phi_1$. Mainly, we will prove the following proposition.

\begin{proposition}
\label{p:Delta-phi-1-expansion}
Let $\phi_1$ be the $3$-form defined in \eqref{e:def-phi-1}, then we have 
\begin{align} 
\Delta\phi_1=& -\frac{H^\alpha y_\alpha}{2r^3}dy_1\wedge dy_2\wedge dy_3-\Omega_{ij\alpha\beta} (\frac{1}{4r} dy_{\widehat{\alpha\beta}} +\frac{1}{4r^2} y_{\widehat{\alpha\beta}}dr )\wedge dx_i\wedge dx_j
\nonumber\\
&+A_{ij\alpha\beta}(\frac{1}{4r^2}y_{\widehat{\alpha\beta}}dr-\frac{1}{4r}dy_{\widehat{\alpha\beta}})\wedge dx_i\wedge dx_j +r^{-5} \Pi_3^{(4)}+O'(1).
\label{e:Delta-phi-1}
\end{align}

\end{proposition}
 
\begin{proof}

The proof consists of two steps. 
The first step focuses on the computation for $d^*d\phi_1$. Starting with the expansion of $\phi_1$ in \eqref{e:phi-1-expansion}, we have \begin{equation}
d\phi_1 = \frac{-1+H_{\alpha}y_{\alpha}+\tO(r^2)}{2r^3} \cdot rdr\wedge \dvol_N +  \frac{1}{2r} d\Big((1-H^{\alpha}y_{\alpha}+{\tO}(r^2)) \dvol_N\Big).\label{e:d-phi-1}
\end{equation}
To deal with the first term, we use Lemma \ref{l:coe-crossing} and \eqref{Gausslemma} in Lemma \ref{l:generalized-Gauss}, then \begin{eqnarray} \label{eqn2-14}
y_\alpha \eta_\alpha&=&y_{\alpha}dy_{\alpha}-y_{\alpha}p_{i\alpha}dx_i\nonumber\\
&=&rdr-y_\alpha g_{i\alpha}dx_i+\tO(r^4)
\nonumber\\
&=&rdr+\tO(r^4),\label{e:gauss-cancel}
\end{eqnarray}
which yields
\begin{equation}
rdr\wedge \dvol_N=(y_\alpha\eta_\alpha)\wedge \dvol_N+\tO(r^4)=\tO(r^4).\label{e:normal-times-vertical-v}
\end{equation}
So it follows that
\begin{equation}
d\phi_1 = \frac{1}{2r} d\Big((1-H^{\alpha}y_{\alpha}+\tO(r^2)) \dvol_N\Big) + O'(r).
\end{equation}
It is easy to see that
\begin{equation}d(\tO(r^2)\dvol_N)=\tO(r^2).\end{equation}
So we obtain 
\begin{equation}
d\phi_1= \frac{1}{2r} d\Big((1-H^{\alpha}y_{\alpha}) \dvol_N\Big) + O'(r).\label{e:d-phi-1-dv-N}
\end{equation}
Next, let us compute the expansion for $d(\dvol_N)$. By definition, \begin{equation}d(\dvol_N)=d(\eta_1\wedge\eta_2\wedge\eta_3)=d\eta_\alpha\wedge\eta_{ \widehat\alpha}.\end{equation}
By \eqref{e:symmetry-of-coefficient-p},
\begin{eqnarray}
\label{detaalpha}
d\eta_\alpha&=&
d(p_{i\alpha})\wedge dx_i
\nonumber\\
&=&A_{i\alpha\beta}dy_{\beta} \wedge dx_i+A_{ji\alpha\beta}  y_{\beta} dx_j\wedge dx_i+B_{i\alpha\beta\gamma}y_{\beta}dy_{\gamma}\wedge  dx_i+\tO(r^2).\end{eqnarray}
So we have 
\begin{equation}d(\dvol_N)=A_{i\alpha\beta}dy_{\beta} \wedge dx_i \wedge \eta_{\widehat{\alpha}}+ y_\beta (A_{ji\alpha\beta}  dx_j\wedge dx_i+B_{i\alpha\beta\gamma} dy_\gamma \wedge dx_i )\wedge dy_{\widehat\alpha}
+\tO(r^2).
\label{e:dv-N-splitting}\end{equation}
Now we need to rearrange the above expansion. Since $A_{i\alpha\beta}$ is skew symmetric in $\alpha$ and $\beta$, we have for $\alpha\in \{1, 2, 3\}$, 
\begin{equation}A_{i\alpha\alpha}=0,\end{equation} so the leading order in the first term vanishes, hence 
\begin{eqnarray}A_{i\alpha\beta}dy_\beta\wedge  dx_i \wedge \eta_{\widehat{\alpha}}
&=&A_{i\alpha\beta}A_{j\mu\gamma}y_{\gamma}dy_\beta \wedge dx_i\wedge dx_j\wedge dy_{\widehat{\alpha\mu}}+\tO(r^2)
\nonumber\\
&=&A_{i\alpha\beta}A_{j\beta\gamma}y_\gamma dy_{\widehat\alpha}\wedge dx_i\wedge dx_j+ \tO(r^2)
\nonumber\\
&=&\frac{1}{2}(A_{i\alpha\beta}A_{j\beta\gamma}-A_{i\gamma\beta}A_{j\beta\alpha})y_\gamma dy_{\widehat\alpha}\wedge dx_i\wedge dx_j+ \tO(r^2).
\end{eqnarray}
Therefore, 
\begin{equation}
d(\dvol_N) =\Omega_{ij\alpha\beta}\cdot y_\beta \cdot dy_{\widehat\alpha}\wedge dx_i\wedge dx_j+B_{i\alpha\beta\alpha}\cdot y_\beta \cdot dy_1\wedge dy_2\wedge dy_3\wedge dx_i+\tO(r^2).\label{e:1-d-N}
\end{equation}
By \eqref{e:dvol-N-expansion} we have
\begin{equation} \label{e:H-d-N}
d(H^{\alpha}y_{\alpha})\wedge \dvol_N=(H^\alpha A_{i\alpha\beta}-\partial_i(H^\beta))\cdot y_{\beta}\cdot dy_1\wedge dy_2\wedge dy_3\wedge dx_i+\tO(r^2).
\end{equation}
Now substituting \eqref{e:1-d-N} and \eqref{e:H-d-N} into \eqref{e:d-phi-1-dv-N},
\begin{align} \label{e:d-phi-1-exp}
d\phi_1
 &= \frac{1}{2r}\Omega_{ij\alpha\beta}\cdot y_\beta \cdot  dy_{\widehat\alpha}\wedge dx_i\wedge dx_j\nonumber\\
 &+\frac{1}{2r}\Big(B_{i\alpha\beta\alpha}-(H^\alpha A_{i\alpha\beta}-\partial_i(H^\beta)\Big)\cdot y_\beta  \cdot dy_1\wedge dy_2\wedge dy_3\wedge dx_i+O'(r).
 \end{align}
Now we need to take $d^*$ of this. Notice that the leading order of $d^*d\phi_1$ can be obtained using the flat Euclidean model, so we obtain
\begin{eqnarray} \label{e:d*dphi1-curvature-to-be-simplified}
d^*d\phi_1&=& \Omega_{ij\alpha\beta} (\frac{1}{2r} dy_{\widehat{\beta\alpha}} -\frac{1}{2r^3} y_\mu y_\beta dy_{\widehat{\mu\alpha}})dx_i\wedge dx_j+r^{-3}\Pi_3^{(2)}+O'(1).
\end{eqnarray}

In our next step, let us compute $dd^*\phi_1$.  First, 
\begin{equation}*\phi_1=\frac{1}{2r}\dvol_T.\end{equation}
Notice that $d(\dvol_T)=0$, so 
\begin{equation}d*\phi_1=-\frac{1}{2r^3} \cdot rdr\wedge \dvol_T.\end{equation}
By \eqref{e:gauss-cancel}, $rdr= y_{\alpha}\eta_{\alpha}+{\tO}(r^4)$, then
\begin{equation}
d*\phi_1=-\frac{y_{\alpha}}{2r^3}{}\eta_{\alpha}\wedge \dvol_T +O'(r).\end{equation}
Applying item (2) of Lemma \ref{l:v-T-N}, 
\begin{equation}*d*\phi_1=-\frac{y_{\alpha}}{2r^3}{}(1-H^{\beta}y_{\beta} + \tO(r^2))\eta_{\widehat{\alpha}} + O'(r).\end{equation} 
So it follows  that
\begin{eqnarray} \label{eqn2-19}
d^*\phi_1
&=&- *d*\phi_1= \frac{y_{\alpha}}{2r^3}{}(1-H^{\beta}y_{\beta})\eta_{\widehat{\alpha}}{} + \frac{\tO(r^2)}{r^3}y_\alpha dy_{\widehat\alpha} + O'(r).
\end{eqnarray}
Taking $d$ and applying Lemma \ref{l:generalized-Gauss},
\begin{align}
dd^*\phi_1=&(1-H^{\beta}y_{\beta})\Big(-\frac{3y_{\alpha}}{2r^5}{} rdr\wedge \eta_{\widehat{\alpha}}+\frac{1}{2r^3} d(y_\alpha \eta_{\widehat{\alpha}})\Big){}
-\frac{H^{\beta}y_{\alpha}}{2r^3}\eta_{\widehat{\alpha}}\wedge dy_{\beta}+r^{-5}\Pi_3^{(4)}+O'(1)
\nonumber\\
=&(1-H^{\beta}y_{\beta})\Big(-\frac{3y_{\alpha}}{2r^5}{} rdr\wedge \eta_{\widehat{\alpha}}+\frac{1}{2r^3} d(y_\alpha \eta_{\widehat{\alpha}})\Big)
-\frac{H^{\alpha}y_{\alpha}}{2r^3}dy_1\wedge dy_2\wedge dy_3 
\nonumber\\
+&r^{-5}\Pi_3^{(4)}+O'(1).\label{e:dd^*}\end{align}
Now we simplify this expression. By \eqref{eqn2-14}, 
\begin{equation}
-\frac{3y_{\alpha}}{2r^5}{} rdr\wedge \eta_{\widehat{\alpha}} =- \frac{3y_\alpha y_\beta}{2r^5} {}\eta_\beta \wedge \eta_{\widehat{\alpha}}+O'(1)=-\frac{3}{2r^3}{}\dvol_N+{O}'(1).\label{e:first-term}\end{equation}
Also
\begin{eqnarray}
\frac{1}{2r^3} d(y_\alpha \eta_{\widehat{\alpha}}) &=& \frac{1}{2r^3}dy_{\alpha}\wedge \eta_{\widehat{\alpha}} +\frac{1}{2r^3}y_{\alpha}d\eta_{\widehat{\alpha}}
\nonumber\\
&=&\frac{1}{2r^3}(\eta_{{\alpha}} - p_{i\alpha} dx_i)\wedge \eta_{\widehat{\alpha}} +\frac{1}{2r^3}y_\alpha (d{\eta_{\alpha+1}}\wedge {\eta_{\alpha+2}}-{\eta_{\alpha+1}}\wedge d{\eta_{\alpha+2}})
\nonumber\\
&=& \frac{3}{2r^3}\dvol_N -\frac{1}{2r^3}(A_{i\alpha\beta}y_\beta+\frac{1}{2}B_{i\alpha\beta\gamma}y_\beta y_\gamma)dx_i\wedge\eta_{\widehat{\alpha}}
\nonumber\\
&+&\frac{1}{2r^3}y_\alpha (d{\eta_{\alpha+1}}\wedge {\eta_{\alpha+2}}-{\eta_{\alpha+1}}\wedge d{\eta_{\alpha+2}})+{O'}(1).
\end{eqnarray}
So it follows that
\begin{align}
dd^*\phi_1 =& -\frac{1}{2r^3}\Big(A_{i\alpha\beta}\cdot y_{\beta}\cdot dx_i\wedge\eta_{\widehat{\alpha}} -y_\alpha (d{\eta_{\alpha+1}}\wedge {\eta_{\alpha+2}}-{\eta_{\alpha+1}}\wedge d{\eta_{\alpha+2}})\Big)
\nonumber\\
-&\frac{H^{\alpha}y_{\alpha}}{2r^3}dy_1\wedge dy_2\wedge dy_3+r^{-5}\Pi_3^{(4)}+O'(1).\label{e:dd*phi1-before-key-cancellation}
\end{align}

Next, we will show a crucial cancellation for the first term of the above $dd^*\phi_1$, which gives a further order improvement.

\begin{lemma}[Cancellation Lemma]\label{l:key-cancellation} The following holds:
\begin{align} 
& A_{i\alpha\beta}\cdot y_{\beta}\cdot dx_i\wedge \eta_{\widehat{\alpha}}
-y_{\alpha} (d{\eta_{\alpha+1}}\wedge{\eta_{\alpha+2}}-{\eta_{\alpha+1}}\wedge d{\eta_{\alpha+2}})
\nonumber\\
=& -A_{ij\alpha\beta} y_\beta y_\mu dy_{\widehat{\mu\alpha}}\wedge dx_i\wedge dx_j
+\Pi_3^{(2)}+{\tO}(r^3).\end{align}
\end{lemma}

\begin{proof}
Directly applying the definition of $\eta_{\alpha}$,  we have
\begin{align} \label{e:cancellation-first-term}
& A_{i\alpha\beta}\cdot y_{\beta}\cdot  dx_i\wedge \eta_{\widehat{\alpha}}\nonumber\\
=& A_{i\alpha\beta} y_\beta dx_i\wedge dy_{\widehat{\alpha}}+ A_{i\alpha\beta} y_\beta y_\gamma (A_{j, \alpha+1, \gamma} dy_{\alpha+2}-A_{j, \alpha+2, \gamma} dy_{\alpha+1})\wedge dx_i\wedge dx_j+\tO(r^3).
\end{align}
 By \eqref{detaalpha},  we get 
\begin{align} \label{e:eta-d-eta}
&y_\alpha (d{\eta_{\alpha+1}}\wedge{\eta_{\alpha+2}}-{\eta_{\alpha+1}}\wedge d{\eta_{\alpha+2}})\nonumber\\
=& y_\alpha (A_{i, \alpha+1, \beta}dy_\beta \wedge dx_i \wedge dy_{\alpha+2}-A_{i, \alpha+2, \beta}dy_\beta \wedge dx_i \wedge dy_{\alpha+1})\nonumber\\
+&y_\alpha y_\gamma (A_{i, \alpha+1, \beta} A_{j, \alpha+2, \gamma}-A_{i, \alpha+2, \beta} A_{j, \alpha+1, \gamma} )dy_\beta\wedge dx_i\wedge dx_j\nonumber\\
+&y_\alpha y_\beta (A_{ij, \alpha+1, \beta} dy_{\alpha+2}-A_{ij, \alpha+2, \beta} dy_{\alpha+1})\wedge dx_i \wedge dx_j\nonumber\\
+&\Pi_3^{(2)}+\tO(r^3).
\end{align}
Rearranging the subscripts of the first groups of terms in \eqref{e:eta-d-eta}, 
\begin{eqnarray}
&&y_\alpha (A_{i, \alpha+1, \beta}dy_\beta \wedge dx_i \wedge dy_{\alpha+2}-A_{i, \alpha+2, \beta}dy_\beta \wedge dx_i \wedge dy_{\alpha+1})\\
&=& y_\alpha (A_{i, \alpha+1, \alpha}dy_\alpha \wedge dx_i \wedge dy_{\alpha+2}-A_{i, \alpha+2, \alpha}dy_\alpha \wedge dx_i \wedge dy_{\alpha+1})\\
&=& y_{\alpha+2} A_{i, \alpha, \alpha+2}dy_{\alpha+2} \wedge dx_i \wedge dy_{\alpha+1}-y_{\alpha+1}A_{i, \alpha, \alpha+1}dy_{\alpha+1} \wedge dx_i \wedge dy_{\alpha+2}\\
&=& A_{i\alpha\beta} y_\beta dx_i \wedge dy_{\widehat{\alpha}},
\end{eqnarray}
which matches the first term of \eqref{e:cancellation-first-term}. As in the proof of Lemma \ref{l:rearrange}, one can see that the second groups of terms in \eqref{e:cancellation-first-term} and \eqref{e:eta-d-eta} are both equal to 
\begin{equation}
 (A_{i\alpha, \alpha+1}A_{j\alpha, \alpha+2}-A_{i\alpha, \alpha+2}A_{j\alpha, \alpha+1})y_\alpha \cdot rdr\wedge dx_i\wedge dx_j. 
\end{equation}
Next, the third group of terms in \eqref{e:eta-d-eta} can be rewritten as follows,  
\begin{eqnarray}
&&y_\alpha y_\beta (A_{ij, \alpha+1, \beta} dy_{\alpha+2}-A_{ij, \alpha+2, \beta} dy_{\alpha+1})\wedge dx_i \wedge dx_j
\nonumber\\
&=&
A_{ij\alpha\beta} y_\beta (y_{\alpha+2} dy_{\alpha+1}-y_{\alpha+1}dy_{\alpha+2})\wedge dx_i\wedge dx_j
\nonumber\\
&=&A_{ij\alpha\beta} y_\beta y_\mu dy_{\widehat{\mu\alpha}}\wedge dx_i\wedge dx_j.
\end{eqnarray}
The conclusion just follows.
\end{proof}

Let us return to the expansion of 
$dd^*\phi_1$ given by 
\eqref{e:dd*phi1-before-key-cancellation}. Applying Lemma \ref{l:key-cancellation}, we have 
\begin{align}
dd^*\phi_1 &= -\frac{H^\alpha y_\alpha}{2r^3}dy_1\wedge dy_2\wedge dy_3+\frac{1}{2r^3} A_{ij\alpha\beta} y_\beta y_\mu dy_{\widehat{\mu\alpha}}\wedge dx_i\wedge dx_j\nonumber\\
&+r^{-5} \Pi_3^{(4)}+O'(1). \label{e:dd*phi1-A-to-be-simplified}
\end{align}

To further simplify,   we need the following lemma.  
\begin{lemma}[Re-arrangement Lemma] The following identities hold:
\begin{align}
\Omega_{ij\alpha\beta} (\frac{1}{2r} dy_{\widehat{\beta\alpha}} -\frac{1}{2r^3} y_\mu y_\beta dy_{\widehat{\mu\alpha}}) &=-\Omega_{ij\alpha\beta} (\frac{1}{4r} dy_{\widehat{\alpha\beta}} +\frac{1}{4r^2} y_{\widehat{\alpha\beta}}dr ), \label{e:curvature-simplification}
\\
\frac{1}{2r^3} A_{ij\alpha\beta} y_\beta y_\mu dy_{\widehat{\mu\alpha}}&=A_{ij\alpha\beta}(\frac{1}{4r^2}y_{\widehat{\alpha\beta}}dr-\frac{1}{4r}dy_{\widehat{\alpha\beta}}).\label{e:connection-simplification}
\end{align}
\end{lemma}
\begin{proof}We only prove \eqref{e:curvature-simplification} because the other equality follows from the same computations. 
Using the fact that $\Omega_{ij\alpha\beta}=-\Omega_{ij\beta\alpha}$, we can write out the left hand side as 
\begin{align}
&\Omega_{ij\alpha\beta} (\frac{1}{2r} dy_{\widehat{\beta\alpha}} -\frac{1}{2r^3} y_\mu y_\beta dy_{\widehat{\mu\alpha}})\nonumber\\
=&
\Omega_{ij\alpha, \alpha+1}\Big(-\frac{1}{r} dy_{\alpha+2} -\frac{1}{2r^3} (y_{\alpha+2} y_{\alpha+1} dy_{y_{\alpha+1}}-y_{\alpha+1}^2dy_{\alpha+2} )+\frac{1}{2r^3} (y_\alpha^2 dy_{\alpha+2}-y_\alpha y_{\alpha+2} dy_{\alpha})\Big)\nonumber\\
=&\Omega_{ij\alpha, \alpha+1}\Big(-\frac{1}{2r}dy_{\alpha+2}-\frac{1}{2r^2} y_{\alpha+2}dr\Big)\nonumber\\
=& -\frac{1}{2}\Omega_{ij\alpha, \beta}\Big(\frac{1}{2r}dy_{\widehat{\alpha\beta}}+\frac{1}{2r^2} y_{\widehat{\alpha\beta}}dr\Big).
\end{align}
\end{proof}
Applying the above lemma, \eqref{e:d*dphi1-curvature-to-be-simplified} and \eqref{e:dd*phi1-A-to-be-simplified} can 
be simplified as follows,
\begin{align}
d^*d\phi_1 &= -\Omega_{ij\alpha\beta}(\frac{1}{4r}dy_{\widehat{\alpha\beta}}+\frac{1}{4r^2}y_{\widehat{\alpha\beta}}dr)\wedge dx_i\wedge dx_j + r^{-3}\Pi_3^{-2} + O'(1),
\\
dd^*\phi_1 &=  -\frac{H^\alpha y_\alpha}{2r^3}dy_1\wedge dy_2\wedge dy_3+ A_{ij\alpha\beta} (\frac{1}{4r^2}y_{\widehat{\alpha\beta}}dr-\frac{1}{4r}dy_{\widehat{\alpha\beta}})\wedge dx_i\wedge dx_j\nonumber\\
&+r^{-5} \Pi_3^{(4)}+O'(1).
\end{align}
Therefore,
\begin{align}
\Delta \phi_1 = &  (d^*d+dd^*)\phi_1 
\nonumber\\
= & -\frac{H^\alpha y_\alpha}{2r^3}dy_1\wedge dy_2\wedge dy_3-\Omega_{ij\alpha\beta} (\frac{1}{4r} dy_{\widehat{\alpha\beta}} +\frac{1}{4r^2} y_{\widehat{\alpha\beta}}dr )\wedge dx_i\wedge dx_j
\nonumber\\
&+A_{ij\alpha\beta}(\frac{1}{4r^2}y_{\widehat{\alpha\beta}}dr-\frac{1}{4r}dy_{\widehat{\alpha\beta}})\wedge dx_i\wedge dx_j +r^{-5} \Pi_3^{(4)}+O'(1).
\end{align}
The proof is done. 
\end{proof}

\noindent{\bf Step 3:} In this step we modify $\phi_1$ to kill the unbounded terms on the right hand side of \eqref{e:Delta-phi-1}.
We  first we recall some elementary computations involving the standard Euclidean Hodge Laplacian.
\begin{lemma}\label{l:euclidean-laplacian}
Let $\Delta_0$ be the standard Hodge Laplacian on the Euclidean space $\R^3$, then the following holds:
\begin{enumerate}
\item Let $\{y_1,y_2,y_3\}$ be the Cartesian coordinates of $\R^3$, then 
\begin{align}
\begin{cases}
\Delta_0r=-\frac{2}{r},\\
\Delta_0(\frac{y_\alpha}{r})=\frac{2y_\alpha}{r^3},\\
\Delta_0(\frac{y_\alpha y_\beta}{r})=\frac{4y_\alpha y_\beta}{r^3}, &  \alpha\neq \beta,\\
\Delta_0((\frac{y_\alpha^2}{r}-r))=\frac{4y_\alpha^2}{r^3}.
\end{cases}
\end{align}
\item Denote by $\mathcal P_4$ the space of all homogeneous degree 4 polynomials on $\R^3$,   then the operator 
\begin{equation}\square: \mathcal P_4\rightarrow \mathcal P_4; f\mapsto r^{5}\Delta_0(r^{-3}f)\end{equation}
is an isomorphism.

\end{enumerate}
\end{lemma}
\begin{proof}
Item (1) is a direct calculation. An convenient way to see this is to use the following two facts: 
\begin{enumerate}
\item A homogeneous polynomial degree $k$ polynomial restricts to an eigenfunction of the Hodge-Laplacian $\Delta_{S^2}$ on the unit sphere $S^2$, with eigenvalue $k(k+1)$. 
\item Given an eigenfunction $h$ of $\Delta_{S^2}$ on the unit sphere with eigenvalue $k$, for any $l$, we can extend $h$ to a homogeneous function $h_l$ on $\R^3\setminus\{0\}$ of degree $l$, and
\begin{equation}\Delta_0 h_l=r^{-2}(k-l(l+1))h_l.\end{equation}
\end{enumerate}

For the second item it is possible to write down an explicit inverse to $\Delta_0$.  Here we provide a quick abstract proof. First we notice $\square:\mathcal{P}_4\to\mathcal{P}_4$ is a well-defined linear map. This follows from the standard computations
\begin{eqnarray*}
r^5\Delta_0(r^{-3}f)&=&r^5\Delta_0(r^{-3})\cdot f-2 r^5\nabla (r^{-3})\cdot \nabla f+r^{2}\Delta_0 f\\
&=& -6 f-3\nabla (r^2)\cdot \nabla f + r^2\Delta_0f.
\end{eqnarray*}
Since each term in the above formula is a polynomial in $\mathcal{P}_4$, so $\square f\in \mathcal{P}_4$. 

Now to prove $\square$ is an isomorphism  it suffices to prove it has a trivial kernel. Let $u\equiv r^{-3}f$, then $u=O(r)$ for both $r\to0$ and $r\to\infty$. If $\Delta_0(u)=0$, then $u$ is harmonic on $\R^3\setminus\{0\}$. The removable singularity theorem implies that $u$ extends smoothly on $\R^3$. Since $u=O(r)$ as $r\to\infty$, by Liouville theorem,  $u$ is a linear function. Since $f\in\mathcal{P}_4$, we conclude $f\equiv 0$. The proof is done. 
\end{proof}

Next, we want to find a bounded correction $3$-form 
$\mathfrak{B}_0=O'(1)$ such that  
\begin{equation}\Delta(\phi_1+\mathfrak{B}_0)=O'(1)\ \text{on}\ \mathcal{U}\setminus P.\end{equation}
Now the main part is to eliminate the unbounded terms in $\Delta\phi_1$ which relies on the following explicit calculations for $\Delta\mathfrak{B}_0$.
In fact, the leading terms of $\Delta\mathfrak{B}_0$ are exactly given by the Euclidean Laplacian $\Delta_0$ acting on the normal components
such that the explicit computations in Lemma \ref{l:euclidean-laplacian} can be effectively used in our context. 
Precisely, we have the following lemma.

\begin{lemma}
\label{l:laplacian-error} Let $\Delta_0$ be the Hodge Laplacian on $\dR^3$, then the following holds:
\begin{enumerate}
\item 
Denote by $\nu_y $ one of the following differential forms $dy_{\alpha}$, $dy_{\alpha}\wedge dy_{\beta}$ or $dy_1\wedge dy_2\wedge dy_3$. Similarly, let $\tau_x$ be a  tangential $p$-form given by $\tau_x \equiv dx_1\wedge\ldots \wedge dx_{\alpha_p}$ with $0\leq p\leq m-3$.  
Let $\omega  \equiv f(x)h(y) \nu_y \wedge  \tau_x$, where $f(x)$ is a smooth function defined on $U$ and $h(y)=O'(|y|^k)$ for some $k\in\dZ_+$. Then
\begin{equation}
\Delta \omega - f(x)\cdot \Delta_0(h(y))\cdot \nu_y\wedge \tau_x = O'(|y|^{k-1}).
\end{equation}

\item Let $f$ be a smooth function defined on $U\subset P$ 
and let
$\omega \equiv f(x)\cdot \frac{y_{\alpha}}{r}dy_1\wedge dy_2\wedge dy_3$. Then
\begin{align}
\Delta \omega = 2r^{-2}\omega+ r^{-5}\Gamma_3^{(4)}+O'(1),
\end{align}
where the $3$-form $\Gamma_3^{(4)}$ has the form of  \eqref{e:3-form-polynomial-coe} in Notation \ref{n:differential-form-notation}.
\end{enumerate}

 \end{lemma}

\begin{proof}

First, we prove item (1). 
By definition, $\Delta=dd^* + d^*d$.
We only prove the case $\nu_y= dy_{\alpha}$ for $1\leq \alpha\leq 3$ and $1\leq p\leq m-3$. The proof of the remaining cases is identical. 

First, we compute $d^*d\omega$. 
\begin{align}
d\omega &= f(x) \cdot \frac{\p h(y)}{\p y_{\beta}} dy_{\beta}\wedge dy_{\alpha}\wedge \tau_x + \frac{\p f(x)}{\partial x_j} \cdot h(y)\cdot dx_j\wedge \nu_y\wedge \tau_x,\end{align}
which implies that
\begin{align}
*d\omega &=(-1)^p f(x)\cdot\frac{\p h(y)}{\p y_{\beta}} dy_{\widehat{\beta\alpha}}\wedge *_T(\tau_x) + O'(|y|^k)
\nonumber\\
&=(-1)^p f(x)\Big(\frac{\p h(y)}{\p y_{\alpha-1}}dy_{\alpha+1} -\frac{\p h(y)}{\p y_{\alpha+1}}dy_{\alpha-1}\Big)\wedge *_T(\tau_x)+ O'(|y|^k).
\end{align}
Differentiating the above equality,
\begin{align}
d*d\omega & =(-1)^p f\cdot\Big(\frac{\p^2h}{\p y_{\alpha}\p y_{\alpha-1}}dy_{\alpha}\wedge dy_{\alpha+1} 
- \frac{\p^2h}{\p y_{\alpha}\p y_{\alpha+1}}dy_{\alpha}\wedge dy_{\alpha-1}
\nonumber\\
&+ \Big(\frac{\p^2h}{\p y_{\alpha-1}^2}+\frac{\p^2h}{\p y_{\alpha+1}^2}\Big)dy_{\alpha-1}\wedge dy_{\alpha+1}\Big)\wedge *_T(\tau_x)+ O'(|y|^{k-1}).
\end{align}
Then it follows that
\begin{align}
d^*d\omega & = (-1)^{mp+m+1}*d*d\omega
\nonumber\\
&=f\cdot\Big(\frac{\p^2h}{\p y_{\alpha}\p y_{\alpha-1}}dy_{\alpha-1}+\frac{\p^2h}{\p y_{\alpha}\p y_{\alpha+1}}dy_{\alpha+1}-\Big(\frac{\p^2h}{\p y_{\alpha-1}^2}+\frac{\p^2h}{\p y_{\alpha+1}^2}\Big)dy_{\alpha}\Big)\wedge \tau_x
\nonumber\\
& + O'(|y|^{k-1}).\label{e:d*d-omega}
\end{align}
On the other hand, 
\begin{align}
*\omega = f(x)h(y)dy_{\widehat{\alpha}}\wedge *_T(\tau_x),
\end{align}
which implies
\begin{equation}
d*\omega = f(x)\frac{\p h}{\p y_{\alpha}}\dvol_N\wedge *_T(\tau_x) + O'(|y|^k).
\end{equation}
So it follows that
\begin{equation}
d^*\omega = (-1)^{mp+1}*d*\omega= -f\cdot\frac{\p h}{\p y_{\alpha}}\cdot\tau_x + O'(|y|^k),
\end{equation}
and hence
\begin{equation}
dd^*\omega = - f\cdot \Big( \frac{\p^2h}{\p y_{\alpha-1}\p y_{\alpha}} dy_{\alpha-1} + \frac{\p^2h}{\p y_{\alpha}^2}dy_{\alpha}+\frac{\p^2h}{\p y_{\alpha+1}\p y_{\alpha}}dy_{\alpha+1}\Big)\cdot dy_{\alpha} \wedge \tau_x + O'(|y|^{k-1}).\label{e:dd*-omega}
\end{equation}
 Therefore, combining \eqref{e:d*d-omega} and \eqref{e:dd*-omega},
 \begin{equation}
 \Delta\omega= (d^*d+dd^*)\omega = - f\cdot (\Delta_0 h(y) ) dy_{\alpha}\wedge \tau_x + O'(|y|^{k-1}),
 \end{equation}
 where $\Delta_0(h(y))=-\frac{\p^2 h}{\p y_1^2}-\frac{\p^2 h}{\p y_2^2}-\frac{\p^2 h}{\p y_3^2}$. 
The proof of (1) is done.

Now we prove item (2). Let $\omega = f\cdot \frac{y_{\alpha}}{r}dy_1\wedge dy_2\wedge dy_3$ and the first step is to compute the term $d^*d\omega$.
By Lemma \ref{l:generalized-Gauss},
$dr=\frac{y_{\gamma}dy_{\gamma}}{r}$, then 
\begin{equation}d(\frac{y_{\alpha}}{r}) \wedge dy_1\wedge dy_2\wedge dy_3=0.\end{equation}
This implies that 
\begin{equation}
d\omega = \frac{\p f}{\p x_i} \cdot \frac{y_{\alpha}}{r} dx_i\wedge dy_1\wedge dy_2\wedge dy_3,\end{equation}
and hence
\begin{equation}
*d\omega =  -  \frac{\p f}{\p x_i} \cdot \frac{y_{\alpha}}{r} *_{T}(dx_i) + O'(r).
\end{equation}
Differentiating the above equality and applying Lemma \ref{l:generalized-Gauss} again,
\begin{align}
d*d\omega = -\frac{\p f}{\p x_i}\Big(\frac{dy_{\alpha}}{r} - \frac{ y_{\alpha}y_{\beta}\cdot dy_{\beta}}{r^3}
\Big)*(dx_i)
+O'(1).
\end{align}
It follows that 
\begin{align}
*d*d\omega  = (-1)^{m+1}\cdot \frac{\p f}{\p x_i}\cdot\frac{dy_{\widehat{\alpha}}\wedge dx_i}{r} + (-1)^m\frac{\p f}{\p x_i}\cdot\frac{ y_{\alpha}y_{\beta}}{r^3}\cdot dy_{\widehat{\beta}} \wedge  dx_i + O'(1).
\end{align}
Therefore,
\begin{align}
d^*d\omega &= (-1)^{m+1}*d*d\omega 
\nonumber\\
&= \frac{\p f}{\p x_i}\cdot\frac{dy_{\widehat{\alpha}}\wedge dx_i}{r} - \frac{\p f}{\p x_i}\cdot\frac{ y_{\alpha}y_{\beta}}{r^3}\cdot dy_{\widehat{\beta}} \wedge  dx_i +  O'(1). \label{e:d*d-w}
\end{align}

Now we compute $dd^*\omega$. 
By Lemma \ref{l:v-T-N} and the expansion of $\dvol_N$ in  \eqref{e:dvol-N-expansion},
\begin{align}
*(dy_1\wedge dy_2\wedge dy_3)
&= * (\dvol_N + A_{i\gamma\beta} y_{\gamma}dy_{\widehat{\gamma}}\wedge dx_i) + \tO(r^2)
\nonumber\\
&= (1+H^{\beta}y_{\beta})\dvol_T+A_{i\gamma\beta}y_{\gamma}dy_{\gamma}\wedge *_T(dx_i) + \tO(r^2),
\end{align}
so we have
\begin{align}
*\omega &= f\cdot \frac{y_{\alpha}}{r}\cdot \dvol_T + f\cdot H^{\beta}\cdot \frac{y_{\alpha}y_{\beta}}{r}\cdot \dvol_T+ f\cdot A_{i\gamma\beta}\frac{y_{\alpha}y_{\gamma}}{r}dy_{\gamma}\wedge *_T(dx_i) + O'(r^2)
\nonumber\\
&\equiv \FT_1 + \FT_2 + \FT_3 + O'(r^2).
\end{align}
By collecting the leading terms,  it is easy to compute the leading term in the above equality,
\begin{align}
d*d(\FT_1)&=\frac{\p f}{\p x_i}\Big( \frac{dx_i\wedge dy_{\widehat{\alpha}}}{r} - \frac{y_{\alpha}y_{\beta}dx_i\wedge dy_{\widehat{\beta}}}{r^3}\Big)-2\omega + O'(1)
\\
d*d(\FT_2)&= r^{-5}\Pi_3^{(4)},\ d*d(\FT_3)= r^{-5}\Pi_3^{(4)}.
\end{align}
Therefore,
\begin{equation}
dd^*\omega =  -\frac{\p f}{\p x_i}\cdot \frac{dx_i\wedge dy_{\widehat{\alpha}}}{r} + \frac{\p f}{\p x_i}\cdot \frac{y_{\alpha}y_{\beta}}{r^3} \cdot dx_i\wedge dy_{\widehat{\beta}} + 2\omega + r^{-5} \Pi_3^{(4)} + O'(1).\label{e:dd*-w}
\end{equation}
By \eqref{e:d*d-w} and \eqref{e:dd*-w} we obtain the expansion
\begin{align}
\Delta\omega = (d^*d + dd^*)\omega =  2 \omega + r^{-5} \Pi_3^{(4)} + O'(1).
\end{align}
So the proof is done. 
\end{proof}

Now we finish Step 2 by proving the following

\begin{proposition}
There is some $3$-form $\Lambda_3^{(4)}$ (given in  the form in Notation \ref{n:differential-form-notation}) such that
if we choose \begin{align}
\mathfrak{B}_0 & \equiv \frac{H^\alpha y_\alpha}{4r}dy_1\wedge dy_2\wedge dy_3
\nonumber\\
&+ \Big(\Omega_{ij\alpha\beta}(\frac{1}{16}y_{\widehat{\alpha\beta}}dr-\frac{3}{16} rdy_{\widehat{\alpha\beta}})-\frac{1}{16}A_{ij\alpha\beta}(y_{\widehat{\alpha\beta}} dr+rdy_{\widehat{\alpha\beta}})\Big)\wedge dx_i\wedge dx_j + r^{-3}\Lambda_3^{(4)},
\end{align}
then the corrected $3$-form  \begin{align} 
\phi_2 \equiv & \phi_1 + \mathfrak{B}_0
\label{e:phi_2-correction}
\end{align} satisfies
$\Delta \phi_2 = O'(1)$
in  $\mathcal{U}\setminus P$ 
and has the expansion 
\begin{align}
\phi_2 = & \frac{1}{2r}(1-\frac{H^\alpha y_\alpha}{2}) dy_1\wedge dy_2\wedge dy_3+\frac{1}{2r}y_{\beta}A_{i\alpha\beta}dx_i \wedge dy_{\widehat{\alpha}}
-\frac{1}{4} A_{ij\alpha\beta}r\cdot dy_{\widehat{\alpha\beta}}\wedge dx_i\wedge dx_j\nonumber
\\&+ \frac{3}{16}(A_{i\alpha, \alpha+1}A_{j\alpha, \alpha+2}-A_{i\alpha, \alpha+2}A_{j\alpha, \alpha+1})d(ry_\alpha)\wedge dx_i\wedge dx_j
+r^{-3}\Pi_3^{(4)}+O'(r^2).
\end{align}

\end{proposition}

\begin{proof}

Let $\mathfrak{b}_0 \equiv \frac{H^\alpha y_\alpha}{4r}dy_1\wedge dy_2\wedge dy_3$. Then item (2) of Lemma \ref{l:laplacian-error} tells us that
\begin{equation}
\Delta\mathfrak{b}_0 = \frac{H^\alpha y_\alpha}{2r^3}dy_1\wedge dy_2\wedge dy_3 + r^{-5}\Gamma_3^{(4)} + O'(1).
\end{equation}
Let $\Pi^{(4)}$ be the $3$-form in the expansion of $\Delta\phi_1$ given by \eqref{e:Delta-phi-1} in Proposition \ref{p:Delta-phi-1-expansion}. 

Next,  Lemma \ref{l:euclidean-laplacian} and Lemma \ref{l:laplacian-error} tell us that 
there are $3$-forms $\widehat{\Gamma}_3^{(4)}$
and $\widehat{\Pi}_3^{(4)}$ which are also of the form as in \eqref{e:3-form-polynomial-coe} such that
\begin{align}
\Delta(r^{-3}\widehat{\Gamma}_3^{(4)}) = - r^{-5}\Gamma_3^{(4)} + O'(1), 
\\
\Delta (r^{-3}\widehat{\Pi}_3^{(4)}) = -r^{-5}\Pi_3^{(4)} + O'(1).
\end{align}
Now let
$\mathfrak{b}_1 \equiv  r^{-3}\widehat{\Gamma}_3^{(4)} + r^{-3}\widehat{\Pi}_3^{(4)}$. Then the correction term $\mathfrak{b}_0 + \mathfrak{b}_1$ is chosen as the above such that $\Delta(\mathfrak{b}_0 + \mathfrak{b}_1)$ in fact eliminates the $O'(r^{-2})$-term and implicit 
$O'(r^{-1})$-terms in the expansion of $\Delta\phi_1$ (see Proposition \ref{p:Delta-phi-1-expansion}).

In the following, we will make a further correction such that those explicit $O'(r^{-1})$-terms will be cancelled out as well. In fact, 
we define 
\begin{equation}
\mathfrak{b}_2 \equiv \Big(\Omega_{ij\alpha\beta}(\frac{1}{16}y_{\widehat{\alpha\beta}}dr-\frac{3}{16} rdy_{\widehat{\alpha\beta}})-\frac{1}{16}A_{ij\alpha\beta}(y_{\widehat{\alpha\beta}} dr+rdy_{\widehat{\alpha\beta}})\Big)\wedge dx_i\wedge dx_j,
\end{equation}
applying Lemma \ref{l:euclidean-laplacian} and Lemma \ref{l:laplacian-error} again,  then 
\begin{equation}
\Delta\mathfrak{b}_2 =\Omega_{ij\alpha\beta} \Big(\frac{1}{4r} dy_{\widehat{\beta\alpha}} +\frac{y_{\widehat{\alpha\beta}}}{4r^2} dr\Big)\wedge dx_i\wedge dx_j
-A_{ij\alpha\beta}\Big(\frac{y_{\widehat{\alpha\beta}}}{4r^2}dr-\frac{1}{4r}dy_{\widehat{\alpha\beta}}\Big)\wedge dx_i\wedge dx_j,
\end{equation}
and hence
$\Delta(\phi_1+\mathfrak{b}_0 + \mathfrak{b}_1 + \mathfrak{b}_2) = O'(1)$. Therefore, 
it suffices to choose the correction term 
\begin{equation}
\mathfrak{B}_0 \equiv \mathfrak{b}_0 + \mathfrak{b}_1 + \mathfrak{b}_2,
\end{equation}
which gives
$\Delta(\phi_1+\mathfrak{B}_0) = 
O'(1)$.

Notice that, $\mathfrak{b}_2$ has a further cancellation,
\begin{align}
\mathfrak{b}_2 = &\Big(\Omega_{ij\alpha\beta}(\frac{1}{16}y_{\widehat{\alpha\beta}}dr-\frac{3}{16} rdy_{\widehat{\alpha\beta}})-\frac{1}{16}A_{ij\alpha\beta}(y_{\widehat{\alpha\beta}} dr+rdy_{\widehat{\alpha\beta}})\Big)\wedge dx_i\wedge dx_j,
\nonumber\\
=& -\frac{1}{4}A_{ij\alpha\beta}rdy_{\widehat{\alpha\beta}}\wedge dx_i\wedge dx_j - \frac{1}{16}(A_{i,\alpha,\alpha+1}A_{j,\alpha,\alpha+2}-A_{i,\alpha,\alpha+2}A_{j,\alpha,\alpha+1})y_{\alpha}dr\wedge dx_i\wedge dx_j 
\nonumber\\
& +  \frac{3}{16}(A_{i,\alpha,\alpha+1}A_{j,\alpha,\alpha+2}-A_{i,\alpha,\alpha+2}A_{j,\alpha,\alpha+1})rdy_{\alpha} \wedge dx_i\wedge dx_j.
\end{align}
Therefore, 
\begin{align}
\phi_2 = &
\phi_1 + \mathfrak{B}_0
\nonumber\\
=& \frac{1}{2r}(1-\frac{H^\alpha y_\alpha}{2}) dy_1\wedge dy_2\wedge dy_3+\frac{1}{2r}y_{\beta}A_{i\alpha\beta}dx_i \wedge dy_{\widehat{\alpha}}
-\frac{1}{4} A_{ij\alpha\beta}r\cdot dy_{\widehat{\alpha\beta}}\wedge dx_i\wedge dx_j\nonumber
\\&+ \frac{3}{16}(A_{i\alpha, \alpha+1}A_{j\alpha, \alpha+2}-A_{i\alpha, \alpha+2}A_{j\alpha, \alpha+1})d(ry_\alpha)\wedge dx_i\wedge dx_j
+r^{-3}\Pi_3^{(4)}+O'(r^2),
\end{align}
and 
$\Delta \phi_2 = O'(1)$ on $\mathcal{U}\setminus P$.
\end{proof}

\noindent{\bf Step 4:} In this step we compute $\Delta\phi_2$ as a current on $\mathcal U$. 

\begin{lemma}\label{l:almost-Green's-current}
For any $p\in P$, let  $\mathcal{U}$ be a neighborhood of $p$ in $Q$ and let $U = \mathcal{U}\cap P$, then we have
\begin{equation}
\DelH\phi_2=2\pi\delta_U+O'(1).
\end{equation}
\end{lemma}
\begin{proof}
Let $\chi\in \Omega_0^{m-3}(\mathcal{U})$ be a compactly supported test form. Then applying integration by parts once, 
\begin{equation}(\phi_2, \DelH\chi)=\int_{\mathcal{U}} \phi_2\wedge (dd^*+d^*d)\chi=\int_{\mathcal{U}} d\phi_2\wedge d^*\chi-\int_{\mathcal{U}} d^*\phi_2\wedge d\chi.\end{equation}
Notice that the integration by parts works here because $\phi_2=O(r^{-1})$. Also notice that \eqref{e:d-phi-1-exp} and \eqref{e:phi_2-correction} imply $d\phi_2=O'(1)$, so
\begin{equation}\int_{\mathcal{U}} d\phi_2\wedge d^*\chi=\int_{\mathcal{U}} d^*d\phi_2\wedge \chi. \end{equation}
On the other hand,  by (\ref{eqn2-19}), 
\begin{equation}d^*\phi_2=-\frac{y_\alpha}{2r^{3}} {} dy_{\hat\alpha}+\zeta,\end{equation}
where $\zeta$ is a $2$-form satisfying $\zeta=O'(r^{-1})$.
 Denote by $S_{\epsilon}^2$ the normal geodesic sphere bundle $\{r=\epsilon\}$, then we get that 
\begin{eqnarray}-\int_{\mathcal U} d^*\phi_2\wedge d\chi
&=&\int_{\mathcal U} dd^* \phi_2 \wedge \chi + \int_{S_{\epsilon}^2} d^*\phi_2\wedge \chi
\nonumber\\
&=&\int_{\mathcal U} dd^* \phi_2 \wedge \chi + \lim_{\epsilon\rightarrow 0} \frac{1}{2\epsilon^3}{}\int_{S_\epsilon} (y_\alpha dy_{\hat\alpha} +\epsilon^2 \zeta)\wedge \chi.\end{eqnarray}
By direct calculation of the last term on the right hand side we obtain
\begin{equation}
-\int_{\mathcal U} d^*\phi_2\wedge d\chi=\int_{\mathcal U} dd^* \phi_2 \wedge \chi+2\pi{}\int_{P}\chi.
\end{equation}
This concludes the proof. 
\end{proof}

\noindent{\bf Step 5:} Now we study the regularity of the solution of $\Delta T=v$ with $v=O'(1)$.

\begin{lemma}\label{l:higher-regularity} 
Let $p\in P$ and let $\mathcal{U}$ be a neighborhood of of $p$ in $Q$. Assume that $v$ is a $3$-form on $\mathcal{U}$ which satisfies $v=O'(1)$ and $v\in C^{\infty}(\mathcal{U}\setminus U)$, where $U = \mathcal{U}\cap P$. Then the equation \begin{equation}\Delta  T= v, \label{e:current-poisson} \end{equation} 
has a weak solution $T \in W^{2,p}(\mathcal{V})$ for any $p>1$ which  satisfies $T=O'(r^2)$ and $T\in C^{\infty}(\mathcal{V}\setminus U)$. Here  $\mathcal{V} \subset\subset \mathcal{U}$ is a smaller neighborhood of $p$ in $Q$, and $r$ is the distance to $P$. 

\end{lemma}

\begin{proof}   Since the assumption implies $v\in L^q(\mathcal{U})$ for any $q>1$,  by the standard elliptic theory, equation \eqref{e:current-poisson} has a solution $T$ which is $W^{2,q}$ for any $q>1$, and $T$ is $C^{\infty}$ away from $U$.  Applying the $W^{2,q}$-regularity to $\Delta T=v$ and using the Sobolev embedding, we have that \begin{equation}T\in W^{2,q}(\mathcal{U}_1)\cap  C^{1, \alpha}(\mathcal{U}_1)\label{e:sobolev-weaker}\end{equation} for any $q>1$ and $0<\alpha<1$, where  $\mathcal{U}_1\subset \subset \mathcal{U}$ a smaller neighborhood of $p$ in $Q$. 
 Therefore,  $|T|=O(1)$ and  $|\p T|=O(1)$.
It remains to show $T=O'(r^2)$. 
The proof consists of two primary steps.

\begin{flushleft}{\bf  Step 1:}\end{flushleft} 

{\it Let us fix any fixed point $p\in P$. In the normal coordinate system around $p$, we will show that} 
\begin{equation}
  |\p^k T|=O(r^{2-(k+\epsilon)}), \label{e:general-nabla}
\end{equation}
{\it where $k\geq 2$ and $\epsilon>0$ are arbitrary, and $\p^k$ denotes the mixed derivatives on the coefficient functions of $T$ of order $k$.}

The regularity \eqref{e:general-nabla} will be proved by the induction on the derivative order $k\geq 2$.

In the base step, we compute $\Delta\Phi$, where $\Phi\equiv \p^2 T$.
Differentiating the equation $\Delta T = v$ by $\p^2$, it schematically yields  
\begin{equation}\DelH\Phi + Q_1* \p\Phi +Q_2*\Phi+Q_3*\p T + Q_4* T=\p^2 v = O'(r^{-2}),\label{e:initial-eq}\end{equation}
where $Q_i$'s involve the derivatives of the metric $g$, and the shorthand $A*B$ denotes a linear combination of $A_I\cdot B_J$ with coefficient functions $A_I$, $B_J$ respectively. Notice that $T \in C^{1,\alpha}(\mathcal{U}_{1})$, so equation \eqref{e:initial-eq} can be rewritten as 
\begin{equation}
\DelH\Phi + Q_1 * \p\Phi + Q_2 * \Phi = \p^2 v-Q_3*\p T - Q_4*T.\label{e:Del-Phi-original}
\end{equation}
Since $\p^2v-Q_3*\p T - Q_4* T= O(r^{-2-\epsilon})$ which is not $L^q$-integrable around $U$ for $q$ sufficiently large, the desired regularity of $\Phi$ does not directly follow from the standard $W^{2,q}$-estimate.

To obtain the refined estimate for $\Phi$, we will need the following rescaling argument.  One can choose a smaller neighborhood $\mathcal U_2\subset\subset\mathcal{U}_1$ such that $B_{r_x/2}(x)\subset \mathcal{U}_1$ for every $x\in \mathcal{U}_2$, where $r_x\equiv r(x)\in(0,1)$. For every fixed $x\in\mathcal{U}_2\setminus U$, let us take the rescaled metric 
$\tilde{g} = (r_x)^{-2} \cdot g$ and the dilation $\widetilde{\Phi}(y)\equiv  \Phi(r_x\cdot y)$ with $y\in B_{1/2}^{\tilde{g}}(x)$. Then equation \eqref{e:Del-Phi-original} in $B_{r_x/2}(x)\subset \mathcal{U}_1$ becomes  
\begin{equation}\widetilde{\Delta} \widetilde{\Phi}+r_x\cdot\widetilde{Q}_1*\widetilde{\p}\widetilde{\Phi}+(r_x)^2\cdot\widetilde{Q}_2*\widetilde{\Phi}= \tilde{\eta} \quad \text{in} \  B_{1/2}^{\tilde{g}}(x),\label{e:rescaled-Phi-eq}\end{equation}
where $|\tilde{\eta}|_{L^{\infty}(B_{\frac{1}{2}^{\tilde{g}}(x)})}\leq C_{\epsilon}\cdot r^{-\epsilon}$ for each $\epsilon>0$. 
Notice that the $L^{\infty}$-norms of the coefficients  in \eqref{e:rescaled-Phi-eq} are uniformly bounded (independent of $x$), and   \eqref{e:sobolev-weaker} implies $|\Phi|_{L^q(\mathcal{U}_1)}=|\p^2T|_{L^q(\mathcal{U}_1)}\leq C_q$ for any $q>1$. Using simple rescaling, we have that for every $q>1$,
$|\widetilde{\Phi}|_{L^q(B_1^{\tilde{g}}(x))}\leq C_q\cdot (r_x)^{-\frac{n}{q}}$, where $C_q>0$ is a uniform constant independent of $x\in\mathcal{U}_2\setminus U$.
Applying the $W^{2,q}$-estimate to \eqref{e:rescaled-Phi-eq} and applying the Sobolev embedding, we have that for each $q>1$\begin{align}|\widetilde{\Phi}|_{W^{2,q}(B_{1/2}^{\tilde{g}}(x))} \leq C_q \cdot (r_x)^{-\frac{n}{q}}.\label{e:Phi-(2,q)}\end{align}
For any fixed $\epsilon>0$, letting $q>1$ be sufficiently large and applying the Sobolev embedding, we obtain that 
\begin{align}
 |\widetilde{\Phi}|_{C^{1,\alpha}(B_{1/4}^{\tilde{g}}(x))}\leq C_{\alpha,\epsilon}\cdot (r_x)^{-\epsilon}.
\end{align}
Rescaling back to the original metric $g$, we have that \begin{equation}|\Phi|_{L^{\infty}(B_{r_x/4}(x))}+|r_x\cdot \p \Phi|_{L^{\infty}(B_{r_x/4}(x))}\leq  C_{\epsilon} \cdot (r_x)^{-\epsilon},\label{e:Phi-base-step}\end{equation}
where $C_{\epsilon}>0$ is independent of the base point $x\in \mathcal{U}_{2}\setminus U$. 
This completes the base step.

Based on the initial step $k=2$, we are now ready to finish the induction step. For any fixed integer $k\geq 3$, suppose that the estimate in item (1) holds for all $2\leq j \leq k-1$. That is, for every $\epsilon>0$, $2\leq j \leq k-1$ and $q>1$, there is a smaller neighborhood $\mathcal{U}_{j}\subset\subset \mathcal{U}$ on which the following estimates hold:\begin{align}
&|\p^{j-2} \Phi|=O(r^{2-j-\epsilon}),\label{e:hypothesis-1}
\end{align}
where $C_{k,q,\epsilon}>0$ is independent of the choice of the base point $x\in \mathcal{U}_{j}\setminus U$. Then we will show that for all $\epsilon>0$,  $|\p^k T|=O(r^{2-(k+\epsilon)})$. 
Differentiating equation \eqref{e:initial-eq} by $\p^{k-3}$, we have that for every $\epsilon>0$, 
\begin{equation}
\Delta(\p^{k-3}\Phi) + \sum\limits_{j=0}^{k-2}Q_j * \p^j \Phi = O(r^{1-k-\epsilon}),  \label{e:Delta-nabla-(k-2)}
\end{equation}
where $Q_j$'s arise from the derivatives of the metric coefficients. 

As before, to obtain the higher regularity of $\Phi$, we will study \eqref{e:Delta-nabla-(k-2)} under the rescaling $\tilde{g} = (r_x)^{-2}\cdot g$ and the dilation $\widetilde{\Phi}(y) = \Phi(r_x\cdot y)$ for every fixed $x\in \mathcal{U}_k\setminus U$, where $\mathcal{U}_k$ is some smaller neighborhood.
Applying the induction hypothesis \eqref{e:hypothesis-1} and absorbing the lower order terms of the \eqref{e:Delta-nabla-(k-2)} in the rescaled form, we have that for every $\epsilon>0$,
\begin{equation}
\widetilde{\Delta}(\widetilde{\p}^{k-3}\widetilde{\Phi}) + r_x \cdot \widetilde{Q}_{j-1} * \widetilde{\p}(\widetilde{\p}^{k-3} \widetilde{\Phi}) + \widetilde{Q}_{j-3} * (\widetilde{\partial}^{k-3}\widetilde{\Phi}) = O(r^{-\epsilon}).\end{equation}
Then the elliptic regularity implies that for each $\epsilon>0$ and $q>1$, 
\begin{equation}
|\widetilde{\p}^{k-3}\widetilde{\Phi}|_{W^{2,q}(B_{1/4}^{\tilde{g}}(x))} \leq C_{k,q,\epsilon}\cdot (r_x)^{-\epsilon},
\end{equation}
where $C_{k,q,\epsilon}>0$ is independent of the base point $x\in\mathcal{U}_k\setminus U$. Applying the Sobolev embedding and scaling back to the original metric $g$,  we have that for each $\epsilon>0$,
\begin{equation}|\p^k T|=|\p^{k-2}\Phi|_{L^{\infty}(B_{r_x/4}(x))}\leq C_{k,\epsilon}\cdot (r_x)^{2-k-\epsilon}\end{equation} where $C_{k,\epsilon}>0$ is independent of the base point $x\in\mathcal{U}_k\setminus U$. 
The proof of \eqref{e:general-nabla} is done.

 \begin{flushleft}{\bf Step 2:}\end{flushleft} 

{\it We prove that  for each $\epsilon>0$, $k\geq 2$, $\ell \in \dZ_+$, } 
\begin{equation} |\p_t^{\ell}T|=O(1),\quad  |\p_t^{\ell}\p T|=O(1),\quad 
|\p_t^{\ell} \p^k T|=O(r^{2-k-\epsilon}), \label{e:mixed-nabla}
\end{equation}
{\it  where $\p_t^{\ell}$ denotes the derivatives in the tangential directions of order $\ell$. This will be proved by induction on the both $\ell$ and $k$.}
 
Before going through the proof, we first outline the induction scheme and several sub-steps that we will establish: 
\begin{itemize}
\item Step 2.0: For every $\ell\in\dZ_+$, the following tangential regularity holds for $T$ and $\p T$,    \begin{align}|\p_t^{\ell}T|=O(1),
\quad |\p_t^{\ell}\p T|=O(1). \label{e:tangential-T-pT}\end{align}

\item Step 2.1: For every $\epsilon>0$ and $k\geq 2$, the following holds
\begin{align}
|\p_t\p^k T| = O(r^{2-k-\epsilon}).\label{e:first-order-tangential}
\end{align}

\item Step 2.2: Under the induction hypothesis that the higher regularity 
\begin{align}
|\p_t^i\p^k T| = O(r^{2-k-\epsilon})
\label{e:main-induction-hypothesis}\end{align}
 holds for every $\epsilon>0$, $k\geq 2$,  $1\leq i \leq \ell -1$, we will prove that for every $\epsilon>0$ and $k\geq 2$,
 \begin{align}
 |\p_t^{\ell}\p^k T| = O(r^{2-k-\epsilon}).\label{e:mixed-regularity-induction-step}
 \end{align}

\end{itemize}

{\bf Step 2.0:}  {\it To prove \eqref{e:tangential-T-pT}, it suffices to show $\p_t^{\ell} T\in W^{2,q}$ for any $q>1$.}

To this end,  differentiating $\Delta T=v$ by the tangential derivative $\p_t$, we have that
\begin{equation}
\Delta(\p_t T) +  Q_1 * \p^2 T + Q_2 * \p T = \p_t v = O'(1),\label{e:tangential-T}
\end{equation}
where $Q_1$ and $Q_2$ involve the derivatives of the metric.
Applying the elliptic regularity to \eqref{e:tangential-T}, in some smaller neighborhood $\mathcal{U}'\subset \subset \mathcal{U}$, we have that for any $q>1$
\begin{align}\p_t T\in W^{2,q}(\mathcal{U}').\label{e:tangent-T-(2,q)}\end{align}  Furthermore, given any $\ell\in\dZ_+$, computing $\p_t^{\ell}(\Delta T) =\p_t^{\ell} v$, we have
\begin{align} 
\Delta(\p_t^{\ell} T) + \sum\limits_{j=0}^{\ell-1} \Big(E_j*\p^2 \p_t^j T + F_j*\p\p_t^j T + H_j*\p_t^j T\Big)=\p_t^{\ell} v= O'(1).
\end{align}
Based on \eqref{e:tangent-T-(2,q)} and simple induction argument, we obtain 
 $\p_t^{\ell} T \in W^{2,q}(\mathcal{U}'')$
 for all $q>1$ and $\ell\in\dZ_+$ in a smaller neighborhood $\mathcal{U}''\subset\subset \mathcal{U}'$.

{\bf Step 2.1:} {\it This step is to prove the estimate \eqref{e:first-order-tangential} by induction on $k\geq 2$.}

 Let $\Phi = \p^2 T$ and the base step here is to prove \eqref{e:first-order-tangential} for $k=2$. 
Using \eqref{e:general-nabla}, it is straightforward to check that
\begin{align}
\DelH(\p_t \Phi) + Q_1 * \p(\p_t\Phi) + Q_2 *(\p_t\Phi)=\sum\limits_{j=0}^4Q_j'*\p^j T + \p_t\p^2v = O(r^{-2-\epsilon}) ,\label{e:Delta-p_t-Phi}\end{align} for all $\epsilon>0$. Since by \eqref{e:tangent-T-(2,q)}, $\p_t\Phi=\p_t\p^2 T\in L^q$ for any $q>1$, applying the previous rescaling argument, we have that \begin{align}|\p_t\Phi|_{L^{\infty}(B_{r_x/2}(x))}+r_x\cdot |\p_t\p\Phi|_{L^{\infty}(B_{r_x/2}(x))}\leq C_{\epsilon}\cdot (r_x)^{-\epsilon}\label{e:first-tangential-Phi}\end{align} in terms of the original metric $g$,
where $C_{\epsilon}>0$ is independent of the choice 
of the base point $x$. Next, computing $\Delta(\p_t\p^{k-3}\Phi)$ for any $k\geq 3$, it follows that 
\begin{align}
\Delta(\p_t\p^{k-3}\Phi)
=\p^{k-3}\Delta(\p_t\Phi) + \sum\limits_{j=1}^{k-2}Q_j*\p_t\p^j\Phi.\label{e:Delta-1-(k-1)}
\end{align}
Applying the induction hypothesis $|\p_t\p^j\Phi|=O(r^{-j-\epsilon})$ for $0\leq j\leq k-3$, equation \eqref{e:Delta-1-(k-1)} can be rewritten as the following for any 
\begin{align}
\Delta(\p_t\p^{k-3}\Phi) + Q_1 * \p(\p_t\p^{k-3}\Phi) + Q_2*(\p_t\p^{k-3}\Phi) = O(r^{1-k-\epsilon}).
\end{align}
The same rescaling argument implies that
$|\p_t\p^kT|=|\p_t\p^{k-2}\Phi|=O(r^{2-k-\epsilon})$ for every $\epsilon>0$. This completes Step 2.1.

{\bf Step 2.2:} {\it We will prove \eqref{e:mixed-regularity-induction-step} under the induction hypothesis \eqref{e:main-induction-hypothesis}.
}

The proof is by induction on $k\geq 2$ and the arguments consist of two parts.

To begin with, we will show that the following regularity\begin{align}
|\p_t^{\ell}\Phi | =O(r^{-\epsilon})\label{e:base-case-of-(2.0)}
\end{align}
holds for any $\ell\in\dZ_+$ and $\epsilon>0$. Recall that the case $\ell=1$ 
has been proved in \eqref{e:first-tangential-Phi}.
Now let us compute $\Delta(\p_t^{\ell}\Phi)$ for $\ell>2$. By straightforward computations,
\begin{align}
\Delta(\p_t^{\ell}\Phi) = \p_t^{\ell}\Delta\Phi + \sum\limits_{\ell'=1}^{\ell-1}(E_j*\p^2\p_t^{\ell'}\Phi + F_j*\p\p_t^{\ell'}\Phi).\label{e:Del-tangential-ell-Phi}
\end{align}
Notice that by \eqref{e:initial-eq} and the lower regularity \eqref{e:tangential-T-pT}
\begin{align}
\p_t^{\ell}\Delta\Phi =   Q_1*\p(\p_t^{\ell}\Phi)+Q_2*(\p_t^{\ell}\Phi) + O(r^{-2-\epsilon}).\label{e:tangential-Delta-Phi}
\end{align}
Plugging \eqref{e:tangential-Delta-Phi} into \eqref{e:Del-tangential-ell-Phi}, and applying
the induction hypothesis \eqref{e:main-induction-hypothesis} to the above lower order terms, we eventually have that
\begin{align}
\Delta(\p_t^{\ell}\Phi) + Q_1*\p(\p_t^{\ell}\Phi)+Q_2*(\p_t^{\ell}\Phi) =O(r^{-2-\epsilon})\end{align}
for every $\epsilon>0$.
Since we have proved in Step 2.0 that $|\p_t^{\ell}\Phi|\in L^q$ for any $q>1$, the previous rescaling argument implies that for every $\epsilon>0$,
\begin{align}
|\p_t^{\ell}\Phi|=O(r^{-\epsilon}), \quad |\p_t^{\ell}\p\Phi|=O(r^{-1-\epsilon}).
\end{align}
This completes the proof of \eqref{e:base-case-of-(2.0)}. 

Let $k\geq 3$ be any fixed positive integer. 
Based on \eqref{e:base-case-of-(2.0)} and the hypothesis $|\p_t^{\ell}\p^j \Phi| = O(r^{-j-\epsilon})$
for every $\epsilon>0$, $\ell\in\dZ_+$ and $0\leq j\leq k-3$, we proceed to prove the induction step. That is, 
we will show that \begin{align}
|\p_t^{\ell}\p^{k-2}\Phi| = O(r^{2-k-\epsilon}).
\end{align}

The proof is the same as before. Computing the Laplcian of $\p_t^{\ell}\p^{k-3}\Phi$,  we have that
\begin{align}
\Delta(\p_t^{\ell}\p^{k-3}\Phi)  & =  \p^{k-3}\Delta(\p_t^{\ell}\Phi) + Q_1 * \p (\p_t^{\ell}\p^{k-3}\Phi) + Q_2 * (\p_t^{\ell}\p^{k-3}\Phi) +\sum\limits_{j=1}^{k-4} F_j*\p_t^{\ell}\p_j\Phi.
\end{align}
Using the induction hypothesis, one can eventually obtain
\begin{align}
\Delta(\p_t^{\ell}\p^{k-3}\Phi) + Q_1 * \p (\p_t^{\ell}\p^{k-3}\Phi) + Q_2 * (\p_t^{\ell}\p^{k-3}\Phi) = O(r^{1-k-\epsilon})\end{align}
for every $\epsilon>0$. The same rescaling argument implies that
\begin{align}
|\p_t^{\ell}\p^{k-3}\Phi| = O(r^{3-k-\epsilon}), \quad |\p_t^{\ell}\p^{k-2}\Phi| = O(r^{2-k-\epsilon}).
\end{align}
which proves \eqref{e:mixed-regularity-induction-step}. Therefore, the proof is done. 
\end{proof}

For our purpose later, we also need the following lemma.

\begin{lemma}Let $\Delta$ denote the Hodge Laplacian on $Q$, then 
\begin{equation}\Delta(r^2)=-6+\tO(r^2).\end{equation}
\end{lemma}

\begin{proof}
This follows from similar, and simpler arguments as above. First,
\begin{equation}dr^2=2rdr=2y_\alpha \eta_\alpha+\tO(r^4),\end{equation}
so it follows that 
\begin{align}*dr^2&=2y_\alpha \eta_{\widehat\alpha}\wedge \dvol_T+\tO(r^3)\\
d*dr^2&=6\dvol_N\wedge \dvol_T+\tO(r^2)=6\dvol_g+\tO(r^2).\end{align}
Hence $\Delta(r^2)=d^*dr^2=-*d*dr^2=-6+\tO(r^2)$.\end{proof}

\subsection{Green's currents on a cylinder} \label{ss:complex-greens-currents}

In this subsection we assume $Q\equiv D\times\dR$ is a Riemannian product of a K\"ahler manifold $(D, \omega_D, J_D)$ of complex dimension $n-1$, and the real line $\R$ with coordinate $z$.
Given a smooth divisor $H\subset D$, let  $P\equiv H\times \{0\}\subset D\times\{0\}$. In our discussion sometimes we also naturally identify $H$ with $P$. 
The results of this subsection will be purely local so $D$ and $H$ are not necessarily compact. Our goal is to prove a few local expansion results, which will be used in Section \ref{s:neck}. 

The splitting of a line $\dR$ allows us to study the normal exponential map in $Q$ in terms of the normal exponential map in $D$.  Notice that the normal bundle of $P$ in $Q$ is naturally a Riemannian direct sum 
\begin{equation}{N}=N_0\oplus \R_z,\end{equation}
where $N_0$ is the normal bundle of $H$ in $D$ given as the orthogonal complement $(TH)^{\perp}$ of $TH$ in $TD|_H$ (with respect to $\omega_D$). So  $N_0$ is naturally a  hermitian line bundle. We also naturally identify $N_0$ with the holomorphic normal bundle $(TD|_H)/TH$, as complex line bundles. Therefore, $N_0$ can be viewed as a holomorphic hermitian line bundle. The {\it normal exponential map} of $H$ in $D$ is defined by 
\begin{equation}
\Exp_H: N_0\to D,\ (p,v)\mapsto\Exp_p(v),\end{equation}
 which gives a local diffeomorphism from a neighborhood of the zero section in $N_0$ to a tubular neighborhood of $H$ in $D$. Immediately, \begin{equation}
\label{eqn3-1}
d\text{Exp}_H: TN_0|_{H} \longrightarrow TD|_{H}\end{equation}
is the identity map under the natural isomorphisms $TN_0|_{H}\cong N_0\oplus TH$ and $TD|_{H}\cong N_0\oplus TH
$. 

Given any point $p\in H$, we may choose local holomorphic coordinates $\{w_i\}_{i=1}^{n-1}$ on $D$, centered at $p$, such that $H$ is locally defined by $w_1=0$, and the coordinate vector fields $\p_{w_i}$ are orthonormal at $p$. 
Then $dw_1, w_2, \ldots, w_{n-1}$ induces local holomorphic coordinates on $N_0$, which we denote by $\{\zeta, w_2', \ldots, w_{n-1}'\}$. 
 Given any $(p,v)\in N_0$, its coordinates are by definition given as
\begin{align}
\begin{cases}
\zeta= (dw_1)_p(v), \\
w_j'=w_j(p), & j\geq 2.
\end{cases}\end{align}
Under the normal exponential map $\Exp_H$, these coordinates can also be viewed as local (non-holomorphic) coordinates on $D$, and when restricted to $H$ we have $w_j'=w_j(j\geq 2)$ and $d\zeta=dw_1$.  In particular, $\{w_2', \cdots, w_{n-1}'\}$ still give holomorphic coordinates on $H$. 

Similarly using $\Exp_H$, the coordinate vector field $\p_\zeta$, originally defined on the normal bundle $N_0$, can also be viewed as a local (non-holomorphic) vector field on $D$. When restricted to $H$, the vector field $\p_\zeta$ can hence be  identified with the local section $\sigma$ of $(TH)^{\perp}\subset TD|_H$ given by the orthogonal projection of the holomorphic vector field $\p_{w_1}$. 
Then we obtain a local unitary frame $e$ of $(TH)^{\perp}$ given by
\begin{equation}e\equiv\sigma/|\sigma|.
\end{equation}
These  generate fiber coordinates $y, \bar y$ on $N_0$ such that 
\begin{equation}y =|\sigma|\cdot \zeta.\end{equation}  
In this way we obtain local coordinates $\{y,\bar{y}, y_3=z, w_2', \bar w_2', \ldots, w_{n-1}', \bar w_{n-1}'\}$ in a neighborhood of $p$ in $Q$. To match with the notation in the previous subsection, with respect to the local orthonormal basis $\{\sigma+\bar\sigma, \sqrt{-1}(\sigma-\bar\sigma), \p_z\}$, the normal geodesic coordinates are given by $\{y_1=Re(y), y_2=Im(y), y_3=z\}$, and 
$r^2=|y|^2+z^2.$
Also, the convention for the orientation is given such that 
\begin{equation}2^{-(n-1)} \sqrt{-1} dy\wedge d\bar y\wedge  dz \wedge  \prod\limits_{j=2}^{n-1}(\sqrt{-1}dw_j' \wedge d\bar w_j')\end{equation} defines a positive volume form. In the following, we  will also use $\langle\cdot, \cdot\rangle$  to denote the hermitian inner product on $(1,0)$-type vectors. The relation with the Riemannian inner product is seen as 
\begin{equation}\label{unitary}\langle \xi, \xi\rangle=2\langle Re(\xi), Re(\xi)\rangle.\end{equation}
What the notation means will be clear in the context. 

By making $P$ and $Q$ smaller we get local existence of Green's current $G_P$ for $P$ in $Q$, by 
Theorem \ref{t:Green-expansion}, with the expansion given there. In our case the formula can be written in terms of the above complex coordinates  as 

\begin{proposition}\label{p:complex-Green-expansion}
Let $G_P$ be a Green's current  for $P\equiv H\times\{0\}$ in $Q$, then locally
\begin{equation}G_P=\psi\wedge dz+\mathcal{R},\end{equation}
 where $\mathcal{R}\in C^{\infty}$ and $\psi$ is family of real-valued   $(1,1)$-forms on $D$ parametrized by $z$,  satisfying 
 $\psi(-z)-\psi(z)\in C^\infty.$ Moreover, in terms of the above local coordinates we can write
\begin{equation}
\psi=\frac{\sqrt{-1}}{4r} dy\wedge d\bar y+\frac{1}{2r}(yd\bar y+\bar yd y)\wedge\Gamma +r\cdot d\Gamma+r^{-3}\Pi_2^{(4)}+ O'(r^2),\label{e:singular-2-form-expansion}
\end{equation}
where  $\Gamma$ is a  smooth real-valued $1$-form locally defined on $H$ given by 
\begin{equation}\Gamma(v)\equiv-\frac{\sq}{2} \langle\nabla_{v}\p_y, \p_y\rangle\Big|_H, \ v\in T_pH,\end{equation}
and $\Pi_2^{(4)}$ is the $2$-form given by Notation \ref{n:differential-form-notation} and each term in it contains at least one of the $dy$ or $d\bar{y}$.

\end{proposition}

\begin{proof}
This essentially follows from the fact that $P$ is located on the slice $\{z=0\}$ and $H$ is a complex submanifold of $D$. Indeed, we can decompose 
$G_P=\psi_1\wedge dz+\mathcal{R}_1,$
where $\mathcal{R}_1$ does not involve $dz$.
Given any compactly supported test form $\chi\in\Omega_0^{2n-4}(Q)$, we can write 
$\chi=\beta\wedge dz+\gamma,$
where $\gamma$ does not involve $dz$. Immediately,
\begin{align}
(\mathcal R_1, \DelH\gamma)&=\int_Q\mathcal R_1\wedge \DelH\gamma=0,
\\
(\psi_1\wedge dz, \DelH(\beta\wedge dz)) &= \int_Q\psi_1\wedge dz\wedge  \DelH(\beta\wedge dz)=0.
\end{align}
So it follows that \begin{equation}(\mathcal R_1, \DelH\chi)=(\mathcal R_1, \DelH(\beta\wedge dz))=(G_P, \DelH(\beta\wedge dz))=2\pi\int_P(\beta\wedge dz)=0.\end{equation}
This implies that $\DelH\mathcal R_1=0$ in the distributional sense. By the elliptic regularity, we have $\mathcal R_1\in C^{\infty}$. 

Now we write $\psi_1=\psi+\psi_2,$
where $\psi$ is $J_D$-invariant, i.e., of type $(1,1)$ in $D$, and $\psi_2$ is anti-$J_D$-invariant. Since $H$ is a complex submanifold of $D$,  the Dirac current $\delta_P$ is $J_D$-invariant, and hence $J_D(G_P)$ is also a Green's current for $P$. So we see that $\psi_2=\frac{1}{2}(G_P-J(G_P))$ is smooth. Then we have that 
\begin{equation}G_P=\psi\wedge dz+\mathcal{R} ,\end{equation}
where $\mathcal R=\mathcal R_1+\psi_2\wedge dz$ is smooth. 
Similarly, since  $\delta_P$ is invariant under $z\mapsto -z$, the difference $\psi(z)-\psi(-z)$ is smooth. 

To see the expansion of $\psi$, notice that $H$ is  K\"ahler, and hence it is a minimal submanifold of $D$. So the mean curvature  of $H$ in $D$ vanishes. Also notice that $\p_z$ is parallel on $Q$ so $A_{i\alpha\beta}=0$ if either $\alpha=3$ or $\beta=3$. This implies that 
\begin{equation}
\psi=\frac{\sqrt{-1}}{4r} dy\wedge d\bar y+\frac{1}{2r}(yd\bar y+\bar yd y)\wedge\Gamma +r\cdot \mathbb A+r^{-3}\Pi_2^{(4)}+ \tO(r^2),
\end{equation}
where 
\begin{align}
\Gamma =-\frac{1}{2}A_{i12}dx_i\quad  \text{and}\quad  \mathbb{A} =-\frac{1}{4}A_{ij\alpha\beta}dx_i\wedge dx_j.\end{align}
In particular, $\mathbb A=d\Gamma$. Re-writing 
$A_{i12}=\langle \nabla_{\p_{x_i}}\p_{y_2}, \p_{y_1}\rangle$ 
in terms of the coordinates $y, \bar y$ and keeping in mind \eqref{unitary}, we obtain the desired formula for $\Gamma$. 
\end{proof}

Proposition \ref{p:complex-Green-expansion} has a quick corollary which will be used in our later calculations. 
\begin{corollary}\label{c:psipower}
For any positive integer $k\geq 2$, we have
$
\psi^k = O'(r^{k-1}).
$

\end{corollary}

\begin{proof}
By Proposition \ref{p:complex-Green-expansion}, we may write $\psi =  \FT_1 + \FT_2 + \FT_3$, where
\begin{align}
\FT_1 & \equiv \frac{\sqrt{-1}}{4r} dy\wedge d\bar y,
\\
\FT_2 & \equiv \frac{1}{2r}(yd\bar y+\bar yd y)\wedge\Gamma=O'(1),
\\
\FT_3 & \equiv r\cdot d\Gamma+r^{-3}\Pi_2^{(4)}+ O'(r^2)=O'(r).
\end{align}
Immediately we have that
$(\FT_1)^2 = (\FT_2)^2=\FT_1\wedge \FT_2=0$ and for all $k\geq 1$, $\FT_2\wedge (\FT_3)^{k}=(\FT_3)^k=O'(r^k)$. 
Moreover, 
\begin{equation}
\FT_1\wedge \FT_3 = \frac{\sqrt{-1}}{4}dy\wedge d\bar{y}\wedge d\Gamma +\frac{\sqrt{-1}}{4r^4}dy\wedge d\bar{y}\wedge \Pi_2^{(4)} + O'(r).
\end{equation}
Notice that $\frac{\sqrt{-1}}{4}dy\wedge d\bar{y}\wedge d\Gamma$ is a smooth term and by definition $dy\wedge d\bar{y}\wedge\Pi_2^{(4)}=0$. Then 
$
\FT_1\wedge \FT_3 =  O'(r).
$
So for all $k\geq 1$,  
$
\FT_1\wedge (\FT_3)^k=O'(r^k).
$
Now by direct calculation, 
\begin{eqnarray}
\psi^{k}=k\Big((\FT_1)\wedge(\FT_3)^{k-1} +  (\FT_2) \wedge (\FT_3)^{k-1}\Big) + (\FT_3)^k.
\end{eqnarray}
The conclusion then follows. 
\end{proof}

Notice that the above local coordinates $\{y, \bar y\}$ are not canonical, and depend on the initial choice of the local coordinates $\{w_i\}_{i=1}^{n-1}$ on $D$. However,  a different choice of local holomorphic coordinates on $D$ will induce the coordinates $\tilde y, \bar{\tilde y}$ on fibers of $N_0$ such that 
\begin{equation}y=e^{\sq\phi}\cdot\tilde y \label{e:y-coordinate-change}\end{equation}
for some real function $\phi$ on $H$. In particular, we have the transformations
\begin{align}
dy\wedge d\bar y&=d\y \wedge d\bar\y-\sq d|y|^2\wedge d\phi, 
\\
\Gamma&=\widetilde{\Gamma}-\frac{1}{2}d\phi.\label{e:Gamma-transform}
\end{align}
This suggests viewing $\Gamma$ as a connection 1-form on the normal bundle. Indeed this is exactly the case. 

\begin{lemma} \label{l:lemmagamma}
$2\sq\cdot\Gamma$ is the Chern connection 1-form of the normal bundle $N_0$ with respect to the above hermitian holomorphic structure, in the local holomorphic frame $\sigma$. In other words, 
\begin{equation}\Gamma=\frac{1}{2}d^c_H \log |\sigma|.\end{equation}
\end{lemma}
\begin{proof}
By definition,  
$\sigma=f\p_y=\p_{w_1}-\sum_{j\geq 2}\mu_j \p_{w_j},$
where $f=|\sigma|>0$ is local real valued function on $H$, and $\mu_2, \cdots, \mu_{n-1}$ are local complex valued function on $H$. The key property we will use is that along $H$,  $\nabla_{\p_{\bar w_k}} \sigma$ is tangential to $H$ for $k\geq 2$. In fact, the K\"ahler condition implies $\nabla_{\bp_{w_k}}\p_{w_j}=0$ for all $j$, and hence
\begin{equation}\nabla_{\p_{\bar w_k}}\sigma=\nabla_{\p_{\bar w_k}}\Big(\p_{w_1}-\sum_{j=2}^{n-1}\mu_j\p_{w_j}\Big)=-\sum_{j= 2}^{n-1}\p_{\bar w_k}(\mu_j) \p_{w_j}. \end{equation}
Therefore, 
\begin{equation}\p_{w_k}f=\p_{w_k} \langle \p_y, \sigma \rangle= \langle \nabla_{ \p_{w_k} }\p_y, f\p_y\rangle+\langle \p_y, \nabla_{\p_{\bar w_k}} \sigma\rangle=f\langle \nabla_{\p_{w_k}}\p_y, \p_y\rangle,\end{equation}
and hence
\begin{equation}\langle \nabla_{\p_{w_k}}\p_y, \p_y\rangle=f^{-1}\p_{w_k}f=\p_{w_k}(\log f).\end{equation}
Differentiating  $|\p_y|^2=1$, we get
\begin{equation}\langle \nabla_{\p_{w_k}}\p_y, \p_y\rangle+\langle \p_y, \nabla_{\p_{\bar w_k}}\p_y\rangle=0,\end{equation}
which implies
\begin{equation}\langle\nabla_{\p_{\bar w_k}}\p_y, \p_y\rangle=-\p_{\bar w_k} \log f.\end{equation}
Therefore,
\begin{align}\Gamma
&=-\frac{\sq}{2} \Big(\sum_{k\geq 2}\langle \nabla_{\p_{w_k}}\p_y, \p_y\rangle dw_k+\sum_{k\geq 2}\langle \nabla_{\p_{\bar w_k}}\p_y, \p_y\rangle d\bar w_k\Big)
\nonumber\\
&=-\frac{\sq}{2} \Big(\sum_{k\geq 2}\p_{w_k}(\log f) dw_k - \sum_{k\geq 2}\p_{\bar w_k}(\log f) d\bar w_k\Big)
\nonumber\\
&=-\frac{\sq}{2}(\p_H \log f-\bp_H \log f)
\nonumber\\
&=\frac{1}{2}d^c_H\log f.\end{align}
\end{proof}

For later applications, we need a few more local expansion results.  We will also use the notation $O'$ and $\tO$ in Definition \ref{d:normal-regularity} when we discuss the expansion in a neighborhood of $H$ in $D$, and the distance function is locally given by $|y|$. Notice the following expansions are given in the local (non-holomorphic) coordinates $\{y, \bar y, w_2', \bar w_2', \cdots, w_{n-1}, \bar w_{n-1}'\}$, and by definition we have $w_j'|_H=w_j|_H$ for $j\geq 2$. 
 
\begin{proposition} \label{p:d^c|y|^2}
Locally near the point $p\in H$ we have 
\begin{equation} \label{eqn3-4}
d^c_D|y|^2=\sq(yd\bar y-\bar ydy)+4|y|^2\cdot\Gamma+\tO(|y|^3).
\end{equation}
\end{proposition}
The proof  relies on 
the following expansions 
of the holomorphic coordinate functions $w_j$.

\begin{lemma}\label{l:w_1-expansion}
We have the expansion
\begin{align}\begin{cases}w_1=a_1y+a_2y^2+\tO(|y|^3)
\\
w_j=w_j'+c_jy+d_jy^2+ \tO(|y|^3), & j\geq 2,
\end{cases}
\end{align}
where $a_1=|\sigma|^{-1}>0$, $a_2$, $c_j$, $d_j$ are local smooth functions on $H$.
\end{lemma}

\begin{proof}
By definition, $\sigma$ is the orthogonal projection of $\p_{w_1}$ onto $(TH)^{\perp}$, so  we have that
\begin{equation}\label{eqn3-188}a_1^{-1}\p_y|_H=\p_{w_1}+\sum_{j= 2}^{n-1} b_j \p_{w_j},\end{equation}
where $a_1=|\sigma|^{-1}>0$ and $b_j$ are smooth functions on $H$. Now write 
\begin{equation}\label{eqn2-22}
\p_y=\sum_{j=1}^{n-1} \frac{\p w_j}{\p y} \p_{w_j}+\sum_{j= 1}^{n-1} \frac{\p \bar w_j}{\p { y}} \p_{\bar w_j}.
\end{equation}
Then we get
\begin{align}\label{eqn2-34}
\frac{\p \bar w_j}{\p y}\Big|_H=0, \  j\geq 1,
\end{align}
which in particular implies 
\begin{equation}
\frac{\p  w_j}{\p \bar{y}}\Big|_H=\overline{\frac{\p \bar w_j}{\p y}}\Big|_H=0,
 \  j\geq 1.
\end{equation}
Now by the definition of the normal exponential map,  we have that at $p$, 
\begin{equation}\nabla_{\p_y}\p_y=\nabla_{\p_{\bar y}}\p_y=\nabla_{\p_{\bar y}}\p_{\bar y}=0. \end{equation}
Using the K\"ahler condition we have
\begin{equation}\nabla_{\p_{w_j}}{\p_{\bar w_k}}=\nabla_{\p_{\bar w_j}}\p_{w_k}=0, \ \ j, k\geq 1.\end{equation}
Then by \eqref{eqn2-22} we get
\begin{equation}
\frac{\p^2 w_j}{\p y\p \bar y}\Big|_H=\frac{\p ^2 w_j}{\p\bar y^2}\Big|_H=0, \ \ \ j\geq 1.
\end{equation}
Thus the conclusion follows. 
\end{proof}

\begin{proof}[Proof of Proposition \ref{p:d^c|y|^2}]
Given the above lemma we first obtain that 
\begin{align}
d\bar w_1&=a_1 d\bar y+\bar y (da_1+2\bar a_2d\bar y)+\tO(|y|^2),
\\
w_1d\bar w_1&=a_1^2yd\bar y+|y|^2a_1da_1+a_1y(2\bar a_2\bar y+a_2 y)d\bar y+\tO(|y|^3).
\end{align}
Hence 
\begin{equation} \label{eqn1001}
d_D^c |w_1|^2=\sq a_1^2(yd\bar y-\bar ydy)+\sq a_1\bar a_2 \bar y (2yd\bar y-\bar ydy)-\sq a_1 a_2 y(2\bar y dy-yd\bar y)+\tO(|y|^3).
\end{equation}
On the other hand, we have 
\begin{equation}
|w_1|^2=a_1^2|y|^2+a_1(a_2y+\bar a_2 \bar y)|y|^2+\tO(|y|^4).
\end{equation}
So
\begin{equation}\label{eqn1000}
d_D^c |w_1|^2=a_1^2 d_D^c|y|^2+|y|^2 d_D^c a_1^2+d_D^c (a_1(a_2y+\bar a_2 \bar y)|y|^2)+\tO(|y|^3).
\end{equation}
Now by Lemma \ref{l:w_1-expansion},
\begin{equation}d_D^c (a_1y)=d_D^c w_1+\tO(|y|)=-\sqrt{-1}dw_1+\tO(|y|),\end{equation}
so
\begin{equation}
d_D^cy=-\sq dy+\tO(|y|).
\end{equation}
Similarly, $d_D^c\bar{y}=\sq d\bar{y}+\tO(|y|)$.
Plugging these into \eqref{eqn1000}, and compare with \eqref{eqn1001} we obtain 
\begin{align}d_D^c |y|^2=\sq(yd\bar y-\bar ydy)-2|y|^2d_D^c \log a_1+\tO(|y|^3).
\end{align}
Thanks to Lemma \ref{l:w_1-expansion}, $a_1=|\sigma|^{-1}$ which is a smooth function on $H$, so
\begin{align}
d_D^c |y|^2=\sq(yd\bar y-\bar ydy)+2|y|^2d_H^c \log |\sigma|+\tO(|y|^3).
\end{align}
By Lemma \ref{l:lemmagamma}, $\Gamma=\frac{1}{2}d_H^c\log |\sigma|$, so we conclude 
\begin{equation}d_D^c |y|^2=\sq(yd\bar y-\bar ydy)+4|y|^2\Gamma +\tO(|y|^3).\end{equation}
\end{proof}

Now we prove an expansion result for the trace of $\psi$.
\begin{proposition} \label{p:trace-expansion} Let $\psi$ be the $2$-form on $Q$ given as in \eqref{e:singular-2-form-expansion}, then we have  the following expansion in a neighborhood of $p$ in $Q$
\begin{equation}
\label{e:trace expansion equation}\Tr_{\omega_D}\psi=\frac{1}{2r}+O'(r).\end{equation}
\end{proposition}
Using \eqref{e:singular-2-form-expansion} it is easy to see $\Tr_{\omega_D}\psi$ admits an expansion of the form 
\begin{equation}\Tr_{\omega_D}\psi=\frac{A_0}{r}+\frac{A_1y+\bar A_1\bar y}{r}+O'(r).\label{e:tr-exp}\end{equation}
for local functions $A_0, A_1$ defined on $H$. It suffices to show $A_0\equiv 1$ and $A_1\equiv 0 $. Since the left hand side is independent of the choice of local holomorphic coordinates,  it suffices to  work on the slice $z=0$ with  special local holomorphic coordinates in a neighborhood of $p\in H$, and  it suffices to understand the Taylor expansion along the fiber $N_0(p)$ of $N_0$ over the fixed point $p$. 

\begin{lemma}\label{l:holo-coor-at-a-point}
We may choose the above holomorphic coordinates $\{w_i\}_{i=1}^{n-1}$ centered at $p$,  so that $H$ is locally given by $w_1=0$ and 
\begin{equation}
\omega_D=\frac{\sq}{2}g_{i\bar j}dw_i\wedge d\bar w_j, 
\end{equation}
where
\begin{align}
\begin{cases}
g_{i\bar j}(0)=\delta_{ij}, &   1\leq i, j\leq n-1,\\ 
\p_{w_1}g_{i\bar j}(0)=0, & 1\leq  i, j\leq n-1,
\\
\p_{w_k}g_{1\bar 1}(0)=\p_{w_k}g_{i\bar j}(0)=0, &  2\leq i,j, k\leq n-1.
\end{cases}\label{e:at-a-point}
\end{align}

\end{lemma}

\begin{remark}
In fact, the only non-trivial Christoffel symbols at $p$ are
\begin{equation}\Gamma_{ij}^1 (0)=\p_{i}g_{j\bar 1}(0), \ \ \Gamma_{\bar i\bar j}^{\bar 1}=\p_{\bar i}g_{1\bar j}(0)
\end{equation}
for  $i, j\geq 2$. This is due to the constraint that  the equation $w_1=0$ defines $H$, which prevents us from  using substitutions like \begin{equation}w_1=z_1+\sum\limits_{i, j=2}^{n-1} C_{1ij}z_i z_j.\end{equation} Intrinsically, $\{\Gamma^1_{ij}\}_{i, j\geq 2}$ captures the second fundamental form of the complex hypersurface $H$ at $p$. 
\end{remark}

\begin{proof}[Proof of Lemma \ref{l:holo-coor-at-a-point}]
This follows from elementary manipulation. First, the holomorphic coordinates $\{w_i\}_{i=1}^{n-1}$  can be chosen such that $g_{i\bar j}(0)=\delta_{ij}$ for all $1\leq i, j\leq n-1$. By the substitution of the form 
\begin{equation}\begin{cases}
w_i=z_i+\frac{1}{2}\sum\limits_{j, k= 2}^{n-1} C_{ijk}z_j z_k+\sum\limits_{j= 2}^{n-1} D_{ij} z_1z_j+E_{i}z_1^2, \ \  2\leq i\leq n-1,\\
w_1=z_1+\sum\limits_{j= 1 }^{n-1} F_j z_1z_j,
\end{cases}\label{e:co-change}
\end{equation}
with suitable choices of coefficients, where $C_{ijk}=C_{ikj}$ for $2\leq i,j,k\leq n-1$.
One can plug both the Taylor expansion of $g_{i\bar{j}}$ along $z_k$'s and \eqref{e:co-change} into $\omega_D$. Comparing the coefficients, then it follows that,
\begin{align}
\begin{cases}
C_{ijk} = - \p_{w_k}g_{j\bar{i}}(0),
\\
D_{ij} = -\p_{w_j} g_{1\bar{i}}(0),
\\
E_i = -\frac{1}{2}\p_{w_1}g_{1\bar{j}}(0),
\\
F_i = - \p_{w_j}g_{1\bar{1}}(0),
\\
F_1 = -\frac{1}{2}\p_{w_1}g_{1\bar{1}}(0),
\end{cases}
\end{align}
where $2\leq i,j,k\leq n-1$. Then we can achieve \eqref{e:at-a-point} with $\{w_i\}_{i=1}^{n-1}$ replaced by $\{z_i\}_{i=1}^{n-1}$.
\end{proof}

Now we prove Proposition \ref{p:trace-expansion}.

\begin{proof}[Proof of Proposition \ref{p:trace-expansion}] The goal is to show $A_0=1$ and $A_1=0$ in the expansion \eqref{e:tr-exp}.
  We work in the above special holomorphic coordinates given by Lemma \ref{l:holo-coor-at-a-point}.  
 The first step is to show that the $O'(1)$-term in the expansion of $\psi$
given by Proposition \ref{p:complex-Green-expansion}
  in fact vanishes along $N_0(p)$. 
  To this end, notice that $\p_y=\sigma=\p_{w_1}$ at $p\in H$ and hence by Lemma \ref{l:w_1-expansion},
 \begin{equation}w_1 = y + a_2 y^2 + \tO(|y|^3).\label{e:unit-leading-coe}
 \end{equation}
  Since the only non-trivial Christofell symsbols at $p$ are $\Gamma_{ij}^1$ and $\Gamma_{\bar i\bar j}^{\bar 1}$ for $i, j\geq 2$, it easily follows that 
\begin{equation}d|\sigma|(p)=0.\label{e:d-sigma=0}\end{equation}
Combining \eqref{e:d-sigma=0} and Lemma \ref{l:lemmagamma}, \begin{equation}\Gamma(p)=\frac{1}{2}(d_H^c\log|\sigma|)(p)=0,\end{equation} for each $p\in H$. Therefore,  along the fiber $N_0(p)$ of the normal bundle $N_0(p)$,
 the expansion of $\psi$ in Proposition \ref{p:complex-Green-expansion}  becomes\begin{equation}\psi=\frac{\sq}{4|y|} dy\wedge d\bar y+O(|y|).\end{equation}

The next is to compute $A_0(p)$ and $A_1(p)$  in \eqref{e:tr-exp}.
As in the proof of Lemma \ref{l:w_1-expansion}, we obtain that 
\begin{align}\frac{\p w_j}{\p y}(p)&=\frac{\p w_j}{\p \bar y}=0, \ \  j\geq 2,
\\
\frac{\p^2 w_j}{\p y^2}(p)&=\frac{\p^2 w_j}{\p y\p\bar y}(p)=\frac{\p^2 w_j}{\p \bar y^2}(p)=0, \ \ j\geq 1. \end{align}
This particularly implies that $a_2(p)=0$ and 
along the fiber $N_0(p)$,
\begin{equation}
\label{eqn3.291}w_j=O(|y|^3), \ \  j\geq 2.\end{equation}
By Lemma \ref{l:holo-coor-at-a-point}, $\p_{w_1}g_{i\bar{j}}(p)=0$ for all $1\leq i,j\leq n-1$,  then the expansion of $\omega_D$ along the fiber $N_0(p)$ is at least quadratic in the $w_1$-direction, i.e., 
\begin{eqnarray}
\omega_D&=&\frac{\sq}{2} \Big(dw_1\wedge d\bar w_1+\sum_{j=2}^{n-1} dw_j\wedge d\bar w_j\Big)+O\Big(\sqrt{\sum_{j=2}^{n-1}|w_j|^2}\Big)+O(|w_1|^2)\nonumber\\
&=&\frac{\sq}{2} \Big(dw_1\wedge  d\bar w_1+\sum_{j=2}^{n-1} dw_j\wedge d\bar w_j\Big)+O(|y|^2). 
\end{eqnarray}
By \eqref{e:unit-leading-coe} and \eqref{eqn3.291},  along the fiber $N_0(p)$, we have 
\begin{equation}dw_1=dy+O(|y|^2),\quad dw_j=dw_j'+O(|y|^2), \ \ j\geq 2.
\end{equation}
So we get 
\begin{equation}
\omega_D=\frac{\sq}{2}\Big(dy\wedge d\bar y+\sum_{j=2}^{n-1} dw_j'\wedge d\bar w_j'\Big)+O(|y|^2)
\end{equation}
Since by definition,
\begin{equation}\Big(\Tr_{\omega_D}\psi\Big)\cdot \frac{\omega_D^{n-1}}{(n-1)!}=\psi\wedge \frac{ \omega_D^{n-2}}{(n-2)!}.\end{equation}
by elementary manipulations we get that $A_0(p)=1$ and $A_1(p)=0$. 
\end{proof}

We close this subsection by proving an expansion of a local holomorphic volume form on $D$. Given the choice of local holomorphic coordinates on $D$ as before, let $\Omega_D$ be a local holomorphic volume form in a neighborhood of $p$, then we can always write 
\begin{equation}\Omega_D=f\cdot dw_1\wedge dw_2\cdots \wedge dw_{n-1},\label{e:Omega-D-in-w-coordinates}\end{equation}
for a local nowhere vanishing holomorphic function $f$. 
Denote the local holomorphic volume form on $H$
\begin{equation}
\Omega_H\equiv dw_2'\wedge\cdots dw_{n-1}'=(dw_2\wedge\cdots \wedge dw_{n-1})|_H.
\end{equation}
Then $\Omega_H$ can be naturally viewed as a complex $(n-2)$-form in some neighborhood of $p$ in $D$, in the coordinate system given by $\{y, \bar y, w_2', \bar w_2', \cdots, w_{n-1}', \bar w_{n-1}'\}$.

\begin{proposition}\label{lem3-6} We have the following expansion
\begin{equation}\Omega_D=F(dy+2\sq y \Gamma)\wedge \Omega_H+\tO(|y|)dy+\tO(|y|^2) \label{e:Omega-D-expansion}\end{equation}
for some smooth function $F$ locally defined on $H$. 
\end{proposition}
\begin{proof}
We need to calculate the expansion for $dw_1\wedge \ldots\wedge dw_{n-1}$ in \eqref{e:Omega-D-in-w-coordinates}.
First, by Lemma \ref{l:w_1-expansion},\begin{equation}
dw_1=a_1(dy+yd_H\log a_1)+\tO(|y|) dy+\tO(|y|^2),
\end{equation}
where $a_1=|\sigma|^{-1}$.
Notice that
\begin{equation}
d_H\log a_1 =2\p_H\log a_1 - \sqrt{-1}d_H^c \log a_1.
\end{equation}
Applying Lemma \ref{l:lemmagamma},
\begin{equation}
d_H\log a_1=2(\p_H\log a_1+\sqrt{-1}\Gamma)
.\end{equation}
Next, applying Lemma \ref{l:w_1-expansion} to $w_j$'s for $j\geq 2$,
\begin{equation}
dw_j=dw_j'+c_jdy+ydc_j+\tO(|y|)dy+\tO(|y|^2).
\end{equation}
Since it holds that
\begin{equation}
\p_H \log a_1\wedge\Omega_H\equiv 0,
\end{equation}
then taking the wedge product,\begin{equation}
dw_1\wedge\cdots \wedge dw_{n-1}=a_1(dy+2\sq y\Gamma)\wedge \Omega_H+\tO(|y|)dy+\tO(|y|^2).
\end{equation}
On the other hand, we have the expansion of $f$ in \eqref{e:Omega-D-in-w-coordinates},  
\begin{equation}
f=f|_H+\frac{\p f}{\p y}\Big|_H\cdot y+\frac{\p f}{\p\bar y}\Big|_H\cdot \bar y+\tO(|y|^2).
\end{equation}
Therefore, 
\begin{equation}
\Omega_D=f|_H \cdot a_1\cdot (dy+2\sq y\Gamma)\wedge \Omega_H+\tO(|y|)dy+\tO(|y|^2).
\end{equation}
So we obtain the conclusion by taking  $F\equiv f|_H\cdot a_1$. 
\end{proof}
\subsection{A global existence result}\label{ss:global-existence}
We assume the same set-up as in Section \ref{ss:complex-greens-currents}.
 We further assume that $D$ is compact K\"ahler, and $H$ is a smooth divisor in $D$ which is  Poincar\'e dual to $\frac{k}{2\pi}[\omega_D]$ for some positive integer $k$. 
The following proposition establishes a global existence for Green's current in this setting. We thank Lorenzo Foscolo for discussions concerning the constructive proof.

\begin{proposition} \label{p:existence-Greens-current}
Given any constants $k_-, k_+\in \R$ with $
 k_--k_+=k,$ 
 there exists a unique global Green's current $G_P$ for $P$ in $Q$ such that the following properties hold:
\begin{enumerate}
\item  $G_P$ is of the form \begin{equation}G_P=\psi(z)\wedge dz, \end{equation}
where $\psi(z)$ is a family of  real-valued closed $(1,1)$-forms on $D$ parametrized by $z$ and satisfies the expansion \eqref{e:singular-2-form-expansion} near $P$.

\item 
For any  $k\in\dN$ and  $\delta\in(0,10^{-2})$, we have
\begin{align}
\label{e:greens-current-exp-asymp}
\begin{cases}
|\nabla^k(\psi(z)-(k_-z)\cdot \omega_D)|=O(e^{(1-\delta)\sqrt{\lambda_1}z}),& z\rightarrow -\infty,\\
|\nabla^k(\psi(z)-(k_+z)\cdot \omega_D)|=O(e^{-(1-\delta)\sqrt{\lambda_1} z}), & z\rightarrow \infty,
\end{cases}
\end{align}
where $\lambda_1>0$ is the first eigenvalue of the Hodge Laplacian  acting on real-valued closed $(1,1)$-forms on $D$, and 
\eqref{e:greens-current-exp-asymp} are with respect to the fixed product metric on $Q$. 
\end{enumerate}

\end{proposition}

\begin{proof}[Proof of Proposition \ref{p:existence-Greens-current}]
We first prove the existence part.  Let $\{\phi_j\}_{j=0}^{\infty}$ be a complete  $L^2$-orthonormal basis of eigenvectors for the Hodge Laplacian $\Delta_D$ acting on real-valued closed  $(1,1)$-forms on $D$. We suppose $\Delta_D \phi_j=\lambda_j\phi_j$, and $\lambda_j\geq 0$ is increasing in $j$.  Our basic strategy is to first construct a formal series and then prove the convergence.

To begin with, the Dirac $3$-current $\delta_P$ of $P\subset Q$ has a formal expansion along the $D$ direction, 
\begin{equation}2\pi \delta_P=\sum_{j=0}^{\infty} f_j(z)\phi_j\wedge dz. \end{equation}
Here $f_j(z)$ is a $0$-current on $\R$ given by 
\begin{equation}f_j(z) \equiv 2\pi\Big( \int_H *_D\phi_j \Big)\delta_{0}(z), 
\end{equation}
where $\delta_0(z)$ is the standard Dirac $0$-current acting on functions on $\R$, supported at  $\{z=0\}$. 
Furthermore, if $\Delta_D\phi_j=0$, then $*_D\phi_j$ is closed and 
\begin{equation}\int_H*_D\phi_j=\int_D \frac{k}{2\pi} \omega_D\wedge *_D\phi_j=\frac{k}{2\pi}\langle \omega_D, \phi_j\rangle_{L^2(D)}. \end{equation}
It follows that there is exactly one $j$, which we may assume to be $0$,  such that $\lambda_j=0$ and $f_j$ is non-zero. The corresponding eigenform $\phi_0$ is a multiple of $\omega_D$. Now suppose $G_P$ is given as a formal series
\begin{equation}G_P=\sum\limits_{j=0}^{\infty} h_j(z) \phi_j \wedge dz,\end{equation}
then we need $h_j$ to satisfy \begin{equation}\label{eqn3333}-h_j''(z)+\lambda_j\cdot h_j(z)=f_j(z).\end{equation}
If $\lambda_j>0$, we can write down a solution 
 \begin{align}h_j(z)&=\frac{1}{-2\sqrt{\lambda_j}}\Big(e^{-\sqrt{\lambda_j}\cdot z}\int_{-\infty}^z e^{\sqrt{\lambda_j}\cdot u}f_j(u)du+e^{\sqrt{\lambda_j}\cdot z} \int_z^\infty e^{-\sqrt{\lambda_j}\cdot u}f_j(u)du\Big)
 \nonumber\\
 &=\begin{cases}\frac{\pi}{-\sqrt{\lambda_j}}\cdot e^{-\sqrt{\lambda_j}\cdot z}\cdot \int_H *_D\phi_j, & z>0,
\\
\frac{\pi}{-\sqrt{\lambda_j}}\cdot e^{\sqrt{\lambda_j}\cdot z}\cdot \int_H *_D\phi_j, & z\leq 0.
\end{cases}
\label{e:formal-solution-h-j}\end{align}
If $j>0$ and $\lambda_j=0$ we simply set $h_j(z)=0$. 
If $j=0$, we can write down a solution
 \begin{align}h_0(z)=\begin{cases}
 k_+z \langle \omega_D, \phi_0\rangle_{L^2(D)}, & z\geq 0,\\
 k_-z \langle \omega_D, \phi_0\rangle_{L^2(D)}, &  z\leq0.
 \end{cases}
 \end{align}
With this choice of $h_j$, we can define $G_P$ as the formal series given above. 
Next we claim that  $G_P$ is well-defined as a $3$-current on $Q$. For any test form $\chi\in\Omega_0^{m-3}(Q)$, we need to show the sum $\sum\limits_{j=0}^N\Big(h_j(z) \phi_j \wedge dz, \chi\Big)$ converges as $N\rightarrow\infty$. 
 So it suffices to show that 
\begin{equation}
\sum\limits_{j=0}^\infty\Big|\Big(h_j(z)\phi_j \wedge dz, \chi\Big)\Big|<\infty.\label{e:bounded-sum-current}\end{equation}
 To see this, for each $j$, we write
\begin{eqnarray}\Big(h_j(z)\phi_j \wedge dz, \chi\Big)&=&\int_Q h_j(z) \phi_j\wedge dz \wedge \chi
\nonumber\\
&=&\int_{\dR} \Big( h_j(z) \cdot \int_D\langle \chi, *_D\phi_j\rangle\dvol_{\omega_D}\Big) dz.\label{e:fubini}\end{eqnarray}
We first derive a uniform bound on  the integral $\int_D\langle \chi, *_D\phi_j\rangle\dvol_{\omega_D}$. If $\lambda_j>0$, then for all $\ell\geq 1$, 
\begin{eqnarray}\int_D\langle \chi(z), *_D\phi_j\rangle \dvol_{\omega_D} &=&
\frac{1}{(\lambda_j)^{\ell}}\int_D\langle \chi(z), *_D(\DelH_D)^{\ell}(\phi_j)\rangle \dvol_{\omega_D}
\nonumber\\
&=&\frac{1}{(\lambda_j)^{\ell}}\int_D\langle (\DelH_D)^{\ell}\chi(z), *_D\phi_j\rangle\dvol_{\omega_D}.\end{eqnarray}
By standard elliptic regularity we have \begin{equation}
\|\phi_j\|_{C^0(D)} \leq C \cdot (\lambda_j)^{\frac{n-1}{2}}, 
\end{equation}
where $C$ depends only on $n$ and the metric $\omega_D$. 
So it follows that
\begin{equation}
\Big|\int_D\langle \chi(z), *_D\phi_j\rangle \dvol_{\omega_D}\Big| \leq C\cdot \|\chi\|_{C^{2\ell}(Q)}\cdot \frac{1}{(\lambda_j)^{\ell-\frac{n-1}{2}}}.\end{equation}
Notice this estimate is independent of $z$. 
Next, we have
\begin{eqnarray}
\int_{\dR}|h_{j}(z)|dz &\leq&  \int_{-\infty}^{-1}|h_j(z)|dz +  \int_{-1}^{1}|h_j(z)|dz + \int_{1}^{+\infty} |h_j(z)|dz 
\nonumber\\
&\leq & C(1+\lambda_j^{\frac{n}{2}}).
\end{eqnarray}
Combining the above estimates, we have
\begin{equation}
|(h_j(z) \phi_j \wedge dz, \chi)|=\Big|\int_{\dR} \Big( h_j(z) \cdot \int_D\langle \chi, *_D\phi_j\rangle\dvol_{\omega_D}\Big) dz\Big| \leq C\cdot \|\chi\|_{C^{2\ell}(Q)}\cdot \frac{C(1+\lambda_j^{\frac{n}{2}})}{(\lambda_j)^{\ell-\frac{n}{2}}}.
\end{equation}
Applying Weyl's law, we see that if we fix $\ell$ sufficiently large,  then \eqref{e:bounded-sum-current} holds, and this completes the proof of the claim.

Now we show $G_P$ as defined above satisfies the current equation
\begin{equation}\label{eqn3347}
\Delta G_P =2\pi\delta_P .
\end{equation}
Given a test form $\chi\in\Omega_0^{m-3}(Q)$,
applying the definition of $f_j$, $h_j$ and integration by parts, it is straightforward to see that for each $j$,
\begin{equation}(h_j(z) \phi_j \wedge dz, \DelH\chi)=(f_j(z) \phi_j \wedge dz, \chi).\end{equation}
So 
\begin{equation}
(G_P, \Delta\chi)=\sum_j (f_j(z) \phi_j\wedge dz, \chi)=\sum_j 2\pi (\int_H *_D\phi_j) \int_D \phi_j\wedge \chi|_{z=0}. 
\end{equation}
By Hodge decomposition we can write  the $(n-2,n-2)$ component of $\chi|_{z=0}$ as $ d_D\alpha+\beta$ with $d_D^*\beta=0$. Since $\phi_j$ is a closed $(1,1)$ form, we have
\begin{equation}
\int_{D} \phi_j\wedge \chi|_{z=0}=\int_{D}\phi_j\wedge \beta=\langle \beta, *_D\phi_j\rangle_{L^2(D)}. 
\end{equation}
Notice $\beta=\sum_{j} \langle \beta, *_D\phi_j\rangle_{L^2(D)}*_D\phi_j$, we see $(G_P, \Delta\chi)=2\pi  \int_{P}\chi, $ which proves \eqref{eqn3347}. In particular, we know $G_P$ is smooth away from $P$. Clearly $G_P$ satisfies (1) with $\psi=\sum_j h_j(z) \phi_j$. 

Finally we study the asymptotics of  $G_P$ as $z\rightarrow\pm\infty$. If $\lambda_j>0$,  by standard elliptic regularity we have for all $\ell\geq 0$
\begin{equation}
\| \nabla^{\ell}\phi_j \|_{C^0(D)}\leq C\cdot (\lambda_j)^{\frac{n-1}{2}+\frac{\ell+1}{2}}.
\end{equation}
This implies that for any  $z>10^{n^2+\ell^2}$, we have \begin{equation}
 \|\nabla^{\ell}_Q(h_j(z)\cdot \phi_j)\|_{C^0(D\times\{z\})} \leq C(\lambda_j)^{\frac{n+\ell}{2}}e^{-\sqrt{\lambda_j} z}.
 \end{equation}
 By elementary computations, for each $\delta\in(0,10^{-2})$ and for $z$ sufficiently large, we have 
 \begin{equation}
\sum\limits_{\lambda_j>0} \|\nabla^{\ell}_Q(h_j(z)\cdot \phi_j)\|_{C^0(D\times\{z\})}  \leq C e^{-(1-\delta)\sqrt{\lambda_1}z}\cdot \sum\limits_{j=0}^{\infty}(\lambda_j)^{-6n}.
\end{equation}
By Weyl's law, the above series converges, and hence for each $\ell\in\dN$, 
 $\nabla_Q^{\ell}(G_P-h_0(z)\phi_0\wedge dz)$ is exponentially decaying as $z\rightarrow +\infty$. The argument is identical  for $z<0$. Then $G_P$ satisfies the \eqref{e:greens-current-exp-asymp}. 
 
To see the uniqueness,  suppose there is another Green's current $\widetilde G_P=\tilde \psi(z)\wedge dz$ also satisfying (1) and (2), then $\widetilde G_P-G_P=(\tilde \psi(z)-\psi(z))\wedge dz$ is a global harmonic 3-form with exponential decay at infinity. Applying Fourier expansion to $\tilde\psi(z)-\psi(z)$ along the $D$ direction, similar to what is done in the above, it is easy to see that $\tilde\psi(z)=\psi(z)$. 
 \end{proof}

The constants $k_-$ and $k_+$ determine some information of the above $\psi$, which will be used later. 

\begin{proposition}\label{p:cohomology-constant}
 Let $\psi$ be the $(1,1)$-current in Proposition \ref{p:existence-Greens-current}, then the following holds:
\begin{enumerate}\item The cohomology class $[\psi(z)]\in H^2(D; \R)$ is given by $k_-z[\omega_D]$ and $k_+z[\omega_D]$  for $z<0$ and $z>0$ respectively.
\item We have
\begin{equation}\label{eqn2-43}
\p_z\psi|_{z=0}=\frac{1}{2}(k_-+k_+)\omega_D.
\end{equation}
 In particular, $\p_z\psi|_{z=0}$ extends smoothly across $P$.
\end{enumerate}
\end{proposition}

\begin{proof}
Since $Q$ is a Riemannian product, we have for $z\neq 0$,  \begin{equation}\frac{d^2}{dz^2}\psi(z)=\Delta_D \psi(z)=d_Dd_D^*\psi(z)\end{equation}
is exact, which implies that the cohomology class $[\p_z\psi(z)]\in H^2(D; \R)$ is locally constant for $z\in \dR\setminus\{0\}$. On the other hand, by the exponential decay property in \eqref{e:greens-current-exp-asymp} we see that  
\begin{equation}\lim_{z\rightarrow\pm \infty}(\psi(z)-k_\pm z[\omega_D])     =0.\end{equation}
So (1) follows.
For item (2), denote
 \begin{equation}\tilde\psi(z)\equiv\psi(-z)+(k_-+k_+)z\omega_D. \end{equation} 
 Then $\tilde\psi\wedge dz$ is also a Green current for $P$ and it is also asymptotic to $(k_{\pm}z)\cdot \omega_D$ as $z\rightarrow \pm\infty$. Therefore by uniqueness, $\tilde\psi(z)=\psi(z)$. 
Taking the $z$-derivative at $z=0$ we get the conclusion.
\end{proof}

\section{The approximately Calabi-Yau neck region}

\label{s:neck}

In this section, we  build the {\it neck region}. It is one of the key geometric ingredients in this paper. 
For all dimensions, we will construct a family of incomplete  K\"ahler metrics with $S^1$-symmetry, on certain singular $S^1$-fibrations over a cylindrical base. These will serve to interpolate between the geometries at the ends of two Tian-Yau metrics.

  Our construction is motivated by the non-linear Gibbons-Hawking ansatz in Section \ref{s:torus-symmetries}. However, as  explained in Section \ref{s:torus-symmetries}, it does not seem easy to solve the non-linear reduced equation directly. Instead, we will use a singular solution to the \emph{linearized} ansatz, namely, the Green's current constructed in Section \ref{s:Greens-currents}, to obtain a family of K\"ahler metrics with $S^1$-symmetry, parametrized by a large parameter $T\gg 1$.  Here are two technical points to note:

\begin{itemize}
\item These metrics will not be shown to be smooth along the fixed loci of the $S^1$-action. Indeed, we  only prove that they are $C^{2, \alpha}$ for all $\alpha\in (0, 1)$ (Proposition \ref{p:C2alpha}). For our gluing construction, we need a further perturbation which lowers the regularity to be $C^{1,\alpha}$. This turns out to be sufficient for our analysis.
\item These metrics  are only \emph{approximately} Calabi-Yau, in an appropriate weighted sense (Proposition \ref{p:CY-error-small}). One can also perturb them to genuine incomplete Calabi-Yau metrics (see Section \ref{s:neck-perturbation}). 
\end{itemize}

Let us first set up some notations. Throughout this section we  fix integers $n\geq 2$ and $k>0$.
Let $(D, \omega_D, \Omega_D)$ be a closed Calabi-Yau manifold of complex dimension $n-1$. Here $\omega_D$ is a K\"ahler form in the class $2\pi c_1(L)$ for some ample holomorphic line bundle $L$,  $\Omega_D$ is a  holomorphic volume form on $D$, and we assume the following normalized Calabi-Yau equation holds \begin{equation}
\label{e:CY equation on D}
\frac{1}{(n-1)!}\omega_D^{n-1}=\frac{(\sq)^{(n-1)^2}}{2^{n-1}}\Omega_D\wedge\bar\Omega_D.
\end{equation}
We fix a hermitian metric on $L$ whose curvature form is $-\sq \omega_D$. This naturally induces a hermitian metric on all tensor powers of $L$. 
We  also fix a smooth divisor $H$ in the linear system $L^{\otimes k}$ and a defining section $S_H$ of $H$. Fix $r_D>0$ such that for any $p\in H$, the local coordinate system $\{y, \bar y, w_2', \bar w_2', \cdots, w_{n-1}', \bar w_{n-1}'\}$   introduced in Section \ref{ss:complex-greens-currents} exists in the ball $B_{10 r_D}(p)$.

Let $Q\equiv D\times \R$ be the Riemannian product of $D$ and the real line $\R$ parametrized by the 
coordinate $z\in (-\infty, \infty)$.  We denote
$P\equiv H\times \{0\}$ the codimension-$3$ submanifold of $Q$.
Using the normal exponential map on $D$ (resp. $Q$), we may always implicitly identify a tubular neighborhood of $P$ in $D$ (resp. $Q$) with a neighborhood of the zero section in the normal bundle $N_0$ (resp. $N=N_0\oplus \R$). Here we adopt the notation in Section \ref{ss:complex-greens-currents}, so $N_0$ is a hermitian line bundle and $N$ is a Riemannian vector bundle.  
  
Let us fix $k_-, k_+\in \Z$ with $k_->0$, $k_+<0$ and $k_--k_+=k$. Applying Proposition \ref{p:existence-Greens-current}, there is a unique  Green's current $G_P=\psi\wedge dz$ for $P$ in $Q$ such that the asymptotics \eqref{e:greens-current-exp-asymp} holds.

We also make the following notational conventions for this section:
\begin{itemize}
\item $\epsilon_T$ denotes a family of smooth functions on $D$, parametrized by $T\geq 1$, such that for each $k\geq 0$, the norm of its $k$-th derivative with respect to $\omega_D$ is of the form $O(e^{-\delta_k T})$ as $T\rightarrow\infty$, for some  $\delta_k>0$ (independent of $T$).
\item $\underline \epsilon_T$ denotes a function of $T$ which is $O(e^{-\delta T})$ as $T\rightarrow\infty$, for some $\delta>0$
\item $\epsilon(z)$ denotes a smooth function  defined on a subdomain in $Q$  for $|z|\gg1$  such that its all derivatives  decay exponentially fast as $|z|\rightarrow\infty$. 
\item $B_T$ denotes a family of smooth functions  on $D$, parametrized by $T\gg1$, such that for each $k\geq 0$, its $k$-th derivatives with respect to $\omega_D$ is bounded independent of $T$. 
\item $\underline B_T$ denotes a function of $T$ which is uniformly bounded as $T\rightarrow\infty$. 
\item $B(z)$ denotes a smooth  function of $z$, such that all its derivatives are uniformly bounded. 
\end{itemize}

The organization of this section is as follows. In Section \ref{ss:kaehler-structures} we construct a  family of $S^1$-invariant K\"ahler structures. Special attention is paid to understand the singularity structure near the fixed loci of the $S^1$-action. We will first construct a smooth compactification and write an explicit local model, and then study the regularity of the K\"ahler structures. In Section \ref{ss:complex-geometry} we show the underlying complex manifold is an open subset in an explicit $\C^*$ fibration over $D$, and derive a formula for the K\"ahler potential of our family of K\"ahler metrics. In Section \ref{ss:regularity-scales}  we study  and classify the limiting geometry of our family of metrics at regularity scales, which forms a foundation for our weighted analysis. In Section \ref{ss:neck-weighted-analysis} we define the relevant weighted H\"older spaces and prove a local weighted Schauder estimate. We also show our family of K\"ahler metrics are approximately Calabi-Yau by providing an estimate of the error in a weighted H\"older space. In Section \ref{ss:perturbation of complex structures} we deal with a perturbation of the complex structures of the underlying complex manifold, and estimate the error in a weighted H\"older space.

\subsection{Construction of a family of $C^{2,\alpha}$-K\"ahler structures}
\label{ss:kaehler-structures}
In this subsection we  use \eqref{omegahequation} to construct a family of $C^{2, \alpha}$-K\"ahler structures on certain singular $S^1$-fibrations over an increasing family of domains in $Q$. In Section \ref{sss:incomplete} we construct smooth K\"ahler metrics on principal $S^1$-bundles over $Q\setminus P$. In Section \ref{sss:smooth compactification} we construct a smooth compactification by adding fixed points of the $S^1$-action which lie over points in $P$. In Section \ref{sss:metric compactification} we study the regularity of the K\"ahler metric near the fixed loci.  Most of the quantities defined in this subsection will depend on the parameter $T$, but for simplicity of notation, we will not always keep track of this if the meaning is clear from the context. 

\subsubsection{K\"ahler metrics on $S^1$-bundles away from $P$}
\label{sss:incomplete}

For $T\gg1$, we define
\begin{equation}
\tilde\omega=T\omega_D+\psi.
\end{equation}
It can be viewed as a family of closed $(1,1)$-forms $\tilde\omega(z)$ on $D$ parametrized by $z\in \R$.  
Using the K\"ahler identities, we obtain
\begin{equation} \label{compatibility}\p_z^2\tilde\omega=\Delta_{D} \tilde\omega=-d_Dd_D^c \Tr_{\omega_D}\tilde\omega, \end{equation}
If we define
\begin{equation}
h\equiv \Tr_{\omega_D}\tilde\omega+q(z)
\end{equation}
for any smooth function $q(z)$, then the pair $(\tilde\omega, h)$ satisfies the first equation in \eqref{omegahequation}:
\begin{equation} \label{eqn4445}
\p_z^2\tilde\omega+d_Dd_D^ch=0.
\end{equation}
For our purpose we need to make a special choice of the function $q(z)$. First we choose a function $q_0(z)$ that satisfies 
 the following cohomological condition
\begin{equation} \label{eqn5-3}
q_0(z)\int_D \omega_D^{n-1}+(n-1)\int_D\tilde\omega(z)\wedge \omega_D^{n-2}=T^{2-n}\int_D \tilde\omega(z)^{n-1},  \ \   \forall z\in\dR.
\end{equation}
By Proposition \ref{p:cohomology-constant}, the cohomology class $[\psi(z)]\in H^2(D; \R)$ is piecewise linear in $z\in \R$, so
\begin{align}\label{q0definition}
q_0(z)=
\begin{cases}
T^{2-n}(T+k_+ z)^{n-1}-(n-1)(T+k_+ z), & z>0,
\\
T^{2-n}(T+k_- z)^{n-1}-(n-1)(T+k_- z), & z<0.
\end{cases}
\end{align}
It follows that $q_0(z)$ is identically zero if $n=2$.  But if $n>2$ then $q_0(z)$ is only $C^{1,1}$ at $z=0$ and we need to smooth it. We  fix throughout this section  a smooth function $L_0: \R\rightarrow \R$ satisfying
\begin{align} \label{e:l(z)}
L_0(z)\equiv
\begin{cases}
k_+z, & z>1,
\\
0 ,  & z=0,
\\
k_-z, & z<-1. 
\end{cases}
\end{align}
and let
\begin{equation} \label{e:LT z}
L_T(z)\equiv T+L_0(z).
\end{equation}
Let us we define
$q(z)\equiv T^{2-n}L_T(z)^{n-1}-(n-1)L_T(z)$.
Then
\begin{align}
q(z)
=
\begin{cases}
q_0(z), & |z|\geq 1,
\\
q_0(z) + T^{-1}B(z), & |z|<1.
\end{cases}
\end{align}
 It is also easy to see from \eqref{eqn5-3} that \begin{align} \label{haverage}
\int_D h\omega_D^{n-1}=
\begin{cases}T^{2-n}\int_D \tilde\omega(z)^{n-1}, &  |z|\geq 1, \\
T^{2-n}\int_D \tilde\omega(z)^{n-1}+T^{-1}B(z), & z\in [-1, 1].
\end{cases}
\end{align}
We refer to Remark \ref{r:error-function} for an explanation of this particular choice of $q(z)$.

To apply the construction in Section \ref{s:torus-symmetries}, we  need to restrict to the region in $Q$ where $\tilde\omega(z)$ is a positive $(1,1)$-form and $h$ is a positive function. 
For $T$ large we define $T_+>0$ and $T_-<0$ by 
\begin{equation}\label{e:define-T-plus-minus}
\begin{cases}
L_T(T_+)=T+k_+T_+=T^{\frac{n-2}{n}}\\
L_T(T_-)=T+k_-T_-=T^{\frac{n-2}{n}}.
\end{cases} 
\end{equation}
and denote by $Q_T\subset Q$ the closed subset defined by $z\in [T_-, T_+]$.  The reason for this choice of $T_\pm$ is also explained in Remark \ref{r:error-function}.

\begin{lemma}
For $T$ large, over $Q_T\setminus P$, both $\tilde\omega$ and $h$ are positive. Moreover, we have 
\begin{align}
h &=T^{2-n}(T+k_\pm z)^{n-1}+\epsilon(z), \quad   |z|\geq 1,\label{e:h-asymptotics}
\\h & =T+\frac{1}{2r}+O'(r) + T^{-1}B(z), \quad |z|\leq 1,\label{e:h-bounded-z}
\end{align}
where $O'(r)$ is independent of $T$, and its singular behavior near $P$ is given by Definition \ref{d:normal-regularity}.
\end{lemma}
\begin{proof}
We first consider $\tilde\omega$.
As $z\rightarrow\pm\infty$ the behavior of $\tilde\omega$ is governed by \eqref{e:greens-current-exp-asymp}, so for $T\gg 1$ we know $\tilde\omega$ is positive over the region where $z\in [-k_-^{-1}(T-1), -k_+^{-1}(T-1)]\setminus [-C, C]$ for some number $C>0$ independent of $T$. By the expansion of $\psi$ in a neighborhood of $P$ given in Proposition \ref{p:complex-Green-expansion}, for $T$ sufficiently large,  $\tilde\omega$ is also positive when $z\in [-C, C]$. Hence $\tilde\omega$ is positive over the region where $z\in [-k_-^{-1}(T-1), -k_+^{-1}(T-1)].$  Since this contains $Q_T$ we see in particular $\tilde\omega$ is positive over $Q_T\setminus P$. 

To deal with $h$ we need to analyze $q(z)$.  When $|z|\geq 1$, we have 
\begin{equation}\label{qexpansion}
q(z)=q_0(z)=T^{2-n}(T+k_\pm z)^{n-1}-(n-1)(T+k_\pm z),
\end{equation}
where the choice of $+$ or $-$ depends on whether $z>0$ or $z<0$. 
By \eqref{e:greens-current-exp-asymp} we then get 
\begin{equation} 
h=T^{2-n}(T+k_\pm z)^{n-1}+\epsilon(z).
\end{equation}
So we can find $C>0$ such that $h$ is positive when $z\in [T_-, T_+]\setminus [-C, C]$. On the other hand, on $[-C, C]$ we know by definition
\begin{equation}\label{e:q-bounded-distance}
q(z)=(2-n)T+T^{-1}B(z).
\end{equation}
Hence by the expansion in Proposition \ref{p:trace-expansion} we obtain \eqref{e:h-bounded-z}. This implies that for $T\gg1$,   $h$ is also positive  when $z\in [-C, C]$.
\end{proof}

Now on $Q\setminus P$ we define the 2-form 
\begin{equation}
\Upsilon\equiv \p_z\tilde\omega-dz\wedge d_D^ch.
\end{equation}
Then \eqref{compatibility} implies that $\Upsilon$ is closed on $Q\setminus P$ and hence $[\Upsilon]\in H^2(Q\setminus P, \mathbb R)$. 
Moreover, we have
\begin{lemma}\label{l:integrality}
The cohomology class $\frac{1}{2\pi}[\Upsilon]\in H^2(Q\setminus P; \mathbb R)$  is integral. 
\end{lemma}
\begin{proof}
As mentioned in the beginning of this section, we identify a tubular neighborhood of $P$ in $Q$ with a neighborhood of the zero section in its normal bundle $N=N_0\oplus \R$. For simplicity we may assume this neighborhood is given by $\mathcal B_\epsilon$, the 2-ball bundle over $P$ consisting of the set of all elements in $N_0\oplus \R$ with norm at  most $\epsilon$, and we denote by $\mathcal S_\epsilon$ the boundary of $\mathcal B_\epsilon$.

Fix $z_0>0$, then the composition of the natural maps 
\begin{equation}
D\simeq D\times \{z_0\}\hookrightarrow Q\setminus P \hookrightarrow Q\rightarrow D
\end{equation}
is the identity map, which implies that for all $k$, the map $H_k(Q\setminus P; \Z)\rightarrow H_k(Q;\Z)$ is surjective and we have a natural splitting 
\begin{equation}
H_2(Q\setminus P; \Z)=H_2(D; \Z)\oplus K
\end{equation}
for some $K$.
By assumption for $z>0$,  
\begin{equation}[\p_z\tilde\omega(z)]=[\p_z\psi(z)]=k_+[\omega_D]=2\pi k_+c_1(L),
\end{equation}
 so $\frac{1}{2\pi}[\Upsilon]|_{D\times \{z_0\}}=k_+c_1(L)$ is integral. Hence it suffices to show the integral of $\frac{1}{2\pi}\Upsilon$ over any element in $K$ is also an integer. 

By the Mayer-Vietoris sequence applied to $Q=(Q\setminus P)\cup \mathcal B_\epsilon$, we get 
\begin{equation}
0\rightarrow H_2(\mathcal S_\epsilon; \Z)\rightarrow H_2(Q\setminus P; \Z)\oplus H_2(\mathcal B_\epsilon; \Z)\rightarrow H_2(Q;\Z)\simeq H_2(D;\Z)\rightarrow 0.
\end{equation}
So we obtain the  exact sequence 
\begin{equation} \label{e: 4-23}
0\rightarrow K\rightarrow H_2(\mathcal S_\epsilon; \Z)\rightarrow H_2(\mathcal B_\epsilon; \Z)\simeq H_2(P; \Z).
\end{equation}
 On the other hand, 
by the Gysin sequence applied to the 2-sphere bundle $p:\mathcal S_\epsilon\rightarrow P$ we get
\begin{equation}
0\rightarrow H^2(P; \mathbb Z) \xrightarrow{p^*} H^2(\mathcal S_\epsilon; \mathbb Z)\xrightarrow{\int} H^0(P; \mathbb Z)\xrightarrow{\wedge e} H^3(P; \mathbb Z)\rightarrow\cdots
\end{equation}
where  $\int$ denotes integration over the 2-sphere fibers, and $\wedge e$ denotes the wedge product with Euler class of $\mathcal S_\epsilon$.
Since the Euler class $e$ of $N_0\oplus \R$ vanishes, the above becomes 
\begin{equation}
\label{e: 4-25}
0\rightarrow H^2(P; \mathbb Z)\xrightarrow{p^*}  H^2(\mathcal S_\epsilon; \mathbb Z)\xrightarrow{\int} H^0(P; \mathbb Z)\simeq \Z\rightarrow 0.
\end{equation}
\eqref{e: 4-23} and \eqref{e: 4-25} together imply that modulo torsion, $K$ is generated by the homology class of a 2-sphere fiber of $p$. So we just need to show $\int \frac{1}{2\pi}[\Upsilon]|_{\mathcal S_\epsilon}$ is an integer.

By the expansion of $\psi$ and $h$ in Proposition \ref{p:complex-Green-expansion} and Proposition \ref{p:trace-expansion}, it is easy to check that by restricting to the fiber of $N$ over $p$, we have 
\begin{equation}
\Upsilon|_{N(p)}=-\frac{\sq}{4r^3}(zdyd\bar y+(yd\bar y-\bar ydy)dz)+O(1).
\end{equation}
Further restricting to the $2$-sphere with radius $\epsilon$, we get 
\begin{equation}
\Upsilon|_{\mathcal S_\epsilon(p)}=-\frac{1}{2\epsilon^2}\dvol_{S^2_\epsilon}+O(1),
\end{equation}
where $\dvol_{S^2_\epsilon}$ is the area form of the standard $\epsilon$-sphere in $\R^3$. Taking the integral and letting $\epsilon\rightarrow 0$,\begin{equation}
\int_{\mathcal S_\epsilon(p)} \Upsilon=-2\pi. 
\end{equation}
\end{proof}

By Lemma \ref{l:integrality}, standard theory yields a $U(1)$ connection $1$-form $-\sq\Theta$ on a principal $S^1$-bundle
\begin{equation}\pi: \mathcal M^* \rightarrow Q_T\setminus P\end{equation} 
with curvature form $-\sq\Upsilon$. Moreover,  $\mathcal M^*$ restricts  to the standard  Hopf bundle on each normal $S^2$ to $P$ (it has degree $-1$ if we use the natural orientation). Then we have the second equation in \eqref{omegahequation}: 
\begin{equation} \label{d Theta equation}
d\Theta=\p_z\tilde\omega-dz\wedge d_D^c h.
\end{equation}
On ${\mathcal M^*}$ we define a real-valued 2-form and a  complex-valued $n$-form 
as follows  
\begin{align}
\omega &\equiv T^{\frac{2-n}{n}}(\pi^*\tilde\omega+dz\wedge \Theta),\label{e:def-omega}
\\ 
\Omega &\equiv\sq(hdz+\sq\Theta)\wedge \pi^*\Omega_D.\label{e:def-Omega}
\end{align}
One can directly check that both $\omega$ and $\Omega$ are closed. By the discussion in Section \ref{s:torus-symmetries}, we know $(\omega, \Omega)$ defines a smooth K\"ahler metric on ${\mathcal M^*}$, so that $\Omega$ is the holomorphic volume form and $\omega$ is the K\"ahler form. Also $h^{-1}$ has an intrinsic geometric meaning as the norm squared of the Killing field generating the $S^1$-action. 
By \eqref{e:CY equation on D} and straightforward calculations, we have
\begin{equation} 
\frac{(\sq)^{n^2}2^{-n}\Omega\wedge\bar\Omega}{\omega^{n}/n!}=\frac{T^{-1}h\omega_D^{n-1}} {(\omega_D+T^{-1}\psi)^{n-1}}.\label{e:error}
\end{equation}

\begin{definition}
\label{d:error-function} Given the above constructed K\"ahler metric $\omega$,  the error function is defined by 
\begin{equation}
\Err\equiv \frac{T^{-1}h\omega_D^{n-1}} {(\omega_D+T^{-1}\psi)^{n-1}}-1.
\end{equation}
\end{definition}
In particular, $\omega$ is a Calabi-Yau metric if $\Err=0$.

\begin{remark} \label{r:error-function}
 We now explain the reason for the various of choices of constants in the definition above.
 
  First, the constants $T_{\pm}$ are chosen so that on the boundary $\{z=T_\pm\}$, the size of the base $D\times \{z\}$ and the size of the $S^1$-circles are comparable, both of order $O(1)$. This can be seen from \eqref{e:h-asymptotics} and \eqref{e:greens-current-exp-asymp}.

 Secondly, the function $q(z)$ and the rescaling factor $T^{\frac{n-2}{n}}$ in the above definition of $\omega$ are chosen to make the K\"ahler metric $(\omega, \Omega)$ approximately Calabi-Yau in the following sense:
 \begin{enumerate}
 \item Applying \eqref{e:greens-current-exp-asymp} and \eqref{e:h-asymptotics}, we have $\Err=T^{-2}\epsilon(z(\bx))$ for $\bx\in \M^*$ satisfying $|z(\bx)|\geq C$.
 \item Applying  \eqref{e:trace expansion equation} and \eqref{e:h-bounded-z}, we  have  $\Err=O(T^{-2})$ for
$\bx\in\M^*$ satisfying $|z(\bx)|\leq C$ and $d_Q(\bx, P)\geq d_0>0$, where $d_0>0$ is some definite constant.
\end{enumerate}
Later we need a more precise weighted estimate on $\Err$, which will be shown in Proposition \ref{p:CY-error-small}.
 \end{remark}

 \begin{remark} \label{remark4-2-2}
As explained in Section \ref{s:torus-symmetries},  a priori these structures depend on the choice of $\Theta$.  Given two such $\Theta$ and $\Theta'$, the difference $\Theta'-\Theta$ is a closed 1-form on $Q_T\setminus P$. Since $P$ has codimension $3$ in $Q$, it follows that $H^1(Q_T\setminus P; \R)\simeq H^1(Q; \R)\simeq H^1(D; \R)$.  Then  
$
\Theta'-\Theta=df+\beta
$
for a function $f$ on $Q_T\setminus P$ and a harmonic 1-form $\beta$ on $D$.  If $b_1(D)=0$ then $\beta=0$, and the isomorphism class of the K\"ahler structure $(\omega, \Omega)$ is independent of the choice of $\Theta$.  If $b_1(D)>0$, up to gauge equivalence, $\Theta-\Theta'$ is the pull-back of a flat connection on $D$. In Section \ref{sss:complex manifold}, we will fix this choice of flat connection, by complex-geometric considerations, so then we have a unique choice of $\Theta$ up to gauge equivalence.  
\end{remark}

\begin{remark}\label{r:bundle fixed}
Later we will study the behavior of these metrics as $T\rightarrow\infty$. It is convenient to notice that the above $U(1)$ bundle and the connection 1-form $-\sqrt{-1}\Theta$ can be defined over the entire $Q\setminus P$. Hence as $T$ varies we may view $(\omega, \Omega)$ as a family of pairs of forms on the fixed $U(1)$ bundle, but they only define a K\"ahler structure over $Q_T\setminus P$.

\end{remark}

\subsubsection{Smooth compactification}
\label{sss:smooth compactification}
Next we move on to the study the compactified geometry of ${\mathcal M^*}$ near $P$. We first construct a smooth model for the compactification and then study the regularity of the K\"ahler structure $(\omega, \Omega)$ on this model. 

As before we  always identify a neighborhood $\mathcal U$ of $P$ in $Q$ with a tubular neighborhood of the zero section in $N_0\oplus \R$ over $H$. 
Denote by $\mathbb L_1$ and $\mathbb L_2$ the complex line bundles over $H$ given by the restriction
\begin{equation}
\mathbb L_1\equiv L^{\otimes -k_+}|_H; \ \  \mathbb L_2\equiv L^{\otimes k_-}|_H.
\end{equation}
Then as complex line bundles  $N_0$ is isomorphic to $L^{\otimes k}|_H\simeq \mathbb L_1\otimes \mathbb L_2$, and we fix such an isomorphism now.   Notice $N_0$ is equipped with a natural hermitian metric induced from the K\"ahler metric $\omega_D$ on $D$ (c.f. Section \ref{ss:complex-greens-currents}). This then determines a hermitian metric on $L$ hence on $\mathbb L_1$ and $\mathbb L_2$. Define
\begin{equation}\mathbb L\equiv \mathbb L_1\oplus \mathbb L_2,\end{equation}
and consider the map
\begin{equation}
\tau: \mathbb L \rightarrow N_0\oplus \R;  (s_1, s_2)\mapsto (s_1\otimes s_2, \frac{|s_1|^2-|s_2|^2}{2}).
\end{equation}
Away from the zero section in $\mathbb L$, $\tau$ is a principal $S^1$-bundle, with the $S^1$-action given by 
\begin{equation}\label{S1action equation}
e^{\sq\ft}\cdot (s_1, s_2)=(e^{-\sq t}s_1, e^{\sq t}s_2).
\end{equation}
As in Section \ref{ss:complex-greens-currents}, we choose local holomorphic coordinates $\{w_1, \cdots, w_{n-1}\}$ on $D$ centered at $p\in H$. These give rise to local coordinates $\{y, \bar y, w_2', \bar w_2', \cdots, w_{n-1}', \bar w_{n-1}'\}$ on $N_0$, and also a local unitary section of  $N_0$  in the form $$\e\equiv|\sigma|^{-1}\cdot \sigma.$$
Then we choose a local section $\e_L$ of $L|_H$ with $\e_L^{\otimes k}=\e$.  Correspondingly we get local unitary sections $$\e_1\equiv \e_L^{\otimes -k_+}, \ \ \ \  \e_2\equiv \e_L        ^{\otimes k_-}$$ of $\mathbb L_1, \mathbb L_2$ respectively. Then we obtain local fiber coordinates $u_1, u_2, y$ on $\mathbb L_2, \mathbb L_1, N_0$ respectively by writing 
\begin{equation}s_1=u_1\e_1, \ \ \ s_2=u_2\e_2, \ \ \  s= y\e. \end{equation} 
Then the map $\tau$ can be represented in coordinates as 
 \begin{equation} \label{e:y definition equation}
 \begin{cases}
 y=u_1u_2\\
 z=\frac{1}{2}(|u_1|^2-|u_2|^2)
 \end{cases}\end{equation}
Hence $\tau$ is the standard Hopf fibration $\C^2\rightarrow\R^3$ when restricting to each fiber. 

\begin{lemma}\label{lem4-3}
Over \  $\U\setminus P$, the principal $S^1$-bundles ${\mathcal M^*}$ and $\mathbb L$ are isomorphic.
\end{lemma}

\begin{proof}
Notice a principal $S^1$-bundle is topologically determined by its first Chern class. It suffices to compare the first Chern classes of ${\mathcal M^*}$ and $\mathbb L$ over the sphere bundle $\mathcal S_\epsilon$ for a small $\epsilon$.  As in the proof of Lemma \ref{l:integrality} the Gysin sequence gives 
\begin{equation}0\rightarrow H^2(P; \mathbb Z)\xrightarrow{p^*}  H^2(\mathcal S_\epsilon; \mathbb Z)\xrightarrow{\int} H^0(P; \mathbb Z)\simeq \Z\rightarrow 0.\end{equation}
From the proof of Lemma \ref{l:integrality} we know 
\begin{equation}
\int c_1({\mathcal M^*})=\int_{\mathcal S_\epsilon(p)} \frac{1}{2\pi}\Upsilon=-1. 
\end{equation}
Also by \eqref{modelcurvature} we have 
\begin{equation}
\int c_1(\mathbb L)=\int_{S^2\subset\R^3}\frac{1}{2\pi}\Upsilon_0=-1.
\end{equation}
So we have that 
${\mathcal M^*}=\mathbb L\otimes p^*  L'$ 
for some $U(1)$ bundle $L'$ over $P$. Now we restrict both ${\mathcal M^*}$ and $\mathbb L$ to the subset $H_0\subset \mathcal U$ where $y=0$ and $z=z_0$ for a fixed $z_0<0$. We can identify $H_0$ with $H$ by the projection map.  Now we claim both restrictions have first Chern class equal to $k_- c_1(\mathbb L_2)$.  For ${\mathcal M^*}$ this follows from construction and for $\mathbb L$ we notice that $z=z_0<0$ implies that $s_2\neq 0$ and $s_1=0$, so the projection map $(s_1, s_2)\mapsto |2z_0|^{1/2}\cdot s_2$ gives an isomorphism between the restriction of $\mathbb L$ and the unit circle bundle in $\mathbb L_2$. This also explains the choice of the weight of the $S^1$-action in \eqref{S1action equation}. 

It follows from the claim that $L'$ is indeed a trivial principal $S^1$-bundle, and this finishes the proof. 
\end{proof}

By Lemma \ref{lem4-3} we may glue ${\mathcal M^*}$ and $\mathbb L$ together to obtain  a differentiable compactfication $\M$ of ${\mathcal M^*}$. The projection map $\pi$ naturally extends to a map
\begin{equation}\pi: \M\rightarrow Q_T\end{equation}
which is a singular $S^1$-fibration,  with discriminant locus given by $P$. We  identify
\begin{equation}
 \cP\equiv\pi^{-1}(P)
 \end{equation}
 with the zero section in $\mathbb L$, and  identify a neighborhood of $\cP$ with a neighborhood of the zero section in $\mathbb L$ and the projection map $\pi$ with the above $\tau$.  

\subsubsection{Regularity of the K\"ahler structures}
\label{sss:metric compactification}
To study the regularity of the K\"ahler  structure $(\omega, \Omega)$ on the compactification $\mathcal M$, we will make a special choice of the connection 1-form $-\sq\Theta$ on a neighborhood $\mathcal V$ of $\cP$ in $\mathbb L$, with curvature form $-\sqrt{-1}\Upsilon$, which has explicit regularity behavior across $\mathcal{P}$. To do this, we need to proceed in a few steps.  First, we notice that $\{u_1, \bar u_1, u_2, \bar u_2, w_2', \bar w_2', \cdots, w_{n-1}', \bar w_{n-1}'\}$ provides local coordinates on $\mathbb L$, and we can define a local model connection 1-form on $\mathbb L$ by simply taking the model formula \eqref{e:model-connection}:
\begin{equation}
\Theta_0=-\sq\cdot\frac{\bar u_1 du_1-u_1d\bar u_1-\bar u_2du_2+u_2d\bar u_2}{2(|u_1|^2+|u_2|^2)}.
\end{equation}
Just as in the discussion in Section 2, we see $\Theta_0(\p_t)=-1$, where $\p_t$ is the vector field generating the $S^1$-action. It is clear that the definition of $\Theta_0$ depends only on the choice of $\sigma$ and does not depend on the choice of $\e_1$ and $\e_2$ (which has the freedom of multiplying by a constant root of unity). 
 
To make a globally defined connection 1-form, we need to add a correction term, and we define 
\begin{equation}
\Theta_1=\Theta_0+\frac{z}{r}\Gamma-\frac{k_-+k_+}{k_--k_+}\Gamma,
\end{equation}
where $\Gamma$ is the local 1-form given in Section \ref{ss:complex-greens-currents}, and we have implicitly viewed forms on $H$ as forms on $\mathbb L$ using the pull-back $\pi^*$. 

\begin{proposition}\label{p: prop4-4}
$-\sq\Theta_1$ is a globally-defined connection 1-form on the $S^1$-bundle  $\tau: \mathbb L\setminus \cP\rightarrow N\setminus H$, and we have 
\begin{equation}
d\Theta_1-\tau^*\Upsilon=O'(s)
\end{equation}
where 
\begin{equation}
s^2\equiv |u_1|^2+|u_2|^2=2r,
\end{equation}
 and we have adopted the $O'$ notation in Section \ref{ss:geodesic-coordinates} for the submanifold $\cP\subset \mathbb L$.  In particular, $d\Theta_1-\tau^*\Upsilon$ extends continuously across $\cP$. Moreover, it is cohomologically trivial in a neighborhood of $\cP$. \end{proposition}
\begin{proof}
To see $\Theta_1$ is a well-defined, we  consider the change of unitary frame $\e$ on $N_0$ to $\tilde \e=e^{\sq k\phi}\e$,  then we have 
\begin{equation}
\tilde y=e^{-(k_--k_+)\sq\phi}y; \ \ \tilde u_1=e^{k_+\sq\phi}u_1;\ \ \tilde u_2=e^{-k_-\sq\phi}u_2.
\end{equation}
for some local real-valued function $\phi$ on $H$. Then we get 
\begin{align}
\bar u_1 du_1-u_1d\bar u_1 &=\bar {\tilde u}_1d\tilde u_1-\tilde u_1d\bar {\tilde u}_1-2k_+\sq |u_1|^2d\phi,
\\
\bar u_2 du_2-u_2d\bar u_2 &=\bar {\tilde u}_2d\tilde u_2-\tilde u_2d\bar {\tilde u}_2+2k_-\sq |u_2|^2d\phi,
\\
\Gamma&=\widetilde\Gamma-\frac{k_--k_+}{2}d\phi.
\end{align}
Then it is a straightforward to compute that $\widetilde{\Theta}_1=\Theta_1$, which shows that $\Theta_1$ is globally defined. 

Now we consider the local expansion of $\Upsilon$. First, differentiating the expansion of $\psi$ in Proposition \ref{p:complex-Green-expansion},\begin{equation}
\p_z\tilde\omega=-\sq\frac{z}{2r^3}dy\wedge 
d\bar y-\frac{z}{2r^3} (yd\bar y+\bar ydy)\wedge\Gamma+\frac{z}{r} d\Gamma+O'(1)dy+O'(1)d\bar y+O'(r).
\end{equation}
Next, applying Proposition \ref{p:trace-expansion} and Proposition \ref{p:d^c|y|^2}, we obtain
\begin{equation}
d_D^ch=d_D^c(\frac{1}{2r}+O'(r))=-\frac{1}{4r^3} d_D^c |y|^2+O'(1)=-\frac{\sq(yd\bar y-\bar y dy)+4|y|^2\Gamma}{4r^3}+O'(1).
\end{equation}
Putting together these, and noting that $d\Theta_0$ is given as in \eqref{e:d Theta0}, we obtain
\begin{equation}
d\Theta_1-\tau^*\Upsilon=O'(1)dy+O'(1)d\bar y+O'(r)+O'(1)dz.
\end{equation}
Now translating into the coordinates $u_1, u_2$ on $\mathbb L$, and noticing that $\tau^*dx, \tau^*dy, \tau^*dz$ are all in $\tilde O(s)$ and $2r=s^2$, we obtain that $d\Theta_1-\tau^*\Upsilon=O'(s)$.  

To see the last statement, 
 we can restrict to the slice with $z>0$ and $y=0$,  then by Proposition \ref{p:cohomology-constant} we know $\Upsilon$ is cohomologous to $k_+\omega_D$. On the other hand, by definition $d\Theta_1$ on this slice is given by $(1-\frac{k_++k_-}{k_--k_+})d\Gamma$, which is also cohomologous to $k_+\omega_D$, using  Lemma \ref{l:lemmagamma} and the fact that $2\pi\cdot  c_1(N_0)=2\pi c_1(L)=(k_--k_+)[\omega_D]$. 
\end{proof}

The next Lemma allows us to correct the $O'(s)$ term on the right hand side. We fix an arbitrary $S^1$-invariant Riemannian metric on $\mathbb L$. 

\begin{lemma}\label{l:theta-regularity}
There exists a local 1-form $\theta$ on a neighborhood of $\cP$ in $\mathbb L$ with the following properties:\begin{enumerate}
\item $\theta=O'(s^2)$,
\item $\theta$ is smooth away from $\pi^{-1}(P)$,
\item $\mathcal L_{\p_t}\theta=0$,
\item $\p_t\lrcorner \theta=0$,
\item $d(\Theta_1+\theta)=\tau^*\Upsilon$.
\end{enumerate}
\end{lemma}

\begin{proof}
From Proposition \ref{p: prop4-4}, the above proposition we know that  in a neighborhood of $\mathcal P$,  $d\Theta_1-\tau^*\Upsilon$ is in $C^{\alpha}$ for all $\alpha\in (0, 1)$ and is cohomologous to zero. The existence of a solution $\theta$ to $d(\Theta_1+\theta)=\tau^*\Upsilon$ is obtained by adding the gauge fixing condition $d^*\theta=0$, and solving  the elliptic system with Neumann boundary condition
\begin{equation}
\begin{cases}
d\theta=\tau^*\Upsilon-d\Theta_1, \\
d^*\theta=0, \\
\theta(\nu)=0 \ \ \text{on} \ \ \p\mathcal V,
\end{cases}
\end{equation}
on a tubular neighborhood $\mathcal V$ of $\mathcal P$ in $\mathbb L$. 
 Standard elliptic regularity guarantees a solution $\theta\in C^{1, \alpha}$ and is smooth away from $\cP$. Since  both $\tau^*\Upsilon$ and $\Theta_1$ are $S^1$-invariant, by averaging we may assume $\theta$ is $S^1$-invariant too, hence $\mathcal L_{\p_t}\theta=0$ on the smooth part. Also since $\tau^*\Upsilon$ and $d\Theta_1$ are pulled-back from the base $Q_T\setminus P$, we have 
$
\p_t \lrcorner \tau^*\Upsilon=\p_t\lrcorner d\Theta_1=0. 
$
So we get 
\begin{equation}
d(\p_t\lrcorner \theta)=\mathcal L_{\p_t}\theta-\p_t\lrcorner(d\theta)=0.
\end{equation}
This implies $\p_t\lrcorner \theta$ is a constant. Now as we approach $\cP$, the norm of $\p_t$, with respect to  the fixed metric on $\mathbb L$, must go to zero, hence we see 
$
\p_t\lrcorner \theta=0.
$
The higher regularity of $\theta$ follows just as in the proof of Lemma \ref{l:higher-regularity}, since $\Delta\theta=d^*d\theta=d^*(\tau^*\Upsilon-d\Theta_1)=O'(1)$. 
\end{proof}

Now we define a fixed 1-form on $\mathbb L$ 
\begin{equation}
\Theta_{m} \equiv \Theta_1 + \theta, 
\end{equation}
Therefore, in a neighborhood of $\cP\subset\mathbb L$ minus $\cP$, the original choice of $\Theta$ can be written as 
\begin{equation}
\Theta=\Theta_m+\theta_f,\label{e:fixed-connection-form}
\end{equation}
where $\theta_f$ is a flat connection, and hence it  is gauge equivalent to the pull-back of a flat connection on $D$. Without loss of generality, we can then assume $\theta_f$ is smooth.

\begin{proposition}\label{p:C2alpha}With respect to the choice of the connection form $\Theta$ given in \eqref{e:fixed-connection-form}, 
$(\omega, \Omega)$ defined by \eqref{e:def-omega} and \eqref{e:def-Omega} gives a $C^{2, \alpha}$ (for all $\alpha\in (0, 1)$) K\"ahler structure on $\mathbb L$ which is invariant under the natural $S^1$-action and is smooth outside $\cP$. 
\end{proposition}

\begin{proof}
We first analyze the regularity of $\omega$, defined in \eqref{e:def-omega}.
To start with, let us compute the lifting $\pi^*\psi$. By \eqref{e:singular-2-form-expansion}, 
\begin{equation}\pi^*\psi=\pi^*\tilde\omega_0+\frac{1}{2r}(yd\bar y+\bar ydy)\wedge \Gamma+rd\Gamma+\pi^*(O'(r)dy+O'(r)d\bar y)+\pi^*O'(r^2),\end{equation}
where
$\tilde\omega_0=\frac{\sq}{4r} dy\wedge d\bar y$
is the standard form in the model setting \eqref{modelquantities}. 
We also notice that
\begin{align}
\pi^*(O'(r)dy+O'(r)d\bar y)&=sO'(s^2),
\\
\pi^*O'(r^2)&=O'(s^4).
\end{align}
Now by definition
\begin{equation}
\Theta =\Theta_0+\frac{z}{r}\Gamma+\frac{k_-+k_+}{k_--k_+}\Gamma +  \theta+\theta_f =\Theta_0+\frac{z}{r}\Gamma +  O'(s^2).\label{e:Theta-near-p}
\end{equation}
Moreover, according to the discussions in Section \ref{s:torus-symmetries}, we have 
$\pi^*\tilde\omega_0+dz\wedge \Theta_0=\omega_{\C^2},
$
where $\omega_{\C^2}=\frac{\sq}{2} (du_1\wedge d\bar u_1+du_2\wedge d\bar u_2)$ is the standard K\"ahler form of $\C^2$. 
Therefore, 
\begin{align}
\pi^*\psi + dz\wedge \Theta =& \omega_{\C^2} + rd\Gamma+
\frac{1}{2r}(yd\bar y+\bar ydy)\wedge \Gamma+dz\wedge (\frac{z}{r}\Gamma) + O'(s^3).
\end{align}
Using the relation $r^2=|y|^2+z^2$ and the simple computation
\begin{equation}
d(r\Gamma)= rd\Gamma+ dr\wedge \Gamma = rd\Gamma+\frac{1}{2r}(yd\bar y+\bar ydy)\wedge \Gamma+dz\wedge (\frac{z}{r}\Gamma),
\end{equation}
we have
\begin{align}
\pi^*\psi + dz\wedge \Theta =\omega_{\C^2}+ d(r \Gamma) + O'(s^3)
\nonumber=\omega_{\C^2}+ O'(s^3),
\end{align}
where we use the fact that $r=\frac{1}{2}s^2$ and hence $r\Gamma=s^2\Gamma$ is smooth on $\mathbb L$. Then it follows that 
\begin{equation}T^{\frac{n-2}{n}}\omega =T\pi^*\omega_D+\omega_{\C^2}+O'(s^3).\end{equation}
Hence we see the $(1,1)$-form  $\omega$ locally extends to a $C^{2, \alpha}$-form across the subset $\{u_1=u_2=0\}$; see item (4) of Example \ref{ex:regularity}.

Now we analyze the regularity of the holomorphic volume form $\Omega$, as defined by \eqref{e:def-Omega}.
By Lemma \ref{lem3-6}, locally we have
\begin{equation}\pi^*\Omega_D=F(u_1du_2+u_2du_1+2\sq u_1u_2\Gamma)\wedge \pi^*\Omega_H+\tO(s^2)(u_1du_2+u_2du_1)+\tO (s^3).
\end{equation}
Also 
\begin{equation}
hdz+\sq\Theta=q(z)dz+\frac{1}{|u_1|^2+|u_2|^2}(-\bar u_2 du_2+\bar u_1 du_1+\sq(|u_1|^2-|u_2|^2)\Gamma)+O'(s^2).
\end{equation}
Therefore,  
\begin{equation}\Omega=Fdu_1\wedge du_2\wedge \Omega_H +\sq F(u_2du_1-u_1du_2)\wedge \Gamma\wedge \Omega_H+\tO(s^2)+sO'(s^2).\end{equation}
This implies that $\Omega$ also extends to a $C^{2, \alpha}$ form across $\{u_1=u_2=0\}$. This is equivalent to saying that the almost complex structure $J$ determined by $\Omega$ extends to a $C^{2, \alpha}$ almost complex structure on $\M$. 
\end{proof}

\begin{remark}
\label{r:C2alpharegularity}
The reason we need the $C^{2, \alpha}$ regularity will be explained in Remark \ref{r:C1alphavsC2alpha}. If we only need weaker regularity then we only need to work out fewer terms in the expansion in Theorem \ref{t:Green-expansion}; similarly, with more work one can probably improve the regularity here, which is not needed for our purpose. But we point out that in general there is no reason to expect the above defined K\"ahler structures to be smooth. Only after we perturb these metrics to exactly Calabi-Yau metrics (see Section \ref{s:neck-perturbation} and \ref{s:gluing}) we can gain smoothness via elliptic regularity. This is in sharp contrast to the 2 dimensional case, where the Gibbons-Hawking ansatz yields an exact solution of the Calabi-Yau equation and simultaneously the metric completion near the (isolated) $S^1$-fixed points is smooth.
\end{remark}

Using the Newlander-Nirenberg theorem , we may locally find $C^{3, \alpha}$ holomorphic coordinates, making the complex structure locally standard while still keeping the K\"ahler form with $C^{2,\alpha}$  regularity . 

By construction the K\"ahler structure $(\omega, \Omega)$ is preserved by the natural $S^1$-action. The corresponding Killing field is given by \begin{equation}\p_t\equiv-\sq(u_1\p_{u_1}-u_2\p_{u_2})+\sq(\bar u_1\p_{\bar u_1}-\bar u_2\p_{\bar u_2}).
\end{equation}
The zero set $\mathcal P$ 
 is a complex submanifold of $\M$ which bi-holomorphic to $H\subset D$. 
We denote the corresponding holomorphic vector field by \begin{equation} \label{xi 10}
 \xi^{1,0}\equiv\frac{1}{2}(\p_t-\sq J\p_t).
 \end{equation}
We also have a smooth holomorphic projection $\pi: \mathcal M\rightarrow D\setminus H$ whose fibers are holomorphic cylinders (isomorphic to annuli in $\C$).  
In the next subsection we will understand the underlying complex manifold and the K\"ahler potentials on $\M$. 

\subsection{K\"ahler geometry}

\label{ss:complex-geometry}
A key feature in the analysis on K\"ahler manifolds is that we can describe the geometry locally in terms of a single \emph{potential function}. This has led to a vast simplification of formulae in K\"ahler geometry as compared to more general Riemannian geometric setting, and it also has allowed various techniques from PDE and several complex variables etc. to be exploited. 
The goal of this subsection is to derive a formula for the relative K\"ahler potential for our K\"ahler manifold $(\M, \omega, \Omega)$. This is one of the most crucial observations in this paper.

In Section \ref{sss:complex manifold} we  identify the underlying complex manifold of the family of K\"ahler metrics constructed in Section \ref{ss:kaehler-structures} as a family of open subsets of a fixed complex manifold. In Section \ref{sss:kahler potential} we  derive a formula for the relative K\"ahler potential.

\subsubsection{The underlying complex manifold}
\label{sss:complex manifold}
We define the following holomorphic line bundles on $D$
\begin{equation}
\begin{cases}
 L_+\equiv L^{-\otimes k_+}, \\
 L_-\equiv L^{\otimes k_-},
\end{cases}
\end{equation}
and we denote by $J_\pm$ the complex structure on the total space of $L_\pm$. 
Denote by $\mathcal N^0$ the hypersurface in the total space of $L_+\oplus L_-$ defined by the equation 
\begin{equation}\label{e: SH definition}\zeta_+\otimes \zeta_-=S_H(x),\end{equation}
where $\zeta_\pm$ denotes points on the fibers of $L_\pm$ over $x\in D$, and $S_H$ is the section we fixed at the beginning of this section.  Since $H$ is smooth, $\mathcal N^0$ is also smooth and the submanifold 
\begin{equation}
\mathcal H\equiv\{\zeta_+=\zeta_-=0\}
\end{equation}
of $\mathcal N^0$
is naturally isomorphic to $H$.
The fixed hermitian metric on $L$ then induces hermitian metrics on $L_\pm$, and these yield  norm functions on $L_\pm$:
\begin{equation}
r_\pm(\zeta_\pm)\equiv\|\zeta_\pm\|.
\end{equation}
Then by pulling back to $\mathcal N^0$ through the projection maps to $L_\pm$ we may also view $r_\pm$ as functions on $\mathcal N^0$.

There is a natural holomorphic volume form on $\mathcal N^0$ given by 
\begin{equation}\label{e:Omega N0 definition}
\Omega_{\mathcal N^0}\equiv\frac{\sq}{2} (\frac{d\zeta_+}{\zeta_+}-\frac{d\zeta_-}{\zeta_-})\wedge \Omega_D,
\end{equation}
where $\Omega_D$ means the pull-back of $\Omega_D$ to $\mathcal N^0$ and for simplicity of notation we  omit the pull-back notation when the meaning is clear from the context. The expression on the right hand side of \eqref{e:Omega N0 definition} should be understood after choosing a local holomorphic frame $\sigma$ of $L$, so that $\zeta_\pm$ becomes local holomorphic functions on $L_\pm$. One can check the definition does not depend on the choice of $\sigma$, and $\Omega_{\mathcal N^0}$ is a well-defined holomorphic volume form on $\mathcal N^0$.

There is a natural $\C^*$ action on $\mathcal N^0$ given by
\begin{equation}
\lambda\cdot (\zeta_+, \zeta_-)\equiv(\lambda^{-1}\zeta_+, \lambda \zeta_-), \ \ \lambda\in \C^*. 
\end{equation}
and we denote by 
\begin{equation}
\xi_{\mathcal N^0}\equiv\sq(-\zeta_+\p_{\zeta_+}+\zeta_-\p_{\zeta_-})
\end{equation}
the corresponding holomorphic vector field (the choice of coefficients is made so that the real part of $\xi_{\mathcal N^0}$ is twice the real vector field generated by the induced $S^1$-action, as given in \eqref{xi 10}).
One checks that $
\xi_{\mathcal N^0}\lrcorner\   \Omega_{\mathcal N^0}=\Omega_D
$ and $\mathcal L_{\xi_{\mathcal N^0}}\Omega_{\mathcal N^0}=0$.

Our main goal is  to holomorphically embed $(\mathcal M, \Omega, \xi^{1,0})$ into $(\mathcal N^0, \Omega_{\mathcal N^0}, \xi_{\mathcal N^0})$, for appropriate choice of $\Theta$ (which is necessary in the case $b_1(D)>0$, see Remark \ref{remark4-2-2}).  This will be summarized in Proposition \ref{p:complex geometry}. Notice for our later purpose the quantitative choice of various constants in the arguments below will  be important.

First let us fix an arbitrary choice of $\Theta$.  Denote
  \begin{equation}
  \begin{cases}
  \M_-\equiv {\mathcal M^*}\setminus \pi^{-1}(H\times [0, \infty))\\
  \M_+\equiv {\mathcal M^*}\setminus \pi^{-1}(H\times (\infty, 0]).
  \end{cases}
  \end{equation}
On $\M_-$ we can trivialize the $U(1)$ connection $-\sq\Theta$ along the $z$ direction so that the $z$ component $\Theta_z$ vanishes identically. Denote by $\Theta|_z$ the restriction of $\Theta$ to the slice $D\times \{z\}$ for $z<0$ and to $(D\setminus H)\times \{z\}$ for $z\geq 0$. From \eqref{d Theta equation} we see that that curvature form of $-\sq\Theta|_z$ is given by $-\sq\p_z\tilde\omega$.

By definition  $\M|_{z=T_-}$ is a principal $S^1$-bundle over $D$ endowed with a unitary connection with curvature $-\sqrt{-1}\p_{z}\tilde\omega|_{z=T_-}$. So we may assume $\M|_{z=T_-}$   embeds into a holomorphic line bundle $\tilde L_-$ over $D$, as the unit circle bundle defined by a hermitian metric $\|\cdot\|_{\sim}^2$, and the connection 1-form $-\sq \Theta|_{T_-}$ agrees with the restriction of the Chern connection form. Denote by $\tilde r_-$ the norm function on $\tilde L_-$ corresponding to the hermitian metric. By Proposition \ref{p:cohomology-constant} we have $[\p_z\tilde\omega]|_{z=T_-}=k_-[\omega_D]\in H^2(D; \R).
$
 If $b_1(D)=0$, then $\tilde L_-$ is isomorphic to $L_-$ as holomorphic line bundles. In general $\tilde L_-$ is isomorphic to $L_-\otimes \mathcal F_-$ for a flat holomorphic line bundle $\mathcal F_-$.

Furthermore, we may extend $-\sq \Theta|_{T_-}$ naturally to the complement of the zero section ${\bf 0}_{\tilde L_-}$ in $\tilde L_-$, via the fiberwise projection, and the resulting 1-form coincides with the Chern connection form   
$\sq\tilde {r}_-^{-1}\tilde J_-d\tilde r_-$, where $\tilde J_-$ denotes the complex structure on $\tilde L_-$.

Now we define a map $\Phi_-: {\mathcal M}_-\rightarrow \tilde L_-\setminus {\bf 0}_{\tilde L_-}$, where ${\bf 0}_{\tilde L_-}$ denotes  the zero section in $\tilde L_-$. First at $z=T_-$ we define $\Phi_-$ to be the  natural inclusion map as above, multiplied by $e^{A_-}$ for some constant $A_-$ to be determined later. Then using the trivialization of the $S^1$-bundle ${\mathcal M}_-$ along  the $z$ direction and the natural scaling map on $\tilde L_-$, we extend the map to the whole ${\mathcal M}_-$ by setting 
\begin{equation}
\label{e:tilde r}\tilde r_-=e^{A_--\int_{T_-}^z h(u) du}.
\end{equation}
Notice both $\mathcal M_-$ and $\tilde L_-$ have natural projection maps to $D$, and $\Phi_-$ clearly commutes with the projection maps to $D$, so  $\Phi_-^*\alpha=\alpha$ for any $1$-form $\alpha$ which is a pull-back from $D$ (again we omit the pull-back notation here). Since 
\begin{equation}
\p_z\Theta|_z=d_D^ch=-J_D d_Dh
\end{equation}
we have
\begin{equation}\tilde r_-^{-1}\Phi_-^*d\rmi=-hdz-\int_{T_-}^z du\wedge d_Dh=-hdz-J_D(\Theta|_{z}-\Theta|_{T_-}),
\end{equation}
noticing that $\Theta|_z-\Theta|_{T_-}$ is a 1-form which is a pull-back from $D$.
So 
\begin{equation}\rmi^{-1}\Phi_-^*(d\rmi+\sq \tilde J_-d\rmi)=-hdz-\sq \Theta|_z-\sq(\Theta|_{T_-}-\Theta|_z)-J_D(\Theta|_{z}-\Theta_{T_-})
\end{equation}
is a $(1, 0)$ form on ${\mathcal M}_-$. It then follows that $\Phi_-$ is holomorphic. It is also clear that $\Phi_-$ is $S^1$-equivariant with respect to the natural $S^1$-action on $\M_-$ and $\tilde L_-$. 

 Since $h$ is positive we see that the image of $\Phi_-$ is bounded in $L_-$. Since $\M\setminus \M_-$ is of complex codimension one,  by the removable singularity theorem for bounded holomorphic functions, $\Phi_-$ extends to a holomorphic map on the entire $\M$. 

Similarly we get a holomorphic  embedding
 $
 \Phi_+: \M_+\rightarrow \tilde L_+
 $
 with 
 \begin{equation}
 \label{tilde r+}
 \tilde r_+=e^{A_+-\int_{z}^{T_+}h(u)du}
 \end{equation}
for a constant $A_+$ to be determined. Here $\tilde L_+$ is the hermitian holomorphic line bundle determined by $\sqrt{-1}\Theta|_{T_+}$, and we have $\tilde L_+=L_+\otimes \mathcal F_+$ for a flat holomorphic line bundle $\mathcal F_+$. Notice there is a sign difference here  due to the fact that $L_+=L^{-\otimes k_+}$.  Again $\Phi_+$ is equivariant  with respect to the natural $S^1$-action on $\M_+$ and the inverse of the natural $S^1$-action on $\tilde L_+$ (due to the sign difference above), and it extends to a holomorphic map on $\M$. 
Together we obtain
\begin{equation}\Phi\equiv(\Phi_+, \Phi_-): \M\rightarrow \tilde L_+\oplus \tilde L_-,
\end{equation}
which is an embedding on $\M\setminus \mathcal P$. It commutes with projections maps to $D$. 

Now we claim $\tilde L_+\otimes \tilde L_-$ is isomorphic to $L^{\otimes k}$ as holomorphic line bundles. First notice that by $S^1$-equivariancy, we know the map 
\begin{equation}
\det\Phi\equiv \Phi_+\otimes \Phi_-: \M\setminus (H\times (-\infty, \infty)) \rightarrow \tilde L_+\otimes \tilde L_-
 \end{equation}
 has image lying on a  holomorphic section, say $\tilde S$, of $\tilde L_+\otimes \tilde L_-$ over $D\setminus H$. Since $h$ is smooth away from $P=H\times \{0\}$, $\tilde S$ is nowhere zero on $D\setminus H$. 
Again since $h$ is positive we know  the image of $\Phi$ is bounded in $\tilde L_+\oplus \tilde L_-$, with respect to the norm  defined by $\tilde r_\pm$, so $\tilde S$ is a bounded section  with respect to the norm $\tilde r\equiv \tilde r_+\otimes \tilde r_-$. Thus again by removable singularity theorem for bounded holomorphic functions it extends to a holomorphic section on the entire $D$. 
By our assumption that $[H]$ is Poincar\'e dual to $c_1(L)=c_1(\tilde L_+\otimes \tilde L_-)$,  we see  $H$ is exactly the zero locus of $\tilde S$ with multiplicity 1. So $\tilde L_+\otimes \tilde L_-$ is isomorphic to the holomorphic line bundle defined by the divisor $H$, which is exactly $L^{\otimes k}$. This proves the claim. 

Now modify the choice of $\Theta$ so that $\tilde L_{\pm}$ is isomorphic to $L_{\pm}$ as holomorphic line bundles. By the correspondence between gauge equivalence classes of flat $S^1$-connections on $D$ and isomorphism classes of flat holomorphic line bundles on $D$ mentioned in Section \ref{ss:Calabi model space}, we can add a flat $S^1$-connection to $\Theta$, and make $\tilde L_-$ isomorphic to $L_-$ as holomorphic line bundle.  The above claim then implies $\tilde L_+$ is also isomorphic to $L_+$ as holomorphic line bundles. So from now on we will simply identify $\tilde L_\pm$ with $L_\pm$.  

 By \eqref{e:greens-current-exp-asymp}, we have 
$\p_z\tilde\omega|_{z=T_\pm}=k_\pm\omega_D+\epsilon_T,
$
so we may assume that  under this identification  $
\log \tilde r_\pm=\log r_\pm+\epsilon_T$. From the above  we also know there is a nonzero constant $C$ such that 
\begin{equation}
\tilde S=C\cdot S_H,
\end{equation}
where $S_H$ is the section of $L^{\otimes k}$ we fixed at the beginning of this section. 
Multiplying $\Phi_-$ by an element in $S^1$-we may assume $C$ is a positive real number.

Notice by definition locally once we choose a holomorphic trivialization $\sigma$ of $L_-$, we can write 
$
\tilde r_-^2=|\zeta_-|^2 \cdot \|\sigma\|_{\sim}^2.
$
So
\begin{equation}
\frac{d\zeta_-}{\zeta_-}=\frac{d\tilde r_-}{\tilde r_-}+\sq J_-\frac{d\tilde r_-}{\tilde r_-}+\p_D \log |\sigma|^2.
\end{equation}
Therefore we obtain
$\Phi_-^*\Omega_{L_-}=\Omega,
$
where 
\begin{equation}\Omega_{L_-}\equiv -\sqrt{-1}\frac{d\zeta_-}{\zeta_-}\wedge \Omega_D
\end{equation}
is a natural holomorphic volume form on $L_-\setminus  {\bf 0}_{L_-}$. In particular $\Phi_-$ is a holomorphic embedding. Also, we have 
\begin{equation}
d\Phi_-(\xi^{1,0})=\sq \zeta_-\p_{\zeta_-}
\end{equation}
is the natural holomorphic vector field on $L_-$.
So similarly we obtain that 
\begin{equation} \label{e:d Phi xi}
d\Phi(\xi^{1,0})=\sq(\zeta_-\p_{\zeta_-}-\zeta_+\p_{\zeta_+}). 
\end{equation}

Now we show that with appropriate choice of $A_\pm$,  $\Phi$ maps $\M$ into $\mathcal N^0$.  
  Notice
\begin{equation}\log C=\frac{1}{\int_D \omega_D^{n-1}} \int_D \log \|{\tilde S}\| \omega_D^{n-1}-\frac{1}{\int_D \omega_D^{n-1}} \int_D \log \|S_H\|\omega_D^{n-1}, \end{equation}
where $\|\cdot\|$ denotes the norm on $L^k$ determined by  the fixed metric on $L$. 
The second term is a constant independent of $T$. For the first term, by definition we have 
\begin{equation}-\log \|\tilde S\|=\int_{T_-}^{T_+}hdz-(A_-+A_+)+\epsilon_T.
\end{equation}
By \eqref{haverage} and Proposition \ref{p:cohomology-constant}, 
\begin{eqnarray}\int_{T_-}^{T_+} \int_D h \omega_D^{n-1}&=&\int_{T_-}^{T_+} T^{2-n} \int_D (T\omega_D+\psi)^{n-1}dz+T^{-1}\underline B_T\\ &=&\frac{1}{n}\int_D \omega_D^{n-1}(\frac{1}{k_-}-\frac{1}{k_+}) (T^2-1)+T^{-1}\underline B_T 
\end{eqnarray}
So we get that
\begin{equation}
-\log C=\frac{1}{n}(\frac{1}{k_-}-\frac{1}{k_+})(T^2-1)-(A_-+A_+)+\frac{1}{\int_D \omega_D^{n-1}} \int_D \log \|S_H\|+T^{-1}\underline B_T
\end{equation}
Setting $C=1$ gives one condition on $A_-$ and $A_+$, but this does not determine $A_-$ and $A_+$. This corresponds to the fact that there is a $\C^*$ action on $\mathcal N^0$.  For our later purposes (c.f. Remark \ref{r:Aplusminus}), we need an additional balancing condition
\begin{equation}\label{balancing condition}
k_-A_-=k_+A_+.
\end{equation}
Together these determine $A_-$ and $A_+$ as 
\begin{align}
A_- &\equiv \frac{1}{nk_-}(T^2-1)-\frac{-k_+}{2(k_--k_+)}\frac{1}{\int_D \omega_D^{n-1}} \int_D \log \|S_H\|+T^{-1}\underline B_T,
\\
A_+ & \equiv \frac{1}{-nk_+}(T^2-1)-\frac{k_-}{2(k_--k_+)}\frac{1}{\int_D \omega_D^{n-1}} \int_D \log \|S_H\|+T^{-1}\underline B_T.
\end{align}
So we  obtain the following

\begin{proposition}\label{p:complex geometry}
With the above choice of $\Theta$, we have  a holomorphic embedding $\Phi: (\M, \Omega)\rightarrow \mathcal N^0$ as a relatively compact open subset containing $\mathcal H$, such that the following holds
\begin{enumerate}
\item $\Phi$ commutes with the projection maps to $D$; 
\item $\Phi^*\Omega_{\mathcal N^0}=\Omega$;
\item $d\Phi(\xi^{1,0})=\xi_{\mathcal N^0}$. In particular, $\Phi$ maps $\mathcal P$ bi-holomorphically onto $\mathcal H$. 
\end{enumerate}
\end{proposition}

For our purpose later, we list a few more results here.  First we   compare the function $z$ with the norm $r_-$ and $r_+$ near each end.
Given $C>0$ fixed, then by \eqref{e:h-asymptotics} we have
\begin{equation}\label{e: compare r and z}
\begin{cases}
-\log r_-=\frac{1}{nk_-}T^{2-n}(T+k_-z)^{n}-A_-+\epsilon_T+\epsilon(z),  \  \ z\leq -C; \\
-\log r_+=-\frac{1}{nk_+}T^{2-n}(T+k_+ z)^{n}-A_++\epsilon_T+\epsilon(z), \ \ z\geq C,
\end{cases}
\end{equation}
by noticing that for example
\begin{equation}
\int_{T_-}^z\epsilon(z)dz=\epsilon_T+\epsilon(z), z\leq -C.
\end{equation}
So we have
\begin{equation} \label{e: compare z and r}
\begin{cases}
(T+k_-z)^n=T^{n-2}nk_-(A_--\log r_-+\epsilon_T+\epsilon(z)), \ \ z\leq -C;\\
(T+k_+z)^n=-T^{n-2}nk_+(A_+-\log r_-+\epsilon_T+\epsilon(z)), \ \  z\geq C.
\end{cases}
\end{equation}

Next we give a description of the behavior of the metric $\omega$ when we restrict to the region $|z|\geq 1$. From the asymptotics of $\tilde\omega$ and $h$ we know the metric is asymptotic to the ends of the Calabi model space in Section \ref{ss:Calabi model space}. Locally on $D$ we fix holomorphic coordinates $\{w_1, \cdots, w_{n-1}\}$ and choose a holomorphic trivialization of $L$ as before, then we obtain fiber holomorphic coordinates $\zeta_\pm$ on $L_\pm$. Denote 
\begin{equation}
\omega_{\pm, cyl}\equiv \sum_{i\geq 1} \sq dw_i\wedge d\bar w_i+\frac{\sq d\zeta_\pm\wedge d\bar\zeta_\pm}{|\zeta_\pm|^2}
\end{equation}
the local cylindrical type metrics on $L_\pm$ respectively. Then we have 

\begin{lemma}\label{l:neck cylindrical compare}
On $|z|\geq 1$, we have 
\begin{equation}
C^{-1}T^{\frac{(n-2)(1-n)}{n}}(T+k_\pm z)^{1-n} \omega_{\pm, cyl}\leq \omega\leq C T^{\frac{2-n}{n}}(T+k_\pm z)\cdot  \omega_{\pm, cyl}.
\end{equation}
Furthermore, for all $k\geq 1$, there exists $m_k, C_k$ such that 
\begin{equation}
|\nabla^k_{\omega_{\pm, cyl}}\omega|_{\omega_{\pm, cyl}}\leq C_k (T^{\frac{2-n}{n}}(T+k_\pm z))^{m_k}.
\end{equation}
\end{lemma}

\begin{proof}
We only consider the case $z\leq -1$. Since
\begin{equation}
\frac{d\zeta_-}{\zeta_-}=\frac{dr_-}{r_-}+\sq J\frac{dr_-}{r_-}=-hdz+\epsilon_T-\sq J hdz
\end{equation}
Then we can estimate the coefficient of the metric $\omega$ in the frame given by $dw_1, \cdots dw_{n-1}$ and $\frac{d\zeta_-}{\zeta_-}$, using the asymptotics of $h$ \eqref{e:h-asymptotics} and $\tilde\omega$
\eqref{e:greens-current-exp-asymp}. From this the conclusion follows.
\end{proof} 

We also need to understand the level set $r_\pm=C$ under the projection to $D\times \R$, for a fixed $C>0$ and for $T$ large. First we have    

\begin{proposition} \label{l:zeta formula}
We have 
\begin{equation}
A_--\int_{T_-}^0 h(u)du=\frac{1}{2}\log {\|S_H\|}+B_T,
\label{e:asympotics-A-}
\end{equation}
\begin{equation}
A_+-\int_{0}^{T_+} h(u)du=\frac{1}{2}\log {\|S_H\|}+B_T.\label{e:asympotics-A+}
\end{equation}
\end{proposition}
\begin{proof}
We denote
\begin{equation}
\hat h_-=A_--\int_{T_-}^0 h(u)du-\frac{1}{2}\log \|S_H\|.
\end{equation}
By the Poincar\'e-Lelong equation we have
\begin{equation}
d_Dd_D^c \log {\|S_H\|}^2=4\pi \delta_H-(k_--k_+)\omega_D,
\end{equation}
where $\delta_H$ denotes the current of integration along $H$. 
By directly taking derivatives and use \eqref{eqn4445} we obtain that outside $H$, 
\begin{equation}
d_Dd_D^c (\int_{T_-}^{0}h(z)dz)=\int_{T_-}^{0}d_Dd_D^c h(z)dz=-\int_{T_-}^{0} \p_z^2\tilde\omega(z)dz=\p_z\tilde\omega|_{z=T_-}-\p_z\tilde\omega|_{z=0}
\end{equation} 
By \eqref{e:greens-current-exp-asymp} and \eqref{eqn2-43}, the right hand side is given by $\frac{1}{2}(k_--k_+)\omega_D+\epsilon_T$. Now using the asymptotics of $h$ near $P$ in \eqref{e:h-bounded-z}, one sees that $\hat h_-$ is bounded near $H$. So globally as currents on $D$, we have $
d_Dd_D^c\hat h_-=\epsilon_T.
$ Now 
\begin{equation}
\int_{D} \hat h_- \omega_D^{n-1}=A_-\int_D \omega_D^{n-1}-\int_{D}\int_{T_-}^{0}h\omega_D^{n-1} dz=B_T
\end{equation}
By standard elliptic regularity we get the conclusion for $\hat h_-$. The proof of \eqref{e:asympotics-A+} is the same. 
\end{proof}

\begin{remark}\label{r:Aplusminus}
This proposition explains the reason for choosing the constants $A_\pm$ to satisfy the balancing condition \eqref{balancing condition}: this makes the image under $\Phi$ of the slice $\{z=0\}$ lie in a bounded region in $\mathcal N^0$, which is not distorted as $T\rightarrow\infty$. 
\end{remark}

For $|z|\leq 1$, since $h(z)=T+\frac{1}{2r}+O'(r)+O(T^{-1})$, we easily see that in a fixed distance (with respect to $\omega_D$) away from $H$, $r_\pm \leq C$ is equivalent to $ B_T\cdot  T^{-1}\mp z\geq 0$.  We need a refinement of this. Fix a normal coordinate chart $(y, \bar y, w_2', \cdots, \bar w_{n-1}')$ on $D$ as given in Section \ref{ss:complex-greens-currents}. Recall we have locally $r^2=|y|^2+z^2$.   
\begin{proposition}
\label{p:z r- relation}
If $|z|\leq 1$, then  we have 
\begin{align}
\log r_-&=-Tz+\frac{1}{2}\log (r-z)+ B_T,
\\
\log r_+&=Tz+\frac{1}{2}\log (r+z)+ B_T.
\end{align}
\end{proposition}
\begin{proof}
By the previous proposition,
\begin{equation}
A_--\int_{T_-}^zh(u)=\frac{1}{2}\log {\|S_H\|}+B_T+\int_{z}^0 h(u)du.
\end{equation}
When $|z|\leq 1$ we have by \eqref{e:h-bounded-z} that 
\begin{equation}
\int_{0}^z h(u)du=B_T+Tz+\frac{1}{2}(\log(r+z)-\log |y|).
\end{equation}
Comparing \eqref{e:y definition equation} and \eqref{e: SH definition} we see  $\log {\|S_H\|}=\log |y|+B_T$ , it follows that 
\begin{equation}
\log r_-=B_T-Tz+\frac{1}{2}\log (r-z).
\end{equation}
Similarly we get the estimate for $
\log r_+$.  
\end{proof}

The following corollary will be used frequently later. 
\begin{corollary}\label{c:r zeta relation}
The following  hold:
\begin{enumerate}
\item Let $C>0$ be fixed. Then for $T$ large,  $r_\pm\leq C$ implies $\frac{3}{4}T^{-1}\log T\mp z\geq0$.
\item Let $c\in (0, 1/2)$ be fixed. Then for $T$ large $r \leq cT^{-1}\log T$ implies 
$\log r_\pm\leq -\frac{1}{2}(\frac{1}{2}-c) \log T.$
\item Let $C\geq 1$ be fixed. Then for $T$ large,  $C\pm z\geq 0$ implies $\log r_\mp\leq (C+1)T$.
\end{enumerate}

\end{corollary}
\begin{proof}
The first two items are easy consequences of the previous proposition. For the last item to see the bound on $r_-$ we simply notice that for $C\geq 1$, 
\begin{equation}
\int_{-C}^{-1} h(u)du=\int_{-C}^{-1} (T^{2-n}(T+k_-u)^{n-1}+\epsilon(u))du\leq CT. 
\end{equation}
The bound for $r_+$ can be proved similarly. 
\end{proof}

\subsubsection{K\"ahler potentials}\label{sss:kahler potential}We look for an $S^1$-invariant function $\phi$ on $\M$ satisfying the equation
\begin{equation} \label{eqn7-1}
T\pi^*\omega_{D}+dd^c\phi=T^{\frac{n-2}{n}}\omega
\end{equation}
We write
$d\phi=d_D\phi+\phi_zdz$
where as before $d_D\phi$ is the differential along $D$ direction   and $\phi_z=\p_z\phi$ is the derivative along $z$ direction. Then
$d^c\phi=d^c_D\phi+\phi_zh^{-1} \Theta, $
and 
\begin{equation}dd^c\phi=d_Dd_D^c\phi+dz\wedge (d^c_D\phi_z)+d(\phi_z h^{-1})\wedge\Theta+\phi_z h^{-1} (\p_z\tilde\omega-dz\wedge d_D^ch)\end{equation}  
Since 
$T^{\frac{n-2}{n}}\omega=\pi^* \tilde\omega+dz\wedge \Theta,
$
we see (\ref{eqn7-1}) is equivalent to the system of equations
\begin{equation}\label{eqn7-9}
\begin{cases}
\tilde\omega=T\omega_D+d_Dd^c_D\phi+\phi_zh^{-1} \p_z\tilde\omega\\
d_D^c\phi_z-\phi_zh^{-1} d_D^ch=0\\
d(\phi_z h^{-1})=dz.
\end{cases}
\end{equation}
To solve these equations, we first notice that the last equation in (\ref{eqn7-9}) is equivalent to 
\begin{equation}\phi_zh^{-1}=z+C\end{equation}
for a constant $C$. 
So we obtain \footnote{In the case when $n=2$ for the classical Gibbons-Hawking ansatz this formula was derived by the authors together with Hans-Joachim Hein in the office of the first author at Stony Brook in the Fall of 2017. }
\begin{equation}\label{eqn7-2}
\phi(z)=\int_{z_0}^z (u+C)h du+\phi(z_0) 
\end{equation}
for a function $\phi(z_0)$ on $D$.

The second equation of (\ref{eqn7-9}) then holds automatically, and if we take $\p_z$ on the first equation  then it also holds.  So in order for $\phi$ defined in (\ref{eqn7-2}) to satisfy (\ref{eqn7-9}),   it suffices that at a fixed $z=T_+$ the following holds
\begin{equation}T\omega_D+d_Dd^c_D \phi= \tilde\omega-(T_++C)\p_z\tilde\omega.\end{equation}
Comparing the cohomology classes of both sides yields that $C$ must be zero. 
Then we can solve $\phi(T_+)$ uniquely up to addition of a constant. After fixing a choice of $\phi(T_+)$ we may define $\phi$ by
\begin{equation} \label{eqn7-10}
\phi(z)=\int_{T_+}^z uh du+\phi(T_+),
\end{equation}
 and we can  view it as either a function on $Q_T$ or an $S^1$-invariant function on $\M$. 

\begin{proposition}
The function $\phi$ is smooth on ${\mathcal M^*}$, and $C^{3, \alpha}$ on $\M$ with respect to the smooth topology as defined in Section \ref{ss:kaehler-structures}, and satisfies (\ref{eqn7-1}).
\end{proposition}

\begin{proof}
Since $h$ is smooth on $Q_T\setminus H\times (-\infty, 0]$, so is $\phi$.  Using \eqref{e:h-bounded-z} it is easy to see that $\phi$ extends continuously on the whole $Q_T$. Hence for all fixed $z$, the following equation holds as currents on $D$
\begin{equation}
\label{eqn4149}T\omega_D+d_Dd_D^c\phi(z)=\tilde\omega(z)-z\p_z\tilde\omega(z). 
\end{equation}
Elliptic regularity implies that $\phi$ is smooth on each slice $\{z\}\times D$ for $z\neq 0$. Now for $z\leq 0$  we can write 
\begin{equation}\phi(z)=\int_{T_-}^{z} uh du+\phi(T_-). 
\end{equation}
We then see that $\phi$ is indeed smooth on $ Q_T\setminus P$. Over the $S^1$-fibration $\M$, we know $\phi$ is globally continuous, and it is smooth and satisfies  the equation (\ref{eqn7-1}) on ${\mathcal M^*}$. Since $\mathcal P=\mathcal M\setminus \mathcal M^*$ is a complex submanifold, and $\phi$ is continuous across $\mathcal P$, by standard theory on extension of pluri-subharmonic functions we conclude the current equation \eqref{eqn4149} holds globally on $\M$. Since $\omega$ is $C^{2, \alpha}$ in local holomorphic coordinates on $\M$, elliptic regularity gives that $\phi$ is  $C^{4, \alpha}$ in local holomorphic coordinates. This implies that $\phi$ is $C^{3, \alpha}$ in the smooth topology we defined, since the holomorphic coordinate functions are $C^{3,\alpha}$. 
\end{proof}

\begin{remark}
\label{r:Calabi model potential}As a by-product we can also recover the formula of the Calabi model metric in terms of K\"ahler potentials as mentioned in Section \ref{ss:Calabi model space}. In this case as in \eqref{e:Calabi model solution} we take $\tilde\omega= z\omega_D$  and $h=z^{n-1}$. Then  we can write 
$\tilde\omega=dd^c\phi$
with 
\begin{equation}\phi=\int_0^z u^{n} du=\frac{1}{n+1}z^{n+1}\end{equation}
To match with the formula for Calabi ansatz  in \eqref{calabiansatz}, we notice that $z^{n+1}=(-\log |\xi|)^2$, and there is a factor of $\frac{n}{2}$ due to the normalization of the Calabi-Yau equation and that $dd^c=2\sq \p\bp$. 
\end{remark}

\begin{remark}
\label{r:TaubNUT potential}
Notice the argument above does not essentially require the compactness of $D$, except to solve the equation \eqref{eqn4149} on one slice. Using similar idea one can get the expression of the Taub-NUT metric on $\C^2$ in terms of K\"ahler potentials, as mentioned in Section \ref{ss:2d standard model}. Here we take $D$ to be $\C$ with the standard flat structure, and 
$
\tilde\omega(z)=\frac{\sq}{2} Vdy\wedge d\bar y$, $h=V$,
with $
V=\frac{1}{2r}+T.$
Suppose we want to find $\phi$ with 
$
\omega=dd^c\phi,
$
then by \eqref{eqn7-2} we have 
$
\phi(z)-\phi(0)=\frac{1}{2}r-\frac{1}{2}|y|+\frac{T}{2}z^2.
$
The equation \eqref{eqn4149} for $z=0$ becomes  
$
4\p_y\p_{\bar y}\phi(0)=\tilde\omega(0)=\frac{1}{2|y|}+T,
$
and a solution is given by 
$
\phi(0)=\frac{1}{2}|y|+\frac{T}{4}|y|^2.
$
So we get 
\begin{equation}
\phi=\frac{1}{2}r+\frac{T}{2}z^2+\frac{T}{4}|y|^2=\frac{1}{4}(|u_1|^2+|u_2|^2)+\frac{T}{8}(|u_1|^4+|u_2|^4).
\end{equation}
This agrees with formula \eqref{e:TaubNUT potential} up to a constant $2$, again caused by the fact that $dd^c=2\sq \p\bp$.

\end{remark}

For our purpose later we need to express $\omega$ as $dd^c$ of an explicit function on the two ends $z\rightarrow\pm\infty$. 

\begin{proposition}
When $z\leq -C$, we have 
\begin{equation}
T^{\frac{n-2}{n}}\omega=T^{\frac{n-2}{n}}dd^c\phi_-,
\end{equation}
with 
\begin{equation} \label{neck potential negative side}
\phi_-\equiv \frac{1}{n+1}n^\frac{n+1}{n}k_-^{-\frac{n-1}{n}} (A_-+\epsilon_T+\epsilon(z)-\log r_-)^{\frac{n+1}{n}}-T^{\frac{2}{n}}k_-^{-1} A_-+\epsilon_T+T^{\frac{2}{n}}\epsilon(z).
\end{equation}
Similarly for $z\geq C$, we have 
\begin{equation}
T^{\frac{n-2}{n}}\omega=T^{\frac{n-2}{n}}dd^c\phi_+,\end{equation}
with
\begin{equation} \label{neck potential positive side}
\phi_+\equiv \frac{1}{n+1}n^\frac{n+1}{n}(-k_+)^{-\frac{n-1}{n}} (A_++\epsilon_T+\epsilon(z)-\log r_+)^{\frac{n+1}{n}}-T^{\frac{2}{n}}k_+^{-1} A_++\epsilon_T+T^{\frac{2}{n}}\epsilon(z).
\end{equation}
\end{proposition}
The goal of the rest of this subsubsection is to prove this proposition.
First notice from the above discussion we know for each fixed $z$, $\phi(z)$ is uniquely determined up to a constant on $D$ by the equation 
\begin{equation}\label{eqn7-11}
T\omega_D+d_Dd^c_D \phi(z)=\tilde\omega(z)-z\p_z \tilde\omega(z),  
\end{equation}
and the integration formula (\ref{eqn7-10}) exactly gives a coherent way of fixing all the constants for each $z$, so the overall freedom in only up to a global constant.  

 Notice by \eqref{e:greens-current-exp-asymp} we have for $z\gg 1$, 
\begin{equation}\tilde\omega(z)-z\p_z\tilde\omega(z)-T\omega_D=\psi(z)-z\p_z\psi(z)=\epsilon(z).\end{equation}
Standard elliptic estimate allows us to fix a solution $\phi(T_+)$ which is  $\epsilon_T$.  By \eqref{e:h-asymptotics}  we obtain that for $z\geq C$
\begin{equation} \label{eqn8-7}
\phi(z)=C_++T^{2-n}k_{+}^{-2} (\frac{(k_+z+T)^{n+1}}{n+1}-\frac{T(k_+z+T)^{n}}{n}),
\end{equation}
where 
\begin{equation}C_{+}=\epsilon_T+\epsilon(z)-T^{2-n}k_+^{-2}(\frac{1}{n+1}T^{\frac{(n+1)(n-2)}{n}}-\frac{1}{n}T^{n-2}).\end{equation}
For the other end $z\leq -C$, similarly we have 
\begin{equation} \label{eqn8-8}
\phi(z)-\phi(T_-)=C_-+T^{2-n}k_{-}^{-2} (\frac{(k_-z+T)^{n+1}}{n+1}-\frac{T(k_-z+T)^{n}}{n}),
\end{equation}
where 
\begin{equation}C_-=\epsilon_T+\epsilon(z)+T^{2-n}k_-^{-2}(\frac{1}{n+1}T^{\frac{(n+1)(n-2)}{n}}-\frac{1}{n}T^{n-2}).\end{equation}
To understand $\phi(T_-)$ we need the following 
\begin{lemma}
We have
\begin{equation}\label{eqn7-6}
\phi(T_-)=\epsilon_T+T^{-1}\underline B_T.
\end{equation}
\end{lemma}
\begin{proof}
We have $\phi(T_-)=\phi(T_+)-\Psi,$ where 
$\Psi=\int_{T_-}^{T_+}zhdz.$
Away from $H$ we have 
\begin{equation}d_Dd_D^c\Psi=\int_{T_-}^{T_+} zd_Dd_D^c h dz=-\int_{T_-}^{T_+} z\p_z^2\tilde\omega dz.\end{equation}
Integration by parts we get
\begin{equation}\label{e:small-ddc}
d_Dd_D^c\Psi=(-z \p_z\tilde\omega+\tilde\omega)|^{T_+}_{T_-}=\epsilon_T.
\end{equation}
Notice since there is a factor $z$ in the integrand we do not get residue term at $z=0$. Notice $\Psi$ is continuous on $D$, and the right hand side is smooth on $D$, so elliptic regularity implies that $\Psi$ is indeed smooth on $D$, and the equation holds globally on $D$.

On the other hand, we have
\begin{equation}\int_D \Psi\omega_D^{n-1}=\int_{T_-}^{T_+}z \int _D h\omega_D^{n-1} dz.\end{equation}
Using \eqref{haverage}, we see
\begin{align}\label{e:quotient-int}
\frac{\int_D \Psi\omega_D^{n-1}}{\int_D \omega_D^{n-1}}= & T^{2-n}k_{+}^{-2} (\frac{(k_+z+T)^{n+1}}{n+1}-\frac{T(k_+z+T)^{n}}{n}) 
\nonumber\\&-T^{2-n}k_{-}^{-2} (\frac{(k_-z+T)^{n+1}}{n+1}-\frac{T(k_-z+T)^{n}}{n})+T^{-1}\underline B_T,
\end{align}
where we used the definition of $T_-$ and $T_+$. 
(\ref{e:small-ddc}) and \eqref{e:quotient-int} together yield the conclusion. 
\end{proof}

Now notice that by \eqref{e: compare r and z} 
\begin{equation}-T^{3-n}k_{-}^{-2} \frac{(k_-z+T)^{n}}{n}= \frac{T}{k_-} (\log r_-- A_-+\epsilon_T+\epsilon(z))+\epsilon_T.\end{equation}
We may also write by definition
$\omega_D=-\frac{1}{k_-}dd^c\log  r_-$. Then \eqref{neck potential negative side} and \eqref{neck potential positive side} follow from 
\eqref{eqn7-1} and simple computation. 

\

\subsection{Geometries at regularity scales}
\label{ss:regularity-scales}

In this subsection, we take a closer look at the Riemannian geometric behavior of the family of incomplete K\"ahler metrics $(\M_T, \omega_T)$ constructed in Section \ref{ss:kaehler-structures} as $T\rightarrow \infty$. For clarity we now re-install the parameter $T$ throughout the rest of this section. The K\"ahler metric $\omega_T$ is given by 
\begin{equation}
\omega_T\equiv T^{\frac{2-n}{n}}\cdot\Big(\pi^*(T\omega_D+\psi)+dz\wedge \Theta\Big).\end{equation}
Denote the corresponding Riemannian metric by $g_T$, which has  the form 
\begin{equation}g_T=T^{\frac{2-n}{n}}\cdot\Big(\pi^*(Tg_0+g_1+h_Tdz^2)+h_T^{-1}\Theta^2\Big),\label{e:g_T-submersion-1}
\end{equation}
where $g_0$ is the Riemannian metric corresponding to $\omega_D$, and $g_1$ is the symmetric 2-tensor corresponding to $\psi$. 
Notice that by Remark \ref{r:bundle fixed} we may fix the $U(1)$-connection $\Theta$ over the entire $Q\setminus P$. Then $\omega_T$ and $g_T$ can be viewed as families of tensors  on a fixed space. Moreover, they are positive definite when restricted to $T_-\leq z\leq T_+$ and $T_{\pm}$ are defined in \eqref{e:define-T-plus-minus}.

It is easy to see that as the parameter $T\to+\infty$, the curvatures are unbounded around the singular set 
 $\mathcal{P}\subset\M_T$ such that the standard uniform elliptic estimates fail.
Instead, we will define some appropriate weighted H\"older spaces
and establish uniformly weighted a priori estimates, which will be done in Section \ref{ss:neck-weighted-analysis}. 
Geometrically, the weighted elliptic estimate that we pursue is intimately connected with 
the {\it effective regularity at definite scales} of the points in $(\M_T,\omega_T)$. More rigorously, we  need the following notion. 
\begin{definition}[Local regularity]
\label{d:local-regularity} Let $(M^n,g,p)$ be a Riemannian manifold and $p\in M^n$. Given $r>0$, $\epsilon>0$, $k\in\dN$, $\alpha\in(0,1)$, we say $(M^n,g,p)$ is $(r,k+\alpha,\epsilon)$-regular at $p$ if the metric $g$ is at least $C^{k+\alpha}$ in $B_{2r}(p)$ and satisfies the following property: Let $(\widetilde{B_{2r}(p)},\tilde{p})$ be the Riemannian universal cover of $B_{2r}(p)$. Then $B_r(\tilde{p})$ is diffeomorphic to a disc $\mathbb{D}^n$ or a half disc $\mathbb{D}_+^n$ in the Euclidean space $\dR^n$ such that $g$ in coordinates  satisfies
\begin{equation}
|g_{ij}-\delta_{ij}|_{C^0(B_r(\tilde{p}))}+\sum\limits_{m=1}^k r^m\cdot|\nabla^m g_{ij}|_{C^0(B_r(\tilde{p}))} +  r^{k+\alpha}[g_{ij}]_{C^{k,\alpha}(B_r(\tilde{p}))} < \epsilon.
\end{equation}

\end{definition}

 \begin{remark}
The case that $B_r(\tilde{p})$ is diffeomorphic to a half Euclidean disc $\mathbb D_+^n$ will be used to discuss  the regularity of a manifold with boundary.
\end{remark}

\begin{definition}
[$C^{k,\alpha}$-regularity scale] Let $(M^n,g)$ be a Riemannian manifold with a $C^{k,\alpha}$-Riemannian metric $g$. The $C^{k,\alpha}$-regularity scale at $p$, denoted by $r_{k,\alpha}(p)$, is defined as
the supremum of all $r>0$ such that $M^n$  is $(r,k+\alpha,10^{-6})$-regular at $p$.
\end{definition}

Intuitively, the $C^{k,\alpha}$-regularity scale 
is the maximal zooming-in scale at which 
the {\it nontrivial} $C^{k,\alpha}$-geometry is uniformly bounded on the local universal cover.   Clearly, if we work in a scale smaller than the regularity scale, then the corresponding $C^{k,\alpha}$ geometry is also uniformly bounded. So in the following we are mostly interested in a lower bound of the regularity scale.
    
\begin{example} 
 If $g$ is a $C^{k,\alpha}$-metric on $M^n$, then for any $p\in M^n$, we have $r_{k,\alpha}(p)>0$. Here the size of $r_{k,\alpha}(p)$ depends on $p$. 
\end{example}

\begin{example} 
Let $(M^n,g)$ satisfy $|\Rm_g|\leq 1$ in $B_2(p)$. Then the following holds:
\begin{enumerate} \item  there exists a dimensional constant $r_0(n)>0$ such that $r_{1,\alpha}(x)\geq r_0(n)>0$ for all $x\in B_1(p)$ and $\alpha\in(0,1)$. Moreover, 
$r_{1,\alpha}(p) \geq r_0(n)\cdot r_{|\Rm|}(p) >0$, where
\begin{equation}r_{|\Rm|}(p)\equiv \sup\Big\{r>0\Big| |\Rm|_{C^0(B_r(p))}\leq r^{-2}\Big\}\end{equation} denotes the curvature scale at $p$.

 \item In particular, if $\Rm_g\equiv 0$ on a complete manifold $M^n$, then $r_{k,\alpha}(x)=+\infty$ for all $x\in M^n$, $k\in\dZ_+$ and $\alpha\in(0,1)$.

\end{enumerate} 
\end{example}

Notice that the construction of the K\"ahler manifolds $(\mathcal M_T, \omega_T)$  in Section \ref{ss:kaehler-structures} are fairly explicit. In this subsection, we estimate a lower bound of the $C^{k, \alpha}$-regularity scale on $\mathcal M_T$ for $T$ large.

Before the technical  discussion, it is helpful to present the scenario of geometric transformations on $\M_T$  from the singular set $\mathcal{P}$ to the boundary $\p \M_T$.
First, as $T\to+\infty$,  curvatures blow up if the reference point $\bx$ is located around $\mathcal{P}$, and suitably rescaling the metric $\omega_T$
gives rise to a product bubble limit $\dC_{TN, \lambda}^2\times\dC^{n-2}$, where $\dC_{TN, \lambda}^2$ is the Taub-NUT space (c.f. Section \ref{ss:2d standard model}) for some $\lambda>0$.  This is a {\it deepest bubble (rescaling limit)} in our context.    
When the distance from $\bx$ to $\mathcal{P}$ is increasing, 
the length of $S^1$-fiber at the infinity 
of the Taub-NUT space $\dC_{TN, \lambda}^2\times\dC^{n-2}$ is decreasing which corresponds to $\lambda$ is increasing. The next level of bubble corresponds to $\lambda\rightarrow\infty$, and this amounts to getting the tangent cone at infinity of the product $\dC_{TN}^2\times\dC^{n-2}$, which is $\dR^{2n-1}\equiv \dR^3\times \dC^{n-2}$. This is of codimension-$1$ collapse, with locally uniformly bounded curvature away from $\{0^{3}\}\times \dC^{n-2}$.  When $\bm{x}$ is getting further away from $\mathcal{P}$, the size of $D$ will be shrinking and the next level of bubble is   $D \times \mathbb{R}$. This is again a codimension-$1$ collapse, with locally uniformly bounded curvature away $P=H\times \{0\}\subset D\times \dR$.  As we move further away from $\mathcal P$, we may still see the bubble $D\times \R$, but this time the codimension-1 collapse is with locally bounded curvature and $P$ gets pushed to infinity. 
Finally, as $\bx$ moves close to the boundary $\p\M_T$, the metrics will converge to the incomplete Calabi model metrics $\mathcal C^n_-$ and $\mathcal C^n_+$, which corresponds to applying the construction in Section \ref{ss:Calabi model space} to the line bundle $L^{k_-}$ and $L^{k_+}$ over $D$. 

\vspace{0.5cm}

Now we  make a subdivision for $\M_T$  (see Figure \ref{f:neck-subdivision}).
Given $\bx\in \mathcal M_T$.  Denote by $r(\bx)$ the distance from $\pi(\bx)\in Q$
to $P$ with respect to the product metric $g_Q = g_D + dz^2$ on the base $Q$. 

{\bf Region $\I_1$:}
This region consists of the points $\bm{x}$ which satisfy  
$r(\bx)\leq T^{-1}$. 

{\bf Region $\I_2$:}
This region consists of the points $\bx$  which satisfy 
$\frac{ T^{-1} }{2}\leq r(\bx) \leq 1$. 
  
{\bf Region $\I_3$:}
This region consists of the points 
 $\bx$  which satisfy 
$r(\bx) \geq \frac{1}{2}$ and   
 $T_- \leq z(\bx) \leq T_+$.

\vspace{0.5cm}

Notice that the above regions completely cover the neck $\M_T$, and
 overlapping regions have the same geometric behavior.
So we will just ignore these overlaps in the following discussions.

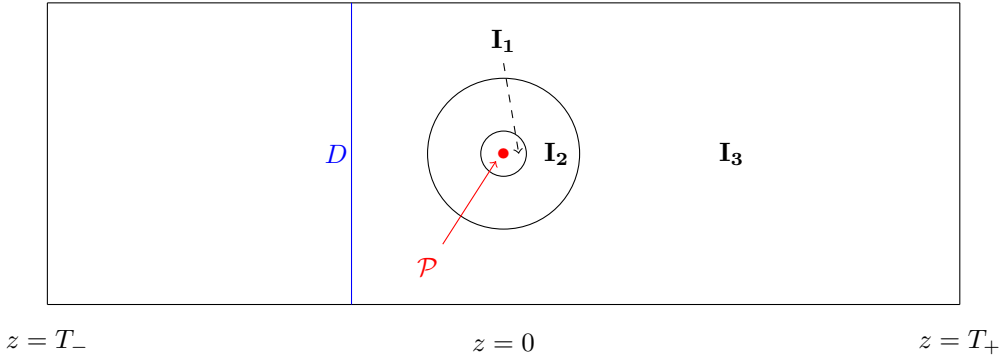
\begin{figure}
\begin{tikzpicture}
\draw (-6, 2) to (6, 2); 
\draw (-6, -2) to (6, -2); 
\draw (0,0) circle (3 mm); 
\draw (0, 0) circle (10mm);
\draw (-6, 2) to (-6, -2); 
\draw (6, 2) to (6, -2); 
\draw[blue] (-2, 2) to (-2, -2);
\node[blue] at (-2.2, 0) {$D$};  

\node[red] at (0, 0) {$\bullet$}; 
\node at (-6, -2.5) {$z=T_-$}; 
\node at (6, -2.5) {$z=T_+$}; 
\node at (0, -2.5) {$z=0$}; 
\node at (0, 1.5) {$\I_1$};
\draw[dashed, ->] (0, 1.2) to (0.2,0); 
\node[red] at (-1, -1.5) {$\mathcal{P}$};
\draw[ ->, red] (-0.8, -1.2) to (-0.1, -0.1);  
\node at (0.7, 0) {$\I_2$};
\node at (3, 0) {$\I_3$}; 
\end{tikzpicture}
\caption{Subdivision of $\M_T$ into various regions}
 \label{f:neck-subdivision}
\end{figure}

 For convenience, we define a continuous function  $\fr$  which is uniformly equivalent to $r(\bx)$: \begin{align}\label{e:definition-fr}
\fr(\bx) \equiv  \begin{cases}
T^{-1},  &  r(\bx) \leq  T^{-1}, \\
r(\bx),  & 2T^{-1}\leq r(\bx) \leq  \frac{1}{4},\\
1,       & r(\bx)\geq \frac{1}{2}.
\end{cases}
\end{align}
The following proposition gives an explicit lower bound estimate of the $C^{k,\alpha}$-regularity scale on $\mathcal{M}_T$.

\begin{proposition}
\label{p:regularity-scale}
Let us define \begin{align}\label{e:explicit-scales}
\fs(\bx) \equiv \Big(\frac{L_T(\bx)}{T}\Big)^{\frac{1}{2}}\cdot \fr(\bx) \cdot T^{\frac{1}{n}}, \quad \bx\in \M_T,
\end{align}
where $L_T(\bx)$ is defined in \eqref{e:LT z}. 
Then the following properties hold:
\begin{enumerate}
\item Given $k\in\dZ_+$ and $\alpha\in(0,1)$, 
there is a uniform constant $\underline{v}_0$ (depending on $k$ and $\alpha$) such that for all sufficiently large $T$ and $\bx\in\mathcal{M}_T$, it holds that
\begin{equation} \label{e: regularity scale bound}
 r_{k,\alpha}(\bx) \geq \underline v_0\cdot \mathfrak{s}(\bx),\end{equation}
where $k\leq 2$ if $\bx$ is in Region $\I_1$, and for all $k\in\dZ_+$ if $\bx$ is in Region $\I_2$ and $\I_3$.

\item 
 There are uniform constants $\underline{v}_0>0$ and $\overline{v}_0>0$ independent of $T\gg1$ such that for each $\bx\in \M_T$, we have
 \begin{equation}
\underline{v}_0 \leq \frac{\fs(\by)}{\fs(\bx)} \leq \overline{v}_0 \quad \text{for all} \quad \by\in B_{\fs(\bx)/4}(\bx).
 \end{equation}

 \item Let $T_j$ be a sequence tending to infinity. Then for a sequence of points $\bx_j\in\M_{T_j}$, the rescaled spaces $(\M_{T_j}, \fs(\bx_j)^{-2} \cdot g_{T_j},\bx_j)$ converge in the pointed Gromov-Hausdorff sense to one of the following as $T_j\to+\infty$:
 \begin{itemize}
 \item the Riemannian product $\dC_{TN}^2\times \dC^{n-2}$ where $\dC_{TN}^2$ is the Taub-NUT space;
 \item the product Euclidean space  $\dR^3\times \dC^{n-2}$;
 \item the cylinder $D\times\dR$;
 \item the Calabi model spaces $(\mathcal C_{\pm}^n, g_{\mathcal C_{\pm}^n})$.
 \end{itemize}

\end{enumerate}

\end{proposition}
\begin{remark}
The Calabi model space $(\mathcal C_{-}^n, \omega_{\mathcal C_{-}^n}, \Omega_{\mathcal{C}_-^n})$ (reps. $(\mathcal C_{+}^n, \omega_{\mathcal C_{+}^n}, \Omega_{\mathcal C_{+}^n})$ ) is defined as the $n$ dimensional Calabi-Yau manifold with boundary, obtained by applying the construction in Section \ref{ss:Calabi model space} with $\omega_D$ replaced by $k_-\omega_D$ (resp. $-k_+\omega_D$), and with the interval $z\in[1, \infty)$. When $b_1(D)>0$, we also make the choice of the corresponding connection 1-form similar to the discussion in Section \ref{sss:complex manifold}, so that the underlying complex manifold is naturally embedded into the holomorphic line bundle given by $L_-$ (resp. $L_+$).

\end{remark}

\begin{remark}
In \eqref{e:explicit-scales},  $\frac{L_T(\bx)}{T} = 1 +  O(T^{-1})$ if $|z(\bx)|$ is bounded. 
\end{remark}

\begin{proof}
[Proof of Proposition \ref{p:regularity-scale}. 
]

We will prove \eqref{e: regularity scale bound}
by contradiction.
Suppose that no such uniform constant $\underline{v}_0$ exists with respect to fixed $k\in\dZ_+$ and $\alpha\in(0,1)$. That is, there are a sequence $T_j\to+\infty$ and  a sequence of points $\bx_j\in \M_{T_j}$ such that 
\begin{align}
\frac{r_{k,\alpha}(\bx_j)}{\fs(\bx_j)}\to 0.\label{e:rescaled-r-limiting-to-zero}
\end{align}
Let us consider the rescaled sequence $(\M_{T_j},\tilde{g}_{T_j},\bx_j)$ with $\tilde{g}_{T_j}\equiv \fs(\bx_j)^{-2} \cdot g_{T_j}$ as $T_j\to+\infty$. In the proof, we will show that $C^{k,\alpha}$-regularity scale at $\bx_j$ with respect to $\tilde{g}_{T_j}$ is uniformly bounded from below as $T_j\to +\infty$ which contradicts \eqref{e:rescaled-r-limiting-to-zero}.
We will produce a contradiction in each of the following cases depending upon the location of $\bx_j$ in $\M_{T_j}$. We will also identify the rescaled limit in each case. 

\vspace{0.5cm}

\begin{flushleft}
{\bf Case (1):}  There is a constant $\sigma_0\geq 0$ such that $r(\bx_j)\cdot T_j\to   \sigma_0$ as $j\rightarrow\infty$.
\end{flushleft}

In this case, by definition $\fs(\bx_j) = (1+O(T_j^{-\frac{1}{2}}))\cdot \fr(\bx_j) \cdot T_j^{\frac{1}{n}}$. It suffices to show that $r_{k,\alpha}(\bx_j)$ with respect to the rescaled metric $\fr(\bx_j)^{-2}\cdot T_j^{-\frac{2}{n}}\cdot g_{T_j}$ (still denoted by $\tilde{g}_{T_j}$) is uniformly bounded from below as $T_j\to+\infty$.

First, we consider the case $\sigma_0\leq 1$. 
Then by definition \eqref{e:definition-fr}, we have  $\fr(\bx_j)=  T_j^{-1}$.
Denote the rescaled K\"ahler form and the rescaled holomorphic volume form by 
\begin{equation}
\tilde{\omega}_{T_j} \equiv  T_j^{\frac{2n-2}{n}}\cdot\omega_{T_j},\quad \widetilde{\Omega}_{T_j} \equiv T_j^{n-1}\cdot\Omega_{T_j}.
\end{equation}
As in \eqref{Taub-NUT}, for a parameter $\lambda>0$, 
let us denote by
$(\omega_{TN,\lambda} ,\Omega_{TN,\lambda})$ the K\"ahler form and the holomorphic volume form of the Taub-NUT space
$\dC_{TN, \lambda}^2$.

Clearly $r(\bx_j)\rightarrow0$.  Passing to a subsequence we may assume $\pi(\bx_j)$ converges to $p\in P$. 
Denote $\bm{p}\equiv \pi^{-1}(p)\in\mathcal P$. 
In the following we will prove that as $T_j\rightarrow\infty$,  \begin{equation} \label{I1convergence}(\M_{T_j}, \tilde{\omega}_{T_j}, \widetilde{\Omega}_{T_j}, \bx_j)\xrightarrow{C^{2,\alpha}}(\dC_{TN,1}^2\times \dC^{n-2}, \omega_{TN,1}\oplus \omega_{\dC^{n-2}}, \Omega_{TN, 1}\wedge \Omega_{\C^{n-2}}, \bm{0}^*)\end{equation} in the pointed $C^{2,\alpha}$-topology 
and the fixed point $\bm{p}$ converges to $\bm{0}^*\equiv (\bm{0}^2, \bm{0}^{n-2})$, where $\bm{0}^2$ is the origin of $\dC_{TN,1}^2$ and $(\omega_{\dC^{n-2}}, \Omega_{\C^{n-2}})$ is the flat K\"ahler structure on $\C^{n-2}$ with the origin $\bm{0}^{n-2}\in\dC^{n-2}$. 

As in Section \ref{sss:metric compactification}, we work with the local coordinate system $\{y, \bar y, z, w_2', \bar w_2', \cdots, w_{n-1}', \bar w_{n-1}'\}$ in a  neighborhood of $p$ in $Q$, which gives a local coordinate system  in a neighborhood of $\bm{p}$ in $\mathcal M_T$, denoted by $\{u_1, \bar u_1, u_2, \bar u_2, w_2', \bar w_2', \ldots, w_{n-1}', \bar w_{n-1}'\}$.  From the computation in Section \ref{sss:metric compactification}, 
\begin{equation}
T_j^{\frac{2n-2}{n}}\cdot\omega_{T_j} =\Big(T_j\cdot\omega_{TN, T_j}\oplus T_j^2\cdot\omega_{\C^{n-2}}\Big)+T_j^2\cdot\pi^*(\omega_D-\omega_{\C^{n-1}})+T_j\cdot O'(s^3)+T_j\cdot d(s^2\Gamma),
\end{equation}
where $\omega_{TN, T_j}$ is the Taub-NUT metric on $\C^2_{u_1, u_2}$ with a parameter $T_j$, and
\begin{equation}\omega_{\C^{n-2}}\equiv   \frac{\sq}{2}\sum_{j= 2}^{n-1} dw_j'\wedge d\bar w_j', \ \ \omega_{\C^{n-1}}\equiv   \frac{\sqrt{-1}}{2}\cdot dy\wedge d\bar y+\frac{\sq}{2}\sum_{j= 2}^{n-1} dw_j'\wedge d\bar w_j'.
\end{equation}
Notice that we have already used the relations \begin{equation}\pi^*(O'(r^p))=O'(s^{2p}) \ \text{for}\  p\geq 1,\quad  \pi^*(dy)=\tO (s),  \quad \pi^*(dz)=\tO (s).\end{equation} 
We perform a change of coordinates
\begin{equation}
z=T_j^{-1}\cdot\underline{z}, \quad y=T_j^{-1}\cdot\underline{y}, \quad  
w_j'=T_j^{-1}\cdot\underline{w_j}',\quad   
u_k=T_j^{-1/2}\cdot\underline{u_k}, 
\end{equation}
and denote 
${\bm w}\equiv (\underline w_2', \cdots,  \underline w_{n-1}')$,   ${\bm u}\equiv (\underline u_1, \underline u_2)$,  $\underline{s}=|{\bm u}|$.
In these rescaled coordinates, the rescaled K\"ahler structure $(T_j\cdot \omega_{TN, T_j},T_j\cdot\Omega_{TN,T_j}) $ can be identified with $(\omega_{TN, 1},\Omega_{TN, 1})$. Moreover, we have that 
\begin{align}
T_j^2\cdot \pi^*(\omega_D-\omega_{\C^{n-1}}) &= O((|{\bm w} |+|{\bm u}|^2)T_j^{-1}),\\
T_j\cdot O'(s^3) &= O(T_j^{-3/2}\underline s^3),\\
T_j \cdot d(s^2\Gamma) &= O(T_j^{-3/2}\underline s).
\end{align}
The above computations imply that on the region $|\bm w|+|\bm u|\leq C$ for a fixed $C>0$. Therefore, \begin{equation}|T_j^{\frac{2n-2}{n}}\omega_T-(\omega_{TN, 1}\oplus\omega_{\C^{n-2}})|_{C^{2, \alpha}}=O(T_j^{-1}),
\end{equation}
where the norm is measured with respect to the product metric 
$\omega_{TN, 1}\oplus\omega_{\C^{n-2}}$.  
Similarly, one can also obtain the expansion for the rescaled holomorphic form $\widetilde{\Omega}_{T_j}$,  
\begin{equation}
\widetilde{\Omega}_{T_j} = T_j^{n-1}\cdot\Omega_{T_j}=\Omega_{TN, 1}\wedge d\underline w_2'\wedge \cdots\wedge d\underline w_{n-1}'+O((|{\bm w}|+|{\bm u}|^2)T_j^{-1}).
\end{equation}
This implies that the convergence \eqref{I1convergence} holds, and hence
 there is  $\underline{v}_0>0$ is independent of $T_j\gg1$ such that under the rescaled metric $\tilde{g}_{T_j}$,
\begin{equation}
r_{2,\alpha}(\bx_j)\geq  \underline{v}_0,
\end{equation}
which contradicts \eqref{e:rescaled-r-limiting-to-zero}.  This completes the proof in the case $\sigma_0\leq 1$.

In the case $\sigma_0>1$, the proof is the same. We only need to notice that the K\"ahler structure on the limiting Taub-NUT space is given by  
\begin{align}\label{e:parameter-TN}
\begin{cases}
\omega_{TN,\sigma_0^2}  \equiv  (\frac{1}{2r}+\sigma_0^2)\cdot \frac{\sq}{2}\cdot dy\wedge d\bar{y}
+dz\wedge \Theta_0  
\\
\Omega_{TN,\sigma_0^2} \equiv \sq ((\frac{1}{2r}+\sigma_0^2)\cdot dz + \Theta_0)\wedge dy.
\end{cases}
\end{align}
The rest of the computations are the same.
So the proof in Case (1) is done.

\vspace{0.5cm}

\begin{flushleft}
{\bf Case (2):}  $r(\bx_j)\cdot T_j\rightarrow\infty$ and $r(\bx_j)\rightarrow 0$ as $j\rightarrow\infty$.\end{flushleft}

In this case, by definition $\fs(\bx_j) =(1+O(T_j^{-\frac{1}{2}}))\cdot r(\bx_j)\cdot T_j^{\frac{1}{n}}$. It suffices to show that  with respect to the rescaled metric 
$r_j^{-2}\cdot T_j^{-\frac{2}{n}}\cdot g_j$ (still denoted by $\tilde{g}_{T_j}$)
with $r_j\equiv r(\bx_j)$, the regularity scale
$r_{k,\alpha}(\bx_j)$ is uniformly bounded from below, which contradicts \eqref{e:rescaled-r-limiting-to-zero}.

In the following computations, it is more convenient to consider the Riemannian metric $g_T$ corresponding to $\omega_{T_j}$, which is given by \begin{equation}g_{T_j}=T_j^{\frac{2-n}{n}}\cdot\Big(\pi^*(T_j\cdot g_0+g_1+h_{T_j}dz^2)+h_{T_j}^{-1}\cdot \Theta^2\Big).\label{e:g_T-submersion-2}
\end{equation}
Again we emphasize that $\Theta$ is  independent of $T_j$.

Since $r_j\rightarrow 0$, we may assume that $\pi(\bx_j)$ converges to a fixed point $p\in P$. In a neighborhood of $P$ in $Q$, we have local coordinates $\{y, \bar y, z, w_2', \bar w_2', \cdots, w_{n-1}', \bar w_{n-1}'\}$ as in Section \ref{ss:complex-greens-currents} and we rescale them by
\begin{equation}
z= r_j\cdot \underline z, \ y= r_j\cdot\underline y, \ 
w_p'= r_j\cdot\underline w_p',\  \  p=2, \ldots, n-1. 
\label{e:rescaled-coordinates}
\end{equation}
Notice that $\pi(\bx_j)$ has a definite distance away from the subspace $\{y=z=0\}\subset P$ under  $\tilde{g}_{T_j}$.

We consider the pointed Gromov-Hausdorff limit of the metrics $(\mathcal M_T, \tilde g_j, \bx_j)$. Notice that
 \begin{equation}\tilde{g}_{T_j}=r_j^{-2}\cdot T_j^{-1}\cdot\Big(\pi^*(T_jg_0+g_1+h_{T_j}dz^2)+h_{T_j}^{-1}\Theta^2\Big).
\end{equation}
As before, we can see that in the rescaled coordinates, the first term converges smoothly to $g_{\C^{n-2}}\oplus dz^2$ away from the subspace $\{y=z=0\}$, where $g_{\C^{n-2}}$ denotes the flat metric on $\C^{n-2}_{w_2', \cdots, w_{n-1}'}$. On the other hand, since $h_{T_j}$ has uniformly positive lower bound in this region (by \eqref{e:h-bounded-z}), and $T_j\cdot r_j\rightarrow +\infty$, it is easy to see that the length of the $S^1$-fibers tends to zero uniformly. 

Analyzing more closely using the behaviors of $\psi, h_{T_j}$ and $\Theta$ near $p$ (see \eqref{e:singular-2-form-expansion}, \eqref{e:h-bounded-z} and \eqref{e:Theta-near-p}), one can see that the metrics $\tilde{g}_j$ are collapsing with uniformly bounded curvature away from the subspace $\{y=z=0\}\subset P$. Moreover, around $\bx_j$, if we pass to the local universal cover, we have $C^{k,\alpha}$-bounded geometry in a ball of definite size. This already shows that there is a constant $\underline{v}_0>0$ independent of $T_j$ such that $r_{k,\alpha}(\bx_j)\geq \underline{v}_0$ with respect to $\tilde{g}_{T_j}$, which contradicts \eqref{e:rescaled-r-limiting-to-zero}. 

With more analysis one can actually show that the rescaled metrics $(\mathcal M_{T_j}, \tilde g_{T_j}, \bx_j)$ converges in the pointed Gromov-Hausdorff sense to the product Euclidean space  $\R^3\times \C^{n-2}$. The details follow from explicit but lengthy tensor computations, so we omit them. Notice that $\R^3\times \C^{n-2}$ is 
the tangent cone at infinity of the product space $\C^2_{TN,1}\times \C^{n-2}$ which appears as the rescaled limit in Case (1).

\vspace{0.5cm}

\begin{flushleft}

{\bf Case (3):} There is a constant $\underline{T}_0>0 $ such that $r(\bx_j)\geq   \underline{T}_0$  and $L_{T_j}(\bx_j)^{-1}\cdot T_j^{\frac{n-2}{n}}\rightarrow 0$ as $j\rightarrow\infty$. 
\end{flushleft}

In this case,  $\fs(\bx_j) = \Big(\frac{L_{T_j}(\bx_j)}{T_j}\Big)^{\frac{1}{2}}\cdot r(\bx_j) \cdot T_j^{\frac{1}{n}}$. Passing to a subsequence there are two sub-cases.

First, we consider the case  $z(\bx_j)\to  z_0$ and $r(\bx_j)\to r_0>0$. Then $\fs(\bx_j) = (1+O(T_j^{-\frac{1}{2}}))\cdot r_0 \cdot  T_j^{\frac{1}{n}}$. So it suffices to work with the rescaled metric  
$r_0^{-2}\cdot T_j^{-\frac{2}{n}}\cdot g_{T_j}$ (again denoted by $\tilde{g}_{T_j}$) and show that the regularity scale
$r_{k,\alpha}(\bx_j)$ is uniformly bounded from below. The rescaled metric $\tilde{g}_{T_j}$ has the explicit form
\begin{equation}
\tilde{g}_{T_j}=r_0^{-2}\cdot T_j^{-1}\cdot\Big(\pi^*(T_j\cdot g_0+g_1+h_{T_j}\cdot dz^2)+h_{T_j}^{-1}\cdot\Theta^2\Big).\end{equation}
Now using the asymptotics of $\psi$ and $h_{T_j}$ in \eqref{e:singular-2-form-expansion} and \eqref{e:h-bounded-z}, we see that $\tilde{g}_{T_j}$ converges smoothly as tensors to $\pi^*(g_0+dz^2)$  away from $P$.  Since the $S^1$-fibers are collapsing, and we have uniformly bounded $C^{k,\alpha}$-geometry on a definite size ball on the local universal cover around $\bx_j$. This implies that there is a constant $\underline{v}_0>0$ such that $r_{k,\alpha}(\bx_j)\geq \underline{v}_0$ under the metric $\tilde{g}_{T_j}$, which contradicts to \eqref{e:rescaled-r-limiting-to-zero}. 
Again with a little more work one can see that the rescaled metrics $(\mathcal M_{T_j}, \tilde g_j, \bx_j)$ converges in the pointed Gromov-Hausdorff topology to the product space $D\times \dR$, and $S^1$-fibers collapse. The collapsing is with uniformly bounded curvature away from $P$.  We also omit the details. Notice that this limit is also related to the rescaled limit in Case (2): Any point $p\in D\times\dR$ has a tangent cone $\R^3\times \C^{n-2}$.

Next, let us consider the case $|z_j|\to+\infty$ and $L_{T_j}(\bx_j)^{-1}\cdot T_j^{\frac{n-2}{n}}\rightarrow 0$, where $z_j\equiv z(\bx_j)$. 
Without loss of generality, we may assume $z_j<0$. First, we have $\fs(\bx_j)= \Big(\frac{L_j}{T_j}\Big)^{\frac{1}{2}} \cdot T_j^{\frac{1}{n}}$, where $L_j\equiv L_{T_j}(\bx_j)$.
Then the rescaled metric $\tilde{g}_{T_j}\equiv \fs(\bx_j)^{-2}\cdot g_{T_j}$ is given by \begin{align}
\tilde{g}_j=L_j^{-1}\cdot\Big(\pi^*((T_j\cdot g_0+g_1)+h_{T_j}\cdot dz^2)+h_{T_j}^{-1}\cdot\Theta^2\Big).\label{e:cylinder-rescaled-metric}\end{align}
We perform a change of coordinate
$z = z_j + (\frac{T_j}{L_j})^{\frac{n-2}{2}}\cdot w$.
Then using the asymptotics of $\psi$ and $h_{T_j}$ in \eqref{e:singular-2-form-expansion} and \eqref{e:h-asymptotics}, we can see that the $S^1$-fibers are collapsing. The pointed Gromov-Hausdorff limit is again the cylinder $D\times\dR$ and the collapsing sequence has uniformly bounded curvature. It is then easy to obtain a contradiction. 

Geometrically, the above two situations are related as follows. The rescaled limit $D\times\dR_z$ in the second situation can be viewed as translating $D\times\dR_z$ along $z$ towards $\pm\infty$, so that the singular set $P$ disappears.

\vspace{0.5cm}

\begin{flushleft}
{\bf Case (4):} There is a constant $c_0>0$ such that $L_{T_j}(\bx_j)^{-1}\cdot T_j^{\frac{n-2}{n}}\to  c_0$.
\end{flushleft}

In this case, by definition $\fs(\bx_j) = (c_0)^{-1}\cdot (1+o(1))$. 
One can see that $d_{g_j}(\bx_j,\p\M_{T_j})$ is uniformly bounded. Without loss of generality we may assume $z(\bx_j)<0$.  We rescale the $z$-coordinate by 
\begin{equation}
w=T_j^{\frac{2-n}{n}}(T_j+k_-\cdot z). 
\end{equation}
Then using the asymptotics of $h_{T_j}, \psi$,  one can obtain convergence of the metric tensor $g_{T_j}$ in the $w$ coordinate. From this one easily see that $(\M_{T_j}, \tilde{g}_{T_j}, \bx_j)$ converges smoothly to the Calabi model space $(\mathcal{C}_-^n, g_{\mathcal{C}_-^n}, \bx_\infty)$ in the pointed Gromov-Hausdorff sense. Clearly we obtain a contradiction to \eqref{e:rescaled-r-limiting-to-zero}. 

Notice that the Calabi model space has a boundary. To relate to the limit in Case (3), one can think of the space $D\times \R$ as a pointed Gromov-Hausdorff limit of certain scale down of the Calabi model space at infinity. We choose a sequence of points towards infinity, and scale down so that the base $D$ has fixed size. Then  $\p\mathcal{C}_-^n$ gets pushed to infinity and also the $S^1$-fibers are collapsing. 
\end{proof}

\subsection{Fundamental estimates in the weighted H\"older spaces}

\label{ss:neck-weighted-analysis}

Based on the above detailed studies of the regularity scales, 
we are now ready to define the weighted H\"older space on the neck. 
To start with, 
let us recall the notation,
\begin{align}
\M_T &\equiv \Big\{\bx\in \mathcal{M}\Big|T_- \leq z(\bm{x})\leq T_+\Big\},
\\
\mathring{\mathcal{M}}_T &\equiv \Big\{\bx\in \mathcal{M}\Big|T_- \leq z(\bm{x}) \leq T_+,\  d_{\omega_T}(\bm{x}, \p\M_T)\geq 1\Big\}.
\end{align}

\begin{definition}
[Weight function] \label{d:weight-function}
Given fixed real parameters $n\geq 2$,  $T\gg1$, $\delta>0$, $\nu,\mu\in\dR$. For each $k\in\dN$, $\alpha\in(0,1)$, 
the weight function $\rho_{\delta,\nu,\mu}^{(k+\alpha)}$
is defined as follows, 
\begin{align}
\label{e:definition of weights}
\rho_{\delta,\nu,\mu}^{(k+\alpha)}(\bx)=e^{\delta\cdot U_T(\bx)}\cdot \fs(\bx)^{\nu+k+\alpha}\cdot T^{\mu},
\end{align}
where $\fs(\bx)$ is the regularity scale at $\bx$ given by Proposition \ref{p:regularity-scale} and 
\begin{align}
U_T(\bx)  \equiv T\Big(1-(\frac{L_T(\bx)}{T})^{\frac{n}{2}}\Big),
\quad 
L_T(\bx)  \equiv   L_T(z(\bx)) = T + L_0(z(\bx)),
\label{d:def-U_T}\end{align}
where the functions $L_T$ and $L_0$ are defined in \eqref{e:LT z}. 
\end{definition}

To better understand the weight function \eqref{e:definition of weights}, we give several remarks. 

\begin{remark}
In the case that $D$ is not flat, by the proof of Proposition \ref{p:regularity-scale}, one can see that $\mathfrak{s}(\bx)$ is uniformly equivalent to the $C^{k,\alpha}$-regularity scale $r_{k,\alpha}(\bx)$ at $\bx$.
\end{remark}

\begin{remark}
The function $e^{\delta\cdot U_T(\bx)}$
 is the dominating 
 factor at large scales on $\M_T$ which 
 behaves like an exponential function.  The term $U_T(\bx)$ is used for unifying the weighted analysis for different ``large scales'' on $\M_T$, which will be seen in the proof of Proposition \ref{p:neck-uniform-injectivity} in Section \ref{s:neck-perturbation}. For intuition, there are two cases in which $U_T$ has simple expressions:
 \begin{align}
 \begin{cases}
 U_T(\bx) = - L_0(z) , & n=2,
 \\
U_T(\bx) \approx - \frac{n}{2}\cdot L_0(z) , & n> 2,\ |z(\bx)|\ll T.
 \end{cases}
 \end{align}

\end{remark}

\begin{remark}
The constant factor $T^\mu$ in the definition of the weight function is needed to deal with the non-linear term in the application of the implicit function theorem (see Proposition \ref{p:nonlinear-neck}). When $n=2$ the non-linear term is quadratic and this constant term is unnecessary, but when $n>2$ we need to choose appropriate $\mu$ (see \eqref{e:fix-mu-neck}) so that the weight function has a uniform lower bound independent of $T$. 
\end{remark}

\begin{lemma}[Lower bound estimate for the weight function]\label{l:weight-function-lower-bound-estimate}
For fixed constants $\delta>0$, $\mu,\nu\in\dR$, $\alpha\in(0,1)$ and $k\in\dN$, then for all $T\gg1$ and $\bx\in\M_T$,
\begin{align}
\rho_{\delta,\nu,\mu}^{(k+\alpha)}(\bx)\geq \begin{cases}
T^{(\frac{1}{n}-1)(\nu+k+\alpha)+\mu}, & \nu+k+\alpha \geq 0, 
\\
T^{\frac{\nu+k+\alpha}{n}+\mu}, & \nu+k+\alpha < 0.
 \end{cases}
\end{align}

\end{lemma}

\begin{proof}

This lower bound estimate can be obtained by analyzing the regularity scale $\fs(\bx)$. 
Denote by $w=w(\bx) \equiv \frac{L_T(\bx)}{T}$ and recall that the two end points $T_-,T_+$ satisfy
\begin{align}
\begin{cases}
L_T(T_-) = T^{\frac{n-2}{n}}
\\
L_T(T_+) = T^{\frac{n-2}{n}}.
\end{cases}
\end{align}
Then we have
$w\in[T^{-\frac{2}{n}},1]$.  So it follows that
\begin{equation}
\rho_{\delta,\nu,\mu}^{(k+\alpha)}(\bx) = F(w)\cdot \fr(\bx)^{\nu+k+\alpha}\cdot  T^{\frac{\nu+k+\alpha}{n}+\mu}, 
\end{equation}
where $F(w) \equiv e^{\delta \cdot T(1-w^{\frac{n}{2}})}\cdot w^{\frac{\nu+k+\alpha}{2}}$.
By the definition of $\mathfrak{r}(\bx)$, immediately we have \begin{align}T^{-1}\leq\fr(\bx)\leq 1\end{align} for all $\bx\in\M_T$, so it follows that 
\begin{align}
\rho_{\delta,\nu,\mu}^{(k+\alpha)} (\bx)
\geq 
\begin{cases}
F(w)\cdot T^{(\frac{1}{n}-1)(\nu+k+\alpha)+\mu}, & \nu+k+\alpha\geq 0, 
\\
F(w)\cdot T^{\frac{\nu+k+\alpha}{n}+\mu}, & \nu+k+\alpha < 0, 
\end{cases}
\end{align}

Now it suffices to 
compute the lower bound of $F(w)$. To this end, 
there are two cases to analyze depending on the sign of $\nu+k+\alpha$. First, let $\nu+k+\alpha\leq 0$, then obviously $F(w)\geq F(1) = 1$
and hence
\begin{equation}
\rho_{\delta,\nu,\mu}^{(k+\alpha)}(\bx)\geq T^{(\frac{1}{n}-1)(\nu+k+\alpha)+\mu}.
\end{equation}
Next, we consider the case $\nu+k+\alpha>0$. Simple calculus shows that 
$F(w)$ achieves its minimum in $[T^{-\frac{2}{n}},1]$ either at $w= 1$ or at $w=T^{-\frac{2}{n}}$. Notice that 
$F(T^{-\frac{2}{n}})\gg F(1)$ as $T\gg1$.
This tells us that
\begin{equation}
\rho_{\delta,\nu,\mu}^{(k+\alpha)}(\bx)\geq T^{\frac{\nu+k+\alpha}{n}+\mu}.\end{equation}
The proof is done. 
\end{proof}

Using the above weight function, we  define weighted H\"older spaces as follows.

\begin{definition}
[Weighted H\"older norm] \label{d:weighted-space} Let $\mathcal{K}\subset\M_T$ be a compact subset. Then the  weighted H\"older norm of a tensor field $\chi\in T^{r,s}(\mathcal{K})$ of type $(r,s)$ is defined as follows:
\begin{enumerate}

\item The weighted $C^{k,\alpha}$-seminorm of $\chi$ is defined by
\begin{align}
[\chi]_{C_{\delta,\nu,\mu}^{k,\alpha}}(\bx) & \equiv  \sup\Big\{\rho_{\delta,\nu,\mu}^{(k+\alpha)}(\bx) \cdot\frac{|\nabla^k\tilde{\chi}(\tilde{\bx})- \nabla^k\tilde{\chi}(\tilde{\by})|}{(d_{\tilde{g}_T}(\tilde{\bx},\tilde{\by}))^{\alpha}}  \ \Big| \  \tilde{\by}\in B_{r_{k,\alpha}(\bx)}(\tilde{\bx})\Big\}\ \text{for}\ \bx\in\mathcal{K},
\\
[\chi]_{C_{\delta,\nu,\mu}^{k,\alpha}(\mathcal{K})} & \equiv \sup\Big\{[\chi]_{C_{\delta,\nu,\mu}^{k,\alpha}}(\bx)\Big|\bx \in \mathcal{K}\Big\},
\end{align}
where $r_{k,\alpha}(\bx)$ is the $C^{k,\alpha}$-regularity scale at $\bx$, $\tilde{\bx}$ denotes a lift of $\bx$ to the universal cover of $B_{2r_{k,\alpha}(\bx)}(\bx)$, 
the difference of the two covariant derivatives is defined in terms of parallel translation in $B_{r_{k,\alpha}(\bx)}(\tilde{\bx})$,  
and $\tilde{\chi}$, $\tilde{g}_T$ are the lifts of $\chi$, $g_T$ respectively.

\item The weighted $C^{k,\alpha}$-norm of $\chi$ is defined by
\begin{align}
\|\chi\|_{C_{\delta,\nu, \mu}^{k,\alpha}(\mathcal{K})}
  \equiv \sum\limits_{m=0}^k\Big\|\rho_{\delta,\nu,\mu}^{(m)} \cdot\nabla^m \chi\Big\|_{C^0(\mathcal{K})} + [\chi]_{C_{\delta,\nu,\mu}^{k,\alpha}(\mathcal{K})},
\end{align}
\end{enumerate}

\end{definition}

\begin{remark}
By definition, it is direct to check that
\begin{equation}
\|\chi\|_{C_{\delta,\nu,\mu}^k(\mathcal{K})} = \sum\limits_{m=0}^k \|\nabla^m \chi\|_{C_{\delta,\nu+m,\mu}^0(\mathcal{K})}.
\end{equation}

\end{remark}

With the above definition of the weighted H\"older space, we are ready to give a local uniform weighted Schauder estimate with respect to the Laplacian on the neck $(\M_T, \omega_T)$.

\begin{proposition}
[Weighted Schauder estimate, the local version] \label{p:local-weighted-schauder}
For every sufficiently large parameter $T\gg1$, let
 $\M_T$ be the neck region with 
an $S^1$-invariant K\"ahler metric $\omega_T$
constructed in Section \ref{ss:kaehler-structures}.
Then the following estimates hold:
\begin{enumerate}

\item (Interior estimate)
Given $k\in\{0,1\}$ and $\alpha\in(0,1)$, there is some uniform constant $C_{k,\alpha}>0$ 
such that for any 
$\bm{x}\in\mathring{\mathcal{M}}(T_-,T_+)$, $r\in(0,1/8]$, $u\in C^{k+2,\alpha}(B_{2 r\cdot \fs(\bx)}(\bx))$,
\begin{align}
&r^{k+2+\alpha}\cdot \|u\|_{C_{\delta,\nu,\mu}^{k+2,\alpha}(B_{r\cdot \fs(\bx)}(\bx))}
\nonumber\\
\leq & C_{k,\alpha} \Big(\|\Delta u\|_{C_{\delta,\nu+2,\mu}^{k,\alpha}(B_{2r\cdot \fs(\bx)}(\bx))} + \| u\|_{C_{\delta,\nu,\mu}^0(B_{2r\cdot \fs(\bx)}(\bx))}\Big),
\label{e:local-schauder}
\end{align}
where $\mathfrak{s}(\bx)$ is the function defined in \eqref{e:explicit-scales}. 
  
\item (Higher order estimate away from $\mathcal{P}$) If $\bx\in\mathring{\M}_T$ satisfies $r(\bx)\geq 2T^{-1}$,  then the uniform Schauder estimate \eqref{e:local-schauder} holds for all $k\in\dZ_+$, $\alpha\in(0,1)$, $r\in(0,r_0)$ with $r_0$ independent of $T$. 
\item (Boundary estimate)
For any $k\in\dZ_+$ and $\alpha\in(0,1)$, there exists some uniform constant $C_{k,\alpha}>0$ such that for all $\bx\in\p\mathcal{M}_T$, $r\in(0,1/8]$ and $u\in C^{k+2,\alpha}(B_{2r\cdot \fs(\bx)}^+(\bx))$,
\begin{align}
& r^{k+2+\alpha}\cdot \|u\|_{C_{\delta,\nu,\mu}^{k+2,\alpha}(B_{r\cdot \fs(\bx)}^+(\bx))}\nonumber
\\  \leq &  C_{k,\alpha} \cdot \Big(\|\Delta u\|_{C_{\delta,\nu+2,\mu}^{k,\alpha}(B_{2r\cdot \fs(\bx)}^+(\bx))}  +  \Big\| \frac{\p u}{\p n}\Big\|_{C_{\delta,\nu+1,\mu}^{k+1,\alpha}(B_{2r\cdot \fs(\bx)}^+(\bx)\cap \p M_T)} \nonumber\\
& + \| u\|_{C_{\delta,\nu,\mu}^0(B_{2r\cdot \fs(\bx)}^+(\bx))}\Big),
\label{e:boundary-local-estimate}
\end{align}
where  $B_s^+(\bx)\equiv B_s(\bx)\cap \mathcal{M}_T$ and $\frac{\p}{\p n}$ is the exterior normal vector field on $\p M_T$.

\end{enumerate}

\end{proposition}

\begin{proof}

We first describe the proof of \eqref{e:local-schauder}. 
Without loss of generality, we only prove the estimate by assuming the scale parameter $r=1/8$ and $k=0$. The estimate in the general case $r\in(0,1)$ can be achieved by simple rescaling.  

Since we have shown in item (1) of Proposition \ref{p:regularity-scale} that, for any $\bx\in\mathring{\mathcal{M}}_T(T_-,T_+)$, under the rescaled metric 
$\tilde{g}_T = \fs(\bx)^{-2}\cdot  g_T$, the $C^{2,\alpha}$-regularity scale of
$B_{1/2}^{\tilde{g}_T}(\bm{x})$ is uniformly bounded from below (independent of $T$) for each $\alpha\in(0,1)$. So there is a uniform constant $C>0$ (independent of $T$) such that the standard Schauder estimate holds for $k\in\{0,1\}$ and for every $u\in C^{2,\alpha}(B_{2 r\cdot \fs(\bx)}(\bx))$ and $\bm{x}\in \mathring{\mathcal{M}}(T_-,T_+)$, \begin{equation}
\|u\|_{C^{2,\alpha}(B_{1/8}^{\tilde{g}_T}(\bm{x}))} \leq C\Big(\| \Delta_{\tilde{g}}u \|_{C^{\alpha}(B_{1/4}^{\tilde{g}_T}(\bm{x}))} +\|u\|_{C^{0}(B_{1/4}^{\tilde{g}_T}(\bm{x}))}\Big).\label{e:ball-standard-schauder}
\end{equation}
Then rescaling back to $g_T$ and using the definition of the weighted functions, we have that
\begin{align}
\begin{split}
&\sum\limits_{m=0}^2\|\rho_{\delta,\nu,\mu}^{(m)}(\bx)\cdot\nabla^m u\|_{C^{0}(B_{\fs(\bx)/8}(\bx))} 
+
[\rho_{\delta,\nu,\mu}^{(2+\alpha)}(\bx)\cdot\nabla^2 u]_{C^{0,\alpha}(B_{\fs(\bx)/8}(\bx))} \\
\leq &  C\cdot\Big(\|\rho_{\delta,\nu+2,\mu}^{(0)}(\bx)\cdot \Delta u\|_{C^{0}(B_{\fs(\bx)/4}(\bx))} 
+[\rho_{\delta,\nu+2,\mu}^{(\alpha)}(\bx)
\cdot \Delta u]_{C^{0,\alpha}(B_{\fs(\bx)/4}(\bx))}
\\
&+\|\rho_{\delta,\nu,\mu}^{(0)}(\bx) \cdot u\|_{C^0(B_{\fs(\bx)/4}(\bx))}\Big).
\end{split}
\end{align}
By the definition of the weighted norms, the next is to verify that for every $\bm{x}\in \mathring{\mathcal{M}}(T_-,T_+)$,
 the weight function $\rho_{\delta,\nu,\mu}^{(k+\alpha)}$ is roughly a constant in the ball $B_{\fs(\bx)/4}(\bm{x})$ in the sense that there is a uniform constant $C=C_{k,\alpha}>0$ (independent of $T$) such that for any $\by\in B_{\fs(\bx)/4}(\bm{x})$,
 \begin{equation}
C^{-1}\cdot  \rho_{\delta,\nu,\mu}^{(k+\alpha)}(\bx)\leq  \rho_{\delta,\nu,\mu}^{(k+\alpha)}(\by) \leq  C\cdot \rho_{\delta,\nu,\mu}^{(k+\alpha)}(\bx).\label{e:weight-function-control}
 \end{equation}
The verification of \eqref{e:weight-function-control} follows from the comparison on $\fs(\bx)$ given in item (2) of Proposition \ref{p:regularity-scale}. So we have that
\begin{align}
\begin{split}
&\sum\limits_{m=0}^2\|\rho_{\delta,\nu,\mu}^{(m)}\cdot\nabla^m u\|_{C^{0}(B_{\fs(\bx)/8}(\bx))} 
+
[\rho_{\delta,\nu,\mu}^{(2+\alpha)}\cdot\nabla^2 u]_{C^{\alpha}(B_{\fs(\bx)/8}(\bx))}\\
\leq &
C\Big(\|\rho_{\delta,\nu+2,\mu}^{(0)}\cdot \Delta u\|_{C^{0}(B_{\fs(\bx)/4}(\bx))}+[\Delta u]_{C_{\delta,\nu+2,\mu}^{0,\alpha}(B_{\fs(\bx)/4}(\bx))}+\|\rho_{\delta,\nu,\mu}^{(0)} \cdot u\|_{C^0(B_{\fs(\bx)/4}(\bx))}\Big).
\end{split}
\end{align}
 Then we obtain the weighted Schauder estimate \eqref{e:local-schauder}.

The proof of item (2) is the same as item (1). We just notice that higher order estimate holds in a geodesic ball without touching the singular locus $\mathcal{P}$.

We now prove
item (3), which 
follows from the Schauder estimate for Neumann boundary problem. 
As before, we only consider $r=1$ and prove the estimate for $\bx\in\{T=T_-\}\subset \p M_T$.  
Let us choose the rescaled metric $\tilde{g} = \fs(\bx)^{-2} \cdot g$ with 
\begin{equation}
\fs(\bx) = (L_T(T_-))^{\frac{1}{2}}\cdot  T^{\frac{2-n}{2n}}=1.\end{equation}
By the proof of Proposition \ref{p:regularity-scale}, for $T\gg1$ sufficiently large, 
$(\M_T, \tilde{g}, \bx)$ is $C^k$-close to a fixed incomplete Calabi space $(\Ca_-^n, g_{\Ca_-^n},\bx_{\infty})$ for any $k\in\dZ_+$.
Moreover, the geodesic ball $B_{1/4}^+(\bx)$ with respect to the rescaled metric $\tilde{g}_T$ has uniformly bounded $C^{k,\alpha}$-geometry (independent of $T$) for any $k\in\dZ_+$. So the standard Schauder estimate for the Neumann boundary problem reads as follows (see Section 6.7 of \cite{GT} for instance):
\begin{align}
\|u\|_{C^{k+2,\alpha}(B_{1/8}^+(\bx))}
\leq   C_{k,\alpha} \Big(\|\Delta u\|_{C^{k,\alpha}(B_{1/4}^+(\bx))} + \Big\| \frac{\p u}{\p n}\Big\|_{C^{k+1,\alpha}(B_{1/4}^+(\bx)\cap \p\M_T)}+ \| u\|_{C^0(B_{1/4}^+(\bx))}\Big),
\end{align}
where $\frac{\p}{\p n}$ is the exterior normal vector field on $\p\M_T$. So we obtain the desired weighted estimate. 
\end{proof}

We finish this subsection with a weighted error estimate for the Calabi-Yau equation.

\begin{proposition}[Weighted error estimate] \label{p:CY-error-small} 
Let $\Err$ be the error function given by Definition \ref{d:error-function}. 
 For fixed parameters $\delta>0$, $\mu,\nu\in\dR$ and $\alpha\in(0,1)$ which satisfy 
 \begin{align}
 0<\delta < \delta_{e} & \equiv \frac{\sqrt{\lambda_1}}{n(|k_-|+|k_+|)},
 \\
 \nu  + \alpha & >0,
 \end{align}
where the constants $\lambda_1>0$, $k_->0$ and $k_+<0$ are given in Proposition \ref{p:existence-Greens-current}. 
Then the weighted $C^{0,\alpha}$-estimate holds, \begin{equation}\|\Err\|_{C^{0,\alpha}_{\delta,  \nu, \mu}(\M_T)}=O(T^{-2+\frac{\nu+\alpha}{n}+\mu}).
\end{equation}

\end{proposition}

\begin{proof}
We again divide into different regions and estimate separately. 

For $|z(\bx)|\leq 1$, applying Corollary \ref{c:psipower}, we have 
\begin{equation}
(\omega_D+T^{-1}\psi)^{n-1}=\omega_D^{n-1}(1+T^{-1}\Tr_{\omega_D}\psi+\sum_{k\geq 2}T^{-k}\Phi_k),
\end{equation}
where $\Phi_k=O'(r^{k-1})$ is independent of $T$. 
By \eqref{e:q-bounded-distance} we have
\begin{equation}
T^{-1}h=1+T^{-1}\Tr_{\omega_D}\psi+T^{-2}\underline B(z).
\end{equation}
Using \eqref{e:trace expansion equation}, it is easy to see that
$\|\Err\|_{C^0(\{|z(\bx)|\leq 1\})}=O(T^{-2})$.
Immediately, by the definition of the weighted $C^0$-norm, we have
\begin{equation}
\|\Err\|_{C_{\delta,\nu,\mu}^0(\{|z(\bx)|\leq 1\})}=O(T^{-2+\frac{\nu}{n}+\mu}),
\end{equation}

Next we consider the region $|z(\bx)|\geq 1$. Then by \eqref{e:greens-current-exp-asymp} we may write 
\begin{align}
\psi=
\begin{cases}
(k_- z)\cdot \omega_D+\xi, & z\leq -1,
\\
(k_+ z)\cdot \omega_D+\xi, & z\geq 1,
\end{cases}
\end{align}
where $\xi=\epsilon(z)$. 
So it follows that
\begin{align}
&(\omega_D+T^{-1}\psi)^{n-1}\nonumber\\
=&\Big((1+T^{-1}k_{\pm} z)\omega_D+T^{-1}\xi\Big)^{n-1}
\nonumber\\
=&\omega_D^{n-1}\Big((1+T^{-1}k_\pm z)^{n-1}+(1+T^{-1}k_\pm z)^{n-2}T^{-1}\Tr_{\omega_D}\xi+O(T^{-2})\Big).
\end{align}
By \eqref{e:h-asymptotics},  we have
$T^{-1}\cdot h=(1+T^{-1}k_\pm z)^{n-1}+T^{-1}\cdot \Tr_{\omega_D}\xi$.
So we obtain 
\begin{equation}
\Err=\Big((1+T^{-1}k_\pm z)^{-1}-(1+T^{-1}k_\pm z)^{-n+1}\Big)T^{-1}\Tr_{\omega_D}\xi+O(T^{-2}).
\end{equation}
Since for $z\in [T_-, T_+]$, 
\begin{equation}
U_T(z)=T-T^{-\frac{n-2}{2}}(T+k_\pm z)^{\frac{n}{2}}=T(1-(1+T^{-1}k_{\pm}z)^{\frac{n}{2}})\leq -\frac{n}{2}\cdot k_{\pm}z.\label{e:U-upper-bound}
\end{equation}
Here we have used the following inequality:  $(1-x)^p\geq 1-px$ for any $p\geq 1$ and $x\in(0,1)$. 
By Proposition \ref{p:existence-Greens-current},  the asymptotics $\xi=\epsilon(z)$ has an explicit exponential decaying rate 
$\epsilon(z)=O(e^{-(1-\tau)\sqrt{\lambda_1}z})$ for any $\tau\in(0,1)$. Applying \eqref{e:U-upper-bound} and the  
 the assumption  
\begin{equation}0<\delta<\delta_e \equiv \frac{\sqrt{\lambda_1}}{n(|k_-|+|k_+|)},\end{equation} we conclude that, as $|z(\bx)|\to+\infty$, 
the growth rate of $e^{\delta\cdot  U_T(z(\bx))}$
is slower than the decaying rate of $\epsilon(z)$. 
Therefore, we have that \begin{align}
\|\Err\|_{C^0(\{\|z(\bx)\|\geq 1\})}  &= O(T^{-2}),
\\
\|\Err\|_{C_{\delta,\nu,\mu}^0(\{\|z(\bx)\|\geq 1\})} &= O(T^{-2+\frac{\nu}{n}+\mu}).
\end{align}

The weighted $C^{0,\alpha}$-estimate can be obtained in a similar way. 
It suffices to analyze the H\"older regularity around the singular set $\mathcal{P}$.
Notice that a fixed function in $O'(r)$ has bounded $C^{0,\alpha}$ norm, so 
the weighted $C^{0,\alpha}$-estimate yields\begin{equation}
\|\Err\|_{C_{\delta,\nu,\mu}^{\alpha}(\M_T)}=O(T^{-2+\frac{\nu+\alpha}{n}+\mu}).
\end{equation}
 The proof is done. 
\end{proof}

\subsection{Perturbation of complex structures}
\label{ss:perturbation of complex structures}

In Section \ref{ss:complex-geometry} we have identified the underlying complex manifold of our family of $C^{2, \alpha}$ K\"ahler metrics $(\M_T, \omega_T)$. In our gluing argument in Section \ref{ss:glued-metrics} we  need to perturb the complex structure and accordingly perturb the K\"ahler forms. This section is devoted to the estimate of error caused by such a perturbation. 

Under the holomorphic embedding of $\M_T$ into $\mathcal N^0$ defined in Section 4.2, $\Omega_T$ is identified with the standard holomorphic volume form $\Omega_{\mathcal N^0}$. 
Fix $C_0>0$, and let $\cV$ be the open neighborhood of $\mathcal P$ in $\mathcal N^0$ defined by 
$\{r_+<C_0, r_-<C_0\}$.
Fix a smooth K\"ahler metric $\omega_{\mathcal N^0}$ on $\mathcal N^0$. 
Assume that there are a family of complex structures $J_T'$ on $\cV$ with holomorphic volume forms $\Omega_T'$ satisfying for all $k\geq 0$, 
\begin{equation}
\sup_{\bx\in \mathcal V}|\nabla^k_{\omega_{\mathcal N^0}}(\Omega_T'-\Omega_{\mathcal N^0})(\bx)|_{\omega_{\mathcal N^0}}\leq \underline\epsilon_{T^2}.
\end{equation}
 We also assume there is a deformation of the form $\pi^*\omega_D$ over $\cV$ to $\omega_{D, T}$, which is a closed $(1,1)$ form with respect $J_T'$, and  satisfies that  for all $k\geq 0$
 \begin{equation}
\sup_{\bx\in \mathcal V}|\nabla^k_{\omega_{\mathcal N^0}}(\omega_{D, T}-\pi^*\omega_D)(\bx)|_{\omega_{\mathcal N^0}}\leq \underline\epsilon_{T^2}.
\end{equation}
These assumptions will be met in our applications.
Let $\phi$ be the K\"ahler potential defined in \eqref{eqn7-10}. Then we define the new family of closed forms on $\mathcal V$ by
 \begin{equation}
 T^{\frac{n-2}{n}}\omega_T'\equiv T\omega_{D, T}+dJ_T'd\phi. 
 \end{equation}
  
 \begin{proposition} \label{p: perturbation of complex structures}
 For all sufficiently large $T$, the above $(\omega_T', \Omega_T')$ defines a family of $C^{1, \alpha}$-K\"ahler structures on $\cV$ which satisfies that for all fixed $\alpha\in(0,1)$, $\delta,\mu, \nu\in \dR$, we have
  \begin{align}\label{e:Omega perturbation} \|\Omega_T'-\Omega_T\|_{C^{2, \alpha}_{\delta, \mu, \nu}(\cV)}&=\underline \epsilon_{T^2},
\\
 \label{e:omega perturbation}
 \|\omega_T'-\omega_T\|_{C^{1, \alpha}_{\delta, \mu, \nu}(\cV)}&=\underline \epsilon_{T^2}.\end{align}
\end{proposition}
 \begin{remark} \label{r:C1alphavsC2alpha}
Notice that we lose one derivative control on $\omega_T'$. This is due to the fact that in the above definition of $\omega_{T}'$ we also have a derivative on the complex structure $J_T'$. Later we need at least $C^{1,\alpha}$ regularity on the metrics to obtain $C^{2, \alpha}$ weighted Schauder estimates for the Laplacian operator. 
This technical issue is the reason why we need to show $(\omega_T, \Omega_T)$ defined in Section \ref{ss:kaehler-structures} is $C^{2,\alpha}$, and it determines the order we need in the expansion of the Green's current in Theorem \ref{t:Green-expansion}.
\end{remark} 

\begin{proof}[Proof of Propsoition \ref{p: perturbation of complex structures}] 
We first reduce the estimate to a local one. To do this, let us  
choose finite holomorphic charts $\{(U_\beta, w_1, \ldots, w_{n-1})\}$ of $D$ and each chart $U_{\beta}$ is given by $\{|w_i|<1\}_{i=1}^{n-1}$,  such that the following holds:
\begin{enumerate}
\item The smaller charts $V_{\beta}\subset U_{\beta}$ given by $\{|w_i|<1/2\}_{i=1}^{n-1}$ also cover $D$. 
\item If $U_{\beta}\cap H\neq \emptyset$, then $U_{\beta}$ satisfies that for some $p\in H$, $w_i(p)=0$ for every $1\leq i\leq n-1$, and $H\cap U_{\beta}$ is given by $\{w_1=0\}$.
\item The line bundle $L$ restricted to every $U_\beta$ has a holomorphic trivialization $\sigma_{\beta}$ under which $\zeta_{\pm}$ can be viewed as local holomorphic functions on $\mathcal{N}^0$,  and $\cV\cap \pi^{-1}(U_\beta)$ is locally defined by $|\zeta_\pm|<C\cdot|\sigma_\beta|^{-|k_\pm|}$.
\end{enumerate}
The above gives an open cover of $\cV$ by $\cV\cap \pi^{-1}(V_\beta)$, and it suffices to prove the estimates in each of those open sets.

We only work with the charts that satisfy $U_\beta\cap H\neq \emptyset$, and the other case can be proved in a similar manner. For such charts, in $\pi^{-1}(U_\beta)$, by the definition of $\mathcal N^0$, we have that
\begin{equation}
\zeta_+\cdot \zeta_-=w_1 \cdot F(w_1, \ldots, w_{n-1})
\end{equation}
for a non-zero holomorphic function $F$. Without loss of generality, we may assume that $\{\zeta_+,\zeta_-, w_2, \ldots, w_{n-1}\}$ are holomorphic coordinates on $\pi^{-1}(U_\beta)$. 

We first prove \eqref{e:Omega perturbation}. 
The hypothesis implies that 
\begin{equation} \label{e: holomorphic volume form difference}
\Omega_T'-\Omega_{T}=G_{i_1\ldots i_n} \cdot e_{i_1}\wedge\ldots\wedge  e_{i_n},
\end{equation}
where each $e_j$ is one of $d\zeta_{\pm}, d\bar\zeta_{\pm}, dw_j, d\bar w_j$ for $2\leq j\leq n-1$, and $G_{i_1\ldots i_n}$ is a smooth function in $\zeta_{\pm}, \bar{\zeta}_{\pm},  w_j, \bar{w}_j $ for $2\leq j\leq n-1$ and its $k$-th derivative over $\mathcal V$ with respect to the fixed metric $\omega_{\mathcal N^0}$ is bounded by $\underline \epsilon_{T^2}$ for all $k$. Since a holomorphic function is automatically harmonic with respect to any K\"ahler metric, we have 
\begin{equation}
\Delta_{\omega_T}\zeta_{\pm}=\Delta_{\omega_T}w_j=0,\quad 2\leq j\leq n-1.\label{e:zeta-harmonic}
\end{equation}
By item (1) of Corollary \ref{c:r zeta relation}, we have that $\cV$ is contained in the region $|z|\leq 1$. Also notice by the discussion in Section \ref{ss:regularity-scales}, there is a constant $C>1$ such that for each $\bx \in \cV\cap \pi^{-1}(V_\beta)$, the ball $B_{C^{-1}\mathfrak{s}(\bx)}(\bx)$ is contained in $\pi^{-1}(U_\beta)\cap \{|z|\leq 2\}$.  In this proof we denote by $C>0$ a constant which is independent of $T$ but may vary from line to line. Again by Corollary \ref{c:r zeta relation}, item (3) on $\pi^{-1}(U_\beta)\cap \{|z|\leq 2\}$, we have 
\begin{equation}
|\zeta_\pm|\leq Cr_\pm \leq Ce^{3T}. 
\end{equation}
By \eqref{e:zeta-harmonic}, one can apply the weighted Schauder estimate in  Proposition \ref{p:local-weighted-schauder} to every $\bx\in \cV\cap \pi^{-1}(V_\beta)$, and we obtain 
\begin{equation}\label{e:eta pm estimate}
 |\zeta_\pm|_{C^{3, \alpha}_{\delta, \mu, \nu}(\cV\cap \pi^{-1}(V_\beta))}\leq  C|\zeta_\pm|_{C^{0}_{\delta, \mu, \nu}(\pi^{-1}(U_\beta)\cap \{|z|\leq 1\})}\leq Ce^{3T}.
 \end{equation}
Similarly, since on $U_\beta$ we have $|w_j|<1$, we get for $2\leq j\leq  n-1$,  
 \begin{equation}
 |w_j|_{C^{3, \alpha}_{\delta, \mu, \nu}(\cV\cap \pi^{-1}(V_\beta))}\leq  |w_j|_{C^{0}_{\delta, \mu, \nu}(\pi^{-1}(U_\beta)\cap \{|z|\leq 1\})}\leq C. 
 \end{equation}
 Then using the chain rule and induction we get that 
\begin{equation}
|G_{i_1\cdots i_n}|_{\wIII(\cV\cap \pi^{-1}(V_\beta))}=\underline\epsilon_{T^2}\cdot Ce^{3T}=\underline \epsilon_{T^2} 
\end{equation}
and hence
$|\Omega_T'-\Omega_T|_{C^{2, \alpha}_{\delta, \mu, \nu}(\cV\cap \pi^{-1}(V_\beta))}=\underline \epsilon_{T^2}$. 
Notice that the complex structure $J_T'$ is pointwise determined by the holomorphic volume form $\Omega_T'$ algebraically, so we have \begin{equation}
|J_T'-J_T|_{\wII(\cV\cap \pi^{-1}(V_\beta))}=\underline\epsilon_{T^2}.
\end{equation}

To prove \eqref{e:omega perturbation}, we write
\begin{equation}
\omega_T'-\omega_T = T(\omega_{D, T}-\pi^*\omega_D)+d((J_T'-J_T)d\phi). 
\end{equation}
By assumption, and the above discussion, using  \eqref{e:eta pm estimate} we get 
\begin{equation}
|\omega_{D, T}-\pi^*\omega_D|_{\wII(\pi^{-1}(V_\beta)\cap \{|z|\leq 1\})}=\underline\epsilon_{T^2}
\end{equation}
It is also easy to see that for two functions $V_1, V_2$, we have
\begin{equation}
|V_1\cdot V_2|_{\wII(\cV\cap \pi^{-1}(V_\beta))}\leq C\cdot T^m \cdot |V_1|_{\wII(\cV\cap \pi^{-1}(V_\beta))}\cdot|V_2|_{\wII(\cV\cap \pi^{-1}(V_\beta))},
\end{equation}
for some $m>0$ independent of $V_1$, $V_2$. We refer to Lemma \ref{l:weight-function-lower-bound-estimate} for the more precise expression of $m$. 
So the proof of \eqref{e:omega perturbation}  is reduced to the following claim. 

\vspace{0.5cm}

{\bf Claim:} There is a constant $C>0$ such that $|\phi|_{\wIII(\cV\cap \pi^{-1}(V_\beta))}\leq e^{C\cdot T}$
holds  for all $T\gg1$.

\vspace{0.5cm}

Let us prove the claim. Since by construction 
$T^{\frac{n-2}{n}}\omega=T\pi^*\omega_D+dd^c\phi$,
we have that 
\begin{equation}
\Delta_{T^{\frac{2n-2}{n}}\omega}\phi=n-T\cdot \Tr_{T^{\frac{2n-2}{n}}\omega}(\pi^*\omega_D).
\end{equation}
Since $\pi^*\omega_D$ is smooth on $\cV$ and $\omega$ is parallel,  again the above discussion gives  that 
\begin{equation}
|\Tr_{T^{\frac{2n-2}{n}}\omega}(\pi^*\omega_D)|_{C^{1, \alpha}_{\delta, \nu, \mu}(\cV\cap\pi^{-1}(V_{\beta}))}\leq e^{CT}.
\end{equation}
Then Proposition \ref{p:local-weighted-schauder} implies that 
\begin{equation}
|\phi|_{C^{3, \alpha}_{\delta, \nu, \mu}(\cV\cap \pi^{-1}(V_\beta))}\leq e^{CT}+C|\phi|_{C^0_{\delta, \nu, \mu}(\pi^{-1}(U_\beta)\cap \{|z|\leq 1\})}.
\end{equation}
To bound the right hand side we use the formula 
\begin{equation}
\phi=\int_{T_+}^z uh(u)du+\phi(T_+)=\int_{T_+}^0 uh(u)du+\phi(T_+)+\int_0^z uh(u)du.
\end{equation}
Hence by \eqref{e:h-asymptotics} and \eqref{e:h-bounded-z},   \begin{equation}
\phi=\frac{T}{2}z^2+\frac{1}{2}r+B_T+O(1),
\end{equation}
which gives 
$|\phi|_{C^0_{\delta, \nu,\mu}(\pi^{-1}(V_\beta)\cap \{|z|\leq 1\})}\leq O(T^m)$
for some $m>0$. The conclusion then follows. 
\end{proof}

\begin{remark}
In principle, it is possible to obtain more refined estimates on the higher order weighted norms of $\zeta_\pm$ and $\phi$ by more direct calculation. The above argument using weighted Schauder estimates avoids the lengthy computations, and it suffices for our purpose since in our setting the error caused by complex structure perturbation is at the scale $e^{-CT^2}$ while the weighted analysis in the region $\{|z|\leq 1\}$ only introduces at most $e^{CT}$ error.
\end{remark}

\section{Perturbation to Calabi-Yau metrics on the neck}

\label{s:neck-perturbation}

In Section \ref{ss:kaehler-structures}, we constructed a family of $C^{2,\alpha}$-K\"ahler structures $(\omega_T,\Omega_T)$ on $\M_T$ which are very close to Calabi-Yau metrics with weighted error estimates in Proposition \ref{p:CY-error-small}. The main result of this section is Theorem \ref{t:neck-CY-metric}.  
Our goal in this section is to perturb $\omega_T$ to a genuine Calabi-Yau metric for any sufficiently large $T$. 
Technically, this amounts to proving the uniform estimates demanded by the implicit function theorem (Lemma \ref{l:implicit-function}). Those estimates will be proved in Sections \ref{ss:perturbation-framework}-\ref{ss:incomplete weighted analysis}. In Section \ref{ss:renormalized-measure}, we will  discuss some geometric information naturally arising in the {\it measured Gromov-Hausdorff convergence} of those Calabi-Yau metrics obtained in Theorem \ref{t:neck-CY-metric}  on appropriate scales. The results in this section may have independent interest. 
As mentioned in the introduction of this paper, it is the proof, but not the statement of Theorem \ref{t:neck-CY-metric} itself, that will be used in the proof of Theorem \ref{t:main-theorem}.

\subsection{Framework of perturbation}
\label{ss:perturbation-framework}
The studies and applications of the implicit function theorem have been well developed 
in various contexts. We refer the readers 
to the book \cite{Krantz} for seeing the comprehensive discussions and the history of the whole methodology. 
For our practical and specific applications, we need the following quantitative version of implicit function theorem (Lemma \ref{l:implicit-function}), which is based on  Banach contraction mapping principle. 

To avoid confusions, we clarify several notations as follows:
\begin{itemize}
\item Let $\mathscr{L}: \fS_1 \to \fS_2$ be a bounded linear operator between normed linear spaces $\fS_1$ and $\fS_2$. Then the operator norm of $\mathscr{L}$ is  defined by  
\begin{equation}
\|\mathscr{L}\|_{op} \equiv \inf\Big\{ M_0\in \dR_+ \Big| \ \|\mathscr{L}(\bv)\|_{\fS_2}\leq M_0\cdot \|\bv\|_{\fS_1} , \ \forall \bv\in \fS_1 \Big\}. 
\end{equation}

\item We use the common notation $\bo$ for the zero vector in every normed linear space. 
\end{itemize}

\begin{lemma}
[Implicit function theorem]
 \label{l:implicit-function}
Let $\mathscr{F}:\fS_1 \to \fS_2$ be a map between two Banach spaces such that for all $\bv\in \fS_1$,
\begin{equation}\mathscr{F}(\bv)-\mathscr{F}(\bo)=\mathscr{L}(\bv)+\mathscr{N}(\bv),\label{e:functional-expansion}\end{equation} where the operator $\mathscr{L}:\fS_1\to\fS_2$ is linear and the operator  $\mathscr{N}:\fS_1\to \fS_2$  satisfies $\mathscr{N}(\bo)=\bo$. Additionally we assume the following properties:
\begin{enumerate} 
\item (Bounded inverse) $\mathscr{L}:\fS_1\to\fS_2$ is an isomorphism and there is some constant $C_L>0$ such that \begin{equation}\| \mathscr{L}^{-1} \|_{op} \leq C_L,\label{e:bounded-inverse}\end{equation} where $\mathscr{L}^{-1}$ is the inverse of $\mathscr{L}$.

\item There are constants $C_N>0$ and   $r_0\in(0,\frac{1}{2C_L C_N})$ which satisfy the following: 
\begin{enumerate}\item  (Controlled nonlinear error) for all $\bv_1,\bv_2\in \overline{B_{r_0}(\bo)}\subset \fS_1$,
\begin{equation}\| \mathscr{N}(\bv_1) - \mathscr{N}(\bv_2) \|_{\fS_2}  \leq C_N\cdot r_0 \cdot  \| \bv_1 - \bv_2 \|_{\fS_1}.\label{e:nonlinear-term-control}\end{equation}
\item (Controlled initial error) $\mathscr{F}(\bo)$
is effectively controlled as follows, 
\begin{equation}\| \mathscr{F}(\bo) \|_{\fS_2}\leq \frac{r_0}{4C_L}.\label{e:very-small-initial-error}\end{equation}
\end{enumerate}
\end{enumerate}
Then the equation $\mathscr{F}(\bx)=\bo$ has a unique solution $\bx\in B_{r_0}(\bo)$ with the estimate 
\begin{equation}
\|\bx\|_{\fS_1} \leq 2C_L \cdot \|\mathscr{F}(\bo)\|_{\fS_2}.\label{e:apriori-estimate-of-the-fixed-point}
\end{equation}

\end{lemma}

\begin{remark}
In our applications, the constants $C_L>0$, $C_N>0$ and $r_0>0$ will be fixed as uniform constants (independent of $T\gg1$),  which will be stated and proved in the next subsections. With the specified weight parameters $\delta,\mu,\nu$, 
the error estimate in Proposition \ref{p:CY-error-small} in fact guarantees 
$\|\Err\|_{\fS_2} \to 0  $ as $T\to\infty$, which particularly implies $\| \mathscr{F}(\bo) \|_{\fS_2}\to 0$  so that  $\mathscr{F}$ satisfies item (2b) in the above lemma.

\end{remark}

To set up the perturbation problem in our setting,  
we define the Banach spaces
\begin{align}
\begin{split}
\fS_1 & \equiv \Big\{\sq\p\bp\phi\in\Omega^{1,1}(\M_T)\Big| \phi\in C^{2, \alpha}(\M_T)\ \text{is}\ S^1\text{-invariant and satisfies} \ \frac{\p \phi}{\p n}\Big|_{\p \M_T}=0\Big\},
\\
\fS_2 & \equiv \Big\{f\in C^{0,\alpha}(\M_T)\Big|f\ \text{is}\ S^1\text{-invariant and} \ \int_{\M_T} f \cdot \omega_T^n= 0\Big\},
\end{split}
\end{align}
endowed with the weighted H\"older norms 
\begin{align} 
\|\sq\p\bp\phi \|_{\fS_1} &\equiv \|\sq\p\bp\phi\|_{C_{\delta,\nu+2,\mu}^{0,\alpha}(X_t)}, \label{e:norm-of-S1-space}
\\
\|f \|_{\fS_2} &\equiv \|f\|_{C_{\delta,\nu+2,\mu}^{0,\alpha}(X_t)}.
\label{e:norm-of-S2-space}
\end{align}
It is worth mentioning that the weighted H\"older norms are defined in a way such that $\fS_1$ and $\fS_2$ are Banach spaces since those weighted norms are equivalent to the standard Banach norms on a closed manifold.
Also notice that an $S^1$-invariant function $\phi$ on $\M_T$ can be identified with a function on the quotient $Q_T$, and the Neumann condition $\frac{\p\phi}{\p n}|_{\p \M_T}=0$ amounts to the condition $\p_z\phi=0$ on $\p Q_T$. 

In this section, the parameters in the weighted norms are specified as follows:  

\begin{enumerate}
\item[(NP1)] (Fix  $\nu$) The parameters $\nu\in\dR$ is chosen such that
\begin{align}
\nu\in(-1,0). \label{e:fix-nu-neck}
\end{align}
In our context, Lemma \ref{l:harmonic-removable-singularity} requires $\nu\in(-1,1)$. Furthermore, we need $\nu\in(-1,0)$ for effectively applying Proposition \ref{p:CY-error-small} to Lemma \ref{l:implicit-function}.

\item[(NP2)] (Fix $\alpha$) The H\"older exponent $\alpha\in(0,1)$ is chosen sufficiently small such that
\begin{equation}
\nu+\alpha<0.\label{e:fix-alpha-neck}
\end{equation}

\item[(NP3)] (Fix $\delta$) $\delta>0$ is chosen such that \begin{equation}
0<\delta < \delta_N \equiv \frac{1}{n\cdot (|k_-| + |k_+|)^n}\cdot \min\{\delta_b, \delta_e,\sqrt{\lambda_D}\},\label{e:fix-delta-neck}
\end{equation}
where 
$\sqrt{\lambda_D}$ is in Lemma \ref{l:liouville-cylinder} (Liouville theorem for harmonic functions on the cylinder $Q$),  
$\delta_e>0$ is in the error estimate  Proposition \ref{p:CY-error-small}, $\delta_b\equiv \min  \{\delta_{\Ca_-^n}, \delta_{\Ca_+^n}\}$ with the constants $\delta_{\Ca_{\pm}^n}$ given in Lemma \ref{l:Liouville-Calabi-space-SZ}  which in turn depend on the two Calabi model spaces $(\mathcal C^n_{\pm}, g_{\mathcal C^n_{\pm}})$ respectively.

\item[(NP4)] (Fix $\mu$)  The parameter $\mu$ is fixed by 
\begin{equation}
\mu= \Big(1-\frac{1}{n}\Big)(\nu+2+\alpha).\label{e:fix-mu-neck}
\end{equation}
This condition guarantees that the weight function $\rho_{\delta,\nu,\mu}^{(\alpha)}$ with parameters specified as the above is uniformly bounded from below, which will be used in  controlling the non-linear error in Proposition \ref{p:nonlinear-neck}.
\end{enumerate}

Now we explicitly write down the nonlinear functional $\mathscr{F}$ in our context. For $T\gg1$, starting with the $C^{2,\alpha}$-K\"ahler structure $(\omega_T,\Omega_T)$, we will solve the Calabi-Yau equation
\begin{equation}
\frac{1}{n!} (\omega_T+\sqrt{-1}\p\bp\phi)^n = \frac{(\sq)^{n^2}}{2^n}\cdot \Omega_T\wedge \overline{\Omega}_T.\label{e:CY-eq-neck-T-large}
\end{equation} 
To begin with, we appropriately normalize the holomorphic volume form $\Omega_T$. 
Notice that the right hand side of \eqref{e:CY-eq-neck-T-large} satisfies 
\begin{align}\int_{\M_T} \frac{(\sq)^{n^2}}{2^n}\Omega_T\wedge\overline{\Omega}_T &= \int_{\M_T} h\cdot \frac{\omega_D^{n-1}}{(n-1)!} dz\wedge \Theta=\int_{T_-}^{T_+}\Big(\int_D h\cdot\frac{\omega_D^{n-1}}{(n-1)!}\Big)dz.\label{e:RHS-CY-eq}
\end{align}
By Proposition \ref{p:cohomology-constant}, we have
\begin{align}
[\tilde{\omega}(z)] 
=
\begin{cases}
(T + k_-\cdot z)[\omega_D], & z<0,
\\
(T + k_+\cdot z)[\omega_D],& z>0,
\end{cases}
\end{align}
as a cohomology class in $H^2(D;\dR)$, which implies \begin{align}
\int_{\M_T} \frac{\omega_T^n}{n!} = T^{2-n}\int_{\M_T} \frac{\tilde\omega(z)^{n-1}}{(n-1)!}dz\wedge \Theta =\Big(\frac{k_+-k_-}{n\cdot k_-\cdot k_+}\Big)(T^2-1).\label{e:LHS-CY-eq}\end{align}
Comparing \eqref{e:RHS-CY-eq} and \eqref{e:LHS-CY-eq} and using \eqref{haverage}, we have that 
\begin{equation}
\int_{\M_T} \frac{(\sq)^{n^2}}{2^n}\Omega_T\wedge\overline{\Omega}_T=(1+O(T^{-2}))\int_{\M_T}\frac{\omega_T^n}{n!}.
\end{equation}
Recall that the error function $\Err$ given in Definition \ref{d:error-function} satisfies \begin{equation}
\frac{(\sq)^{n^2}}{2^n}\Omega_T\wedge\overline{\Omega}_T=(1+\Err) \frac{\omega_T^n}{n!}.
\end{equation}
Therefore,  multiplying $\Omega_T$ by a $T$-dependent constant $c_T$ of the form $1+O(T^{-2})$,  we may assume 
\begin{align}
\int_{\M_T}\Err\cdot\omega_T^n=0.\label{e:integral-error-zero}\end{align}

Applying Proposition \ref{p:CY-error-small} and \eqref{e:fix-mu-neck}, 
\begin{equation}
\label{e:new error estimate}\|\Err\|_{C^{0,\alpha}_{\delta,  \nu+2, \mu}(\M_T)}=O(T^{\nu+\alpha}).
\end{equation}
Now we define 
\begin{equation}
\mathscr{F}: \fS_1\rightarrow C^{0, \alpha}(\mathcal M_T);\quad  \sqrt{-1}\p\bp\phi  \mapsto \frac{(\omega_T+\sq\p\bp\phi)^n}{\omega_T^n}-(1-\Err).
\end{equation}

\begin{lemma}
 $\mathscr{F}$ maps $\fS_1$ into $\fS_2$, and $ \|\mathscr{F}(\bo)\|_{\fS_2}=O(T^{\nu+\alpha}).
$
\end{lemma}

\begin{proof}The lemma amounts to proving that 
$\int_{\M_T} \mathscr{F}(\sqrt{-1}\p\bp \phi)\omega_T^n=0$.
Since \eqref{e:integral-error-zero} holds for the normalized $\Omega_T$, 
by Stokes' theorem, we have that 
\begin{equation}\int_{\M_T} (\omega_T +\sqrt{-1}\p\bp\phi)^n-\int_{\M_T}\omega_T^n=\int_{\p\M_T} \gamma,\end{equation}
where $\gamma$ is the sum of terms which contain a factor $d^c\phi$ and other factors either $dd^c\phi$ or $\omega_T$. We claim that $\gamma$ identically  vanishes on $\p\M_T$. It suffices to show $\p_t\lrcorner\gamma=0$.  Since by assumption $\phi$ is $S^1$-invariant,  we have $\p_t \phi=0$. By the Neumann boundary condition, we also have $d^c\phi(\p_t)=0$ on $\p\M_T$. This follows from the observation that $J\p_t=\nabla z$. 
Now 
\begin{align}\p_t\lrcorner \omega_T|_{\p\M_T}&=dz|_{\p\M_T}=0,
\\
\p_t\lrcorner dd^c\phi&=\mathcal L_{\p_t} (d^c\phi)-d(\p_t\lrcorner d^c\phi)=-d(\p_t\lrcorner d^c\phi).\end{align}
The last term vanishes on $\p \M_T$ since  $d^c\phi(\p_t)=0$ pointwise on $\p\M_T$.
\end{proof}

Now we are ready to state the main result in this section.

\begin{theorem}[Existence of $S^1$-invariant Calabi-Yau metrics] \label{t:neck-CY-metric}
For each sufficiently large $T$, there exists an $S^1$-invariant Calabi-Yau metric of the form $\omega_{T, CY}=\omega_T + \sq\p\bp\phi(T)
$, where  $\phi(T)\in \fS_1$ and
\begin{equation} \label{e:eqn524}
\|\sq\p\bp\phi(T)\|_{\fS_1}\leq C_0\cdot T^{\nu+\alpha},
\end{equation}
where $C_0>0$ is a uniform constant independent of $T\gg1$ and 
the weighted H\"older norm of $\fS_1$ is defined in \eqref{e:norm-of-S1-space} for parameters $\nu$, $\alpha$, $\delta$ and $\mu$ satisfying \eqref{e:fix-nu-neck},
\eqref{e:fix-alpha-neck}, \eqref{e:fix-delta-neck} and \eqref{e:fix-mu-neck}.  
\end{theorem}

To prove Theorem \ref{t:neck-CY-metric}, we write
\begin{equation}
\mathscr{F}(\sq\p\bp\phi) -  \mathscr{F}(\bo) = \mathscr{L}(\sq\p\bp\phi) + \mathscr{N}(\sq\p\bp\phi),
\end{equation}
for any $\sq\p\bp\phi \in \fS_1$, where 
\begin{align}
\mathscr{L}(\sq\p\bp\phi) & \equiv  \Delta \phi,\label{e:linear-MT}
\\
\mathscr{N}(\sq\p\bp\phi)\cdot \omega_T^n & \equiv (\omega_T+\sqrt{-1}\p\bp\phi)^n-\omega_T^n - n \omega_T^{n-1}\wedge \sq\p\bp\phi.\label{e:nonlinear-MT}
\end{align}
By definition of the weight function  and Lemma \ref{l:weight-function-lower-bound-estimate}, we have the following nonlinear error estimate. 
\begin{lemma}
[Nonlinear error estimate]\label{p:nonlinear-neck} There exists a constant $C_N>0$   independent of 
$T\gg1$ such that for all 
$\vr\in (0,\frac{1}{2})
$
and 
\begin{equation}\sqrt{-1}\p\bp\phi_1\in \overline{B_{\vr}(\bo)} \subset \fS_1, \quad \sqrt{-1}\p\bp\phi_2\in \overline{B_{\vr}(\bo)}\subset\fS_1,\end{equation} we have 
\begin{align}  \|\mathscr{N}(\sqrt{-1}\p\bp\phi_1)-\mathscr{N}(\sqrt{-1}\p\bp\phi_2)\|_{\fS_2} 
\leq   C_N \cdot \vr \cdot  \|\sqrt{-1}\p\bp(\phi_1-\phi_2)\|_{\fS_1}.
\end{align}

\end{lemma}

\begin{proof}

By definition,
\begin{align}
& \Big(\mathscr{N}(\sqrt{-1}\p\bp\phi_1)-\mathscr{N}(\sqrt{-1}\p\bp\phi_2)\Big)\cdot \omega_T^n
\nonumber\\
 = & \sum\limits_{k=2}^n\begin{pmatrix} 
n \\
k 
\end{pmatrix} \cdot \omega_T^{n-k}\wedge \Big((\sqrt{-1}\p\bp\phi_1)^k - (\sqrt{-1}\p\bp\phi_2)^k \Big).
\end{align}
By the definition of the norm on $\fS_1$, 
we have
 \begin{equation}
 \|\sqrt{-1}\p\bp\phi_1\|_{C_{\delta,\nu+2,\mu}^{0,\alpha}(\M_T)} \leq \vr , \quad 
 \|\sqrt{-1}\p\bp\phi_2\|_{C_{\delta,\nu+2,\mu}^{0,\alpha}(\M_T)} \leq \vr . 
\end{equation}
With $\mu$ specified by \eqref{e:fix-mu-neck}, Lemma \ref{l:weight-function-lower-bound-estimate} shows that the weight function $\rho_{\delta,\nu+2,\mu}^{(\alpha)}:\M_T\to \dR_+$ has a uniform lower bound $\rho_{\delta,\nu+2,\mu}^{(\alpha)}(\bx) \geq  1$ for any $\bx\in \M_T$,
So we have the pointwise estimates on $\M_T$,
   \begin{equation}
  |\sqrt{-1}\p\bp\phi_1 | \leq C_0\cdot   \vr   \quad \text{and} \quad   
   |\sqrt{-1}\p\bp\phi_2 | \leq C_0\cdot \vr, 
   \end{equation}
where $C_0>0$ is a uniform constant independent of $T\gg1$. This implies the pointwise estimate 
\begin{equation}
|\mathscr{N}(\sqrt{-1}\p\bp\phi_1)-\mathscr{N}(\sqrt{-1}\p\bp\phi_2)
| \leq C_N \cdot \vr  \cdot  |\sqrt{-1}\p\bp(\phi_1-\phi_2)|,
\end{equation}
 where 
$C_N>0$ is a uniform constant independent of $T$.
Write the above in terms of the weighted norms, we have 
\begin{equation}
\|\mathscr{N}(\sqrt{-1}\p\bp\phi_1)-\mathscr{N}(\sqrt{-1}\p\bp\phi_2)
\|_{C_{\delta,\nu+2,\mu}^{0,\alpha}(\M_T)} \leq C_N \cdot \vr  \cdot  \|\sqrt{-1}\p\bp(\phi_1-\phi_2)\|_{C_{\delta,\nu+2,\mu}^{0,\alpha}(\M_T)}. 
\end{equation}
The proof is done. 
\end{proof}

To apply the implicit function theorem, 
we still need to prove the weighted linear estimate, which will be completed in the following subsections.

\subsection{Some Liouville type theorems and  removable singularity theorems}
\label{ss:some liouville theorems}
 In this subsection, we introduce some removable singularity and Liouville type theorems, which will be needed in the proof of Proposition \ref{p:global-injectivity-estimate}.  
For the convenience of discussions, we give precise statement here.

\begin{lemma}[Removable singularity]
\label{l:harmonic-removable-singularity}Let $(M^m,g)$ be a Riemannian manifold with $m\geq 3$ such that $B_R(p)$ 
has a compact closure in $B_{2R}(p)$. Let $K\subset M^m$ be a  smooth submanifold with $\dim(K)=k_0\leq m-3$. If $u$ is harmonic in $B_R(p)\setminus K$ and there is some
$\epsilon\in(0,1)$ such that 
\begin{equation}
|u(x)|\leq \frac{C}{d_g(x,K)^{(m-2-k_0)-\epsilon}},
\end{equation}
 then  $u$ is harmonic in $B_R(p)$.
\end{lemma}

\begin{proof}
The point is to apply integration by parts to show that $u$ is a weak solution to $\Delta u =0$ on $B_R(p)$. The computations are routine and standard in the literature, so we just skip it. 
\end{proof}

\begin{lemma}
[Liouville theorem on $\dR^{m+n}$]\label{l:liouville-product} Given $m,l\in\dZ_+$ with $m+l\geq 3$,  Let $\mu_p\in(-1,1)\setminus\{0\}$ and let $u\in C^{\infty}(\dR^{m+l})$ be a harmonic function on the Euclidean space $(\dR^{m+l}, g_{\dR^m}\oplus g_{\dR^l})$. If $u$ satsifies
\begin{align}
|u(x,y)| \leq \frac{C}{|x|^{\mu_p}}  \quad   \text{for every}\ (x,y)\in (\dR^m\setminus \{0^m\})\times\dR^l,\label{e:partial-control}
\end{align} 
then $u\equiv 0$ on $\dR^{m+l}$.

\end{lemma}

\begin{remark} 
Although the result stated as in Lemma \ref{l:liouville-product} is sufficient for our purpose, it is worth mentioning that
this Liouville type result holds in a much more general setting under weaker assumptions. In fact, the uniform decay condition  \eqref{e:partial-control} can be replaced with some mild boundedness condition. We refer to
\cite[lemma A.1]{Walpuski} for the precise statement. 	
\end{remark}

\begin{proof}

The proof is done by using separation of variables. 
For the simplicity of notations, we denote 
\begin{equation}d\equiv m+l \geq 3.\end{equation}
Let $(r,\Theta)\in\dR^{d}$ be the polar coordinate system  in $\dR^{d}$, so the Laplacian of $u$ can be written as
\begin{equation}
\Delta_{\dR^{d}} u = \frac{\partial^2u}{\partial r^2} + \frac{d-1}{r}\cdot \frac{\partial u}{\partial r} + \frac{1}{r^2}\cdot\Delta_{\mathbb{S}^{d-1}}u.
\end{equation}

We make separation of variables on the punctured Euclidean space $\dR^{d}\setminus\{0^{d}\}$. Let 
\begin{equation}\lambda_j\equiv j(j+d-2),\ j\in\dN,\end{equation} be the spectrum of the unit round sphere $\mathbb{S}^{d-1}$. Correspondingly, let $\vf_j\in C^{\infty}(\mathbb{S}^{d-1})$ satisfy
\begin{equation}
-\Delta_{\mathbb{S}^{d-1}}\vf_j(\Theta) = \lambda_j \vf_j(\Theta).
\end{equation}
Then the function $u(r,\Theta)$ has the expansion along the fiber $\mathbb{S}^{d-1}$,
\begin{equation}
u(r,\Theta) = \sum\limits_{j=0}^{\infty} u_j(r)\cdot\vf_j(\Theta).\label{e:L2-exp-polar}
\end{equation}
Immediately, for each $j\in\dN$, the coefficient function  $u_j(r)$ solves the Euler-Cauchy equation, 
\begin{equation}
u_j''(r) + \frac{d-1}{r}\cdot u_j'(r) - \frac{1}{r^2} \cdot\lambda_j \cdot u_j(r) =0,
\end{equation}
which has a general solution
\begin{equation}
u_j(r)=C_j\cdot r^{p_j} + C_j^*\cdot r^{q_j},\label{e:polar-general-solution}
\end{equation}
where $p_j=\frac{2-d+\sqrt{(d-2)^2+4\lambda_j}}{2}\geq 0$ 
and $q_j=\frac{2-d-\sqrt{(d-2)^2+4\lambda_j}}{2}<0$ solve the quadratic equation
\begin{equation}
w^2 + (d-2)w - \lambda_j = 0.
\end{equation}
So it is obvious
\begin{align}
p_0 & = 0 , \quad q_0=2-d\leq -1,\nonumber\\
 p_j &\geq p_1 =1,\nonumber \\
 q_j & \leq q_1 = 1-d \leq  -2, \quad  j\in\dZ_+.\label{e:radial-gap}
\end{align}

In the following, we will show that, given the growth condition \eqref{e:partial-control} for $\mu_p\in(-1,1)\setminus\{0\}$, then for each $j\in\dN$ and for each $r>0$, the coefficient $u_j(r)$ satisfies 
\begin{equation}
|u_j(r)| \leq \frac{Q_j}{r^{\mu_p}},
\end{equation}
where $Q_j\in\dR$.
In fact, so it follows from the  expansion \eqref{e:L2-exp-polar} that for each $j\in\dN$,
\begin{equation}
u_j(r) = \int_{\mathbb{S}^{d-1}} u(r,\Theta) \cdot \vf_j \dvol_{\mathbb{S}^{d-1}},
\end{equation}
which implies 
\begin{equation}
|u_j(r)| \leq |\vf_j|_{L^{\infty}(\mathbb{S}^{d-1})}\cdot \int_{\mathbb{S}^{d-1}}\frac{1}{|x|^{\mu_p}} \dvol_{\mathbb{S}^{d-1}}.
\end{equation}
Next, we will write the above integral in the polar coordinates \begin{align}\Theta\equiv(\theta_1,\ldots,\theta_{d-1}),\quad \theta_1,\ldots,\theta_{d-2}\in[0,\pi],\quad \theta_{d-1}\in[0,2\pi].\end{align} Denote by  $d\Theta\equiv d\theta_1\wedge d\theta_2\wedge \ldots \wedge d\theta_{d-1}$,  then it is by elementary calculations that, 
$|x| = r^m\cdot  \prod\limits_{k=1}^{d-m}|\sin\theta_{k}|$ and
$\dvol_{\mathbb{S}^{d-1}}  = \prod\limits_{k=1}^{d-2}(\sin^{d-k-1}\theta_k)\cdot d\Theta$.
Therefore, 
\begin{equation}
\int_{\mathbb{S}^{d-1}}\frac{1}{|x|^{\mu_p}} \dvol_{\mathbb{S}^{d-1}}=\frac{1}{r^{\mu_p}}\int_{\mathcal{D}_{\Theta}}	\frac{\prod\limits_{k=1}^{d-2}(\sin^{d-k-1}\theta_k)}{\prod\limits_{k=1}^{d-m}|\sin\theta_{k}|^{\mu_p}}\cdot d\Theta,
\end{equation}
where $\mathcal{D}_{\Theta}\equiv\{0\leq \theta_1,\ldots, \theta_{d-2}\leq \pi, \ 0\leq \theta_{d-1}\leq 2\pi\}$. 
By assumption, $\mu_p\in(-1,1)\setminus\{0\}$, then the following is integrable
\begin{equation}
\mathcal{I}_0\equiv\int_{\mathcal{D}_{\Theta}}\prod\limits_{k=1}^{d-2}(\sin^{d-k-1}\theta_k)(\prod\limits_{k=1}^{d-m}|\sin\theta_{k}|^{\mu_p})^{-1}\cdot d\Theta.
\end{equation}
Therefore,  for each
$j\in\dN$, it holds that for all $r>0$,
\begin{equation}
|u_j(r)| \leq \frac{\mathcal{I}_0\cdot|\vf_j|_{L^{\infty}(\mathbb{S}^{d-1})}}{r^{\mu_p}} \equiv \frac{Q_j}{r^{\mu_p}}
\end{equation}

Now we go back to the representation of $u_j(r)$ in \eqref{e:polar-general-solution} and we analyze the growth behavior of function as  $r\ll1$ and $r\gg1$. Applying the assumption $\mu_p\in(-1,1)\setminus\{0\}$ and the gap obtained in \eqref{e:radial-gap}, we have that, for each $j\in\dN$, $C_j=C_j^*=0$.  Therefore, $u\equiv 0$.
\end{proof}

We will also need the following lemma which holds on a cylindrical space $D\times\dR$.
The proof follows from the rather standard separation of variables. We omit the details.

\begin{lemma}
[Liouville theorem on a cylinder]
\label{l:liouville-cylinder}
Let $(Q,g_Q)\equiv(D\times\dR,g_Q)$ be a cylinder with a product Riemannian metric $g_Q = g_D\oplus dz^2$, where $(D,g_D)$ is a closed Riemannian manifold.  Denote by $\lambda_D>0$ the lowest eigenvalue of the Laplace-Beltrami operator of $(D,g_D)$ acting on functions. If $u$ is a harmonic function on $Q$ satisfying the growth control
\begin{equation}|u| = O(e^{\lambda_c\cdot z})\end{equation} for some $\lambda_c\in(0,\sqrt{\lambda_D})$, then
$u\equiv 0$.

\end{lemma}

 Finally we quote a Liouville theorem on the Calabi model space, which is proved in \cite{SZ-Liouville} (see Corollary 5.3.1 in this paper).
\begin{lemma}[Liouville theorem on Calabi model space,     
 \cite{SZ-Liouville}] \label{l:Liouville-Calabi-space-SZ} Let $(\Ca^n, g_{\Ca^n})$ be a Calabi model space. There exists a constant $\delta_{\Ca^n}>0$ such that the following holds:
 Let $u$ be a harmonic function on  $\Ca^n$ which satisfies the Neumann condition $\frac{\p u}{\p n}|_{\p\Ca^n}=0$. If $u=O(e^{\delta_{\Ca^n}\cdot  z^{\frac{n}{2}}})$ as $z\rightarrow\infty$, then $u$ is a constant on $\Ca^n$. 
\end{lemma}

\subsection{Weighted analysis and existence of incomplete Calabi-Yau metrics}
\label{ss:incomplete weighted analysis}

The main focus of this subsection is to prove the following proposition.

\begin{proposition}[Uniform injectivity estimate on the neck]\label{p:estimate-L-neck}
For any sufficiently large parameter $T\gg1$, the linearized operator defined in \eqref{e:linear-MT}   
\begin{equation}\mathscr{L}: \fS_1\rightarrow \fS_2, \quad \sqrt{-1}\p\bp\phi\mapsto \Delta\phi\end{equation}
is an isomorphism and satisfies the uniform injectivity estimate,
\begin{equation}
\|\sqrt{-1}\p\bp\phi\|_{\fS_1}\leq   C_L \cdot \|\Delta\phi\|_{\fS_2}.
\end{equation}
Here the constant $C_L>0$ is independent of the parameter $T\gg1$.
\end{proposition}

A preliminary ingredient in proving Proposition \ref{p:estimate-L-neck} is the following weighted Schauder estimate.  

\begin{proposition}[Weighted Schauder estimate on the neck]\label{p:neck-weighted-schauder}
There exists a constant $C>0$ such that for $T\gg1$, and all $u\in C^{2, \alpha}(\M_T)$,  the following holds
\begin{equation}
\|u \|_{C_{\delta,\nu,\mu}^{2,\alpha}(\M_T)} \leq C \Big(\|\Delta u\|_{C_{\delta,\nu+2,\mu}^{0,\alpha}(\M_T)} + \Big\|\frac{\p u}{\p n}\Big\|_{C_{\delta,\nu+1,\mu}^{1,\alpha}(\p\M_T)}+\|u\|_{C_{\delta,\nu,\mu}^0(\M_T)}\Big).
\end{equation}
\end{proposition}

\begin{proof}
The proof follows directly from 
Proposition \ref{p:local-weighted-schauder} and standard covering argument. \end{proof}

Next, the key part of the injectivity estimate  in Proposition \ref{p:estimate-L-neck} is the following weighted 
estimate for higher derivatives with respect to the Neumann boundary value problem. 

\begin{proposition}[Uniform derivative estimates on the neck] \label{p:neck-uniform-injectivity}
Let the parameters $\mu$, $\nu$, $\alpha$, $\delta$ satisfy the conditions in \eqref{e:fix-nu-neck}, \eqref{e:fix-alpha-neck}, \eqref{e:fix-delta-neck} and \eqref{e:fix-mu-neck}. Then 
there exists a uniform constant $C>0$  such that for $T\gg1$, and for all 
$u\in C^{2,\alpha}(\M_T)$ satisfying the boundary condition $\frac{\p u}{\p n}|_{\p\M_T}=0$, we have 
\begin{align}
\|\nabla u\|_{C_{\delta,\nu+1,\mu}^{0}(\M_T)} + \|\nabla^2 u\|_{C_{\delta,\nu+2,\mu}^{0}(\M_T)} +  [u]_{C_{\delta,\nu,\mu}^{2,\alpha}(\M_T)}	\leq C \cdot \|\Delta u\|_{C_{\delta,\nu+2,\mu}^{0,\alpha}(\M_T)}. \label{e:hoelder-neck}
\end{align}
 
\end{proposition}

\begin{proof}

The proof of \eqref{e:hoelder-neck} consists of two primary steps:
In the first step, we will prove the weighted $C^1$ and $C^2$ estimates, 
\begin{equation}
		\|\nabla u\|_{C_{\delta,\nu+1,\mu}^{0}(\M_T)} + \|\nabla^2 u\|_{C_{\delta,\nu+2,\mu}^{0}(\M_T)} \leq C \cdot \|\Delta u\|_{C_{\delta,\nu+2,\mu}^{0,\alpha}(\M_T)},\label{e:neck-C2-estimate}\end{equation}
where $C>0$ is independent of $T$. Once \eqref{e:neck-C2-estimate} is accomplished, we will prove the weighted $C^{2,\alpha}$-estimate for some uniform constant $C>0$ independent of $T$, 
\begin{equation}
		[u]_{C_{\delta,\nu,\mu}^{2,\alpha}(\M_T)}	 \leq C \cdot \|\Delta u\|_{C_{\delta,\nu+2,\mu}^{0,\alpha}(\M_T)}.\label{e:neck-top-semi-norm}\end{equation}

\begin{flushleft}
{\bf Step 1} (weighted $C^1$ and $C^2$ estimates):
\end{flushleft}

 We will prove \eqref{e:neck-C2-estimate} by contradiction. 
Suppose no such uniform constant $C>0$ exists. That is, for fixed parameters  $(\nu, \alpha, \delta, \mu)$, 
there are the following sequences: 
\begin{enumerate}
\item 
A sequence $T_j\to+\infty$.

\item A sequence of $C^{2,\alpha}$-functions $u_j$ on $\M_j\equiv \M_{T_j}$ which satisfy 
\begin{align}
\begin{split}
\frac{\p u_j}{\p n}\Big|_{\p\M_j} & = 0, \\
				\|\nabla u_j\|_{C_{\delta,\nu+1,\mu}^{0}(\M_j)} + \|\nabla^2 u_j\|_{C_{\delta,\nu+2,\mu}^{0}(\M_j)} & = 1,
				\\
				\|\Delta u_j\|_{C_{\delta,\nu+2,\mu}^{0,\alpha}(\M_j)} &\to 0, \quad j\to+\infty.
\end{split}
\end{align}
\end{enumerate}
So it follows that either		
$\|\nabla u_j\|_{C_{\delta,\nu+1,\mu}^{0}(\M_j)}\geq \frac{1}{2}$ or 
$\|\nabla^2 u_j\|_{C_{\delta,\nu+2,\mu}^{0}(\M_j)}\geq \frac{1}{2}$. 
Without loss of generality, we only consider the first case and let  $\bx_j\in\M_j$ satisfy 
\begin{equation}
|\rho_{\delta,\nu+1,\mu}^{(0)}(\bx_j)\cdot\nabla u_j(\bx_j)|=\|\nabla u_j\|_{C_{\delta,\nu+1,\mu}^{0}(\M_j)}\geq\frac{1}{2}.\end{equation}	
Let us 	
renormalize $u_j$ by taking
$	v_j(\bx) \equiv u_j(\bx) - u_j(\bx_j)$. 
Immediately, $v_j(\bx_j) = 0$, $\frac{\p v_j}{\p n}|_{\p\M_j}=0$ and
\begin{align}\label{e:v_j-conditions}
\begin{split}
 |\rho_{\delta,\nu+1,\mu}^{(0)}(\bx_j)\cdot\nabla v_j(\bx_j)|=\|\nabla v_j\|_{C_{\delta,\nu+1,\mu}^{0}(\M_j)}&\geq\frac{1}{2},
 \\
\|\nabla v_j\|_{C_{\delta,\nu+1,\mu}^{0}(\M_j)} + \|\nabla^2 v_j\|_{C_{\delta,\nu+2,\mu}^{0}(\M_j)} & = 1, 
\\
\|\Delta v_j\|_{C_{\delta,\nu+2,\mu}^{0,\alpha}(\M_j)} &\to 0,
 \\
 \|v_j\|_{C_{\delta,\nu,\mu}^{0}(\M_j)} &\leq C_0,
 \end{split}
 \end{align}
where $C_0>0$ is independent of $T_j$.
Applying the weighted Schauder estimate in Proposition \ref{p:neck-weighted-schauder}, we have that  
$  \|v_j\|_{C_{\delta,\nu,\mu}^{2,\alpha}(\M_j)} \leq C_0$ for some uniform constant $C_0>0$ independent of $T_j$. 

We will rescale $(\M_j, g_{T_j})$ around the above reference points $\bx_j$  
such that the desired contradiction will arise on the limiting space. 
Denote the rescaling factors as follows:
\begin{enumerate}
\item {\bf Rescaling of the metrics:} 
Let $\tilde{g}_{T_j}\equiv \fs(\bx_j)^{-2}\cdot g_{T_j}$, then with respect to the fixed reference point $\bx_j\in\M_j$ picked as the above and passing to a subsequence, we have the  convergence
$(\M_j, \tilde{g}_{T_j}, \bx_j) \xrightarrow{GH} (\cX_{\infty}, \tilde{d}_{\infty}, \bx_{\infty})$ with the rescaled limit $X_{\infty}$ identified in Proposition \ref{p:regularity-scale}.

\item {\bf Rescaling of the solutions:} 
Let $\kappa_j>0$ be a sequence of rescaling factors which will be determined later, such that $\tilde{v}_j \equiv \kappa_j \cdot v_j$.

\item {\bf Rescaling of the weight functions:}
Denote by $\tilde{\rho}_{j,\delta,\nu,\mu}^{(k+\alpha)}$ be the weight functions on the rescaled spaces $(\M_j, \tilde{g}_{T_j},\bx_j)$ which are given by by 
$\tilde{\rho}_{j,\delta,\nu,\mu}^{(k+\alpha)} = \tau_j \cdot \rho_{j,\delta,\nu,\mu}^{(k+\alpha)}$, so that $\tilde{\rho}_{j,\delta,\nu,\mu}^{(k+\alpha)}$ converges to $\tilde{\rho}_{\infty,\delta,\nu,\mu}^{(k+\alpha)}$ on  $(\cX_{\infty}, \tilde{g}_{\infty},\bx_{\infty})$. Notice that the rescaling factor $\tau_j$ depends on $k$ and $\alpha$.
\end{enumerate}

In the following, we study the convergence of the renormalized functions $\tilde{v}_j$ with respect to the pointed convergence of the rescaled spaces $(\M_j,\tilde{g}_{T_j},\bx_j)$  for $\bx_j$ in every case in the proof of Proposition \ref{p:regularity-scale}. 
The main goal is to show that $\tilde{v}_{\infty}\equiv 0$ on each rescaled limit $\cX_{\infty}$ which gives the desired contradiction.

 To begin with,  we fix the rescaling factors. Notice that as computing the weighted norms of $\tilde{v}_j$ in terms of $\tilde{g}_{T_j}$ and $\tilde{\rho}_{j,\delta,\nu,\mu}^{(k+\alpha)}$, the property   
$\|\tilde{v}_j\|_{C_{\delta,\nu,\mu}^{k,\alpha}(\M_j)}=\|v_j\|_{C_{\delta,\nu,\mu}^{k,\alpha}(\M_j)}$ holds if we require the rescaling factors to satisfy
\begin{align}\frac{\tau_j \cdot \kappa_j}{(\fs(\bx_j)^{-1})^{k+\alpha}}= 1.\end{align}  So the rescaling factors $\tau_j$ and $\kappa_j$ are fixed as follows such that \eqref{e:v_j-conditions} retains for $\tilde{v}_j$ and $\bx_j$:
\begin{itemize}
\item 
If $|z(\bx_j)|$ is uniformly bounded as $j\to+\infty$, \begin{align}
\begin{cases}
	\tau_j = (\fs_j^{-1})^{\nu+k+\alpha}\cdot T_j^{-\mu}
	\\
	\kappa_j = (\fs_j^{-1})^{-\nu} \cdot T_j^{\mu}.
\end{cases}
\end{align}

\item If $|z(\bx_j)|\to +\infty$,  
\begin{align}
\begin{cases}
\tau_j = (\fs_j^{-1})^{\nu+k+\alpha}  \cdot e^{-T_j}\cdot T_j^{-\mu}
\\
\kappa_j =  (\fs_j^{-1})^{-\nu}  \cdot e^{T_j} \cdot T_j^{\mu}.
\end{cases}
\end{align}
\end{itemize}
As in the proof of Proposition \ref{p:regularity-scale}, we will produce a contradiction in each of the following four cases depending upon the location of $\bx_j\in\M_j$.

\vspace{0.5cm}

\begin{flushleft}
{\it Case (1):  There is a constant $\sigma_0\geq 0$ such that $r(\bx_j)\cdot T_j\to   \sigma_0$}.
\end{flushleft}
\vspace{0.5cm}

Recall the proof of Proposition \ref{p:regularity-scale} that the pointed $C^{2,\gamma}$-convergence
\begin{equation}
(\M_j , \tilde{g}_{T_j}, \bx_j) \xrightarrow{C^{2,\gamma}}	(\dC_{TN}^2 \times \dC^{n-2}, \tilde{g}_{\infty}, \bx_{\infty}) ,
\end{equation}
holds for any $\gamma\in(0,1)$,
where $\tilde{g}_{\infty} \equiv g_{TN} \oplus g_{\dC^{n-2}}$ is the product metric of the Taub-NUT metric $g_{TN}\equiv g_{TN,\sigma_0^2}$ given in \eqref{e:parameter-TN} and the Euclidean metric $g_{\dC^{n-2}}$.
Moreover, the rescaled weight function $\tilde{\rho}_{j,\delta,\nu,\mu}^{(k+\alpha)}$ converges to
\begin{align}
\tilde{\rho}_{\infty,\delta,\nu,\mu}^{(k+\alpha)}(\bx)	
=
\begin{cases}
1, & \bx\in T_1(\Sigma_{0}),\\
(d_{\tilde{g}_{\infty}}(\bx, \Sigma_0))^{\nu+k+\alpha}, & \bx\in(\dC_{TN}^2 \times \dC^{n-2})\setminus T_1(\Sigma_{0}),
\end{cases}
\end{align}
where $\Sigma_0 \equiv \{p_{\infty}\} \times \dC^{n-2}\subset\dC_{TN}^2\times \dC^{n-2}$ for some $p_{\infty}\in \dC_{TN}^2$, is the Gromov-Hausdorff limit of the lifted divisor $\mathcal{P}\equiv \pi^{-1}(P)\subset \M_j$ with respect to the rescaled metrics $\tilde{g}_j$ such that 
and 
\begin{equation}T_1(\Sigma_0) \equiv \{\bx\in \dC_{TN}^2\times \dC^{n-2} | d_{\tilde{g}_{\infty}}(\bx, \Sigma_0) \leq 1\}.\end{equation}
It is straightforward that, the rescaled functions $\tilde{v}_j$ converge to $\tilde{v}_{\infty}$ in the $C_{loc}^{2,\alpha'}$-norm for each $\alpha'\in(0,\alpha)$ such that the following properties hold,
\begin{enumerate}
\item $\|\nabla_{\tilde{g}_{\infty}}\tilde{v}_{\infty}\|_{C_{\delta,\nu+1,\mu}^0(\dC_{TN}^2\times \dC^{n-2},\tilde{g}_{\infty})} +\|\nabla_{\tilde{g}_{\infty}}^2\tilde{v}_{\infty}\|_{C_{\delta,\nu+2,\mu}^0(\dC_{TN}^2\times \dC^{n-2},\tilde{g}_{\infty})} = 1$,
\item $\tilde{v}_{\infty}(\bx_{\infty})=0$,

\item $\Delta_{\tilde{g}_{\infty}}\tilde{v}_{\infty} \equiv 0$ on $\dC_{TN}^2\times \dC^{n-2}$.
\end{enumerate}
We will prove that $\tilde{v}_{\infty} \equiv 0$ on $\dC_{TN}^2\times \dC^{n-2}$.

To start with, we will show that $\tilde{v}_{\infty}$ is constant on the Euclidean factor $\dC^{n-2}$. Indeed,
we write  $\bx\equiv (\bx', \bx'') \in \dC_{TN}^2\times\dC^{n-2}$, so it suffices to prove that for every $1\leq k\leq 2n-4$, we have
\begin{equation}
|\nabla_{k} \tilde{v}_{\infty}| \equiv 0\ \text{on}\ \dC^{n-2},\label{e:gradient-vanishing}
\end{equation}
 where the partial derivative $\nabla_{k} \tilde{v}_{\infty}(\bx) \equiv \frac{\partial \tilde{v}_{\infty}}{\partial x_k''}(\bx',\bx'')$ is taken in the directions of $\dC^{n-2}$.
 Now for every $1\leq k \leq 2n-4$,
\begin{equation}
\Delta_{\tilde{g}_{\infty}} (\nabla_{k} \tilde{v}_{\infty}) = \Delta_{\dC_{TN}^2}(\nabla_{k} \tilde{v}_{\infty}) + \Delta_{\dC^{n-2}} (\nabla_{k} \tilde{v}_{\infty}). 
\end{equation}
Notice that $\tilde{g}_{\infty} = g_{TN} \oplus g_{\dC^{n-2}}$ is a product metric and $\nabla_k$ in effect acts on the Euclidean factor $\dC^{n-2}$, so $\nabla_k$ commutes with both $\Delta_{\dC_{TN}^2}$ and $\Delta_{\dC^{n-2}}$. Therefore,
\begin{eqnarray}
\Delta_{\tilde{g}_{\infty}} (\nabla_{k} \tilde{v}_{\infty}) = \nabla_k (\Delta_{\dC_{TN}^2} \tilde{v}_{\infty} + \Delta_{\dC^{n-2}}\tilde{v}_{\infty}) = 0.
\end{eqnarray}
The weighted bound implies the estimates  
\begin{align}
\begin{cases}
	|\nabla_k\tilde{v}_{\infty}(\bx)|\leq 1, &  d_{\tilde{g}_{\infty}}(\bx,\Sigma_0)\leq 1,
\\
	|\nabla_k\tilde{v}_{\infty}(\bx)|\leq d_{\tilde{g}_{\infty}}(\bx,\Sigma_0)^{-(\nu+1)}, & d_{\tilde{g}_{\infty}}(\bx,\Sigma_0)\geq 1.\label{e:D_k-decay}
	\end{cases}
\end{align}
Since we have assumed $\nu\in(-1,0)$, so
it is straightforward
$-(\nu+1)\in (-1,0)$.
The above implies that $|\nabla_k\tilde{v}_{\infty}| \leq 1$ on $\dC_{TN}^2\times\dC^{n-2}$.
Applying Cheng-Yau's gradient estimate to the harmonic function $\nabla_k\tilde{v}_{\infty}$ on the Ricci-flat manifold $\dC_{TN}^2\times\dC^{n-2}$, we conclude that
 $\nabla_k \tilde{v}_{\infty}$ is constant on $\dC_{TN}^2\times \dC^{n-2}$.
By \eqref{e:D_k-decay}, $\nabla_k\tilde{v}_{\infty}\equiv0$ for every $1\leq k\leq 2n-4$.
Therefore, $\tilde{v}_{\infty}$ is constant
on the Euclidean factor $\dC^{n-2}$.

By the above argument, the limiting function   $\tilde{v}_{\infty}$ can be viewed as a harmonic function on the Taub-NUT space $(\dC_{TN}^2,g_{TN})$. 
Now applying Bochner's formula, 
\begin{equation}
\frac{1}{2}\Delta_{g_{TN}}|\nabla_{g_{TN}}\tilde{v}_{\infty}|^2 = |\nabla_{g_{TN}}^2 \tilde{v}_{\infty}|^2 \geq 0.
\end{equation}
Since $\tilde{v}_{\infty}$
satisfies the weighted bound
\begin{equation}
\|\nabla_{g_{TN}}\tilde{v}_{\infty}\|_{C_{\delta,\nu+1,\mu}^{0}(\dC_{TN}^2)} + \|\nabla_{g_{TN}}^2\tilde{v}_{\infty}\|_{C_{\delta,\nu+2,\mu}^{0}(\dC_{TN}^2)}  =1	,
\end{equation}
so we have for any $\bx\in \dC_{TN}^2\setminus B_1(\bx_{\infty})$,
\begin{equation}
|\nabla_{g_{TN}}\tilde{v}_{\infty}(\bx)| \leq d_{g_{TN}}(\bx,\bx_{\infty})^{-(\nu+1)}.
\end{equation}
By assumption $\nu\in(-1,0)$, then
$|\nabla_{g_{TN}}\tilde{v}_{\infty}| \equiv 0 $ on $\dC_{TN}^2$ and hence $\tilde{v}_{\infty}$ is constant on $\dC_{TN}^2$. Notice that
$\tilde{v}_{\infty}(\bx_{\infty})=0$, so we conclude that 
$\tilde{v}_{\infty}(\bx_{\infty})\equiv 0$.

\vspace{0.5cm}

\begin{flushleft}
{\it  Case (2): $r(\bx_j)\cdot T_j\to \infty$ and $r(\bx_j)\to 0$ as $j\rightarrow\infty$.} \end{flushleft}
\vspace{0.5cm}

As in the proof of Proposition \ref{p:regularity-scale}, we have the pointed Gromov-Hausdorff convergence,
\begin{equation}
	(\M_j , \tilde{g}_{T_j} , \bx_j) \xrightarrow{GH} (\dR^3\times \dC^{n-2}, g_0, \bx_{\infty}),
\end{equation}
where the metric $g_0$ is the product Euclidean metric on $\dR^3\times\dC^{n-2}$. In this rescaled limit, the limiting reference point  $\bx_{\infty}$ satisfies $d_{g_0}(\bx_{\infty},\Sigma_{0^3})=1$ and $\Sigma_{0^3}\equiv \{0^3\}\times\dC^{n-2}\subset \dR^3\times \dC^{n-2}$ is the singular slice. Moreover, the convergence keeps curvature uniformly bounded away from the singular slice $\Sigma_{0^3}$. By passing to the local universal covers, in fact one can see that away from $\Sigma_{0^3}\subset \dR^3\times \dC^{n-2}$, 
  the functions $\tilde{v}_j$ converge to $\tilde{v}_{\infty}$ in the $C_{loc}^{2,\alpha'}$-norm for each $\alpha'\in(0,\alpha)$, such that the following properties hold,
\begin{enumerate}
\item $\|\nabla_{g_0}\tilde{v}_{\infty}\|_{C_{\delta,\nu+1,\mu}^0(\dR^3\times\dC^{n-2})} +\|\nabla_{g_0}^2\tilde{v}_{\infty}\|_{C_{\delta,\nu+2,\mu}^0(\dR^3\times\dC^{n-2})} = 1$,
\item $\tilde{v}_{\infty}(\bx_{\infty})=0$,

\item $\Delta_{g_0}\tilde{v}_{\infty} \equiv 0$ in $(\dR^3\times\dC^{n-2})\setminus\Sigma_{0^3}$,\end{enumerate}
where the limiting weight function is 
\begin{equation}
	\rho_{\infty, \delta, \nu, \mu}^{(k+\alpha)} (\bx)  = (d_{g_0} (\bx, \Sigma_{0^3}))^{\nu+k+\alpha},\ \bx\in\dR^3\times\dC^{n-2}.
	\end{equation}
Our goal is to show that $\tilde{v}_{\infty}\equiv 0$ on $\dR^3\times\dC^{n-2}$, which 
consists of the following ingredients:

First, we will prove that $\tilde{v}_{\infty}$ in fact globally harmonic in $\dR^3\times\dC^{n-2}$. To show the singular slice $\Sigma_{0^3}$ is removable, for each $q\in \Sigma_{0^3}$, we take a unit ball $B_1(q)\subset\Sigma_{0^3}$,  and for any $r\in(0,1)$, we choose the tubular neighborhood $T_r(B_1(q))\subset\dR^3\times \dC^{n-2}$. 
Notice that $\nabla_{g_0}\tilde{v}_{\infty}$ satisfies the uniform estimate 
\begin{equation}\|\nabla_{g_0}\tilde{v}_{\infty}\|_{C_{\delta,\nu+1,\mu}^{0}(\dR^3\times\dC^{n-2})}\leq 1,
\end{equation}
integrating the above weighted bound,  then for any $\bx\in T_r(B_1(q))\setminus B_1(q)$,
\begin{equation}
	|\tilde{v}_{\infty}(\bx)| \leq C\cdot d(\bx,B_1(q))^{-(\nu)}.
\end{equation}
By Lemma \ref{l:harmonic-removable-singularity}, $B_1(q)$ is a removable singular set in $T_r(B_1(q))$ and hence $\tilde{v}_{\infty}$ is harmonic in $T_r(B_1(q))$.

Next, we will show that $\tilde{v}_{\infty}$ is constant in $\dC^{n-2}$. It is straightforward that for each $1\leq k\leq 2n-4$, the partial derivative $\nabla_k\tilde{v}_{\infty}\equiv \frac{\partial}{\partial x_k''}\tilde{v}_{\infty}(\bx',\bx'')$ satisfies
\begin{equation}
\Delta_{g_0}(\nabla_k\tilde{v}_{\infty})=0 \ \text{in}\ \dR^3\times\dC^{n-2}.	
\end{equation}
The weighted condition implies that $\nabla_k\tilde{v}_{\infty}$ satisfies the uniform estimate,
\begin{equation}
|\nabla_k\tilde{v}_{\infty}| \leq d(\bx,\Sigma_{0^3})^{-(\nu+1)} \quad \text{for every}\  \bx\in \dR^3\times\dC^{n-2}.
\end{equation}
Since 
we have assumed $\nu\in(-1,0)$, 
Lemma \ref{l:liouville-product} implies that
$|\nabla_k\tilde{v}_{\infty}| \equiv 0$ on $\dR^3\times\dC^{n-2}$  and hence $\tilde{v}_{\infty}$ is constant in $\dC^{n-2}$.
Therefore, $\tilde{v}_{\infty}$ can be viewed as a harmonic function in the Euclidean space $(\dR^3,g_{\dR^3})$.
By assumption, $\tilde{v}_{\infty}$ satisfies
\begin{equation}
	|\tilde{v}_{\infty}(\bx)| \leq d_{g_{\dR^3}}(\bx, 0^3)^{-\nu}.
\end{equation}
Since $\nu\in(-1,0)$, 
applying the standard Liouville theorem for sublinear growth harmonic functions on a Euclidean space, 
we conclude that $\tilde{v}_{\infty}$
is a constant. The last step is to use the renormalization $\tilde{v}_{\infty}(\bx_{\infty})=0$, which gives $\tilde{v}_{\infty}\equiv 0$. So the proof of Case (2) is done.

\vspace{0.5cm}

\begin{flushleft}
{\it  Case (3): There is a constant $\underline{T}_0>0$ such that $r(\bx_j)\geq \underline{T}_0$ and $L_{T_j}(\bx_j)^{-1}\cdot T_j^{\frac{n-2}{n}}\to 0$.} \end{flushleft}

\vspace{0.5cm}

There are two situations to consider.

First, we consider the case $z(\bx_j)\to z_0$. By Proposition \ref{p:regularity-scale},
 the limit of $(\M_j,\tilde{g}_{T_j},\bx_j)$ is the cylinder
$(Q,g_c,\bx_{\infty})\equiv (D\times\dR, g_{D}\oplus dz^2, \bx_{\infty})$,
where $(D, g_{D})$
is a closed Calabi-Yau manifold.
The limiting function $\tilde{v}_{\infty}$ of $\tilde{v}_j$ satisfies 
\begin{enumerate}
\item $\|\nabla_{g_c}\tilde{v}_{\infty}\|_{C_{\delta,\nu+1,\mu}^0(Q)} +\|\nabla_{g_c}^2\tilde{v}_{\infty}\|_{C_{\delta,\nu+2,\mu}^0(Q)} = 1$,
\item $\tilde{v}_{\infty}(\bx_{\infty})=0$,

\item $\Delta_{g_c}\tilde{v}_{\infty} \equiv 0$  in $Q\setminus P$,
\end{enumerate}
where the limiting weight function is given by
$	\tilde{\rho}_{\infty,\delta,\nu,\mu}^{(k+\alpha)}(\bx)=	e^{-\frac{n}{2}\cdot \delta \cdot L_0(z)}\cdot \mathfrak{r}(\bx)^{\nu+k+\alpha}$ 	and  
	\begin{align}
	L_0(z)=
	\begin{cases}
	k_+z , & z>1,
	\\
	k_-z, & z<-1,
	\end{cases}
	\end{align}
 as defined in \eqref{e:l(z)}.
Similar to Case (2), first we need to extend the limiting function $\tilde{v}_{\infty}$ across 
the singular set $P$. 
Integrating $\nabla\tilde{v}_{\infty}$ around $P$, we have that $\tilde{v}_{\infty}$
satisfies the growth estimate
\begin{equation}
|\tilde{v}_{\infty}(\bx)| \leq C\cdot d_{g_c}(\bx,P)^{-\nu}.
\end{equation}
Since we have assumed $\nu\in(-1,0)$,
 so Lemma \ref{l:harmonic-removable-singularity} implies that the singular set $P$ is removable. Now we have obtained that
$\tilde{v}_{\infty}$ is harmonic on $Q$ and satisfies
\begin{equation}
|\tilde{v}_{\infty}(\bx)| \leq C e^{-\delta\cdot \frac{n}{2}\cdot(|k_+|+|k_-|)\cdot |z(\bx)|},
\end{equation}
for $|z(\bx)|\gg1$.
 By the choice of $\delta$, we conclude that
$\tilde{v}_{\infty}\equiv 0$ on $Q$.

The next situation is that the sequence $\bx_j$ satisfies $|z(\bx_j)|\to \infty$ and $L_{T_j}(\bx_j)^{-1}\cdot T_j^{\frac{n-2}{n}}\to 0$.   
We need to perform the coordinate change centered at the reference point $\bx_j$,
\begin{equation}
z(\bx) = z_j + \Big(\frac{T_j}{L_{T_j}(z_j)}\Big)^{\frac{n-2}{2}}w(\bx), \quad z_j\equiv z(\bx_j).
\end{equation}
Without loss of generality, we only consider the case $z_j \ll 0$. 
Recall the proof of Proposition \ref{p:regularity-scale}, the rescaled limit is $Q=D\times\dR$ with a product metric 
$g_Q = g_D + dw^2$. 
Moreover,  the limiting  weight function is
$\tilde{\rho}_{\infty,\delta,\nu,\mu}^{(k+\alpha)}(\bx) 
= e^{-\frac{  \delta \cdot n\cdot  k_-}{2} \cdot w(\bx)}$.
  Now the growth condition implies that the limiting function $\tilde{v}_{\infty}$ satisfies
\begin{align}
\begin{cases}
\Delta_Q \tilde{v}_{\infty}(\bx) = 0 , & \forall \bx \in Q, 
\\
|\tilde{v}_{\infty}(\bx)| \leq  e^{\frac{\delta\cdot n\cdot k_-}{2} \cdot w(\bx)}, & w\in\dR.
\end{cases}
\end{align}
By \eqref{e:fix-delta-neck}, we have
$\frac{\delta\cdot n\cdot k_-}{2} < \frac{\sqrt{\lambda_D}}{2}$.
Applying Lemma \ref{l:liouville-cylinder}, $ \tilde{v}_{\infty}(\bx) \equiv  0$ on $Q$.

Combining the above situations, the proof of Case (3) is complete.

\vspace{0.5cm}

\begin{flushleft}
{\it  Case (4): There is a constant $c_0>0$ such that $r(\bx_j)\geq \underline{T}_0$ and $L_{T_j}(\bx_j)^{-1}\cdot T_j^{\frac{n-2}{n}}\to c_0$.} \end{flushleft}

\vspace{0.5cm}

Without loss of generality we may assume $z(\bx_j)<0$. Then the reference points $\bx_j$ are close to the boundary $\{z=T_-\}\subset\p\M_j$ such that we can exploit the Neumann boundary condition.    Without loss of generality, we just assume $c_0=1$. We have shown in Section \ref{ss:regularity-scales} that the limit of the rescaled spaces $(\M_j, \tilde{g}_{T_j}, \bx_j)$
is the Calabi model space $(\mathcal{C}^n_- , g_{\mathcal{C}^n_-}, \bx_{\infty})$.
Let us perform the coordinate change  as follows \begin{equation}
k_-\cdot (z(\bx) - T_-) = T_j^{\frac{n-2}{n}}\cdot(w(\bx)-1).
\end{equation}
Moreover, the limiting weight function is 
$\tilde{\rho}_{\infty,\delta,\nu,\mu}^{(k+\alpha)}(\bx) = e^{-\delta\cdot  w^{\frac{n}{2}}}\cdot w^{\frac{\nu+k+\alpha}{2}}$. 
So $\tilde{v}_{\infty}$ satisfies
\begin{align}
\begin{cases}
\Delta_{g_{\mathcal{C}_-^n}} \tilde{v}_{\infty}(\bx) = 0, & \bx\in\mathcal{C}_-^n,
\\
|\tilde{v}_{\infty}(\bx)| \leq e^{\delta\cdot w^{\frac{n}{2}}}\cdot w^{-\frac{\nu}{2}} , & w(\bx)\gg1,
\\
\frac{\p \tilde{v}_{\infty}(\bx)}{\p w} = 0, & w(\bx)=1.
\end{cases}
\end{align}
Applying the renormalization condition $\tilde{v}_{\infty}(\bx_{\infty})=0$ and Lemma \ref{l:Liouville-Calabi-space-SZ}, we have $\tilde{v}_{\infty}\equiv 0 $ on $\Ca_-^n$.

Combining all the above,  we have completed the proof of the weighted estimates in Step 1. 

\begin{flushleft}
{\bf Step 2} (weighted $C^{2,\alpha}$-estimate):
\end{flushleft}

It remains to prove the uniform weighted $C^{2,\alpha}$-estimate
\begin{equation}
[u]_{C_{\delta,\nu,\mu}^{2,\alpha}(\M_T)}	\leq C \cdot \|\Delta u\|_{C_{\delta,\nu+2,\mu}^{0,\alpha}(\M_T)}\label{e:weighted-higher-order}\end{equation}
for some uniform constant $C>0$ independent of $T$. This will again be proved by contradiction. As before, suppose that there are have a sequence of parameters $T_j\to+\infty$ and a sequence of functions  $u_j$
 on $(\M_j, g_{T_j})$
which satisfy 
\begin{align}\label{e:u_j-contradicting-higher-order}
\begin{split}
\frac{\p u}{\p n}\Big|_{\p M_j} & = 0,
\\
[u_j]_{C_{\delta,\nu,\mu}^{2,\alpha}}(\bx_j)=[u_j]_{C_{\delta,\nu,\mu}^{2,\alpha}(\M_j)} &= 1,
\\
\|\Delta_{g_{T_j}} u_j\|_{C_{\delta,\nu+2,\mu}^{0,\alpha}(\M_j)} &\to 0,\quad j\to+\infty.
\end{split}
\end{align}
We normalize the functions $u_j$ by 
$v_j(\bm{x}) \equiv u_j(\bm{x}) - u_j(\bm{x}_j)$. Then \eqref{e:u_j-contradicting-higher-order} retains for $v_j$ and $\bx_j$. Applying \eqref{e:neck-C2-estimate} to $v_j$, we have
\begin{align}
\|\nabla_{g_{T_j}} v_j\|_{C_{\delta,\nu+1,\mu}^0(\M_j)} + \|\nabla_{g_{T_j}}^2 v_j\|_{C_{\delta,\nu+2,\mu}^0(\M_j)} \leq C \cdot \|\Delta_{g_{T_j}} v_j\|_{C_{\delta,\nu+2,\mu}^{0,\alpha}(\M_j)}, 
\end{align}
where $C>0$ is independent of $T_j$. So it follows that 
\begin{align}
\|\nabla_{g_{T_j}} v_j\|_{C_{\delta,\nu+1,\mu}^0(\M_j)} + \|\nabla_{g_{T_j}}^2 v_j\|_{C_{\delta,\nu+2,\mu}^0(\M_j)} \to 0 \label{e:weighted-C2-limit-to-0}
\end{align}
as $j\to +\infty$. Denote $\fs_j\equiv \fs(\bx_j)$, by straightforward computations, there is some uniform constant $C_0>0$ such that
$\|v_j\|_{C_{\delta,\nu,\mu}^0(B_{\fs_j}(\bx_j))} \leq C_0 \cdot \|\nabla v_j\|_{C_{\delta,\nu+1,\mu}^0(B_{\fs_j}(\bx_j))}$ and hence 
\begin{align}
\|v_j\|_{C_{\delta,\nu,\mu}^0(B_{\fs_j}(\bx_j))} \to 0 \ \text{as} \ j\to+\infty.\label{e:weighted-C0-limit-to-0}
\end{align}
By Proposition \ref{p:local-weighted-schauder}, 
$\|v_j\|_{C_{\delta,\nu,\mu}^{2,\alpha}(B_{\fs_j/10}(\bx_j))} \to 0$, which contradicts \eqref{e:u_j-contradicting-higher-order}.  
The proof of \eqref{e:weighted-higher-order}
is done.

Since we have proved \eqref{e:neck-C2-estimate} and \eqref{e:weighted-higher-order}, the proof of the Proposition is complete.
\end{proof}

Combining all the above estimates, we are ready to complete the proof of Theorem \ref{t:neck-CY-metric}.

\begin{proof}
[Proof of  Theorem \ref{t:neck-CY-metric}]

It suffices to verify that $\mathscr{F}$ satisfies all the conditions in Lemma  \ref{l:implicit-function}. Proposition \ref{p:estimate-L-neck} and Proposition \ref{p:nonlinear-neck} show that $\mathscr{F}$ satisfies 
item (1) and item (2a). In our context, $C_L>0$
and $C_N>0$ are uniform constants. $r_0>0$
can be chosen as any fixed constant in $(0,\frac{1}{2C_LC_N})$. To verify item (2b)
in Lemma \ref{l:implicit-function}, we just need to use \eqref{e:new error estimate}. In fact, we have assumed $\nu+\alpha<0$, then 
\begin{equation}
\|\mathscr{F}(\bo)\|_{\fS_2} \leq C\cdot T^{\nu+\alpha} \ll \frac{r_0}{4C_L},
\end{equation}
as $T$ is sufficiently large. This completes the proof. 
\end{proof}

\subsection{Geometric singularity and normalized limit measure}

\label{ss:renormalized-measure}

The goal of this subsection is to understand the measured Gromov-Hausdorff limits of the sequence of incomplete Calabi-Yau metrics $(\M_T, \omega_{T,CY})$ (scaled to fixed diameter) constructed in Theorem \ref{t:neck-CY-metric}. As can be easily seen, the results are parallel to the statements in Theorem \ref{t:main-theorem}. So we will not repeat the arguments in Section \ref{s:gluing}. 

To begin with, we recall the notion of \emph{measured Gromov-Hausdorff convergence}. We refer the readers to \cite{ChC1} for the general theory about this.
 \begin{definition}[Measured Gromov-Hausdorff convergence] Let  $(M_j^m,  g_j, p_j)$ be a sequence of Riemannian manifolds with $\Ric_{g_j}\geq -(m-1)$ such that
\begin{equation}
(M_j^m,  g_j, p_j) \xrightarrow{GH} (X_{\infty},d_{\infty},p_{\infty})
\end{equation}
 for some metric space $(X_{\infty},d_{\infty},p_{\infty})$. Then by passing to  a subsequence, the renormalized measures 
 \begin{equation}
 d\nmv_j \equiv \frac{\dvol_{g_j}}{\Vol_{g_j}(B_1(p_j))}\label{e:renormalized measure}
\end{equation}
converge to a Radon measure $d\nmv_{\infty}$ on $X_\infty$ which is called the {\it renormalized limit measure}. The Gromov-Hausdorff convergence together with the convergence of the renormalzied measures is called the measured Gromov-Hausdorff convergence.
\end{definition}

In the general context of collapsed sequences with Ricci curvature bounded from below,  $d\nmv_{\infty}$ behaves quite differently from the Hausdorff measures on $X_{\infty}$ induced by the limiting metric $d_{\infty}$. 
In our specific context, $d\nmv_{\infty}$ has an explicit form and it effectively reveals the geometric singularity information in the collapsing spaces.

In our context, we are interested in the measured Gromov-Hausdorff limits of $(\M_T, \omega_{T, CY},d\nmv_{T, CY})$, where $d\nmv_{T, CY}$ is the renormalized volume measure of the metric $\omega_{T, CY}$.  By the weighted estimate on the solution $\phi(T)$ given in \eqref{e:eqn524} we see that on the regularity scale the metrics $\omega_{T, CY}$ and $\omega_{T}$ are very close  in $C^{ \alpha}$-topology for $T$ large. Lemma \ref{l:weight-function-lower-bound-estimate} and our choice of parameters ensure that the weight function has a positive lower bound independent of $T$, so  we also know that for $T\gg1$, 
\begin{equation}
(1+O(T^{\nu+\alpha}))\omega_{T}\leq \omega_{T, CY}\leq (1+O(T^{\nu+\alpha}))\omega_T. 
\end{equation}
So on large scales the metrics are also close in $L^\infty$ topology. This allows us to effectively replace  $\omega_{T, CY}$ by $\omega_T$ in the computation of limit geometries. Then we are led to do the following explicit calculations.

\begin{flushleft}{\bf Gromov-Hausdorff limit:} 
\end{flushleft}
 By construction and direct calculation, one can see that $g_T$ in large scale is approximated by the $1$ dimensional metric tensor $T^{\frac{(2-n)(n+1)}{n}} L_T(z)^{n-1} dz^2$. In particular, the diameter is of order $T^{\frac{n+1}{n}}$. This suggests rescaling the metric $g_T$ by $T^{-\frac{2(n+1)}{n}}$ in order to obtain bounded diameter. Indeed, upon the change of variable $z=T\cdot \xi$, we see $T^{-\frac{2(n+1)}{n}}g_T$ converges to the one dimensional metric $(1+k_{\mp}\xi)^{n-1}d\xi^2$, $\xi\in [-k_-^{-1}, -k_+^{-1}]$ in the Gromov-Hausdorff sense. 

The above limit can be transformed into the standard metric on the unit interval $(\bI, dv^2)$ via a constant rescaling and the following coordinate change 
\begin{equation}
\begin{cases}
1+k_+(\frac{1}{k_-}-\frac{1}{k_+})v=(1+k_+\xi)^{\frac{n+1}{2}}, & \xi>0, \\
1+k_-(\frac{1}{k_-}-\frac{1}{k_+})v=(1+k_-\xi)^{\frac{n+1}{2}}, & \xi<0. 
\end{cases}
\end{equation}

\begin{flushleft}
{\bf Renormalized limit measure:} 
\end{flushleft}

Again we first calculate by definition 
\begin{equation}
\Vol_{\omega_T}(\M_{a\leq z\leq b})=\int_a^b \Big(  \frac{T^{2-n}}{(n-1)!} \int_D \tilde\omega(z)^{n-1}\Big)dz.
\end{equation}
Using the change of variable $z=T\cdot \xi$, we have that 
\begin{equation}
\Vol_{\omega_T}(\M_{a\leq \xi\leq b})=T^2 \int_a^b\Big(\int_D \frac{1}{(n-1)!}(1+k_\pm \xi)^{n-1}\omega_D^{n-1}\Big)d\xi.
\end{equation}
So up to constant, the renormalized limit measure has density function given by $(1+k_{\pm}\cdot \xi)^{n-1}d\xi$. Changing to the $v$-variable this becomes (again up to constant multiplication) 
\begin{equation}
d\underline\nu_\infty=
\begin{cases}
(\frac{v}{-k_+}+\frac{1}{k_--k_+})^{\frac{n-1}{n+1}}dv, & v\in [\frac{k_+}{k_--k_+}, 0],\\
(\frac{v}{-k_-}+\frac{1}{k_--k_+})^{\frac{n-1}{n+1}}dv, &  v\in [0, \frac{k_-}{k_--k_+}]. 
\end{cases}
\end{equation}

\begin{flushleft}
{\bf Diameter and volume:} 
\end{flushleft}

By straightforward computations, if the diameter of $\M_T$ is rescaled to be $1$, then the volume of $\M_T$ is collapsing at rate $T^{-2n}$. On the other hand, if the volume of $\M_T$ is rescaled to be $1$, then the diameter of $\M_T$ is of order $T$.

\begin{flushleft}
{\bf Fibration structure:}
\end{flushleft}

There is an obvious fibration of $\M_T$ over $[T_-, T_+]$ using the coordinate function $z$. Composing with above coordinate changes, we obtain a fibration 
\begin{equation}
\mathcal F_T: \M_T\rightarrow \bI,\quad  \bx\mapsto v(\bx). 
\end{equation}
Then for any $v\neq 0$, $\mathcal F_T^{-1}(v)$ is an $S^1$-bundle over $D$, whose first Chern class is given by $c_1(L_\pm)$ depending on the sign of $v$, and $\mathcal F_T^{-1}(0)$ is an singular $S^1$-fibration over $D$ with vanishing circles along $H$.

\begin{flushleft}
\textbf{Bubble classification:}
\end{flushleft}

By Proposition \ref{p:regularity-scale}, it is clear that suitable rescalings around the vanishing circles in $\mathcal{F}_T^{-1}(0)$ converge to $\dC_{TN}^2\times \C^{n-2}$. Also suitable rescalings around the ends $z=T_\pm$ gives the Calabi model spaces. 

\

We close this section by giving the following remarks regarding the renormalized limit measure.

\begin{remark}
It can be seen from the above formulae that the limiting density function $\mathscr{V}_{\infty}=\frac{d\underline\nu_\infty}{dv}$ is a   Lipschitz function on $\bI$ and it is smooth everywhere in the interior of $\bI$ except at $v=0$. On the other hand, the singular fiber of $\mathcal{F}$ precisely appears at $t=0$.
So in our context, the singularity of the renormalized limit measure $d\nmv_{\infty}$ effectively characterizes the singularity behavior of the collapsing geometry.

\end{remark}

\begin{remark}
By Cheeger-Colding (see \cite{ChC3}, theorem 4.6),  in the regular set $\mathcal{R}$ of a general Ricci-limit space, the density function $\mathscr{V}_{\infty}$ of the renormalized limit measure always exists and is H\"older continuous. Our example tells us that, in general, one cannot expect the regularity of $\mathscr{V}_{\infty}$ to be  differentiable in $\mathcal{R}$ (even though $\mathcal{R}$ is  a smooth Riemannian manifold). We thank Shouhei Honda for pointing this out. 
\end{remark}

\begin{remark}
If we use rescale the metrics further around the point $v=0$
such that the sequence of spaces collapse to the complete real line $\dR$, 
then $d\nmv_{\infty}$ coincides with the standard Lebesgue measure on $\dR$. This fact can be quickly seen by scaling-up the coordinate $v$.
This is compatible with the general theory of Ricci-limit spaces. That is, due to Cheeger-Colding, the renormalized limit measure always splits off the Lebesgue measure of $\dR$ if the limit space isometrically splits off $\dR$ (see proposition 1.35 in \cite{ChC1}).  

\end{remark}

\section{Proof of the main theorem}

\label{s:gluing}
The goal of this section is to prove Theorem \ref{t:main-theorem}. We will work with the special family of Calabi-Yau varieties $\mathcal X$ defined in the Introduction.  In Section \ref{ss:algebraic-geometry} we show how to modify the family $\mathcal X$ to a new family $\hX$ such that the new central fiber consists of a chain of three components, with the middle component given by the compactification of the space $\mathcal N^0$ defined in Section \ref{ss:complex-geometry}. Notice in Section \ref{s:neck} a family of neck metrics are constructed on an exhausting family of domains in $\mathcal N^0$. In Section \ref{ss:tian-yau} we review general facts about the Tian-Yau metrics on the complement of a smooth anti-canonical divisor in a Fano manifold. These give Ricci-flat K\"ahler metrics on the other two components of the central fiber in $\hX$.  In Section \ref{ss:glued-metrics} we explain how to graft the above neck metrics and Tian-Yau  metrics on the central fiber of $\hX$ to the nearby smooth fibers, and obtain approximately Calabi-Yau  metrics in a suitable sense. In Section \ref{ss:global analysis}  we finish the proof of Theorem \ref{t:main-theorem}. The arguments are very similar to those in Section \ref{ss:incomplete weighted analysis} and \ref{ss:renormalized-measure}, so we will not provide full details for the final perturbation arguments.

\subsection{Algebro-geometric aspect}

\label{ss:algebraic-geometry}
\subsubsection{Poincar\'e residue}
 We first recall some general facts about \emph{Poincar\'e residues}. Given a smooth divisor $Z$ in a complex manifold $M$ of dimension $m$,  the Poincar\'e residue map
\begin{equation}\Res: H^0(M, K_M\otimes [Z])\rightarrow H^0(Z, K_Z)
\end{equation}
can be defined as follows. Given a holomorphic $m$ form $\Omega$ on $M$ with a simple pole along $Z$, locally if we choose a defining function $h$ of $Z$, then $h\Omega$ is a holomorphic $m$ form, and we can write 
$
h\Omega=dh\wedge \tilde\Omega
$
for some locally defined holomorphic $m-1$ form $\tilde\Omega$. 
The Poincar\'e residue of $\Omega$ along $Z$ is given by
 $\Res(\Omega)\equiv \tilde\Omega|_Z.$
 It is straightforward to check that this definition does not depend on the choice of $h$ and $\tilde\Omega$, and gives rise to a well-defined holomorphic $m-1$ form $\Omega_Z$ globally on $Z$. 
If we choose local holomorphic coordinates $z_1, \cdots, z_m$ on $M$, then we may write 
\begin{equation}\Omega=\frac{g}{h} dz_1\wedge \cdots dz_m, \end{equation}
where $g$ is holomorphic.
At a point  on $Z$ where $\frac{\p h}{\p z_1}\neq 0$, we have then by definition
\begin{equation}\Res(\Omega)=\frac{g}{\frac{\p h}{\p z_1}}dz_2\wedge \cdots \wedge dz_m. 
\end{equation}

From the local expression one can see that if $Z$ is an anti-canonical divisor in $M$,  and we pick a holomorphic volume form $\Omega_M$ on $M\setminus Z$ with a simple pole along $Z$, and then $\Res(\Omega_M)$ gives a holomorphic volume form $\Omega_Z$ on $Z$.

A special case arises when we have a globally defined holomorphic function $h: M\rightarrow \C$, and we are given a holomorphic volume form $\Omega$ on $M$, then for each $w\in \C$, we can apply the above construction to the meromorphic form $(h-w)^{-1}\Omega$. In this way we obtain a nowhere vanishing section  $\Omega'$ of the relative canonical bundle $K_M\otimes (h^*K_{\C})^{-1}$ on the set where $h$ is a submersion, and it satisfies the equation 
$dh\wedge \Omega'=\Omega. $
We may also view $\Omega'$ as a holomorphic varying family of holomorphic volume forms on the fibers of $h$.

\subsubsection{A model partial resolution of singularities}

Let $\mathcal S$ be a two dimensional $A_{k-1}(k\geq 2)$ singularity, which is a hypersurface in $\C^3$ with defining equation
$
z_1z_2+z_3^k=0.
$
Given two positive integers $a_1\geq a_2$ with $a_1+a_2=k$, we can define a \emph{partial} resolution of $\mathcal S$ as follows. Let  $\overline{\mathcal S}$ be the subvariety in the product space $\C^3\times \C\P^2$ cut out by the following system of equations
\begin{equation}
\begin{cases}
z_3^{a_1}u_1=z_1u_3; \\
z_3^{a_2}u_2=z_2u_3; \\
u_1u_2+u_3^2=0;\\
z_3^{a_1-a_2}u_1z_2=u_2z_1; \\
z_3^{a_2}u_3+u_1z_2=0. 
\end{cases}
\end{equation}
where $[u_1:u_2: u_3]$ denotes homogeneous coordinates on $\C\P^2$. Alternatively, $\overline{\mathcal S}$ can also be described as the closure in $\C^3\times\C\P^2$ of the graph of the rational map \begin{align}\mathcal S\rightarrow \C\P^2; (z_1, z_2, z_3)\mapsto \left[\frac{z_1}{z_3^{a_1}}: \frac{z_2}{z_3^{a_2}}:1\right].\end{align} On the affine chart $\{u_i\neq 0\}$ we denote by $v_j=u_j/u_i (j\neq i)$ the affine coordinates. 

\begin{lemma}
$\overline{\mathcal S}$ has at most two  singularities, which are of type $A_{a_1-1}$ and $A_{a_2-1}$ respectively, and the projection map $\overline{\mathcal S}\rightarrow \mathcal S$ is a partial resolution, with exceptional divisor isomorphic to $\C\P^1$. 
\end{lemma}
\begin{proof}
We first show that the system of equations implies $z_1z_2+z_3^k=0$, so that $\overline{\mathcal S}$ does project to $\mathcal S$. To see this, we notice the first three equations imply 
\begin{equation}
u_3^2(z_1z_2+z_3^k)=0. 
\end{equation}
If $u_3\neq 0$, then we get $z_1z_2+z_3^k=0$. If $u_3=0$, then by the third equation we get that either $u_1\neq 0, u_2=0$ or $u_1=0, u_2\neq 0$. In the first case using the remaining equations we get $z_3=z_2=0$. In the second case we get $z_3=z_1=0$. In both cases the equation $z_1z_2+z_3^k=0$ is indeed satisfied. 

Now we study singularities of $\overline{\mathcal S}$. In the affine chart $\{u_1\neq 0\}$, we get
\begin{equation}
\begin{cases}
v_2+v_3^2=0; \\
z_2+z_3^{a_2}v_3=0,
\end{cases}
\end{equation}
so we reduce the defining equations to a single equation in the $z_1, z_3, v_3$ variable given by 
\begin{equation}
z_3^{a_1}=z_1v_3.
\end{equation}
This has exactly one $A_{a_1-1}$ singularity at $\{z_1=z_3=v_3=0\}$. Notice by convention an $A_0$ singularity is  a smooth point. Similarly, on the affine chart $\{u_2\neq 0\}$ we reduce  the equations to  
\begin{equation}
z_3^{a_2}=z_2v_3. 
\end{equation}
This has exactly one $A_{a_2-1}$ singularity at $\{z_2=z_3=v_3=0\}$. 
On the affine chart $\{u_3\neq 0\}$, we reduce the equations to
$v_1v_2+1=0$ which is smooth. 

It is then easy to verify that the projection map $\overline{\mathcal S}\rightarrow \mathcal S$ is an isomorphism outside the point $\{z_1=z_2=z_3=0\}$, and if $z_1=z_2=z_3=0$, we  get the equation 
$u_1u_2+u_3^2=0$,
which gives a conic in $\C\P^2$. 
\end{proof}

From another point of view, we can think of both $\mathcal S$ and $\overline{\mathcal S}$ as  families of algebraic curves by projecting to the $z_3$ variable. For $\mathcal S$ this is simply the standard nodal degeneration of conics in $\C^2$, modified by a base change. The family corresponding to $\overline{\mathcal S}$ is isomorphic to $\mathcal S$ over any general fiber $\{z_3\neq 0\}$, and the special fiber of $\overline{\mathcal S}$ is now given by a chain consisting of three components, two of which are given by the proper transforms of the two lines $\{z_1=0\}$ and $\{z_2=0\}$ in $\C^2$, and the middle component is the conic $\{u_1u_2+u_3^2=0\}$ in $\C\P^2$. In the special case when $a_1=a_2=1$, $\overline{\mathcal S}$ is smooth and the projection map is precisely the minimal resolution of singularity. 

It is well-known that $\mathcal S$ has a canonical singularity, meaning that the canonical line bundle $K_{\mathcal S}$ is trivial. An explicit holomorphic volume form $\Omega_{\mathcal S}$ can be written by applying the Poincar\'e residue to the standard meromorphic form $\frac{1}{z_1z_2+z_3^k}dz_1\wedge dz_2\wedge dz_3$ on $\C^3$. In the chart $\{z_1\neq 0\}$, it is given by 
\begin{equation}
\Omega_{\mathcal S}=\frac{dz_2\wedge dz_3}{z_2}.  
\end{equation}
Notice $\mathcal S$ is isomorphic to the quotient $\C^2/\mathbb Z_k$, and $\Omega_{\mathcal S}$ pulls-back to a multiple of the standard holomorphic volume form on $\C^2$. 
 
Viewing $\mathcal S$ as fibered over $z_3\in\mathbb C$, we further get a relative holomorphic volume form 
\begin{equation}
\Omega'=-\frac{dz_2}{z_2}=\frac{dz_1}{z_1}.        
\end{equation}
One can see $\Omega'$ is smooth away from the singularity $\{z_1=z_2=z_3=0\}$, and on each component of the singular fiber it is a meromorphic 1-form with a simple pole along the singularity.

The partial resolution $\overline{\mathcal S}$ is a \emph{crepant} resolution, i.e., the canonical line bundle $K_{\overline{\mathcal S}}$ is also trivial. Indeed the pull-back $\Omega_{\overline{\mathcal S}}$ of $\Omega_{\mathcal S}$ is nowhere vanishing on $\overline{\mathcal S}$, and by applying the Poincar\'e residue to the function $z_3$,  we then get a meromorphic 1-form on each component of the special fiber. On the conic $\{u_1u_2+u_3^2=0\}$ the meromorphic 1-form  is given by $v_1^{-1}dv_1=-v_2^{-1}dv_2$. The upshot is that we still get a meromorphic section of the relative canonical bundle, which is smooth away from the two singularities  $\{u_1=u_3=z_1=z_2=z_3=0\}$ and $\{u_2=u_3=z_1=z_2=z_3=0\}$ of $\overline S$.

\subsubsection{A modification of the degenerating family}
 We now recall the setup in the introduction. We adopt the notations there, and let $p:\mathcal \mathcal X \rightarrow\Delta$ be the family defined by \eqref{eqn1.1}. We also assume the genericity assumptions (i)-(iv) in the Section \ref{ss:1.1} hold.  The total space $\mathcal X$ is singular along $H\times\{0\}$ and  transverse to $H\times \{0\}$ the singularities are locally modeled on a two dimensional ordinary double point. For our purpose we need to perform certain birational transformations to $\mathcal X$.  We first do a base change $t\mapsto t^{n+2}$, and work on the new family, which we still denote by $\mathcal X$. Then 
 $\mathcal X$ now has singularities along $D\times \{0\}$, transversal to which generically it is a two dimensional $A_{d-1}$ singularity, which becomes worse along $H\times \{0\}$. This is usually referred to as a compounded Du Val (cDV) singularity . 

Now we apply the family version of  the above model partial resolution to $\mathcal X$. Let $\hX$ be  the subvariety  in the projective bundle $\P(\cO(d_2)\oplus \cO(d_1)\oplus \C)$ over $\C\P^{n+1}\times \Delta$ cut out by  the equations
\begin{equation}
\begin{cases}
t^{d_1}s_1=s_3f_2(x);\\
t^{d_2}s_2=s_3f_1(x); \\
s_1\otimes s_2+s_3^2 f(x)=0;\\
t^{d_1-d_2}s_1\otimes f_1(x)=f_2(x)\otimes s_2; \\
t^{d_2}s_3f(x)+s_1\otimes f_1(x)=0.
\end{cases}
\end{equation}
where naturally we view $f_i\in H^0(\C\P^{n+1}, \cO(d_i))$,  $f\in H^0(\C\P^{n+1}, \cO(n+2))$, and  $[s_1: s_2: s_3]$ denotes a point in the fiber of the projective bundle over the point $(x, t)\in \C\P^{n+1}\times \Delta$.  

For our discussion in the rest of this section, we will always take $[x_0:x_1:\cdots: x_{n+1}]$ to be the homogeneous coordinates of a point $x$ on $\C\P^{n+1}$. On the affine chart $\{x_i\neq 0\}$ of $\C\P^{n+1}$ we denote by $u=\{u_j=x_j/x_i, j\neq i\}$ the affine coordinates, and we view $x_i$ as a local trivialization of $\cO(1)$. Then on this chart we can view  holomorphic sections of powers of $\cO(1)$ as local holomorphic functions. In particular, for a homogeneous function $R(x)$, we denote by $R(u)$ the corresponding inhomogeneous function. On the 
affine trivialization of the projective bundle $\{s_i\neq 0\}$, we denote by $\{\zeta_j=s_j/s_i, j\neq i\}$ the affine coordinates on the fibers. 
We also define
\begin{align}
D_1 & \equiv \{f_1(x)=f_2(x)=t=0, s_2=s_3=0\},
\\
D_2 & \equiv \{f_1(x)=f_2(x)=t=0, s_1=s_3=0\}.
\end{align}

\begin{lemma}
$\hX$ is smooth away from the union $D_1\cup D_2$, and transverse to each $D_i$ the singularity is a two dimensional $A_{d_i-1}$ singularity.

\end{lemma}
\begin{proof}
We know $\hX$ is isomorphic to $\mathcal X$ away from $D\times \{0\}$, so it suffices to consider around a point $(x, 0)$ where $f_1(x)=f_2(x)=0$. Locally in an affine chart $\{s_1\neq 0\}$, $\hX$ is then cut out by the equations
\begin{equation}\label{e: s1 nonzero region all equations}\begin{cases}
f_2(u)\zeta_3=t^{d_1};\\
f_1(u)\zeta_3=t^{d_2}\zeta_2;\\
 \zeta_2+\zeta_3^2f(u)=0;\\
 f_2(u)\zeta_2=t^{d_1-d_2}f_1(u);\\
 t^{d_2}\zeta_3 f(u)+f_1(u)=0.
 \end{cases}\end{equation}
These can be reduced to  two equations on the coordinates $u$, $t$ and $\zeta_3$, given by 
\begin{equation} \label{e: s1 nonzero region}
\begin{cases}
f_2(u)\zeta_3-t^{d_1}=0\\
t^{d_2}\zeta_3 f(u)+f_1(u)=0. 
\end{cases}
\end{equation}
By our assumption (iii) locally we may use $v_1=f_1(u)$ and $v_2=f_2(u)$ to replace $u_1, u_2$ (for example) as local holomorphic coordinates on a neighborhood of $x$ in $\C\P^{n+1}$. Then it is easy to see the corresponding subvariety is smooth if $\zeta_3\neq 0$, and has transversal $A_{d_1-1}$ singularities along $D_1$. 
So this gives the local description of $\hX$ in a neighborhood of $D_1$. Similarly on $\{s_2\neq 0\}$ we also know the space is smooth except with transversal $A_{d_2-1}$ singularities along $D_2$. 

On $\{s_3\neq 0\}$, we use $u, t, \zeta_1, \zeta_2$ as coordinates, and we get the constraint equations 
\begin{equation}\label{e: s3 nonzero region}
\begin{cases}
\zeta_1\zeta_2+f(u)=0,\\ 
f_2(u)-t^{d_1}\zeta_1=0,\\ 
f_1(u)-t^{d_2}\zeta_2=0.\end{cases}
\end{equation}
We only need to consider the points where $\zeta_1=\zeta_2=t=0$, so in particular we also have $f(u)=0$. At such a point, the differentials of these three equations are $(\nabla f(u), \nabla f_2(u), \nabla f_1(u))$. This is non-zero by our assumption (iv).
\end{proof}

One can  see that the new central fiber $\hat{X}_0$ consists of a chain of three smooth  components intersecting transversally, given by the proper transforms $\hat Y_1, \hat Y_2$ of $Y_1, Y_2$ respectively and the submanifold $\mathcal N$ in the projective bundle $\P(L_1\oplus L_2\oplus \C)$ over $D$ cut out by the equation $s_1s_2=s_3^2 f(x)$ (so that $\mathcal N$ is a quadric bundle over $D$, and singular fibers are over $H$). Notice $\mathcal N$ itself is a smooth manifold. 

\begin{figure}
\begin{tikzpicture}
\draw (-5, 3) to [out=-75, in=75] (-5, -3); 
\draw (-5, 3) to (-3.5, 4); 
\draw (-5, -3) to  (-3.5, -2); 
\draw (-3.5, 4) to [out=-75, in=75]  (-3.5, -2); 
\draw (-3, 3) to [out=-75, in=75]  (-3, -3); 
\draw (-3, 3) to (-1.5, 4); 
\draw (-3, -3) to (-1.5, -2); 
\draw (-1.5, 4) to [out=-75, in=75]  (-1.5, -2); 
\draw (3.0, 3) to [out=-75, in=75] (3.0, -3); 
\draw (3.0, 3) to (4.5, 4); 
\draw (3, -3) to  (4.5, -2); 
\draw (4.5, 4) to [out=-75, in=75]  (4.5, -2); 
\draw (1.0, 3) to [out=-75, in=75]  (1.0, -3); 
\draw (1.0, 3) to (2.5, 4); 
\draw (1.0, -3) to (2.5, -2); 
\draw (2.5, 4) to [out=-75, in=75]  (2.5, -2); 
\draw (-1, 3) to (0.5, 4); 
\draw (-1, 3) to [out=-65, in=60] (-0.7, 0.5);
\draw[red, very thick] (-0.7, 0.5) to (0.8, 1.5); 
\draw (-0.7, -1.5) to [out=-60, in=65] (-1, -3);
\draw[red, very thick] (-0.7, -1.5) to (0.8, -0.5);
\draw (-1, -3) to (0.5, -2); 
\draw (0.5, 4) to  [out=-65, in=60]  (0.8, 1.5);
\draw  (0.8, -0.5) to [out=-60, in=65] (0.5, -2);
\draw[thick] (-0.7, 0.5) to [out=-65, in=65] (-0.7, -1.5);
\draw[thick] (0.8, 1.5) to [out=-65, in=65] (0.8, -0.5);
\draw[thick] (-0.5, 0.66) to [out=-65, in=65] (-0.5, -1.34);
\draw[thick] (0.4, 1.26) to [out=-65, in=65] (0.4, -0.74);
\draw[very thick] (-0.2, 0.8) to  (0.05, 0); 
\draw[very thick] (0.05, 0) to  (-0.2, -1.19); 

\node[blue] at (0.05, 0) {$\bullet$};
\node[blue] at (-1.75, 0.1) {$\bullet$};
\node[blue] at (-3.75, 0.4) {$\bullet$};
\node[blue] at (2.25, 0.2) {$\bullet$};
\node[blue] at (4.25, 0.6) {$\bullet$};

\node at (-4, 2.5) {$\widehat{X}_t$};
\node at (0, 2.5) {$\hat Y_1$}; 
\node at (0, -1.85) {$\hat Y_2$}; 
\node[red] at (0.3, 1.6) {$D_1$};
\node[red] at (0.3, -1.2) {$D_2$};
\node at (0.3, 0.5) {$\mathcal N$};  
\node[blue] at (-6, 1.4) {$H\times\{t\}$};
\draw[->] (-5.2, 1.3) to (-3.9, 0.5);
\node at (0, -4.5) {$\widehat{X}_0=\hat Y_1\cup_{D_1}\mathcal N\cup_{D_2} \hat Y_2$};
\draw[->] (0, -4) to (0, -2.5);
\draw[thick, blue, dashed] plot[smooth] coordinates {(-3.75, 0.4) (-1.75, 0.1) (0.05, 0) (2.25, 0.2) (4.25, 0.6)};
\end{tikzpicture}
\caption{The modified family $\hX$}
 \label{f: the modified family}
\end{figure}

Then we have
\begin{equation}D_1=\hat Y_1\cap \mathcal N, \ \  D_2=\hat Y_2\cap \mathcal N.\end{equation}
It is straightforward to see that the normal bundle of $D_i$ in $\mathcal N$ is  $L_i^{-1}$.

Next we consider holomorphic volume forms. 
Viewing $\mathcal X$ as an anti-canonical divisor in $\C\P^{n+1}\times \Delta$, then away from $D\times \{0\}$, $\mathcal X$ is smooth and we then obtain a holomorphic volume form $\Gamma$. In the affine chart $\{x_0\neq 0\}\times \Delta\subset \C\P^{n+1}\times \Delta$,  the meromorphic volume form  is given by 
\begin{equation}\frac{1}{F_t(u)} dt\wedge du_1\wedge \cdots \wedge du_{n+1}.\end{equation}
So the Poincar\'e residue on $\mathcal X$ is 
\begin{equation}\Gamma=-\frac{1}{(n+2)t^{n+1}f(u)} du_1\wedge \cdots du_{n+1}. \end{equation}
Using the equation and the genericity assumptions, $\Gamma$ is indeed holomorphic on $\mathcal X\setminus D\times \{0\}$. 

Now applying the above discussion to the global function $t$ on $\mathcal X$, then we get a holomorphic family of holomorphic volume forms $\Gamma_t$ on each $\widehat{X}_t$. 
Differentiating the equation $F_t(u)=f_1(u)f_2(u)+t^{n+2}f(u)=0$, we get
$(n+2)t^{n+1}f(u)dt+d_uF_t=0.$
In the above affine chart,  on the set where $\frac{\p F_t}{\p u_1}\neq 0$, we have
\begin{equation}
\Gamma_t=\frac{1}{\frac{\p F_t(u)}{\p u_1}} du_2\wedge \cdots du_{n+1}.
\end{equation}
This is indeed well-defined on $\widehat{X}_t$ for $t\neq 0$ and also on $X_0\setminus D$. On each component $Y_i$ of $X_0$, it has a simple pole along $D$.  Notice $\Gamma_t$ is also the natural holomorphic volume form on $\widehat{X}_t$ when we apply the Poincar\'e residue to the divisor $\widehat{X}_t$ in $\C\P^{n+1}$.

Now we pass to the resolution $\hX$. Abusing notation we still denote by $\Gamma$ its pull-back. 

\begin{lemma}
$\Gamma$ extends to a global holomorphic volume form on $\hX\setminus (D_1\cup D_2)$.
\end{lemma}

\begin{proof}
We only need to consider around a point $(x, t, s)$ on the exceptional set $\mathcal N$, so $(x, t)\in D\times \{0\}$.  Without loss of generality may assume $x_0\neq 0$. Since $D$ is a complete intersection by assumption (iii),  we may use $v_1=f_1(u)$ and $v_2=f_2(u)$ to replace $u_1, u_2$ (say) as local holomorphic coordinates on a neighborhood of $x$ in $\C\P^{n+1}$.  So we can write 
\begin{equation}\Gamma=-\frac{J^{-1}}{(n+2)t^{n+1}f(u)}dv_1\wedge dv_2\wedge du_3\cdots\wedge du_{n+1}, \end{equation}
 where $J=\frac{\p f_1}{\p u_1} \frac{\p f_2}{\p u_2} -\frac{\p f_1}{\p u_2}\frac{\p f_2}{\p u_1}.$ is the Jacobian.
Suppose first we work on the affine chart $\{s_1\neq 0\}$. Then we get the local equations for $\hX$ given by \eqref{e: s1 nonzero region}. Since we are away from $D_1$, we must have $\zeta_3\neq 0$. Then we can use  $\zeta_3, t, u_3, \cdots, u_{n+1}$ as local holomorphic coordinates on $\hX$. We have
  \begin{align}dv_1 &=-t^{d_2}fd\zeta_3-d_2t^{d_2-1}\zeta_3 f dt-t^{n+2}\zeta_3 df,
  \\
dv_2&=d_1t^{d_1-1} \zeta_3^{-1}dt-\zeta_3^{-2} t^{d_1}d\zeta_3,
\\
df&=\frac{\p f}{\p v_1} dv_1+\frac{\p f}{\p v_2} dv_2+\sum_{j\geq 3} \frac{\p f}{\p u_j} du_j. \end{align}
 So we get
 \begin{align}&(1+t^{d_2} \zeta_3\frac{\p f}{\p v_1}) dv_1
 \nonumber\\
 =&(-t^{d_2}f+t^{n+2} \zeta_3^{-1} \frac{\p f}{\p v_2})d\zeta_3-(d_2t^{d_2-1}\zeta_3 f+d_1t^{n+1}\frac{\p f}{\p v_2})dt & \mod (du_3, \cdots ,du_{n+1}).\end{align}
 Hence we get
\begin{equation} \label{eqn8-3}
\Gamma=\frac{\zeta_3^{-1}}{(1+t^{d_2}\zeta_3\frac{\p f}{\p v_1})} J^{-1}d\zeta_3\wedge dt\wedge du_3\wedge\cdots\wedge du_{n+1}.\end{equation}
Near $t=0$  we see $\Gamma$ is smooth around such a point.  Similarly we can deal with the chart $\{s_2\neq 0\}$.
 
  Now on $\{s_3\neq 0\}$, we only need to  consider a point on $D$ where $f=0$, then by our assumption (iv) we may use $v_3=f$ as a local holomorphic coordinate to replace $u_3$ for instance.  Then  we can write
\begin{equation}\Gamma=-\frac{1}{(n+2)t^{n+1}f} K^{-1} dv_1\wedge dv_2\wedge dv_3\wedge du_4\cdots \wedge du_{n+1},\end{equation}
where $K$ is the Jacobian for the change of coordinates. We have
\begin{equation}dv_3=-(\zeta_1d\zeta_2+\zeta_2d\zeta_1),\end{equation}
\begin{equation}dv_1=t^{d_2}d\zeta_2+d_2\zeta_2 t^{d_2-1} dt,\end{equation}
\begin{equation}dv_2=t^{d_1}d\zeta_1+d_1\zeta_1 t^{d_1-1}dt.\end{equation}
Then we get 
\begin{equation} \label{eqn8-4}
\Gamma=K^{-1}dt\wedge d\zeta_1\wedge d\zeta_2\wedge du_4\cdots \wedge du_{n+1}
,\end{equation}
which is smooth. 
\end{proof}

We can apply the previous Poincar\'e residue to the function $t$ on $\hX$. Since the exceptional set of the resolution lies over $D\times \{0\}$, we still get $\Gamma_t$ for $t\neq 0$. On the central fiber $\hat X_0$, we still get $\Gamma_0$ on $\hat Y_1\setminus D_1$ and $\hat Y_2\setminus D_2$.  Over $\mathcal N\setminus (D_1\cup D_2)$, using  (\ref{eqn8-3}) and (\ref{eqn8-4}) we get the corresponding Poincar\'e residue
\begin{equation}\Gamma_{\mathcal N}=J^{-1}\frac{d\zeta_1}{\zeta_1}\wedge du_3\wedge \cdots\wedge du_{n+1}=-J^{-1}\frac{d\zeta_2}{\zeta_2}\wedge du_3\wedge \cdots\wedge du_{n+1}. 
\end{equation}

Notice by applying Poincar\'e residue twice to the complete intersection $D=\{f_1=f_2=0\}$, we obtain a holomorphic volume form $\Omega_D$ on $D$, which in the above local coordinates can be written as
\begin{equation}\Omega_D=J^{-1}du_3\wedge \cdots\wedge du_{n+1}.\end{equation}
So we have that  
$\Gamma_{\mathcal N}=\frac{d\zeta_1}{\zeta_1} \wedge \Omega_D$. This means that up to multiplying by $-\sq$, $\Gamma_{\mathcal N}$ agrees with the natural holomorphic volume  form $\Omega_0$ on $\mathcal N_0$ defined in Section \ref{ss:complex-geometry}, under the identification $k_-=d_2, k_+=-d_1$.

\begin{remark}
A priori there could be various different choices of birational transformations extracting an extra component as above. For example, one can directly perform a blow-up on the original family $\mathcal X$ (without doing the base change $t\mapsto t^{n+2}$).  The reason for our particular choice of birational transform is related to a matching condition required when we glue the Calabi-Yau metrics on the three components on the central fiber $\hat{X}_0$ to the nearby smooth fiber $X_t$. See \eqref{eqn6-70} and \eqref{eqn6-71}. 
\end{remark}

\subsection{Tian-Yau metrics}

\label{ss:tian-yau}
In this subsection we briefly review the complete Ricci-flat K\"ahler metrics, constructed in \cite{TY} on the complement of a smooth anti-canonical divisor in a Fano manifold. We will state some facts on the asymptotics of these metrics. Interested readers are referred to \cite{HSVZ} (Section 3) for the proof and more details.

Let $Y$ be an $n$ dimensional Fano manifold, $D$ a smooth anti-canonical divisor in $Y$, and denote $Z=Y\setminus D$.  By adjunction formula $D$ itself is Calabi-Yau, and we can find a Ricci-flat K\"ahler metric $\omega_D\in 2\pi c_1(L_D)$, where $L_D$ is the restriction of $K_Y^{-1}$ to $D$. 
Fixing a defining section $S$ of $D$, we can view $S^{-1}$ as a holomorphic $n$-form $\Omega_{Z}$ on $Z$ with a simple pole along $D$. Rescaling suitably we may assume that the Poincar\'e residue of $\Omega_{Z}$ gives a holomorphic volume form $\Omega_D$ on $D$ which satisfies the normalization condition \eqref{e:CY equation on D}.

As before we can fix the hermitian metric $|\cdot|$ on $L_D$ whose curvature form is $-\sq\omega_D$  and we also fix a smooth extension to $Y$ with strictly positive curvature. 
Then
 \begin{equation}\omega_{Z} \equiv\frac{n}{n+1}\sq \p\bp (-{\log |S|^2})^{\frac{n+1}{n}}  \end{equation}
defines a K\"ahler form on a neighborhood of infinity in $Z$. The Tian-Yau metric $\omega_{TY}$ on $Z$ is then obtained by solving a Monge-Amp\`ere equation with reference metric $\omega_{Z}$. Let $\mathcal{C}^n$ be the Calabi model space constructed using $(D, L_D, \omega_D)$, as in Section \ref{ss:Calabi model space}.

\begin{proposition}[\cite{TY}, see also \cite{HSVZ}] \label{t:hein}
There is a smooth function $\phi$ on $Z$ such that  $\omega_{TY}\equiv \omega_{Z}+\sqrt{-1}\p\bp\phi$ is a complete Ricci-flat K\"ahler metric on $Z$ solving the Monge-Amp\`ere equation
 \begin{equation}\omega_{TY}^n=\frac{1}{n\cdot 2^{n-1}}(\sq)^{n^2}\Omega_{Z}\wedge\overline\Omega_{Z}.  \end{equation}
Moreover, there is a diffeomorphism $\Phi: \mathcal{C}^n\setminus K'\rightarrow Y\setminus K$, where $K \subset Z$ is compact and $K' = \{|\xi| \geq \frac{1}{2}\}$ and constant $\delta_{Z}>0$, such that the following holds uniformly for all $z$ large:

\begin{enumerate}
\item The K\"ahler potential $\phi$ satisfies the asymptotics
 \begin{equation}\label{lalilu}|\nabla_{g_Z}^k \phi|_{g_Z} = O(e^{-\delta_Z \cdot (-\log |S|^2)^{\frac{1}{2}}}) \ \text{for all} \ k\in\dN.
\end{equation}

\item  We have the complex structure asymptotics
\begin{equation}
|\nabla_{g_{\mathcal{C}^n}}^k(\Phi^*J_{Z}-J_{\Ca^n})|_{g_\mathcal{C}^n}=O(e^{-(\frac{1}{2}-\epsilon)z^n})\ \text{for all} \ k \in\dN, \epsilon > 0.
\end{equation}

\item We have the holomorphic $n$-form asymptotics
\begin{equation}\label{e:n-form-asympt}
|\nabla_{g_{\mathcal{C}^n}}^k(\Phi^*\Omega_Z-\Omega_{\Ca^n})|_{g_{\mathcal C}^n}=O(e^{-(\frac{1}{2}-\epsilon)z^n})\ \text{for all} \ k \in\dN, \epsilon > 0.
\end{equation} 

\item We have the K\"ahler form asymptotics \begin{equation}|\nabla_{g_{\Ca^n}}^k(\Phi^*\omega_{TY}-\omega_{\Ca^n})|_{g_{\mathcal{C}^n}}=O(e^{-{\delta_Z} z^{n/2}}).\end{equation} 
\item There is a constant $C>0$ such that 
\begin{equation}
C^{-1}\cdot z\leq \Phi^*((-\log |S|^2)^{\frac{1}{n}})\leq C\cdot z.
\end{equation}

\end{enumerate}

\end{proposition}
In particular, the space $(Z, \omega_{TY})$ is $\delta_Z$-asymptotically Calabi in the sense of \cite{SZ-Liouville}. For later purposes we also need a simple observation regarding the asymptotics of $\omega_{TY}$. Fix a local holomorphic chart $\{U, w_1, \cdots, w_n\}$ centered at a point $p\in D$, i.e., $w_i(p)=0$ for all $i$, and such that $S$ is locally defined by $w_1=0$. Define a cylindrical type K\"ahler metric as follows
\begin{equation}
\omega_{cyl}\equiv \sum_{j=2}^n \sq dw_j\wedge d\bar w_j+\sq  |w_1|^{-2}dw_1\wedge d\bar w_1.
\end{equation}
By a straightforward computation, we have the following. 
\begin{lemma} \label{l: TY cylindrical compare}
On $U\setminus D$,  there is a constant $C>0$ such that 
\begin{equation}
C^{-1}(-\log|S|^2)^{\frac{1}{n}-1}\omega_{cyl}\leq \omega_{TY}\leq C(-\log |S|^2)^{\frac{1}{n}}\omega_{cyl},
\end{equation}
and for all $k\in\dZ_+$, there are constants $C_k, m_k>0$ such that 
\begin{equation}
|\nabla^k_{\omega_{cyl}}\omega_{TY}|_{\omega_{cyl}}\leq C_k(-\log |S|^2)^{m_k}.
\end{equation}
\end{lemma}
Using this lemma, later when we do estimates for quantities using the Tian-Yau metric, we can do computations using the cylindrical metric which becomes much simpler, and in the end we only get an error which is of polynomial order in $-\log |S|^2$. 

Finally we need a crucial Liouville theorem on the Tian-Yau spaces, which is proved in \cite{SZ-Liouville}.

\begin{theorem}[Theorem 1.2, \cite{SZ-Liouville}]
\label{t:Liouville on Tian-Yau}
There exists a constant $\epsilon_{Z}>0$ such that if $u$ is a harmonic function on the above Tian-Yau space $(Z, \omega_{TY})$ and $|u|\leq e^{\epsilon_Z\cdot z^{\frac{n}{2}}}$ as $z\rightarrow\infty$,
then $u$ is a constant. 
\end{theorem}

\subsection{Construction of approximately  Calabi-Yau metrics}

\label{ss:glued-metrics}

We will work on the setup of Section \ref{ss:algebraic-geometry}. Let us recall some notation from previous discussion. 
The algebro-geometric setup is

\begin{itemize}
\item We have the family of Calabi-Yau varieties $p:\hX \rightarrow \Delta$ in $\C\P^{n+1}\times \Delta$. Let us denote by $\widehat{X}_t$ the fiber $p^{-1}(t)$. By construction, for $t\neq 0$ we know $\widehat{X}_t$ can be identified with $X_{t^{n+2}}$ in the original family.

\item The central fiber $\widehat{X}_0$ is given by the union of three smooth components: $\hat Y_1$, $\hat Y_2$ and $\mathcal N$, with $\hat Y_j\cap \mathcal N=D_j$ both canonically isomorphic to $D$.
\item Under the identification $k_-=d_2$, $k_+=-d_1$, and $L=\mathcal O(1)|_D$,  $\mathcal N\setminus (D_1\cup D_2)$ is naturally identified with the space $\mathcal N^0$ defined in Section \ref{ss:complex-geometry}.
\item The normal bundle of $D_j$ in $\hat Y_j$ is $L_j=\cO(d_{3-j})|_{D}$ and in $\mathcal N$ is $L_j^{-1}$. 
\item There is a relative holomorphic volume form $\Gamma_t(t\in\Delta)$ defined on $\hX\setminus \{D_1\cup D_2\}$. We denote 
\begin{equation}
\begin{cases}
\Gamma_{0, 1}\equiv \Gamma_0|_{Z_1}\\
\Gamma_{0,2}\equiv \Gamma_0|_{Z_2},
\end{cases}
\end{equation}
where $Z_j\equiv \hat Y_j\setminus D_j$, and we know 
\begin{equation}
\Gamma_0|_{\mathcal N^0}=-\sq \Omega_0,
\end{equation}
where $\Omega_0$ is the holomorphic volume form on $\mathcal N^0$ defined in \eqref{e:Omega N0 definition}.

\end{itemize}
The corresponding metric ingredients are 

\begin{itemize}
\item We have the Calabi-Yau metric $\omega_D\in 2\pi c_1(L)$ on $D$, where $L=\cO(1)|_D$. We fix a hermitian metric on $L$ with curvature $-\sq\omega_D$. We also extend this hermitian metric to the whole $\C\P^{n+1}$ such that its curvature form  defines a smooth K\"ahler metric $\omega_{\C\P^{n+1}}$. This then induces hermitian metrics on $\cO(l)$ for all $l$, and also on the pull-back of $\cO(l)$ to the projective bundle $\P(\cO(d_2)\oplus \cO(d_1)\oplus \C)$. Later when $s$ is a holomorphic section of some $\cO(l)$,  $|s|$ will always mean the norm of $s$ with respect to this fixed hermitian metric. 
\item Applying the construction in Section \ref{ss:tian-yau} to the line bundle $L_j\rightarrow D_j$, we have the Tian-Yau metrics $\omega_{TY, j}$ on $Z_j$ for $j\in\{1, 2\}$, and the Calabi-Yau metrics $\omega_{D_j}=d_{3-j}\cdot \omega_D$. So $\omega_{TY, j}$ is asymptotic to 
\begin{equation}
\omega_{Z_j}=\frac{n}{n+1}\sq \p\bp (-\log |f_{3-j}|^2)^{\frac{n+1}{n}},
\end{equation}
and 
\begin{equation}
\omega_{TY, j}^n=\frac{(\sq)^{n^2}}{n\cdot 2^{n-1}}\cdot d_{3-j}^{n-1} \cdot \Gamma_{0, j}\wedge \bar\Gamma_{0, j},
\end{equation}
where the coefficient $d_{3-j}^{n-1}$ arises from the fact we are using $\omega_{D_j}$ instead of $\omega_D$ in the construction.

\item The family of incomplete $C^{2, \alpha}$ approximately Calabi-Yau metrics $(\omega_T, \Omega_T)$ on $\M_T$, and $(\M_T, \Omega_T)$ is embedded in $(\mathcal N^0, \Omega_0)$ as in Section \ref{ss:complex-geometry}, with $k_-=d_2$ and $k_+=-d_1$. 
\end{itemize}

Our goal in this subsection is to construct for each $t$ small a $C^{1,\alpha}$ K\"ahler metric $\omega(t)$ on $\widehat{X}_t$ which is approximately Calabi-Yau in a suitable weighted sense.

\subsubsection{Matching between the parameters $t$ and $T$}

The relationship between the parameters $t$ and $T$ can be determined by studying the matching between the Tian-Yau ends and the neck region.

In our setting, we need to normalize the Tian-Yau metrics  $\omega_{TY, i}$ on $Z_i$ (as in Section \ref{ss:tian-yau}) by defining \begin{equation}
\tilde\omega_{TY, j}\equiv 2^{\frac{-1}{n}}n^{\frac{1}{n}}d_{3-j}^{-\frac{n-1}{n}}\omega_{TY, j}.
\end{equation}
Then we have 
\begin{equation}
\tilde\omega_{TY, j}=\frac{(\sq)^{n^2}}{2^n} \Gamma_{0, j}\wedge \bar\Gamma_{0,j}.
\end{equation}
By definition we can write 
\begin{equation}
\tilde\omega_{TY, j}=  dd^c\phi_j=2\sq \p\bp\phi_j,\end{equation}
where 
$\phi_j=\eta_j+\psi_j$
such that 
\begin{align}
\eta_j=\frac{1}{n+1}\cdot k_{3-j}^{\frac{1-n}{n}}\cdot n^{\frac{n+1}{n}}\cdot (-\log |f_{3-j}|)^{\frac{n+1}{n}}\quad  \text{and} \quad  |\nabla^k \psi_1|=O(e^{-\delta_0 (-\log |f_{3-j}|^2)^{1/2}}),\end{align}
for all $k\in \dN$, and where the derivatives and norms are taken with respect to the Tian-Yau metric itself (which is equivalent to taking with respect to the metric $\omega_{Z_j}$).

\

Now on the neck $\M_T$ we have the asymptotics of the K\"ahler potential given in Section \ref{ss:complex-geometry}. By the discussion there we identify $\M_T$ with an open set in $\mathcal N^0$, and the latter is naturally an open set in $\mathcal N$. Moreover, we can write 

\begin{equation}
\label{e:asymptotics of phi-}
T^{\frac{n-2}{n}}\omega_T=dd^c\phi_T, 
\end{equation}
with
\begin{equation}
\phi_T=
\begin{cases}
\phi_-\equiv\vf_-+\psi_-, \ \ \ \ z< 0;\\
\phi_+\equiv\vf_++\psi_+, \ \ \ \ z> 0,
\end{cases}\label{e:phi_T}
\end{equation}
where 
\begin{equation}
\label{e:definition of phi-}
\begin{cases}
\vf_-=\frac{1}{n+1}n^{\frac{n+1}{n}} k_-^{-\frac{n-1}{n}}(A_--\log |s_1/s_3|)^{\frac{n+1}{n}};\\
\vf_+=\frac{1}{n+1}n^{\frac{n+1}{n}} (-k_+)^{-\frac{n-1}{n}}(A_+-\log |s_2/s_3|)^{\frac{n+1}{n}}.\end{cases}
\end{equation}
Notice $\omega_T$ is defined by $z\in [T_-, T_+]$, so  by \eqref{e: compare r and z} $\varphi_-$ is well-defined for $A_--\log |s_1/s_3|\gg1$ and 
$\varphi_2$ is well-defined for $A_+-\log |s_2/s_3|\gg1$. Moreover
 for $|z|\geq 1$, we have 
\begin{equation}
|\psi_\pm|=\epsilon(z)+\epsilon_T.
\end{equation}
Now on $\M_T$ for $|t|>0$ small we have \begin{equation}
t^{d_1}s_1=s_3f_2(x)
\end{equation}
which gives
\begin{equation}-d_1\log |t|-\log \frac{|s_1|}{|s_3|}=-\log |f_2|.\end{equation}
So if we want to graft the metrics on the three components of $\widehat{X}_0$ to nearby $\widehat{X}_t$ so that they match with small errors, then we need 
\begin{equation} \label{eqn6-70}
d_1\log |t|=-A_-.
\end{equation}
Similarly at the positive end we need 
\begin{equation}\label{eqn6-71}
d_2\log |t|=-A_+.
\end{equation}
This suggests that we should choose 
\begin{equation}
\label{e:t T relation}
|t|=e^{-\frac{1}{d_1}A_-}=e^{-\frac{1}{d_2}A_+}.
\end{equation}
Given $|t|$ small we can find $T$ big so that \eqref{e:t T  relation} holds. It is not necessary that $T$ is uniquely determined by $t$, but we always fix a particular choice for each $t$ throughout this section so that \eqref{e:t T relation} holds.  With this choice it is easy to see that 
\begin{equation} \label{eqn6667}
C^{-1}e^{-\frac{1}{d_1d_2n} T^2}\leq |t|\leq Ce^{-\frac{1}{d_1d_2n} T^2}.
\end{equation}

\subsubsection{Fixing the constants in the definition of weighted spaces}
\label{sss:parameters-fixed}

From now on, we fix weight parameters in the definition of weight spaces,  which allows us to prove the uniform injectivity estimate in Proposition \ref{p:global-injectivity-estimate} and apply the implicit function theorem to complete the proof of the main theorem in Section \ref{ss:global analysis}.  The parameters $\delta$, $\mu$, $\nu$ are fixed as follows (similar to the specification of the parameters in Section \ref{ss:perturbation-framework}):

\begin{enumerate}
\item[(GP1)] (Fix $\nu$) The parameter $\nu\in\dR$ is chosen such that
$\nu\in(-1,0)$.

\item[(GP2)] (Fix $\alpha$) The H\"older constant $\alpha\in(0,1)$ is chosen such that
$\nu+\alpha<0$.

\item[(GP3)] (Fix $\delta$) The constant $\delta>0$ is chosen such that \begin{equation}
0<\delta < \delta_G \equiv \frac{1}{n\cdot (|k_-| + |k_+|)}\cdot \min\{\delta_e, \delta_{Z_1}, \delta_{Z_2}, \epsilon_{Z_1}, \epsilon_{Z_2}, \sqrt{\lambda_D}\},\label{e:global-delta}
\end{equation}
where $\sqrt{\lambda_D}$ is in Lemma \ref{l:liouville-cylinder},  
$\delta_e>0$ is in Proposition \ref{p:CY-error-small},  $\delta_{Z_1}, \delta_{Z_2}$ are the constants in Proposition \ref{t:hein} applied to $Z_1, Z_2$, and $\epsilon_{Z_1},\epsilon_{Z_2}$ are the constants in Theorem \ref{t:Liouville on Tian-Yau} applied to $Z_1, Z_2$.

\item[(GP4)] (Fix $\mu$) The parameter $\mu>0$ is chosen as 
$\mu=(1-\frac{1}{n})(\nu+2+\alpha)$.

\end{enumerate}

\subsubsection{Construction of $\omega(t)$}

We will divide a neighborhood of $\widehat{X}_0$ into various regions (c.f. Figure \ref{f: the division of region}):

\begin{itemize}
\item Region $\I$ is given by $2|s_3|\geq \max(|s_1|, |s_2|)$;
\item Region $\II_-$ is given by $s_1\neq 0$, and $|s_3|\leq 2|s_1|, -\log |f_2|\geq -\frac{d_1}{2}\log |t|$;
\item Region $\III_-$ is given by $s_1\neq 0$, and $|f_2|\leq 1/2$,  $-\log |f_2|\leq -\frac{d_1}{2}\log |t|+1$;
\item Region $\IV_-$ is 
given by $s_1\neq 0$ and $|f_2|\geq 1/4$;
\item Region $\II_+$ is given by $s_2\neq 0$, and $|s_3|\leq 2|s_2|, -\log |f_1|\geq -\frac{d_2}{2} \log |t|$;
\item Region $\III_+$ is 
given by $s_2\neq 0$ and $|f_1|\leq 1/2$, $-\log |f_1|\geq -\frac{d_2}{2} \log |t|+1$;
\item Region $\IV_+$ is 
given by $s_2\neq 0$ and $|f_1|\geq 1/4$.
\end{itemize}

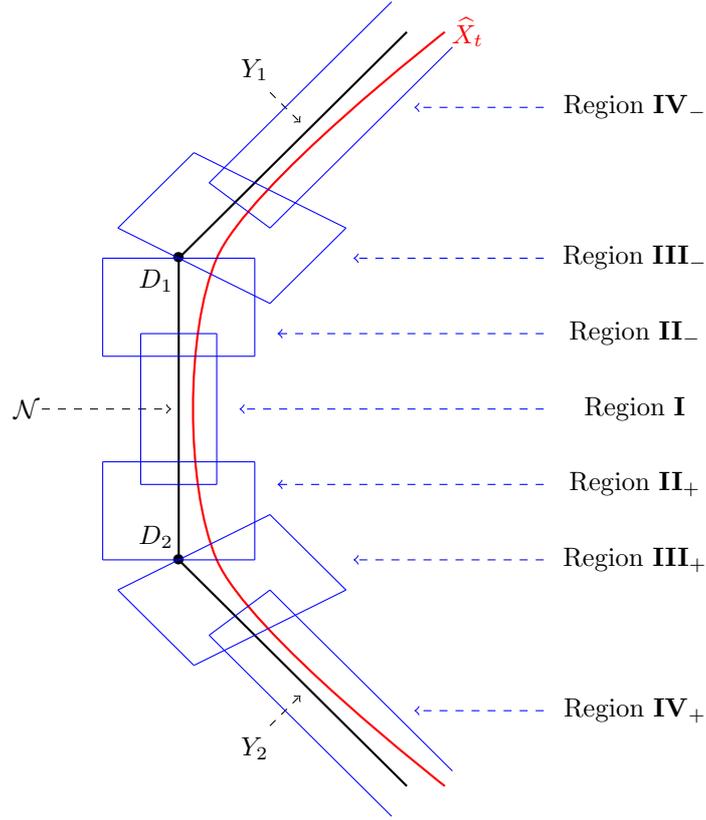
\begin{figure}
\begin{tikzpicture}
\draw[thick] (0, 2) to (3, 5); 
\draw[thick] (0, 2) to (0, -2); 
\draw[thick] (0, -2) to (3, -5); 
\node at (0, 2) {$\bullet$};
\node at (0, -2) {$\bullet$};
\node at (1, 4.5) {$Y_1$}; 
\draw[dashed, ->] (1.2, 4.2) to (1.6, 3.8); 
\node at (1, -4.5) {$Y_2$}; 
\draw[dashed, ->] (1.2, -4.2) to (1.6, -3.8); 
\node[red] at (3.8, 5) {$\widehat{X}_t$}; 
\node at (-0.3, 1.7) {$D_1$}; 
\node at (-0.3, -1.7) {$D_2$}; 
\node at (-2, -0) {$\mathcal N$};
\draw[dashed, ->] (-1.8, 0) to (-0.1, 0); 
\node at (6, 4) {Region $\IV_-$};
\node at (6, 2) {Region $\III_-$};
\node at (6, 1) {Region $\II_-$};
\node at (6, 0) {Region $\I$};
\node at (6, -4) {Region $\IV_+$};
\node at (6, -2) {Region $\III_+$};
\node at (6, -1) {Region $\II_+$};
\draw[blue,->, dashed] (4.8, 4) to (3.1, 4);
\draw[blue,->, dashed] (4.8, 2) to (2.3, 2);  
\draw[blue,->, dashed] (4.8, 1) to (1.3, 1);  
\draw[blue,->, dashed] (4.8, 0) to (0.8, 0);  
\draw[blue,->, dashed] (4.8, -4) to (3.1, -4);
\draw[blue,->, dashed] (4.8, -2) to (2.3, -2);  
\draw[blue,->, dashed] (4.8, -1) to (1.3, -1);

\draw[red, thick] plot[smooth] coordinates {(3.5, 5) (0.5, 2) (0.5, -2) (3.5, -5)};
\draw[blue] (-0.5, 1) to (0.5, 1); 
\draw[blue] (-0.5, 1) to (-0.5, -1); 
\draw[blue] (-0.5, -1) to (0.5, -1);
\draw[blue] (0.5, 1) to (0.5, -1);  

\draw[blue] (-1, 0.7) to (1, 0.7); 
\draw[blue] (-1, 2) to (1, 2);
\draw[blue] (-1, 0.7) to (-1,2); 
\draw[blue] (1, 0.7) to (1, 2); 

\draw[blue] (-0.8, 2.4) to (1.2, 1.4);

\draw[blue] (0.2, 3.4) to (2.2, 2.4); 
\draw[blue] (-0.8, 2.4) to (0.2, 3.4); 
\draw[blue] (2.2, 2.4) to (1.2, 1.4); 

\draw[blue] (0.4, 3) to (1.2, 2.4); 
\draw[blue] (2.8, 5.4) to (0.4, 3);
\draw[blue] (1.2, 2.4) to (3.6, 4.8);

\draw[blue] (-1, -0.7) to (1, -0.7); 
\draw[blue] (-1, -2) to (1, -2);
\draw[blue] (-1, -0.7) to (-1,-2); 
\draw[blue] (1, -0.7) to (1, -2); 

\draw[blue] (-0.8, -2.4) to (1.2, -1.4);

\draw[blue] (0.2, -3.4) to (2.2, -2.4); 
\draw[blue] (-0.8, -2.4) to (0.2, -3.4); 
\draw[blue] (2.2, -2.4) to (1.2, -1.4); 

\draw[blue] (0.4, -3) to (1.2, -2.4); 
\draw[blue] (2.8, -5.4) to (0.4, -3);
\draw[blue] (1.2, -2.4) to (3.6, -4.8);
\end{tikzpicture}
\caption{Division of a neighborhood of $\widehat{X}_0$}
 \label{f: the division of region}
\end{figure}

For all $|t|$ sufficiently small, then we also get a division of $\widehat{X}_t$ into 7 regions. 
Notice that we have non-empty intersections  between these regions so that we need a cut-off on the overlap.

For the convenience of later analysis, we now fix a finite cover $\U=\{U_\beta^1, U_\gamma^2, U_{\cN}, U_{-}, U_{+}\}$ of a neighborhood of $\widehat{X}_0$ in $\hX$ obtained as follows. 

We first cover a neighborhood of $D_1$. 
Given any point in $(x, t, [s_1:s_2:s_3])\in D_1$, we have $t=f_1(x)=f_2(x)=s_2=s_3=0, s_1\neq0$.  On the open subset $\{x_j\neq 0\}$ in $\C\P^{n+1}$, we can view $\sigma=x_j$ as a trivialization of $\cO(1)$. Without loss of generality we may assume $j=0$. Then we get affine coordinates $\{u_i=x_i/x_0(i=1,\cdots, n+1)\}$, and we can $v_1=f_1(u)$ and $v_2=f_2(u)$ as local holomorphic functions on $\C\P^{n+1}$. Further without loss of generality we can assume $\{v_1, v_2, u_i=x_i/x_0 (i=3, \cdots)\}$ yield local holomorphic coordinates in a neighborhood of $x$ in $\C\P^{n+1}$. Correspondingly we can pull back these to be local holomorphic functions on the projective bundle $\P(\cO(d_2)\oplus \cO(d_1)\oplus \C)$. As before we also introduce local holomorphic functions $\zeta_3=s_3/s_1, \zeta_2=s_2/s_1$ on the projective bundle, and the space $\hX$ is then defined by the equations as in \eqref{e: s1 nonzero region all equations}, which essentially reduces to one relation $v_2\zeta_3=t^{d_1}$ in the three variables $v_2, \zeta_3, t$. We denote by $U_\beta^1$ an open  subset in $\hX$ defined by the inequalities $|\zeta_3|< 3 |\sigma|^{d_2}$, $|v_2|<3|\sigma|^{d_2}$, and $|u_i|< C(i=3, \cdots)$ for some fixed $C>0$.  Denote the trivializing section $\sigma$ by $\sigma_\beta^1$.  For $|t|$ small, $U^1_{\beta, t}\equiv U^1_\beta\cap \widehat{X}_t$ is then defined by the equation $v_2\zeta_3=t^{d_1}$. 

We have the natural projection maps
\begin{equation}\pi_\beta^1: U^1_{\beta, t}\rightarrow U^{1}_{\beta, 0}\cap Y_1; (x, t, v_2, \zeta_3)\mapsto (v_2, 0),\end{equation}
\begin{equation}\pi_\beta^{ \mathcal N}: U^1_{\beta, t}\rightarrow U^{1}_{\beta, 0}\cap \mathcal N; (x, t, v_2, \zeta_3)\mapsto (0, \zeta_3),\end{equation}
\begin{equation}\pi^{1, D}_{\beta}: U^1_{\beta, t}\rightarrow D; (x, t, v_2, \zeta_3)\rightarrow x.\end{equation}
Then the union of images $\pi^{1, D}_\beta(U^1_{\beta, t})$ form an open cover of $D$. By compactness we can choose and then fix finitely many of them which also cover $D$, and we put these $U^1_\beta$'s in $\U$. Then we obtain also a cover of a neighborhood of $D_1$ in $Y_1$ by $\{U^{1}_{\beta, 0}\cap Y_1\}$ and a cover of a neighborhood of $D_1$ in $\mathcal N$ by $\{U^{1}_{\beta, 0}\cap \mathcal N\}$ so that on each element in the cover we have holomorphic coordinates.  Without loss of generality we may assume these cover the neighborhood defined by $|f_2|\leq 3$ and $|s_3/s_1|\leq 3$.  So in particular they contain Regions $\II_-$ and $\III_-$.

 We can do the same with $D_2$, and add the corresponding elements $U_\gamma^2$ to $\U$. Now away from $D_1\cup D_2$ we may find a trivialization of the fibration $\hX\rightarrow\Delta$. So we can obtain three open subsets of $\hX$, each of which has a differentiable trivialization over $\Delta$. Call these $U_{\mathcal N}$, $U_-$, $U_+$.  Adding these to $\U$ we then obtain an open cover of a neighborhood of $\widehat{X}_0$. Over each of the three subsets we also have the projection map $\pi_{-},  \pi_{+}$, and $\pi_{{\mathcal N}}$ from them into $\widehat{X}_0\setminus (D_1\cup D_2)$. We may assume that Region $\I$ is contained in $U_{\mathcal N}$, Region $\IV_\pm$ is contained in $U_\pm$.

Let us fix a partition of unity $\chi_\beta^1$ of $D$ subordinate to the cover $\pi_\beta^{1, D}(U^1_{\beta, t})$, and ${\chi_\gamma^2}$ of $D$ subordinate to the cover $\pi_\gamma^{1, D}(U^2_{\gamma, t})$. We view these naturally as functions on the corresponding $U^1_{\beta, t}$ and $U^2_{\gamma, t}$, though not compactly supported (along the fiber direction).

\

Below we define the approximately Calabi-Yau  metric $\omega(t)$ on $(\widehat{X}_t, \Gamma_t)$ for each region above, and we also define the weight function $\rho_t(\bx)$ simultaneously and measure the following error of the Calabi-Yau equation in the weighted sense:
\begin{equation}
\Errt\equiv \frac{(\sq)^{n^2}2^{-n}\Gamma_t\wedge\bar\Gamma_t}{\omega(t)^n/n!}-1.
\end{equation}
Obviously $\Errt=0$ if and only if $\omega(t)$ is Calabi-Yau. 
In the course we also discuss the gluing in the intersection of neighboring regions. 

\

{\bf Region $\I$}. In this region we define
\begin{equation}\omega(t)=T^{\frac{2-n}{n}}(T\omega_{\C\P^{n+1}}|_{\widehat{X}_t}+dd^c\phi_{t, \mathcal N}),
\end{equation}
where 
$\phi_{t,\mathcal N}=\pi_{\mathcal N}^*\phi_T$. 
Using the fixed diffeomorphism $\pi_{\mathcal N}$ we may view the $C^{1,\alpha}$ K\"ahler structures $(\omega(t), \Omega(t)=\Gamma_t)$ on $\widehat{X}_t$ as a perturbation of the K\"ahler structure $(\omega_T, \Omega_0)$ on $\mathcal N^0$. 

Notice  by Corollary \ref{c:r zeta relation} it is not difficult to see that $\I\cap \mathcal N$ is contained in the union $\I_1\cup \I_2$ (as defined in Section \ref{ss:neck-weighted-analysis}). So we can define 
\begin{equation}
z(\bx)\equiv z(\pi_{\cN}(\bx)), \ \ L_t(\bx)\equiv L_t(\pi_{\mathcal N}(\bx)), \ \  \fr(\bx)\equiv \fr(\pi_{\mathcal N}(\bx))
\end{equation}
and then use \eqref{e:definition of weights} to define the weight function $\rho_t(\bx)$. Applying Proposition \ref{p: perturbation of complex structures}, we conclude that 
 \begin{equation}
\begin{cases}|\Omega(t)-\Omega_0|_{\wI}=\underline\epsilon_{T^2}, \\
|\omega(t)-\omega_t|_{\wI}=\underline\epsilon_{T^2}. 
\end{cases}
\end{equation}
 Then by Proposition \ref{p:CY-error-small} we get an error estimate 
\begin{equation}
\label{e:error region I}
\|\Errt\|_{C^{0, \alpha}_{\delta,\nu+2, \mu}({\I}\cap \widehat{X}_t)}=O(T^{\nu+\alpha}).
\end{equation}
At the two ends of ${\I}\cap {\widehat{X}_t}$, we can write down the metric $\omega(t)$ in potential form. In the negative end we have $f_2\neq 0$, so we can write 
\begin{equation}
\omega_{\C\P^{n+1}}|_{\widehat{X}_t}=-\frac{1}{d_2}dd^c \log |f_2|,
\end{equation}
and
\begin{equation}
\omega(t)=T^{-\frac{n-2}{n}}dd^c \phi_{t, \I_-}, 
\end{equation}
where
\begin{equation}
\label{e:Kahler potential I-}
\phi_{t, \I_-}=-\frac{T}{d_2}\log |f_2|+\phi_{t,\cN}. 
\end{equation}
Similarly at the positive end  we have
\begin{equation}
\omega(t)=T^{-\frac{n-2}{n}}dd^c\phi_{t, \I_+},
\end{equation}
where
\begin{equation}
\phi_{t, \I_+}=-\frac{T}{d_1}\log |f_1|+\phi_{t, \cN}. 
\end{equation}

\

\

{\bf Region $\IV_\pm$.}  We only consider the region $\IV_-$, and $\IV_+$ is similar. We define 
\begin{equation}
\omega(t)=dd^c(\phi_1\circ \pi_-).
\end{equation}
Then for $|t|$ small we can view $(\widehat{X}_t\cap {\IV_-}, \omega(t))$ as a perturbation of the Tian-Yau metric $\tilde\omega_{TY, 1}$. It is easy to see that in the intersection $\widehat{X}_t\cap \IV_-$, for all $k\geq 0$ we have
\begin{equation}
\begin{cases}
|\nabla_{\tilde\omega_{TY, 1}}^k(\omega(t)-\tilde\omega_{TY, 1})|_{\tilde\omega_{TY, 1}}=\underline\epsilon_{T^2};\\
|\nabla_{\tilde\omega_{TY, 1}}^k((\Gamma(t)-\Gamma_{0, 1}))|_{\tilde\omega_{TY, 1}}=\underline\epsilon_{T^2}. 
\end{cases}
\end{equation}
To define the weight we let
\begin{equation}
L_t(\bx)\equiv T^{\frac{n-2}{n}}(nk_-)^{\frac{1}{n}} (-\log 4)^{\frac{1}{n}}, 
\end{equation} 
\begin{equation}
U_t(\bx)=T-T^{1-\frac{n}{2}}L_t(\bx)^{\frac{n}{2}},
\end{equation}
 and then define $\rho_t(\bx)$ as  in \eqref{e:definition of weights}.Then we obtain that
\begin{equation}
\label{e:error region IV-}
\|\Errt\|_{C^{0, \alpha}_{\delta, \nu+2, \mu}({\IV_-}\cap \widehat{X}_t)}=\underline\epsilon_{T^2}.
\end{equation}
We also have by assumption the asymptotics at the end 
\begin{equation}
\label{e:Kahler potential region IV-}
\phi_1\circ\pi_-=\eta_1\circ\pi_-+\psi_1\circ\pi_-. 
\end{equation}

\

\

{\bf Region $\II_\pm$.} Again we only consider the region $\II_-$. 
 We define
\begin{equation}
\omega(t)=T^{\frac{2-n}{n}}dd^c \phi_{t, \II_-},
\end{equation}
where 
\begin{equation}
\phi_{t, \II_-}=\sum {\chi_\beta^1} \cdot \phi_-\circ \pi_{\beta}^{\mathcal N}.
\end{equation}
By definition of $\phi_-$ in \eqref{e:phi_T}, we know this is well-defined in region $\II_-$ for $|t|>0$ small. 
We need the following Lemma. 
\begin{lemma}
\label{l:transition function expansion}
We have the following
\begin{enumerate}
\item On $\pi^1_{\beta}(U^1_{\beta, t}\cap U^1_{\beta',t})$, we write $\pi^{1}_{\beta'}\circ( \pi^1_\beta)^{-1}(\bx)=\bx'$. Suppose $\bx$ and $\bx'$ have coordinates given by $(v_2, 0, u_i)$ and $(v_2', 0, u_i')$ in the chart $U^1_{\beta, 0}\cap Y_1$.  Then we have 
\begin{equation}
\begin{cases}
v_2'=v_2 \cdot (1+v_1 F_2)
\\
u_i'=u_i+v_1 G_i,
\end{cases}
\end{equation}
where $F_2$ and $G_i$ are smooth functions in $v_1, v_2, u_i$, and $v_1$ is implicitly determined by $v_2, u_i$ and $t$ by the equation \eqref{e: s1 nonzero region}
\item On $\pi^{\mathcal N}_{\beta}(U^1_{\beta, t}\cap U^1_{\beta',t})$, we write $\pi^{\mathcal N}_{\beta'}\circ( \pi^{\mathcal N}_\beta)^{-1}(\bx)=\bx'$. Suppose $\bx$ and $\bx'$ have coordinates given by $(0, \zeta_3, u_i)$ and $(0, \zeta_3', u_i')$ in the chart $U^1_{\beta, 0}\cap\mathcal N$.  Then we have 
\begin{equation}
\begin{cases}
\zeta_3'=\zeta_3 \cdot (1+v_2 \tilde F_3)
\\
u_i'=u_i+v_2 \tilde G_i,
\end{cases}
\end{equation}
where $\tilde F_2$ and $\tilde G_i$ are smooth functions in $v_1, v_2, u_i$, and $v_1$ is implicitly determined by $\zeta_3, u_i$ and $t$ by the equation \eqref{e: s1 nonzero region}.

\end{enumerate}
\end{lemma}
\begin{proof}
This involves only local discussion.   By construction we get overlapping local holomorphic charts on $\C\P^{n+1}$ given by $\{v_1, v_2, u_i(i\geq 3)\}$ and $\{v_1', v_2',  u_i' (i\geq 3)\}$. Given a point in this overlap with coordinates $(v_1, v_2, u_i)$ and $(v_1', v_2', u_i')$ in these two coordinate charts respectively, then we have 
\begin{equation}
\begin{cases}
v_1'=v_1\cdot Q_1(v_1, v_2, u_i)\\
v_2'=v_2\cdot Q_2(v_1, v_2, u_i)\\
u_i'=R_i'(v_1, v_2, u_i).
\end{cases}
\end{equation}
where $Q_1, Q_2$ are smooth and non-vanishing along $D$. More precisely, we have 
\begin{equation}
Q_i=(\sigma_{\beta'}^1/\sigma_\beta^1)^{d_i}.
\end{equation}
Correspondingly we obtain the transition maps on  $U^1_\beta\cap U^1_{\beta'}$ given by 
\begin{equation}
\begin{cases}
v_2'=v_2\cdot Q_2(v_1, v_2, u_i); \\
\zeta_3'=\zeta_3\cdot Q_2^{-1}(v_1, v_2, u_i); \\
u_i'=R_i'(v_1, v_2, u_i);
\end{cases}
\end{equation}
where using \eqref{e: s1 nonzero region} we can write $v_1$ implicitly as a function of $v_2, \zeta_3'$ and $u_i$. In particular, we obtain the transition function on $Y_1\cap U_\beta^1\cap U_{\beta'}^1$ given by 
\begin{equation}
\begin{cases}
v_2'=v_2\cdot Q_2(0, v_2, u_i); \\
u_i'=R_i'(0, v_2, u_i),
\end{cases}
\end{equation}
and  on $\mathcal N\cap U_\beta^1\cap U_{\beta'}^1$ given by 
\begin{equation}
\begin{cases}
\zeta_3'=\zeta_3\cdot Q_2^{-1}(v_1, 0, u_i); \\
u_i'=R_i'(v_1, 0, u_i). 
\end{cases}
\end{equation}
Then the conclusion follows by a direct calculation. 
\end{proof}

\begin{proposition}
\label{p:region II- error on overlapping region}
In the Region ${\II_-}\cap U_{\beta, t}^1$, we have for all $k\geq 0$
\begin{equation}
|\nabla^k(\phi_{t, \II_-}\circ (\pi_{\beta}^{\mathcal N})^{-1}-\phi_-)|=\underline\epsilon_{T^2},
\end{equation}
where the derivative and norm are taken with respect to the metric $\omega_T$. 
\end{proposition}
\begin{proof}
We may write 
\begin{equation}
\phi_{t, \II_-}\circ (\pi_{\beta}^{\mathcal N})^{-1}(\bx)-\phi_-(\bx)=\sum_{\beta': q\in U_{1, \beta'}} \chi_{\beta'}^1(\bx) (\phi_-\circ \pi_{\beta'}^{\mathcal N}\circ(\pi_{\beta}^{\mathcal N})^{-1}(\bx)-\phi_-(\bx)).
\end{equation}
Write 
\begin{equation}
\pi_{\beta'}^{\cN}\circ (\pi_{\beta}^{\mathcal N})^{-1}(\bx)=\bx'=(\zeta_3', u_i').
\end{equation}
Then we write 
\begin{equation} \label{eqn-6109}
\phi_-(\bx')-\phi_-(\bx)=\int_0^{1}\langle\nabla_{\omega_T}\phi_-(t\bx'+(1-t)\bx), \bx'-\bx\rangle_{\omega_T} dt.
\end{equation}
\textbf{Claim:} For any $k\geq 0$, there is a $C_k>0$ such at for all $\bx\in \II_-$, 
\begin{equation}
|\nabla^k_{\omega_T} \phi_-(\bx)|\leq e^{C_kT}.
\end{equation}
To see this we notice by definition $\phi_-$ satisfies the equation 
\begin{equation}
\Delta_{\omega_T}\phi_-=\Tr_{\omega_T}\omega_T=n.
\end{equation}
 Notice by Corollary \ref{c:r zeta relation} item (2), given $c\in (0, 1/2)$ we have for all $\bx\in \II_-$, 
\begin{equation}
r(\bx)\geq cT^{-1}\log T.
\end{equation}
Hence by Proposition \ref{p:regularity-scale}, for every point $\by$ in the regularity ball $B_{\mathfrak s(x)}(\bx)$, we have
\begin{equation}
r(\by)\geq \frac{c}{2}T^{-1}\log T.
\end{equation}
So we can apply the item (2) in Proposition \ref{p:local-weighted-schauder}, and it suffices to show a bound on the $C^0$ norm of $\phi_-$. By \eqref{e:definition of phi-} it suffices to bound $-\log |f_2|$. By our definition for $\bx\in \II_-$ we have 
\begin{equation}
-\log |f_2|=-d_1\log |t|-\log \frac{|s_1|}{|s_3|}\leq -\frac{d_1}{2}\log |t|-\log 2\leq CT^2.\end{equation}
This then proves the Claim. 

Now it suffices to bound the norm of the vector field $\bx'-\bx$ and its convariant derivatives. To this end we divide into two cases. 

\textbf{Case 1}: $z\leq -1$. 
Notice by Lemma \ref{l:neck cylindrical compare} comparing with the cylindrical metric, we obtain the norm of the tangent vectors  $|\bx'-\bx|\leq |v_2|T^{m}$ for some $m>0$. On the other hand we have $|v_2|\leq C|f_2|\leq \underline\epsilon_{T^2}$. So we obtain 
\begin{equation}
|\phi_-(\bx')-\phi_-(\bx)|=\underline\epsilon_{T^2}.
\end{equation}
The higher order derivatives follows similarly by differentiating \eqref{eqn-6109} and Lemma \ref{l:neck cylindrical compare}, using the fact that all derivatives of the vector field $\bx'-\bx$ in the cylindrical metric is bounded by $C|v_2|$.

\

\textbf{Case 2}. $z\geq -1$. Then we instead compare the metric $\omega_T$ with the standard metric 
\begin{equation}\omega_{std}\equiv \sum_{j=1}^{n-1}\sq dw_j\wedge d\bar w_j+\sq d\zeta_3\wedge d\bar \zeta_3.\end{equation}
As in the proof of Proposition \ref{p: perturbation of complex structures} we first notice 
\begin{equation}\Delta_{\omega_T}w_j=\Delta_{\omega_T}\zeta_3=\Delta_{\omega_T}\zeta_3^{-1}=0.\end{equation}
By assumption we have $|\zeta_3|\leq C$ in this case, and also by Corollary \ref{c:r zeta relation}, item (3) we get $|\zeta_3^{-1}|\leq Ce^{CT}$. 
Then we again apply Schauder estimates Proposition \ref{p:local-weighted-schauder}, item (2), to get \begin{equation}
|\nabla^k w_j|\leq Ce^{C_kT}, |\nabla^k \zeta_3|\leq Ce^{C_kT}. 
\end{equation}
Hence we get for all $k\geq 0$. 
\begin{equation}|\nabla^k_{\omega_T} \omega_{std}|_{\omega_T}\leq Ce^{C_kT}.\end{equation}
Now we use
\begin{equation}
\omega_T^{n}\leq C\Omega_T\wedge\bar\Omega_T\leq C|\zeta_3|^{-2} \omega_{std}^n.
\end{equation}
to get that 
\begin{equation}
\omega_{std}\geq Ce^{-CT}\omega_T. 
\end{equation}
Now we again can first estimate the norm of $\bx'-\bx$ and its derivatives using the standard metric, and use the above information to conclude. 
\end{proof}

Now we define the weight function $\rho_t$.  We first define

\begin{equation}
L_t(\bx)=\sum_\beta \chi_\beta(\bx)\cdot  L(\pi_{\beta}^{\mathcal N}(\bx)), \ \ \  \  \fr(\bx)\equiv e^{\sum_\beta \chi_\beta(\bx)\cdot \log \fr(\pi_\beta^{\mathcal N}(\bx))}. 
\end{equation}
Then we define the weight function $\rho_t(\bx)$ as \eqref{e:definition of weights}.
Notice we have that on ${\II_-}\cap \widehat{X}_t\cap U_\beta^1$,
\begin{equation}
\label{e:l function on II-}
L_t(\bx)=T^{\frac{n}{2}-1}(nk_-)^{\frac{1}{n}}(A_--\log |r_-(\pi_\beta^{\mathcal N}(\bx))|+\epsilon(z))^{\frac{1}{n}}\end{equation}
From this we get that 
\begin{equation}|\omega(t)-(\pi_{\beta}^{\mathcal N})^*\omega_T|_{C^{1, \alpha}_{\delta, \nu, \mu}({\II_-}\cap \widehat{X}_t)}=\underline\epsilon_{T^2}. \end{equation}
Now we understand the holomorphic volume form. Using \eqref{eqn8-3} we get that 
\begin{equation}
\Gamma_t= (1+H)(\pi_{\beta}^{\mathcal N})^*\Omega_0,
\end{equation}
where $H$ is a holomorphic function in $\zeta_3, v_2, w_2, \cdots, w_{n-1}$, and its derivatives is of order $\underline\epsilon_{T^2}$ in these coordinates. Then we again apply weighted Schauder estimates to get 
that 
\begin{equation}
|H|_{C^{0,\alpha}_{\delta, \nu, \mu}({\II_-}\cap \widehat{X}_t) }=\underline\epsilon_{T^2}. 
\end{equation}
So by Proposition \ref{p:CY-error-small}  we obtain
\begin{equation}
\label{e:error in region II-}
\|\Errt\|_{C^{0, \alpha}_{\delta, \nu+2, \mu}({\II_-}\cap \widehat{X}_t)}=O(T^{\nu+\alpha}).
\end{equation}

\

Notice ${\II_-}\cap \widehat{X}_t$ has two ends. Along one end it is close to the negative end of Region $\I$. 

\begin{proposition}
\label{p:potential difference II- I}
On the intersection ${\II_-}\cap {\I}\cap \widehat{X}_t$ we have for all $k\geq 0$
\begin{equation}|\nabla^k(\phi_{t, \II_-}-\phi_{t, \I_-}-d_1\log t)|=\underline\epsilon_{T^2},
\end{equation}
where the derivative and norm are taken with respect to $\omega(t)$. 
\end{proposition}
\begin{proof}
We work in $U_\beta^1$ for a fixed $\beta$.  We have 
\begin{equation}
\phi_{t, \I_-}(\bx)=\frac{T}{d_2}\log |f_2(\bx)|+\pi_{\mathcal N}^* \phi_t(\bx)
\end{equation}
and 
\begin{equation}
(\pi_{\beta}^{\mathcal N})^* \phi_-(\bx)=\phi_t(\by)-\frac{T}{d_2}\log r_-(\by)
\end{equation}
where $\by=\pi_\beta^{\mathcal N}(\bx)$. By definition it is easy to see that $\by-\bx$ is of order $\underline\epsilon_{T^2}$ in the coordinates in $v_2, \zeta_3, w_2, \cdots, w_{n-1}$. 
By our choice of $T$ in terms of $t$ we have 
\begin{equation}-\log |r_-(\by)|=d_1\log |t|-\log |f_2(\by)|.\end{equation}
Then by Lemma \ref{l:transition function expansion}, and using weighed Schauder estimates as above we get the conclusion. 
\end{proof}

By Proposition \ref{p:potential difference II- I}, we  glue  the potentials in Region $\I_-$ and $\II_-$ using the potential\begin{equation}
\phi(t)\equiv\chi(r_-(\bx))\cdot (\phi_{t, \I_-}+d_1\log t)+(1-\chi(r_-(\bx)))\cdot \phi_{t, \II_-}
\end{equation}
where $\chi$ is a cut-off function in $s$ satisfying
\begin{equation}
\chi(s)=
\begin{cases}
1,  s\leq 3/4\\
0, s\geq 5/4.
\end{cases}
\end{equation}
Then we define $\omega(t)=dd^c\phi(t)$.

Along the other end, Region $\II_-$ is close to the region $\III_-$.

\begin{proposition}
\label{p:potential difference II- III-}
On the intersection ${\II_-}\cap {\III_-}\cap \widehat{X}_t$, we have for all $k\geq 0$
\begin{equation}
|\nabla^k(\phi_{t, \II_-}-\eta_1)|=O(e^{-\delta_e T}),
\end{equation}
where the derivative and norm are taken with respect to $\omega(t)$, and $\delta_e$ is defined as in Proposition \ref{p:CY-error-small}.
\end{proposition}
\begin{proof}
The proof is similar to the previous Proposition. One works in a fixed $U_\beta^1$, and then we use the asymptotics of $\phi_-$ (c.f. \eqref{e:asymptotics of phi-}) and the relation between $t$ and $T$ (c.f. \eqref{e:t T relation}). We omit the details.  
\end{proof}

\

{\bf Region $\III_\pm$.} Again we only consider the Region $\III_-$. The discussion here is very similar to the case of Region $\II_-$ so the computations are sketchy. We define
\begin{equation}
\omega(t)=dd^c \phi_{t, \III_-}, 
\end{equation}
where 
\begin{equation}
\phi_{t, \III_-}(\bx)=\sum\chi_\beta^1(\bx)\cdot \phi_{1}\circ \pi_\beta^1(\bx).
\end{equation}

\begin{proposition}
In the intersection ${\III_-}\cap  {U_{\beta, t}^1}$, we have for all $k\geq 0$
\begin{equation}
|\nabla^k(\phi_{t, \III_-}\circ (\pi_{\beta}^{1})^{-1}-\phi_1)|=\underline\epsilon_{T^2},
\end{equation}
where derivative is taken with respect to the metric $\omega_{TY, 1}$. 
\end{proposition}
The proof is very similar to that of Proposition \ref{p:region II- error on overlapping region}, except that one compares with the cylindrical metric and uses Lemma \ref{l: TY cylindrical compare}. We omit the details.

To define the weight, we also define the function $L_t$ by setting 
\begin{equation}
L_t(\bx)=T^{\frac{n}{2}-1}(nk_-)^{\frac{1}{n}}(-\log |f_2(\bx)|)^{\frac{1}{n}}
\end{equation}
and correspondingly the weight $\rho_t$ using \eqref{e:definition of weights}.

Similar to the case of Region $\II_-$, we have the holomorphic volume form 
\begin{equation}
\Gamma_t=(1+H)(\pi_\beta^1)^*\Omega_0
\end{equation}
where $H$ is a holomorphic function in $v_2, \zeta_3, w_2, \cdots, w_{n-1}$ and is of order $\underline\epsilon_{T^2}$. Therefore, we obtain 
\begin{equation}
\label{e:error in region III-}
\|\Errt\|_{C^{0, \alpha}_{\delta, \nu+2, \mu}({\III_-}\cap \widehat{X}_t)}=\underline \epsilon_{T^2}.
\end{equation}

Region $\III_-$ has two ends. One end intersects Region $\IV-$. 
\begin{proposition}
\label{p:potential difference III- IV-}
On $\III_-\cap \IV_-$, we have for all $k\geq 0$
\begin{equation}
|\nabla^k(\phi_{t, \III_-}-\phi_1)|=\underline\epsilon_{T^2},
\end{equation}
where the derivative and norm are taken with respect to $\omega(t)$. 
\end{proposition}
This is fairly easy to see, by working in a fixed $U^1_\beta$. 

\

The other end is close to the Region $\II_-$.

\begin{proposition}
\label{p:potential difference III- II-}
On $\III_-\cap \II_-$ we have for all $k\geq 0$
\begin{equation}
|\nabla^k(\phi_{t, \III_-}-\eta_1)|=O(e^{-\delta_{Z_1} T}), 
\end{equation}
where the derivative and norm are taken with respect to $\omega(t)$, and $\delta_{Z_1}$ is the constant in Proposition \ref{t:hein} applied to $Z_1$. 
\end{proposition}
To see this we only need to work in a fixed $U_\beta^1$ and use the asymptotics of the Tian-Yau metric $\tilde\omega_{TY, 1}$.

Now by Proposition \ref{p:potential difference II- III-} and \ref{p:potential difference III- II-}, we can choose a cut-off function to glue together $\phi_{t, \II_-}$ and $\phi_{t, \III_-}$. Similarly we may also  glue the corresponding weight function $\rho_t(\bx)$. Here we need to use \eqref{e:l function on II-}, the fact that 
\begin{equation}
-\log |f_2|=-\log \frac{|s_1|}{|s_3|}-d_1\log |t|,
\end{equation}
and the relation between $|t|$ and $T$  \eqref{e:t T relation}.

We choose a cut-off function to glue $\phi_{t, \III_-}$ and $\phi_{t, \IV_-}$, and also glue the corresponding weight function $\rho_t(\bx)$. 
Similarly, we can define the metrics $\omega(t)$ on $\II_+, \III_+, \IV_+$ and glue them together in the overlapping regions, and we also glue the weight functions.

\

To sum up, we have constructed a family of $C^{1, \alpha}$ K\"ahler metrics $\omega(t)$ on $\widehat{X}_t$ for $|t|$ small such that in the above defined weighted norm
$\|\Errt\|_{C^{0, \alpha}_{\delta, \nu+2, \mu}(\widehat{X}_t)}=O(T^{\nu+\alpha})$.

\begin{remark}
By the above gluing construction, $\omega(t)$ lies in the cohomology class $T^{\frac{2}{n}}\cdot 2\pi c_1(\mathcal O(1)|_{\widehat{X}_t})$. Hence we get the volume  
\begin{equation}
\int_{\widehat{X}_t}\frac{\omega(t)^n}{n!}=C\cdot T^2\sim (-\log |t|)^{-1}.
\end{equation}
The above error estimate in particular gives 
\begin{equation}
\int_{\widehat{X}_t}\Gamma_t\wedge\bar\Gamma_t\sim T^2\sim (-\log |t|)^{-1}. 
\end{equation}
\end{remark}
For our analysis in the next subsection we define the normalized holomorphic volume form as
\begin{equation}
\Omega(t)\equiv (\frac{2^n\int_{\widehat{X}_t}\omega(t)^n}{(\sq)^{n^2} \int_{\widehat{X}_t}\Gamma_t\wedge\bar\Gamma_t})^{\frac{1}{2}} \cdot \Gamma_t.
\end{equation}
Abusing notation we define $\Errt$ by
\begin{equation}
\frac{(\sq)^{n^2}}{2^n}\Omega(t)\wedge\bar\Omega(t)=(1+\Errt)\frac{\omega(t)^n}{n!}, 
\end{equation}
where 
$
\int_{\widehat{X}_t}\Errt\cdot\omega(t)^n=0
$
and 
\begin{equation}
\|\Errt\|_{C^{0, \alpha}_{\delta, \nu+2, \mu}(\widehat{X}_t)}=O(T^{\nu+\alpha}).\label{e:global-error-estimate}
\end{equation}

\subsection{Global weighted analysis on $\widehat{X}_t$ and the proof of the main theorem}
\label{ss:global analysis}

We are now in a position to set up the global weighted analysis on the glued manifold.

To begin with,
let  $(\omega(t),\Omega(t))$
 be the  $C^{1,\alpha}$-K\"ahler structure on $\widehat{X}_t$ constructed in Section \ref{ss:glued-metrics}. So we define the Banach
\begin{align}
\fS_1 & \equiv \Big\{\sq\p\bp\phi\in\Omega^{1,1}(\widehat{X}_t)\Big| \phi\in C^{2, \alpha}(\widehat{X}_t)\Big\},
\nonumber\\
\fS_2 & \equiv \Big\{f\in C^{0,\alpha}(\widehat{X}_t)\Big| \int_{\widehat{X}_t} f\cdot \frac{\omega(t)^n}{n!}  = 0\Big\},
\end{align}
which are equipped with the weighted norms defined similar to \eqref{e:norm-of-S1-space} and \eqref{e:norm-of-S2-space}, with parameters $\delta, \nu, \mu, \alpha$ given in  
Section \ref{sss:parameters-fixed}

For $|t|\ll1$, we want to solve
\begin{equation}
\label{e:CY-eq}\frac{1}{n!} (\omega(t)+\sqrt{-1}\p\bp\phi)^n = (\sq)^{n^2}2^{-n}\cdot \Omega(t)\wedge\overline{\Omega}(t).\end{equation}
Let $\mathscr{F}: \fS_1\rightarrow \fS_2$ be defined by
\begin{equation}
\mathscr{F}(\sq\p\bp\phi)\cdot\omega(t)^n\equiv (\omega(t)+\sq\p\bp\phi)^n-\omega(t)^n(1-\Errt).
\end{equation}
Then \eqref{e:CY-eq} is equivalent to $
\mathscr{F}(\sq\p\bp\phi) = 0.\label{e:converted-CY-eq}
$

Now we write 
\begin{equation}
\mathscr{F}(v) -  \mathscr{F}(0) = \mathscr{L}(v) + \mathscr{N}(v),
\end{equation}
for any $v\in \fS_1$, where
$\mathscr{L}(\sq\p\bp\phi)  = \Delta \phi$ is the linearization of $\mathscr{F}$ and
\begin{align}
\mathscr{N}(\sq\p\bp\phi)\cdot \omega(t)^n & =(\omega(t)+\sqrt{-1}\p\bp\phi)^n-\omega(t)^n - n \omega(t)^{n-1}\wedge \sq\p\bp\phi.\label{e:nonlinear-term}
\end{align}

The proof of the following is identical to Proposition \ref{p:nonlinear-neck}.

\begin{proposition}
[Nonlinear error estimate]\label{l:nonlinear-error-estimate} 
There exists a constant $C_N>0$   independent of 
$0<|t|\ll 1$ such that for all 
$\vr\in (0,\frac{1}{2})$
 and 
\begin{equation}\sqrt{-1}\p\bp\phi_2\in \overline{B_{\vr}(\bo)} \subset \fS_1, \quad \sqrt{-1}\p\bp\phi_2\in \overline{B_{\vr}(\bo)}\subset\fS_1,\end{equation} we have the pointwise estimate
\begin{align}  \|\mathscr{N}(\sqrt{-1}\p\bp\phi_1)-\mathscr{N}(\sqrt{-1}\p\bp\phi_2)\|_{\fS_2} 
\leq   C_N \cdot \vr \cdot  \|\sqrt{-1}\p\bp(\phi_1-\phi_2)\|_{\fS_1}.
\end{align}

\end{proposition}

The following is an analogue of Proposition \ref{p:neck-weighted-schauder}.

\begin{proposition}[Weighted Schauder estimate, the global version]  \label{p:global-weighted-schauder}
There exists a uniform constant $C>0$ (independent of $0<|t|\ll1 $) such that for every
$u\in C^{2,\alpha}(\widehat{X}_t, \omega(t))$
\begin{align}
\| u\|_{C_{\delta,\nu,\mu}^{2,\alpha}(\widehat{X}_t)}	\leq C \Big( \|\Delta u\|_{C_{\delta,\nu+2,\mu}^{0,\alpha}(\widehat{X}_t)} + \| u\|_{C_{\delta,\nu,\mu}^0(\widehat{X}_t)}	\Big).
\end{align}
 
\end{proposition}

The local version of Proposition
\ref{p:global-weighted-schauder} is given in Proposition \ref{p:local-weighted-schauder} and they share very similar proof. From the construction of the metric $\omega(t)$ in Section \ref{ss:glued-metrics} we see that the rescaled limit geometries as $|t|\to 0$ are the same as those of the neck $\M_T$ as $T\to\infty$ studied in Section \ref{ss:regularity-scales}, except that the two incomplete Calabi model space limits are replaced by the two complete Tian-Yau metrics on the ends. Another difference to note is that due to the perturbation of the neck the metric $\omega(t)$ now has only $C^{1,\alpha}$ regularity so that we can only obtain the $C^{2,\alpha}$ estimate here. We omit the details of the proof.

\begin{proposition}[Global injectivity estimates]  \label{p:global-injectivity-estimate}
There exists a uniform constant $C>0$ (independent of $0<|t|\ll1$) such that for every
$u\in C^{2,\alpha}(\widehat{X}_t, \omega(t))$,
\begin{align}
\|\nabla u\|_{C_{\delta,\nu+1,\mu}^{0}(\widehat{X}_t)} + \|\nabla^2 u\|_{C_{\delta,\nu+2,\mu}^{0}(\widehat{X}_t)}	+[u]_{C_{\delta,\nu,\mu}^{2,\alpha}(\widehat{X}_t)}\leq C \cdot \|\Delta u\|_{C_{\delta,\nu+2,\mu}^{0,\alpha}(\widehat{X}_t)}. \label{e:global-injectivity-estimate}
\end{align}
\end{proposition}

The proof is very similar to the proof of Proposition \ref{p:neck-uniform-injectivity},  using a contradiction argument and applying various Liouville theorems. We omit the details and only mention two points. The first point is that from our construction of $\omega(t)$ on $\widehat{X}_t$, if we rescale around points in Region $\IV_\pm$, then we obtain the Tian-Yau spaces as limits, instead of the incomplete Calabi model spaces, and we need to use Theorem \ref{t:Liouville on Tian-Yau}. The second point is that the other rescaled limits will be exactly the same as considered in the proof of Proposition \ref{p:neck-uniform-injectivity}, and this follows from the fact that by construction our metric $\omega(t)$ away from the region $\IV_\pm$ is essentially a small perturbation of the neck region $(\M_T, \omega_T)$.

As $|t|>0$ is sufficiently small, the existence of solution of the Calabi-Yau equation \eqref{e:CY-eq} is the same as the proof of Theorem \ref{t:neck-CY-metric}. In fact, combining Proposition \ref{p:global-injectivity-estimate} with the error estimate \eqref{e:global-error-estimate} and nonlinear estimate in Proposition \ref{l:nonlinear-error-estimate}, one can apply the implicit function theorem (Lemma \ref{l:implicit-function}), which gives a solution  $\phi(t)\in \fS_1$ of  \eqref{e:CY-eq}. Since $\omega(t)$ lies in the cohomology class $T^{\frac{2}{n}}\cdot 2\pi c_1(\mathcal O(1)|_{\widehat{X}_t})$, by the well-known uniqueness of the solution to the Calabi-Yau equation, the rescaled metric $T^{-\frac{2}{n}}\cdot(\omega(t)+\sq \p\bp \phi(t))$ must agree with the Calabi-Yau metric $\omega_{CY, t^{n+2}}$ on $X_{t^{n+2}}\simeq \widehat X_{t}$ in the Introduction. The geometric statements in Theorem \ref{t:main-theorem} then follow from similar arguments as in Section \ref{ss:renormalized-measure}. Notice that  $\diam_{\omega(t)}(\widehat{X}_t)$ is of order $T^{\frac{n+1}{n}}$, and the relation between $T$ and $|t|$ is given by \eqref{eqn6667}. Then we have that 
\begin{align}
C^{-1}(\log|t|^{-1})^{\frac{1}{2}} \leq \diam_{\omega_{TY,t^{n+2}}}(X_{t^{n+2}})\leq C\cdot(\log|t|^{-1})^{\frac{1}{2}},
\end{align}
 where $C>0$ is independent of $T$. We omit other details.

\section{Extensions and conjectures}
\label{s:discussions}

In this section we discuss some possible extensions and questions related to our results in this paper.

\subsection{More general situation}
\label{ss:general-situations}
As discussed in Section \ref{ss:dimension reduction}, our motivation is to more general degenerations of Calabi-Yau manifolds. During the preparation of this paper in the Fall of 2018, we also made some preliminary progress towards understanding the case of maximal degenerations (which is related to the SYZ Conjecture in mirror symmetry), based on similar ideas to that of Section \ref{s:torus-symmetries} and \ref{s:neck}, and partly motivated by \cite{Morr}. We hoped in a future paper to work out the details of constructing local models generalizing the Ooguri-Vafa metric to higher dimensions. In January 2019,  we received a preprint from Yang Li (\cite{Li}) who, partly motivated by \cite{HSVZ}, had essentially achieved most of what we were planning to do (in complex dimension three). Thus we decided not to expand in this direction beyond what we have written at the time we learned about \cite{Li}. On the other hand, we still present the original brief discussions here (so the arguments are rather sketchy and there will be NO theorems). We hope this may still be of some interest to the readers, since it seems to shed a slightly different light from \cite{Li}.

We start with a lemma on Green's function on certain non-compact spaces. 
\begin{lemma}
 \label{l:green-function-asymp}	Let $(X^{m+n},g)\equiv (\dR^m\times K^n, g_{\dR^m}\oplus h)$ be a Riemannian product of a Euclidean space $(\dR^m, g_{\dR^m})$ and a compact Riemannian manifold $(K^n, h)$. For any point $p=(p_1, p_2)\in X^{m+n}$, there exists a Green's function $G_p$  on $X$ such that
\begin{enumerate}
\item $-\Delta_g G_p = 2\pi\delta_p$.
\item There are constants $\epsilon>0$, $R>0$ and $C>0$, independent of $p$,  such that 
\begin{equation}
	|G_p(x) - \Phi_{m, p_1}(x_1)| \leq C \cdot e^{-\epsilon \cdot |x_1-p_1|}
\end{equation}
for any $x=(x_1, x_2)\in X^{m+n}\setminus B_R(p)$, where $\Phi_m:\dR^m\setminus\{0^m\}\to\dR$ is the standard Green's function given by
\begin{align}
\Phi_{m,p_1}(x_1) 
\equiv 
\begin{cases}
\frac{2\pi}{(m-2)\cdot \Area(\p B_1(0^m))}\cdot |x_1 - p_1|^{2-m},& m \geq 3
\\
-\log |x_1 - p_1|, & m = 2,
\\
-2\pi |x_1-p_1|, & m=1.
\end{cases}
\end{align}
 \end{enumerate}
\end{lemma}

\begin{proof}
The proof is by separation of variables, and is similar to Proposition \ref{p:existence-Greens-current}. So we will not provide all the details, except pointing out one key point. For simplicity of notation we may assume $p_1=0$. After separation of variables we need to solve a PDE of the form on $\dR^m$
\begin{equation} 
-	\Delta_{\dR^m}	u_{\lambda}  + \lambda \cdot u_{\lambda} =2\pi\cdot \delta_{0^m}, \label{e:delta-Rm}
\end{equation}
where $\lambda$ is non-negative. When $\lambda=0$, a solution is given by the Green's function of $-\Delta_{\dR^m}$, so we only deal with the case $\lambda>0$.  When $m=1$, this is the equation \eqref{eqn3333}. When $m\geq 2$, we look for a radial solution $u_{\lambda}=u_{\lambda}(r)$, then \eqref{e:delta-Rm} reduces to an ODE
\begin{equation}
-u_{\lambda}''(r) - \frac{m-1}{r}\cdot u_{\lambda}'(r)  +  \lambda \cdot u_{\lambda}(r) = 0, r\in (0, \infty)\end{equation}
We make the transformation 
$f(r) \equiv u_{\lambda}(r)\cdot r^{-\alpha}$,
where $\alpha$ is to be determined. 
Then it follows that 
\begin{equation}
r^2f''(r) + (2\alpha+m-1)\cdot r\cdot  f'(r)	 + \Big(\alpha(\alpha+m-2)-\lambda \cdot  r^2\Big)\cdot  f(r) = 0.
\end{equation}
Now let $2\alpha+m-1=1$, i.e., $\alpha=\frac{2-m}{2}$, and let $\sqrt{\lambda}\cdot r =s$, then we get the modified Bessel equation \begin{equation}
	s^2f''(s) + s f'(s)  - (\alpha^2 + s^2) f(s) = 0. 
\end{equation}
Then we get a solution $u(r)=K_\alpha(\sqrt{\lambda} r)\cdot r^{\alpha}$, where $K_{\alpha}$ is the modified Bessel function.
So it follows that \begin{align}
u(r)  \sim 
\begin{cases}
	\vf_{\lambda}(p)\cdot r^{2-m}, & m\geq 3,
\\
\vf_{\lambda}(p)\cdot \log r , & m=2,
\end{cases}
\end{align}
as $r\to0$.  
In particular $u$ satisfies the distribution equation \eqref{e:delta-Rm}. Then we can define $G_p$ using a formal expansion, and the convergence and the asymptotic behavior follow from the uniform estimates on $K_\alpha(\sqrt{\lambda} r)$ for $r\geq 1$ (see Proposition 3.3 in \cite{SZ-Liouville}). We omit the detailed proof. \end{proof}

\begin{remark}
Applying this to $\R\times \dT^2$ and $\R^2\times S^1$, we obtain an alternative treatment to the constructions in \cite{HSVZ} (Theorem 2.6) and \cite{GW} (Lemma 3.1).
\end{remark}

We are interested in studying the Green's currents in the situation of Section \ref{ss:global-existence} with $D$ replaced by the non-compact Calabi-Yau manifold $(\C^*)^n$, and with $H$ replaced by a smooth algebraic hypersurface in $D$ defined by a Laurent polynomial $F$. 
Here $D$ is endowed with the standard flat K\"ahler metric
\begin{equation}\omega_D=\sum_{j=1}^n \frac{1}{2}\sqrt{-1}\p\log w_j \wedge \bp \log w_j,\end{equation}
where  $\{w_1, \cdots, w_n\}$ are standard holomorphic coordinates on $(\C^*)^n$. 

Denote $\xi_j=-\log w_j=u_j+\sq v_j$, which gives an identification $(\C^*)^n$ with $\R^n\times (S^1)^n$ equipped with the standard flat product metric. 
Let $\pi: (\C^*)^n\rightarrow \R^n$ be the projection map. The \emph{amoeba} $\mathcal A(F)$ of $F$ is by definition the image $\pi(H)$.

We want to solve 
$\Delta G_P=2\pi\cdot  \delta_P$. 
In terms of the coordinates $\{\xi_j\}$, we can view $\delta_P$ as a matrix of distributions by the decomposition 
\begin{equation}\delta_P=\sum_{\alpha, \beta} f_{\alpha\beta}\widehat{\delta}_P\frac{\sq}{2}d\xi_\alpha \wedge d\bar\xi_\beta \wedge dz, \end{equation}
where $\widehat{\delta}_P$ is a $2n$-current given by setting
$(\widehat\delta_P,  \phi)= \int_P \phi\dvol_P.$
Then by definition it is not difficult to see that at every point on $P$,
\begin{equation}
\label{eqn812}
f_{\alpha\beta} d\xi_\alpha \wedge d\xi_\beta=\frac{\p F\wedge\bp F}{|\p F|^2}.
\end{equation}
If we decompose 
\begin{equation}G_P=\sum h_{\alpha\beta} \frac{\sq}{2}d\xi_\alpha \wedge d\xi_\beta \wedge dz.\end{equation}
Then we need to solve a matrix of distributional equations
$\Delta h_{\alpha\beta}=f_{\alpha\beta}\widehat\delta_P.$
Let us write
\begin{equation}
\widehat\delta_P=\int_{P} \delta_{y} \dvol_P(y).
\end{equation}
Then one can write down a solution in the form
\begin{equation}h_{\alpha\beta}(x) \equiv \int_P(G_y(x)-\mathcal G(y))f_{\alpha\beta}(y)\dvol_P(y),
\end{equation}
where $G_y(x)$ is the Green's function on $Q$ constructed in Lemma \ref{l:green-function-asymp}, and $\mathcal G(y)$ is a renormalization function to make the integral converge. For example, we can take 
\begin{equation}
\mathcal G(y)=\frac{c_n}{|\pi(y)|^{n-2}+1}.
\end{equation}

Now we consider an illustrating example when $n=2$, and 
\begin{equation}
F(w_1, w_2)\equiv w_1+w_2+1. 
\end{equation}
The amoeba $\mathcal A(F)$ is a well-known shape on $\R^2$ with three branches at infinity. Moreover, it is not difficult to show by direct calculation that $\mathcal A(F)$ converges exponentially fast (in the Hausdorff sense) to its \emph{tropicalization}, $T(F)$ which is given by the union of three half lines $P_1, P_2, P_3$ emanating from $0$ in $\R^2$, along the directions of $e_1, e_2, -e_1-e_2$. 
In this case, one also expects that the Green's current $G_P$, viewed as a matrix $(h_{\alpha\beta})$, is asymptotic to the matrix of Green's functions defined using $T(F)$ on $\R^2$. This asymptotics should hold in suitable regions away from $T(F)$. 

The point is that we should remember more information on $T(F)$ than simply a subspace in $\R^2$. Notice each $P_i$ is a straight half line and it has a unit normal $n_i$ in $\R^2$ (well-defined up to sign), and $n_i$ naturally arises if one notices \eqref{eqn812}. Then there is a well-defined matrix valued distribution  
$\delta_{T(F)}\equiv \sum\limits_{i=1}^3 \widehat\delta_{P_i} \cdot n_i\otimes n_i$ on $\dR^3=\dR^2\times \dR$,
where we view $P_i\subset \dR^3$ as $P_i\times \{0\}$. 
Then we can solve for a matrix value Green's function  $G_{T(F)}$ for $T(F)$ in $\dR^3$ such that 
$-\Delta_{\R^3} G_{T(F)}=2\pi \delta_{T(F)}$.
 For this purpose we first solve the Green's function for $P_1$ in $\dR^3$. Again this is easy to write down explicitly as 
\begin{equation}G_{P_1}(x)=\int_{0}^\infty \Big(\frac{1}{\sqrt{(x_1-t)^2+x_2^2+x_3^2}}-\frac{1}{t+1}\Big)dt=-\log (\sqrt{x_1^2+x_2^2+x_3^2}-x_1)+\log 2.\end{equation}
This has interesting asymptotics. Let us write $r^2=x_1^2+x_2^2+x_3^2$ and $u=(x_2, x_3)$. Then 
\begin{align}G_{P_1}(x) 
\sim \begin{cases}  -\log r+\log 2 +\frac{x_1}{r}-\frac{x_1^2}{r^2}+\cdots, & |x_1|\leq C |u|,
\\
 -2\log |u|+\log |x_1|+O(|u|^2x_1^{-2}), & |x_1|\gg|u|. 
\end{cases}
\end{align}
Now the Green's function for $T(F)$ can be written down as a matrix
\begin{equation}G_{T(F)}= \left[ {\begin{array}{cc}
  G_{P_2}+\frac{1}{2}G_{P_3} & -\frac{1}{2}G_{P_3} \\
   -\frac{1}{2}G_{P_3} & G_{P_1}+\frac{1}{2}G_{P_3} \\
  \end{array} } \right].
\end{equation}
Away from the three direction, the  asymptotics as $r\rightarrow\infty$  is given by 
\begin{equation}
 -\left[ {\begin{array}{cc}  \frac{3}{2}\log r & -\frac{1}{2}\log r \\
   -\frac{1}{2}\log r & \frac{3}{2}\log r \\
  \end{array} } \right].
  \end{equation}
  
 Now using the Green's current $G_P$ and its asymptotics at infinity as described above, one can construct an $S^1$-invariant incomplete three dimensional K\"ahler metrics as in Section \ref{ss:kaehler-structures}. Notice as in  \ref{ss:kaehler-structures} there are various parameters. First one can change the flat metric on $\dR^2$. Also in the equation 
 \begin{equation}
 \p_z^2\tilde\omega+d_Dd_D^ch=0, 
 \end{equation}
one is free to add a function of $z$ to $h$, and add a closed $(1,1)$-form on $D=(\C^*)^2$ to $\tilde\omega$. For appropriate choices of parameters one can make this K\"ahler metric approximately Calabi-Yau, and then the goal is to use weighted analysis to perturb to a family of genuine (incomplete) Calabi-Yau metrics. In appropriate scales, these metrics should collapse to a limit which is given as a domain in $\dR^3$. One unsatisfactory point from our point of view is that comparing with the general expectation in SYZ metric collapsing conjecture, these incomplete metrics live on a too small region, since here the collapsing limit is flat whereas in general we should get a limit which is singular along the union of $P_i$'s. In other words, what one constructs here is only an \emph{infinitesimal} model for the collapsing.

In complex dimension three, one can also consider $\dT^2$-invariant Calabi-Yau metrics. As discussed in Section \ref{ss:higher rank torus}, the corresponding dimension reduced equation has slightly different form and the linearized equation in the case when there are stabilizers also motivates us study certain Green's currents. 

 Again we consider the model case $Q=\R^2\times \C^*$ is the quotient space and over $P=P_1\cup P_2\cup P_3\subset \R^2\times\{1\}$ we have stabilizers. 
In this case we are interested in a matrix valued Dirac current
\begin{equation}\delta_P=\sum\limits_{i=1}^3 \widehat\delta_{P_i}\cdot  n_i\otimes n_i,\end{equation}
where $P_i$ is naturally viewed as a submanifold in $\R^2\times \C^*$,
and the corresponding matrix valued Green's function $G_P$ which satisfies
$\Delta G_P=2\pi\delta_P$.

In large scale this is modeled by the corresponding current in $\R^3=\R^2\times \R$, and this has been discussed in the above. Near the vertex of $P$ one can consider the model $\R^2\times \C$, and find the corresponding Green's function for $P\subset \R^2\times 0$. This is similar to the calculation above. For example, one gets 
\begin{equation}G_{P_1}(x)=\int_{0}^\infty \frac{1}{(x_1-t)^2+x_2^2+x_3^2+x_4^2}dt=\frac{1}{v}(\frac{\pi}{2}+\tan^{-1}\frac{x_1}{v}),\end{equation}
where $u=x_3+\sq x_4$ is the coordinate on $\C$, and 
$v^2=|u|^2+x_2^2$. 
We define 
\begin{align}(W_{ij}) \equiv  \left[ {\begin{array}{cc}
  G_{P_2}+\frac12G_{P_3} & -\frac12G_{P_3} \\
   -\frac12G_{P_3} & G_{P_1}+\frac12G_{P_3} \\
  \end{array} } \right] \quad \text{and}
\quad 
\tilde\omega \equiv \frac{\sq}{2}\Tr(W_{ij}) dw\wedge d\bar w.
\end{align}
Then one can check the equation \eqref{e:higher Tk} is satisfied, and one obtains away from the singular locus a $\dT^2$-invariant K\"ahler metric. 

Naively one expects to compactify this metric along singular locus. We compare this with the standard local holomorphic model, which is the standard flat holomorphic structure $(\omega_{\C^3}, \Omega_{\C^3})$ on $\C^3$ under the natural $\dT^2$-action 
\begin{equation}
(e^{\sq \theta_1}, e^{\sq \theta_2})\cdot (z_1, z_2, z_3)=(e^{\sq(\theta_1+\theta_2)}z_1, e^{-\sq \theta_1}z_2, e^{-\sq\theta_2}z_3). 
\end{equation}
The corresponding quotient map is given by
\begin{equation}\mathcal Q: \C^3\rightarrow \R^2\oplus \C,\quad  (z_1, z_2, z_3)\mapsto (\frac{1}{2}(|z_2|^2-|z_1|^2), \frac{1}{2}(|z_3|^2-|z_1|^2), z_1z_2z_3).\end{equation}
Also one can compute 
\begin{align}W_{ij}&=
\frac{1}{|z_1|^2|z_2|^2+|z_3|^2|z_1|^2+|z_2|^2|z_3|^2}  \left[ {\begin{array}{cc}
   |z_1|^2+|z_3|^2 & -|z_1|^2 \\
   -|z_1|^2 & |z_1|^2+|z_2|^2 \\
  \end{array} } \right]
\\
\tilde\omega &=\frac{1}{|z_1|^2|z_2|^2+|z_3|^2|z_1|^2+|z_2|^2|z_3|^2}\cdot \frac{\sq}{2} {dw}\wedge d\bar w.\end{align}
So comparing with the previous formula they do not naturally match. This suggests  that we might need to do something different near the vertex.

Now if we take the above formula of Green's current, but work instead on $\R^2\times \C$, then one can see the above matrix actually has strictly positive lower bound  at infinity. This makes us suspect the existence of a complete Calabi-Yau metric on $\C^3$ which is approximately the above ansatz at infinity. One approach is by using this ansatz as background metric at infinity and solve the Calabi-Yau equation as in \cite{TY}. This should be similar to the result of Yang Li constructing a complete Calabi-Yau metric $\C^3$ with infinity tangent cone $\C^2/\Z_2\times \C$. If such a metric can be constructed, then it should have a  $\dT^2$-symmetry and at infinity has $r^4$ volume growth and the tangent cone at infinity is $\R^2\oplus\C$ with locus of the singular fibration given by the $Y$-vertex.  The situation is analogous to the Taub-NUT space fibered over $\R\oplus \C$. The difference is that here we need to have discriminant locus essentially due to topological reasons. 

The existence of such a complete Calabi-Yau metric on $\C^3$ also resolves the above concern regarding the bad singularity behavior of the ansatz metric near the vertex; see \cite{Li}.

\subsection{Further remarks} 
\label{ss:a few remarks}
We list several further remarks. 

\

{\bf (1)} From the proof of Theorem \ref{t:main-theorem}, it follows that similar results hold in the following more general situation. We leave it for the readers to check the details.

\begin{itemize}
\item $p:\mathcal X\rightarrow \Delta$ is a proper holomorphic map from an $(n+1)$-dimensional normal complex analytic variety onto a disc $\Delta$ in $\C$. 
\item For $t\neq 0$, $X_t\equiv p^{-1}(t)$ is a smooth $n$ dimensional compact complex manifold.
\item $X_0\equiv p^{-1}(0)$ is a union of two smooth $n$ dimensional Fano manifolds $Y_1$ and $Y_2$, and $Y_1\cap Y_2$ is a smooth $(n-1)$-dimensional Calabi-Yau manifold $D$. 
\item $\mathcal L$ is a relatively ample holomorphic line bundle on $\mathcal X$.
\item Two positive integers $d_1, d_2$, and we denote $k=d_1+d_2$.
\item Holomorphic sections $f_1$, $f_2$, $f$ of $\mathcal L^{d_1}, \mathcal L^{d_2}, \mathcal L^{d_1+d_2}$ respectively satisfying 
$
f_1f_2+tf=0.
$
\item $Y_1\cap Y_2$ is  a smooth $(n-1)$-dimensional Calabi-Yau manifold $D$. 
\item $\mathcal X$ is singular along a smooth divisor $H\subset D$ given by $\{f_1=f_2=f=t=0\}$   and transverse to $H$ the singularity is modeled on $\{x_1x_2+tx_3=0\}$.
\item There is a holomorphic volume form on the smooth locus of $\mathcal X$.
\end{itemize}
 More generally, one may also try to understand the situation when the central fiber is given by a chain of irreducible components. This would require constructing a neck region using a Green's current on $D\times \R$ with singularities along the union of $H_j\times \{z_j\}$ for finitely many $z_j$. 

\

{\bf (2)} In connection with algebro-geometric study of degenerations of Calabi-Yau manifolds, Theorem \ref{t:main-theorem} shows that the normalized Gromov-Hausdorff limit in our setting is \emph{topologically} the same as the \emph{essential skeleton} of the degeneration $\mathcal X$. In the other extremal case, namely, the case of large complex structure limit of Calabi-Yau manifolds, it is a folklore conjecture (see Gross-Wilson \cite{GW} and Kontsevich-Soibelman \cite{KonSo, KS-affine}) that the normalized Gromov-Hausdorff limit is \emph{topologically} the same as the essential skeleton of the degeneration. Motivated by our work in this paper, it is then natural to expect the following more general conjecture: 
\begin{conjecture}\label{cj:generalized-SYZ}
Given a normal flat polarized degenerating family of Calabi-Yau manifolds $(\mathcal X, \mathcal L)\rightarrow \Delta$. Let $\omega_{CY, t}$ be the Calabi-Yau metric on $X_t$ in the class $2\pi c_1(\mathcal L|_{X_t})$. Let $\Delta(\mathcal X)$ be the essential skeleton of the degeneration with $\dim(\Delta(\mathcal X))=d\in\dZ_+$. Let us denote \begin{align}
\tilde\omega_{CY,t} \equiv  (\diam_{\omega_{CY, t}}(X_t))^{-2}\cdot \omega_{CY, t},\quad d\nmv_t\equiv (\Vol_{\omega_{CY, t}}(X_t))^{-1}\cdot\dvol_{\omega_{CY, t}}. 
\end{align}
Then $(X_t, \tilde\omega_{CY, t}, d\underline{\nu}_t)$ converges in the measured Gromov-Hausdorff sense, to a compact measured  metric space $(B,d_{B},d\underline{\nu}_0)$ such that $B$ is homeomorphic to $\Delta(\mathcal X)$, and the dimension of $B$ (in the sense of Colding-Naber \cite{Colding-Naber}) is equal to $d$. 
\end{conjecture}
We also expect more refined relationship between the metric collapsing and algebraic geometry:
\begin{itemize}
\item It is interesting to understand the algebro-geometric meaning of the normalized limit measure in Theorem \ref{t:main-theorem}, see \cite{BJ} for related algebro-geometric work. Furthermore, it seems plausible to view the limit metric measure space as a solution to certain non-Archimedean Monge-Amp\`ere equation.   In the case $n=2$, there is also a plausible connection with the compactification of moduli space of hyperk\"ahler metrics on K3 manifolds (see \cite{OO}). 
\item It is interesting to understand the algebro-geometric meaning of the rescaled limits. In this paper, the Tian-Yau spaces naturally fit into some components $Y_1, Y_2$ of an algebraic degeneration, and the collapsing limit $D\times \R$ naturally maps to the intersecting divisor $Y_1\cap Y_2$. But the Taub-NUT limit spaces are not easily directly in the degeneration. This picture resembles the familiar relationship between the nodal degeneration of higher genus algebraic curves and degeneration of hyperbolic metrics. Certainly one expects similar phenomenon in more general settings. 
\end{itemize}

\

{\bf (3)}  As is mentioned in the Introduction, it remains an interesting question to directly glue two Tian-Yau metrics with the same divisor $D$, without a priori assuming the existence of the complex family $\mathcal X$. As mentioned in the Introduction, in the case $n=2$ this was done in \cite{HSVZ} using $SU(2)$-structures, and in the case $n=3$ it is possible to use deformations  of $SU(3)$-structures. This would require certain analysis (in particular Liouville theorem) on forms instead of functions. We leave this for future study.

\

{\bf (4)}
 There is a different class of Tian-Yau spaces, constructed on the complement of a smooth anti-canonical divisor in a projective manifold with trivial normal bundle. In particular the ambient manifold can not be Fano. These spaces have different asymptotics at infinity from the ones we considered in this paper. Namely, they are \emph{asymptotically cylindrical}. Given a smooth Fano manifold $Y$ and a pencil of anti-canonical divisors with smooth base locus $B$, let $Y'$ be the blown-up of $Y$ along $B$, and let $D'$ be the proper transform of a smooth element $D$ in the pencil. Then there is such an asymptotically cylindrical Calabi-Yau metric on $Y\setminus B$ (in every K\"ahler class).  Asymptotically cylindrical Tian-Yau spaces have been important ingredients in the \emph{twisted connected sum} construction of examples of compact $G_2$ holonomy manifolds. It is interesting to see whether the ideas of this paper can be used to construct new examples of $G_2$ holonomy manifolds by gluing together a suitably twisted circle fibration over various pieces. 

\

{\bf (5)}
 Our main result  approximately reduces the understanding on the geometry of part of the Calabi-Yau manifolds $(X_t, \omega_{CY, t})$ (the neck region) for $|t|\ll1$ to the geometry of the Calabi-Yau metric on the one lower dimensional space $D$. One expects this can possibly lead to an inductive way to study the geometry of Calabi-Yau metrics in higher dimensions through iterated degenerations. Correspondingly, it is also interesting to relate the submanifold geometry of the neck to that of $D$. For example, suppose we have a special Lagrangian fibration on a region in $D$, can we construct special Lagrangian fibrations on the neck which are $S^1$-invariant?  At the two ends of the neck it is easy to see the pre-image of a special Lagrangian fibration under the projection map is approximately special Lagrangian. Near the singular fibers of the $S^1$-fibration the situation is more complicated and one expects certain singular perturbation techniques are needed. There are also similar discussions in \cite{Li} in the setting of Section \ref{ss:general-situations}.

\

{\bf (6)}
 Theorem \ref{t:main-theorem} can be viewed as understanding the first order expansion of the family of Calabi-Yau metrics on $\hX$ near $t=0$. One may ask whether it is possible to obtain a refined \emph{asymptotic expansion}.  In spirit, it is similar to the case of a family of hyperbolic metrics on nodal degeneration of Riemann surfaces (see \cite{MZ} by Melrose-Zhu), and it is very likely similar techniques will be useful here. We thank Dominic Joyce and Xuwen Zhu for conversations on this.

\bibliographystyle{amsalpha} 
\bibliography{References_SZ}

\end{document}